\documentclass[12pt,reqno]{amsart}
\usepackage{enumerate}
\usepackage{amsrefs}
\usepackage{amsfonts, amsmath, amssymb, amscd, amsthm, bm, cancel}
\usepackage{mathrsfs}
\usepackage{url}
\usepackage{graphicx}
\usepackage[
linktocpage=true,colorlinks,citecolor=magenta,linkcolor=blue,urlcolor=magenta]{hyperref}
\usepackage{appendix}
\usepackage{multicol}
\usepackage{comment}
\usepackage{tikz}
\usepackage{pgfplots}
\usetikzlibrary{intersections}
\usepackage[margin=1in]{geometry}
\theoremstyle{plain}
\newtheorem{thm}{Theorem}[section]
\newtheorem*{thmA}{Theorem A}
\newtheorem*{thmB}{Theorem B}
\newtheorem*{thmC}{Theorem C}
 \newtheorem*{thmD}{Theorem D}   
  \newtheorem*{thmE}{Theorem E}   
\newtheorem*{thmF}{Theorem F}  
\newtheorem*{thmG}{Theorem G}  
\newtheorem*{thmH}{Theorem H}  
\newtheorem*{thmI}{Theorem I}  
\newtheorem*{thmJ}{Theorem J} 

 \newtheorem{cor}{Corollary}[section]
  \newtheorem{con}[thm]{Conjecture}
  \newtheorem{lem}{Lemma}[section]
 \newtheorem{assump}[thm]{Assumption}
 \newtheorem{prop}{Proposition}[section]
 \theoremstyle{definition}
 \newtheorem{defn}[thm]{Definition}
  \newtheorem*{ack}{Acknowledgments}
  \newtheorem{prob}{Problem}[section]
  
 \theoremstyle{remark}
\newtheorem{rem}{Remark}[section]

 \numberwithin{equation}{section}
 \numberwithin{figure}{section}
 
\allowdisplaybreaks

\renewcommand{\(}{\left(}
\renewcommand{\)}{\right)}

\renewcommand{\a}{\alpha}
\renewcommand{\b}{\beta}
\newcommand{\g}{\varphi}

\newcommand{\s}{\sigma}

\newcommand{\mrm}{\mathrm}

\newcommand{\vol}{\mrm{vol}}
\newcommand{\Vol}{\mrm{Vol}}
\newcommand{\Ric}{\mrm{Ric}}

\newcommand{\divv}{\mrm{div}}
\newcommand{\metric}[2]{\ensuremath{\langle #1, #2\rangle}}  

\begin{document}
\title{ Hyperbolic $p$-sum and Horospherical $p$-Brunn-Minkowski theory in hyperbolic space}

\author{Haizhong Li}
\address{Department of Mathematical Sciences, Tsinghua University, Beijing 100084, P.R. China}
\email{\href{mailto:lihz@tsinghua.edu.cn}{lihz@tsinghua.edu.cn}}

\author{Botong Xu}
\address{Department of Mathematical Sciences, Tsinghua University, Beijing 100084, P.R. China}
\email{\href{mailto:xbt17@mails.tsinghua.edu.cn}{xbt17@mails.tsinghua.edu.cn}}


\begin{abstract}
	The classical Brunn-Minkowski theory studies the geometry of convex bodies in Euclidean space by use of the Minkowski sum. It originated from H. Brunn's thesis in 1887 and H. Minkowski's paper in 1903. Because there is no universally acknowledged definition of the sum of two sets in hyperbolic space, there has been no Brunn-Minkowski theory in hyperbolic space since 1903. In this paper, for any $p>0$ we introduce a sum of two sets in hyperbolic space, and we call it the \emph{hyperbolic $p$-sum}. Then we develop a Brunn-Minkowski theory in the hyperbolic space by use of our hyperbolic $p$-sum, and we call it the horospherical $p$-Brunn-Minkowski theory.
		
	Let $K$ be any smooth horospherically convex bounded domain in the hyperbolic space $\mathbb{H}^{n+1}$. Through calculating the variation of the $k$-th modified quermassintegral of $K$ by use of our hyperbolic $p$-sum, we introduce the \emph{horospherical $k$-th $p$-surface area measure} associated with $K$ on the unit sphere $\mathbb{S}^n$. For $k=0$, we introduce the \emph{horospherical $p$-Minkowski problem}, which is the prescribed horospherical $p$-surface area measure problem. Through designing and studying a new volume preserving flow,
	 we solve the existence of solutions to the horospherical $p$-Minkowski problem for all $p \in (-\infty,+\infty)$ when the given measure is even. For $1 \leq k \leq n-1$, we introduce the \emph{horospherical $p$-Christoffel-Minkowski problem}, which is the prescribed horospherical $k$-th $p$-surface area measure problem. We solve the existence of solutions to the horospherical $p$-Christoffel-Minkowski problem for $p\in(-n, +\infty)$ under appropriate assumption on the given measure. We also study the Brunn-Minkowski inequalities and the Minkowski inequalities for domains in the hyperbolic space.
\end{abstract}

\maketitle
\tableofcontents

\section{Introduction}

\subsection{An overview of the Brunn-Minkowski theory in Euclidean space}\label{subsec-overview} $\ $

Convex geometry studies convex bodies in Euclidean space $\mathbb{R}^{n+1}$.
The classical Brunn-Minkowski theory is the heart of convex geometry, and it originated from Brunn's thesis \cite{Bru1887} and Minkowski's paper \cite{Min1903}.
By introducing the concept of support functions, Minkowski defined a linear combination of convex bodies, which is called the \emph{Minkowski addition} (i.e. Minkowski sum).
By studying the volume of linear combinations of convex bodies, he defined the mixed volume of convex bodies. Many geometric inequalities can be stated by use of the above concepts, such as the Brunn-Minkowski inequality, the Minkowski inequality, and the Alexandrov-Fenchel inequality. Here we refer readers to \cite{Gar02} for a survey on the Brunn-Minkowski inequality. Besides, Osserman \cite{Oss78} emphasized that the Brunn-Minkowski inequality directly implies the isoperimetric inequality. On the other hand, Aleksandrov \cite{Alek37}, Fenchel, and Jessen \cite{FJ38} showed that the integral representation of the mixed volume could induce a Borel measure on the unit sphere $\mathbb{S}^n$, which is known as the surface area measure. The classical Minkowski problem asks the necessary and sufficient conditions for a Borel measure on $\mathbb{S}^n$ to be the surface area measure of a convex body. The polytopal case was solved by Minkowski \cites{Min1897, Min1903}. In general cases, the existence of solutions to this problem was solved independently by Aleksandrov \cite{Alek38}, Fenchel and Jessen \cite{FJ38}. The regularity of solutions to this problem was studied by Lewy \cite{Lew38}, Nirenberg \cite{Nir53}, Cheng, Yau \cite{CY76}, and Pogorelov \cite{Pog78}. Guan and Ma \cite{GM03} studied the Christoffel-Minkowski problem and gave a sufficient condition for the existence of solutions.

In 1962, Firey \cite{F62} found a class of combinations of convex bodies containing the origin for each real number $p \geq 1$, which was called the \emph{Minkowski-Firey $L_p$-addition} (Firey's $p$-sum for short).
In 1993, Lutwak \cite{Lut93} discovered that Firey's $p$-sums could induce a Brunn-Minkowski theory for $p\geq 1$, called the $L_p$ Brunn-Minkowski theory (where $p=1$ is the classical case), see also \cite{Lut96}. He proved the $L_p$ Minkowski inequality and the $L_p$ Brunn-Minkowski inequality for $p>1$, where the latter was obtained earlier by Firey \cite{F62}. For $0 \leq p <1$, the $p$-sums and the Brunn-Minkowski inequalities have also been studied in the recent decade, see e.g. \cites{BLYZ12, KM20, CHLL20}.  Now we focus on the prescribed measure problems in the $L_p$ Brunn-Minkowski theory. Lutwak \cite{Lut93} defined the $L_p$ surface area measure for convex bodies containing the origin in their interiors. In the case $p > 1$, by use of the $L_p$ Minkowski inequality, he showed that the $L_p$ area measures appeared naturally in the integral representation of the mixed $p$-quermassintegrals. The prescribed $L_p$ surface area measure problem is called the $L_p$ Minkowski problem, which reduces to the classical Minkowski problem when $p=1$. In the case $p=0$, it is the prescribed cone-volume measure problem and is also called the logarithmic Minkowski problem \cite{BLYZ13}. In the case $p=-n-1$, it is called the centro-affine Minkowski problem \cite{CW06}. For  $n+1 \neq p> 1$, Lutwak \cite{Lut93} studied the existence and uniqueness of origin symmetric solutions to the $L_p$ Minkowski problem. For $p \geq -n-1$, Chou and Wang \cite{CW06} studied the problem under various assumptions on the given measure. Besides, the polytopal case of the $L_p$ Minkowski problem has also been studied, see e.g. \cites{HLYZ05, BHZ16, Zhu15, Zhu17}. The $L_p$ Brunn-Minkowski theory has developed rapidly for decades, see a recent survey by B\"{o}r\"{o}czky \cite{Bor22}. 

In 2016, Huang, Lutwak, Yang, and Zhang \cite{HLYZ16} introduced and studied the dual Minkowski problem. For other works contributed to this problem, we refer to, e.g. \cite{LSW20,LYZ18} and the references therein. For readers interested in the Brunn-Minkowski theory, we refer to the books by Schneider \cite{Sch14} and Gardner \cite{Gar06}.

\subsection{Motivation and construction of hyperbolic $p$-sum} $\ $

Researchers have tried to figure out the counterpart of results in the Brunn-Minkowski theory for domains in the hyperbolic space $\mathbb{H}^{n+1}$ over decades. 
	
The Minkowski type problems for smooth convex bodies in Euclidean space can be unified as the prescribed measure problems via the Gauss map, e.g. the classical Minkowski problem, the Christoffel-Minkowski problem, and the $L_p$ Minkowski problem. For domains in hyperbolic space, once a suitable Gauss map has been defined, one can also study the prescribed measure problems and the prescribed curvature problems, see e.g. \cite{EGM09} and \cite[Chapter 10]{Ger06}. 
	
In the classical Brunn-Minkowski theory, a special type of the Alexandrov-Fenchel inequalities concerns the comparison of different quermassintegrals of a single convex body, in particular, the classical isoperimetric inequality. For a smooth convex body in Euclidean space, the quermassintegrals  are either volume or can be expressed as curvature integrals on the boundary. For a bounded domain (i.e. compact sets with non-empty interior) in hyperbolic space, the quermassintegrals have also been defined, see \cite{San04}. Especially, when the domain is smooth, its quermassintegrals can be expressed as linear combinations of the volume and the curvature integrals on the boundary. 
Then people turned to study the comparison of the quermassintegrals and curvature integrals for bounded domains in $\mathbb{H}^{n+1}$ and obtained several Alexandrov-Fenchel type inequalities with suitable convex assumptions on the domains, see e.g. \cites{AHL20, LWX14, WX15, HL19, HL21, HLW20}. However, most of the above results in $\mathbb{H}^{n+1}$ are scattered and only for a single domain.

 \emph{The question arises whether there is a Brunn-Minkowski theory for bounded domains in the hyperbolic space with suitable convex assumptions}. 

Our main purpose is to find such a theory so that it can summarize some of the earlier \textit{irrelevant} results for domains in the hyperbolic space. The main difficulty in establishing a Brunn-Minkowski theory for bounded domains in hyperbolic space is that there is no underlying linear structure on the hyperbolic space. Then people do not know how to introduce the concept ``sum" of two domains, which is the first and the most important step. Once it is settled, by following the procedure in the classical Brunn-Minkowski theory (see Subsection \ref{subsec-overview}), one can study the corresponding prescribed measure problems, Brunn-Minkowski inequalities, and Minkowski inequalities in hyperbolic space.
To overcome the above difficulty, we introduce the concept of \emph{hyperbolic $p$-sum} by using two different methods. 

The first method is to construct the hyperbolic $p$-sum of two domains in $\mathbb{H}^{n+1}$ by use of their ``support functions". In Euclidean space $\mathbb{R}^{n+1}$, convex bodies can be viewed as intersections of their supporting half spaces (see \cite[Corollary 1.3.5]{Sch14}), and the support function of a given convex body is defined by the distance from the origin to the hyperplanes which are tangential to it. In hyperbolic space $\mathbb{H}^{n+1}$, we can use horo-balls to replace ``half spaces" and use horospheres to replace ``hyperplanes", where the \emph{horospheres} are the hypersurfaces  in $\mathbb{H}^{n+1}$ with principal curvatures equal to $1$ everywhere, and the \emph{horo-balls} are domains delimited by horospheres. Therefore, it is natural to study the \emph{horospherically convex bounded domains}, i.e. the intersections of closed horo-balls, and we call them \emph{h-convex bounded domains} for short. The \emph{horospherical support function} of a h-convex bounded domain is defined by the distance from a fixed origin in $\mathbb{H}^{n+1}$ to the horospheres which are tangential to it. These concepts can be found in \cite{ACW18}.

To introduce the follow-up hyperbolic $p$-sum in Definition \ref{def-p sum}, let us formulate the definition of the horospherical support function more precisely.
Throughout the paper, we use the hyperboloid model of $\mathbb{H}^{n+1}$ in Minkowski space $\mathbb{R}^{n+1,1}$. For a smooth h-convex bounded domain $\Omega \subset \mathbb{H}^{n+1} \subset \mathbb{R}^{n+1,1}$, let $X \in \partial \Omega$ be the position vector and $\nu$ be the outward unit normal of $\partial \Omega \subset \mathbb{H}^{n+1} $ at $X$. Since $\metric{X}{X}=-1$ and $\metric{X}{\nu}=0$ on $\partial \Omega$, there exist $u \in \mathbb{R}$ and  $z\in \mathbb{S}^{n}$ such that
\begin{equation*}
	X -\nu = e^{-u}(z,1).
\end{equation*} 
Define the \emph{horospherical Gauss map} $G :\partial \Omega \to \mathbb{S}^n$ by $G\(X\)=z$ and call $u$ the \emph{horospherical support function} of $\Omega$. 
For a smooth bounded domain, if the principal curvatures are greater than $1 +\delta$ for some $\delta>0$ on its boundary, then we call it \emph{uniformly horospherically convex} (uniformly h-convex for short). In this case, the map $G$ is a diffeomorphism, see \cite{ACW18}. Then we can view $u$ as a function on $\mathbb{S}^n$ and denote it by $u(z)$ when $\Omega$ is uniformly h-convex. We emphasize that \emph{a single point in $\mathbb{H}^{n+1}$ can be regarded as a degenerated h-convex bounded domain}, since it can be approximated by a sequence of closed geodesic balls.
  
Since the horospherical support function has been defined, it is a natural idea to use them to construct the sum of two uniformly h-convex bounded domains. To this end, we introduce the following definition of \emph{hyperbolic $p$-sum}.
\begin{defn}\label{def-p sum}
	Let $p$, $a$, and $b$ be real numbers that satisfy  $\frac{1}{2} \leq p \leq 2$, $a \geq 0$, $b \geq 0$ and $a +b \geq 1$, and let $K$ and $L$ be two smooth uniformly h-convex bounded domains in $\mathbb{H}^{n+1}$. Denote by $u_K (z)$ and $u_L (z)$ the horospherical support functions of $K$ and $L$ respectively. Allow that both $K$ and $L$ can degenerate to a single point. We define the hyperbolic $p$-sum $\Omega := a \cdot K +_p b \cdot L$ of $K$ and $L$ by the h-convex bounded domain with horospherical support function 
	\begin{equation*}
		u_{\Omega}(z):= \frac{1}{p} \log \(a e^{p u_{K}(z)}+b e^{p u_{L}(z)}\).
	\end{equation*}
\end{defn}
We will show that the above Definition \ref{def-p sum} is well-defined in Theorem \ref{thm-def p sum-well defined}.  Remark that the case $p=1$, $a=1$, and $b=1$ of Definition \ref{def-p sum} was appeared in the paper of Gallego, Solanes, and Teufel \cite{GST13}, where they called it the harmonic sum. Since there exists no good definition of ``dilation", their harmonic sum can't extend to linear combinations of h-convex bounded domains.

The second method to construct the hyperbolic $p$-sum is to give a pointwise definition.
Back to the classical Brunn-Minkowski theory, the pointwise definition of Minkowski addition of convex bodies  is valid for general sets in Euclidean space, i.e.
\begin{equation*}
	\widehat{\Omega} =a \widehat{K}+b \widehat{L} := \{ ax+by:  x \in \widehat{K}, \, y\in \widehat{L} \},
\end{equation*}
where $a,b \geq 0$, and $\widehat{K}, \widehat{L}$ are sets in $\mathbb{R}^{n+1}$. In the $L_p$ Brunn-Minkowski theory, Lutwak, Yang, and Zhang \cite{LYZ12} gave a pointwise definition of Firey's $p$-sum \cite{F62} for each $p \geq 1$. Furthermore, their pointwise addition is also valid for general sets in $\mathbb{R}^{n+1}$. 

Motivated by the above pointwise sums in Euclidean space, for all $p \in (0, +\infty)$, we will give a pointwise definition (see Definition \ref{def-p sum-p>0-new} below) of the hyperbolic $p$-sum for sets in $\mathbb{H}^{n+1}$. Before we state the specific construction, we would like to give a geometric intuition of the hyperbolic $p$-sum for sets. As in Definition \ref{def-p sum}, we need two sets $K, L$ in $\mathbb{H}^{n+1}$ and need three parameters $p$, $a$ and $b$. We first consider the case $p=1$, which is closely related to the classical Minkowski addition from a geometric point of view. For any $X \in K$ and $Y \in L$, it seems natural to define the hyperbolic $1$-sum ``$a \cdot X+_1 b \cdot Y$" of $X$ and $Y$  by $aX+bY$, where we used the vector addition in Minkowski space $\mathbb{R}^{n+1,1}$. However, in general cases, the vector $aX+bY$ is no longer on the hyperboloid $\mathbb{H}^{n+1}$. 

To overcome that obstacle, we \emph{identify the vector $aX+bY$ in $\mathbb{R}^{n+1,1}$ with a geodesic ball in $\mathbb{H}^{n+1}$}, which is enclosed by the intersection of the past light cone of $aX+bY$ and the hyperboloid $\mathbb{H}^{n+1}$. Then we define $a\cdot K+_1b \cdot L$ by the union of those geodesic balls. Now we turn to consider the case $p \in (0,1) \cup (1,+\infty)$. For any $X \in K$ and $Y \in L$, there exists a unique geodesic segment connecting $X$ and $Y$. Then any point on the segment can be parameterized by a new parameter $t \in [0,1]$, see \eqref{P-p-t-a-X-b-Y} below. For each point on the segment, we will define a geodesic ball centered at that point whose radius only depends on $\{a,b,p, X,Y, t\}$, see \eqref{R-p-t-a-X-b-Y} below. Then we can define the hyperbolic $p$-sum of $K$ and $L$ by the combination of those geodesic balls.

To introduce the follow-up pointwise hyperbolic $p$-sum in Definition \ref{def-p sum-p>0-new}, we need to set up some notations. Let $X$ and $Y$ be two points in $\mathbb{H}^{n+1}$, and let $p>0$, $t \in [0,1]$, $a \geq 0$, and $b \geq 0$ be real numbers. Denote by $d_{\mathbb{H}^{n+1}} \( \cdot, \cdot\)$ the distance function in $\mathbb{H}^{n+1}$ and by $\overline{XY}$ the geodesic segment that connects $X$ and $Y$ in $\mathbb{H}^{n+1} $.
\begin{itemize}
	\item If $p \in (0,1) \cup (1, +\infty)$, then we define
		\begin{equation}\label{R-p-t-a-X-b-Y}
			R(p, t; a, X, b, Y) :=  \( (1-t)^{\frac{2}{q}}a^{\frac{2}{p}} + t^{\frac{2}{q}} b^{\frac{2}{p}}+ 2\(1-t\)^{\frac{1}{q}} t^{\frac{1}{q}} a^{\frac{1}{p}} b^{\frac{1}{p}}\cosh d_{\mathbb{H}^{n+1}}(X,Y) \)^{\frac{1}{2}}, 
		\end{equation}
		where $q$ is the H\"older conjugate of $p$, i.e. $\frac{1}{p} + \frac{1}{q} =1$.
		Let $P(p, t; a, X, b, Y)$ be a point lying on $\overline{XY}$ such that the distance between it and $X$ satisfies
		\begin{equation}\label{P-p-t-a-X-b-Y}
			\cosh d_{\mathbb{H}^{n+1}} \( P(p,t; a, X, b, Y), X \) = \frac{(1-t)^{\frac{1}{q}} a^\frac{1}{p} + t^{\frac{1}{q}}b^{\frac{1}{p}} \cosh d_{\mathbb{H}^{n+1}}(X,Y)}{R(p, t; a, X, b, Y)}.
		\end{equation}
	\item Define
		\begin{equation}\label{R-a-X-b-Y}
			R( a, X, b, Y) :=  \( a^{2} +  b^{2}+ 2 a b\cosh d_{\mathbb{H}^{n+1}}(X,Y) \)^{\frac{1}{2}}.
		\end{equation} 
		Let $P(a, X, b, Y)$ be a point lying on $\overline{XY}$ such that the distance between it and $X$ satisfies
		\begin{equation}\label{P-a-X-b-Y}
			\cosh d_{\mathbb{H}^{n+1}} \( P(a, X, b, Y), X \) =  \frac{a + b \cosh d_{\mathbb{H}^{n+1}}(X,Y)}{R(a, X, b, Y)}.
		\end{equation}
\end{itemize}
	
\begin{defn}\label{def-B(p,t,a,X,b,Y)} $ \ $
	\begin{enumerate}
		\item If $p \in (0,1) \cup (1, +\infty)$ and $R(p, t; a, X, b, Y) \geq 1$, then we define $B (p, t; a, X, b, Y)$ as the closed geodesic ball of radius $\log R(p, t; a, X, b,Y)$ centered at $P(p, t; a, X, b, Y)$ in $\mathbb{H}^{n+1}$. 
		\item If $p \in (0,1) \cup (1, +\infty)$ and $R(p, t; a, X, b, Y) < 1$, then we define  $B (p, t; a, X, b, Y)$ as the empty set.
		\item If $R( a, X, b, Y) \geq 1$, then we define $B(a,X,b,Y)$ as the closed geodesic ball of radius	$\log R(a,X,b,Y)$ centered at $P(a,X,b,Y)$ in $\mathbb{H}^{n+1}$.
		\item If $R(a, X, b, Y) < 1$, then we define $B (a, X, b, Y)$ as the empty set.
	\end{enumerate}
\end{defn}	
After the above preparing works, we are in the position to state the pointwise hyperbolic $p$-sum.
\begin{defn}\label{def-p sum-p>0-new}
	Let $p>0$, $a \geq 0$ and $b \geq 0$ be real numbers, and let $K$ and $L$ be sets in $\mathbb{H}^{n+1}$.  Then  the hyperbolic $p$-sum $\Omega$ of $K$ and $L$, i.e.
	\begin{equation*}
		\Omega := a \cdot K+_p b \cdot  L
	\end{equation*}
	is defined by
	\begin{equation*}
		\Omega =
		\left\{ 
		\begin{aligned}
			&\bigcup_{X \in K, \, Y \in L} \bigcup_{t \in [0,1]} B(p,t; a, X, b, Y), \quad &p>1,\\
			&\bigcup_{X \in K, \, Y \in L}  B( a, X, b, Y), \quad &p=1,\\
			&\bigcup_{X \in K, \, Y \in L} \bigcap_{t \in [0,1]} B(p,t; a, X, b, Y), \quad &0 <p<1.
		\end{aligned}
		\right.
	\end{equation*}
\end{defn}
When $p$, $a$, $b$, $K$, and $L$ satisfy the assumptions as in Definition \ref{def-p sum}, we will show in Theorem \ref{thm-new-old-sum-compatible} that the pointwise hyperbolic $p$-sum defined in Definition \ref{def-p sum-p>0-new}  is compatible with the hyperbolic $p$-sum defined in Definition \ref{def-p sum}. 
Then Definition \ref{def-p sum-p>0-new} implies that Definition \ref{def-p sum} does not depend on the hyperboloid model. Specifically, the hyperbolic $p$-sum $a \cdot K+_p b \cdot L$ can be constructed without using the Minkowski space $\mathbb{R}^{n+1,1}$ and the horospherical support functions of $K$ and $L$, it only depends on the relative position of $K$ and $L$ in $\mathbb{H}^{n+1}$. Here we remark that our Definition \ref{def-p sum-p>0-new} is different from the $\lambda$-horocycle Minkowski sum  in $\mathbb{H}^2$ defined recently by Assouline and  Klartag \cite{AK22}. Fillastre \cite{Fill13} studied Fuchsian convex bodies in Minkowski space $\mathbb{R}^{n+1,1}$.

In the next two subsections, we will follow the procedure in the classical Brunn-Minkowski theory and study the prescribed measure problems and geometric inequalities that are both generated by our hyperbolic $p$-sum. At the same time, we will continuously discuss the relationship between these new problems and the earlier results in hyperbolic space. 
 
As we concern about the horospherically convex domains, we sum up the related problems in a theory called the \emph{horospherical $p$-Brunn-Minkowski theory} in hyperbolic space.

\subsection{Prescribed measure problems in hyperbolic space} $ \ $

For each integer $k=0,1,\ldots,n$ and any h-convex bounded domain $K \subset \mathbb{H}^{n+1}$, the $k$-th modified quermassintegral $\widetilde{W}_k \(K \) $ of $K$ was defined by Andrews, Chen, and Wei \cite{ACW18}, see \eqref{def-modified quermassintegral}. They proved that these $\{\widetilde{W}_k \( \cdot\)\}_{k=0}^n$ are monotone with respect to the inclusion of h-convex bounded domains.
 
Let $K$ be a smooth uniformly h-convex bounded domain with horospherical support function $u_K(z)$. Define $\g_K(z) = e^{u_K(z)}$. We will show in Lemma \ref{lem-formula-Wpk-K-L} that the variation of $\widetilde{W}_k \(K+_p t \cdot L\) $  along any h-convex bounded domain $L$ (might be a point) at $t=0$ induces a measure
\begin{equation*}
 	d S_{p,k}\(K,z \) := \g_K^{-p-k}p_{n-k} \(A [\g_K] \) d \sigma
\end{equation*}
on $\mathbb{S}^n$, where  $d \sigma$ is the area measure of the unit sphere $\mathbb{S}^n$; the matrix $A[\g_K]$ is defined by
\begin{equation*}
	A_{ij} [\g_K]:= D_j D_i \g_K - \frac{1}{2} \frac{|D \g_K|^2}{\g_K} \sigma_{ij} + \frac{1}{2} \(\g_K- \frac{1}{\g_K}\)\sigma_{ij},
\end{equation*}
where $\sigma_{ij}$ denotes the standard metric on $\mathbb{S}^n$, $D$ denotes the Levi-Civita connection with respect to $\sigma_{ij}$, and $p_{n-k} \(A[\g_K] \) $ is the modified $(n-k)$-th elementary symmetric function of the eigenvalues $\lbrace x_1, \ldots,x_n  \rbrace$ of $A[\g_K]$, i.e. 
\begin{equation*}
	p_{n-k} \(A[\g_K] \) = \frac{1}{C_n^{n-k}} \sigma_{n-k} \(A[\g_K] \)=\frac{1}{C_n^{n-k}}\sum_{1 \leq i_1<\cdots<i_{n-k} \leq n} x_{i_1} x_{i_2} \cdots x_{i_{n-k}}.
\end{equation*}
Then we call $d S_{p,k}\(K,z\)$ the \emph{$k$-th horospherical $p$-surface area measure of $K$}. Now we introduce the following prescribed $k$-th horospherical $p$-surface area measure problem. 

\textit{For a given smooth positive function $f(z)$, we ask the sufficient and necessary conditions for $f(z)$, such that the equation
\begin{equation}\label{eq-phi-p-CM problem}
	\g^{-p-k}p_{n-k} \(A [\g] \) = f(z)
\end{equation}
admits a smooth positive solution $\g(z)$ with $A[\g(z)]>0$ for all $z \in \mathbb{S}^n$.}

We call it the horospherical $p$-Minkowski problem when $k=0$ and  call it the horospherical $p$-Christoffel-Minkowski problem when $1 \leq k \leq n-1$, see the formal statements of the above problems in Problem \ref{prob-Horospherical p-Minkowski problem} and Problem \ref{prob-Horospherical p-Christoffel-Minkowski problem}.

\subsubsection{Kazdan-Warner type obstructions} If $p=-n$ in \eqref{eq-phi-p-CM problem}, then equation \eqref{eq-phi-p-CM problem} can be formulated as the following prescribed curvature equation for smooth uniformly h-convex hypersurfaces, i.e.
\begin{equation}\label{eq-p=-n-p-CM problem}
	p_{n-k} \(\frac{1}{\kappa-1}\) =f(z),
\end{equation}
where $\kappa = (\kappa_1, \ldots, \kappa_n)$ are the principal curvatures, and $z$ is the image of the horospherical Gauss map defined on the hypersurface.

When $k=n-1$, equation \eqref{eq-p=-n-p-CM problem} is closely related to the Christoffel problem in $\mathbb{H}^{n+1}$ studied by Espinar, G\'alvez, and Mira \cite{EGM09}. They observed that the Christoffel problem in $\mathbb{H}^{n+1}$ is closely related to the Nirenberg problem, and then they gave a Kazdan-Warner type obstruction to it. Motivated by their result, for general $k =0,1,\ldots,n-1$ and constant function $f(z)$, we notice that equation \eqref{eq-p=-n-p-CM problem} relates to the $\sigma_{m}$-Yamabe problems \cite{Via00} ($m=1, \ldots, n-k$) in conformal geometry.  In Theorem \ref{thm-necess-cond-p=-n}, we will give Kazdan-Warner type obstructions to equation \eqref{eq-p=-n-p-CM problem} for all $0 \leq k \leq n-1$. Particularly, for $n \geq 1$ and $p=-n$, we will show that the solutions $\g(z)$ to equation \eqref{eq-phi-p-CM problem} must satisfy
\begin{equation}\label{K-W obstruction}
	\int_{\mathbb{S}^n} \g^{-n} \metric{D f(z)}{D x_i} d\sigma = 0, \quad i=0,1,\ldots, n,
\end{equation}
where $x_0, \ldots, x_n$ are the coordinate functions of $\mathbb{S}^n$ in $\mathbb{R}^{n+1}$.
 
We present two different proofs of the obstruction \eqref{K-W obstruction}.
\begin{itemize}
 	\item Our first method is to consider equation \eqref{eq-phi-p-CM problem} as a linear combination of the equations for $\sigma_m$-Nirenberg problem ($m=1, \ldots, n-k$), and then we will prove the obstruction \eqref{K-W obstruction} by
 	using several Kazdan-Warner type identities proved in \cite{Han06, LLL21, Via00}.
 	\item  The second method is more direct. After observing a key identity for closed hypersurfaces in $\mathbb{H}^{n+1}$ (see Lemma \ref{lem-KZ-non-shifted}), we can prove the obstruction \eqref{K-W obstruction} by a direct calculation.
\end{itemize}
The obstruction \eqref{K-W obstruction} shows that, for general function $f(z)$ defined on $\mathbb{S}^n$, equation \eqref{eq-phi-p-CM problem} may not have a solution when $p=-n$.
\subsubsection{Uniqueness result}
When $f(z)$ is a positive constant and $p \geq -n$, we completely classify the solutions to equation \eqref{eq-phi-p-CM problem} satisfying $\g(z)>1$ and $A_{ij} [\g]>0$, see Theorem \ref{thm-f=c-CM-M-sphere}. Geometrically, the solutions are geodesic balls centered at the origin if $p>-n$, and are geodesic balls if $p=-n$. A general version of Theorem \ref{thm-f=c-CM-M-sphere} will be proved in Proposition \ref{prop-uniq-sphere-higher dim}, where we prove that geodesic balls are the only solutions to a class of curvature equations, see \eqref{eq-uniq-sphere}. The new techniques we use are integral formulas (see Lemma \ref{lem-integral formulas}) for closed hypersurfaces in $\mathbb{H}^{n+1}$ and a Heintze-Karcher type inequality (see Proposition \ref{prop-HK-n=1}) for curves in $\mathbb{H}^2$.  
\subsubsection{Existence result}
Let us consider the equation \eqref{eq-phi-p-CM problem} when $f(z)$ is an even function on $\mathbb{S}^n$, i.e. $f(z) =f(-z)$ for all $z \in \mathbb{S}^n$. We will prove in Theorem \ref{thm-exist-all p-k=0} that, for $k=0$ and $p \in (-\infty, +\infty)$, there exists a constant $\gamma$ such that equation
\begin{equation}\label{eq-CM-gamma}
	\g^{-p-k} p_{n-k}(A[\g]) =\gamma f(z)
\end{equation} 
admits a positive even solution $\g(z)$ with $A_{ij} [\g(z)]>0$, which is an existence result of the horospherical $p$-Minkowski problem. For $1 \leq k \leq n-1$ and $p \geq -n$, by making appropriate assumptions on $f(z)$ (see Assumption \ref{assump-h}), we will prove the existence of positive even solutions to equation \eqref{eq-CM-gamma} in Theorem \ref{thm-exist-all p}, which is an existence result of the horospherical $p$-Christoffel-Minkowski problem.

To prove the above results, we introduce a new volume preserving flow in \eqref{flow-HCMF}. For each real number $p$, we also define a functional $J_p(\Omega)$ for h-convex bounded domain $\Omega$ in \eqref{def-J_p}. We will show that $J_p(\Omega_t)$ is strictly decreasing along the flow \eqref{flow-HCMF} unless the horospherical support function of $\Omega_t$ solves equation  \eqref{eq-CM-gamma}. Then the desired existence results of equation \eqref{eq-CM-gamma} (see Theorem \ref{thm-exist-all p-k=0} and Theorem \ref{thm-exist-all p}) can be proved by use of the long time existence and the subsequential convergence of the flow \eqref{flow-HCMF} (see Theorem \ref{thm-long time existence}). 

\subsection{Geometric inequalities in hyperbolic space} $ \ $

\subsubsection{Steiner formulas and their applications}
Steiner formula for convex bodies in Euclidean space gives an explicit formula for the volume bounded by two convex parallel hypersurfaces. 
Steiner formulas in the hyperbolic space and the sphere can be found in Santal\'{o}'s book \cite[Chapter 18]{San04}.

We will derive the Steiner formulas for modified quermassintegrals of h-convex bounded domains in $\mathbb{H}^{n+1}$ in Theorem \ref{thm-Stein-modi-quemass}. Moreover, we will derive the Steiner formula for volumes with weight $\cosh r$ in Theorem \ref{thm-weighted Steiner formula}, where $r$ denotes the distance function in $\mathbb{H}^{n+1}$ from the origin. As an application, we will prove a new weighted Alexandrov-Fenchel type inequality for h-convex bounded domains in $\mathbb{H}^{n+1}$, see Theorem \ref{thm-new-weighted-inequality}.

\subsubsection{Horospherical $p$-Brunn-Minkowski inequality}
Denote by $B(r)$ the geodesic ball of radius $r$ centered at the origin in $\mathbb{H}^{n+1}$. For any bounded domain $\Omega$ in $\mathbb{H}^{n+1}$, we let $r^k_{\Omega}$ be the positive number satisfying $\widetilde{W}_k(B(r^k_{\Omega})) = \widetilde{W}_k(\Omega)$  for each $0\leq k \leq n$, see also Definition \ref{def-mod k-mean radius}. 

Let $a$, $b$, $p$, $K$ and $L$ satisfy the assumptions in Definition \ref{def-p sum}, we conjecture that the following geometric inequality holds,
\begin{equation}\label{horo BM informal}
	\exp (p r^k_{a\cdot K +_p b \cdot L}) \geq a \exp(p r^k_K) + b \exp \( p r^k_L\), \quad k=0,\ldots,n.
\end{equation}
We call the above inequality as the horospherical $p$-Brunn-Minkowski inequality, and the formal statement will be given in Conjecture \ref{conj-BM}. Especially, when  $K$ and $L$ are geodesic balls, inequality \eqref{horo BM informal} will be proved in Theorem \ref{thm-horo-BM balls}.

Curvature flows are powerful tools in proving geometric inequalities, see e.g. \cites{ACW18, GWW14, GL15, HL21,HLW20, SX19}. Andrews et al. \cite[Corollary 1.9]{ACW18}  proved a class of  Alexandrov-Fenchel type inequalities by using a volume preserving flow, see Theorem C in Subsection \ref{subsec-10.1}.  Later, Hu, Wei, and the first author \cite{HLW20}  proved a class of weighted Alexandrov-Fenchel type inequalities \cite[Theorem 1.8]{HLW20} for h-convex hypersurfaces by using a locally constrained flow, see Theorem E in Subsection \ref{subsec-10.2}. 

In view of the desired inequality \eqref{horo BM informal}, we will show in Proposition \ref{prop-BM-K=L} that the inequality in Theorem C (\cite[Corollary 1.9]{ACW18}) can be interpreted as
\begin{equation*}
	\left. \frac{d}{dt} \right|_{t=0} \exp (p r^k_{K +_p t \cdot K}) \geq \left. \frac{d}{dt} \right|_{t=0}
	\( \exp(p r^k_K) + t \exp \( p r^k_K\)\), \quad  k=0, \ldots, n.
\end{equation*}
On the other hand, we will show that the inequality in Theorem E (\cite[Theorem 1.8]{HLW20}) can be interpreted as 
\begin{equation*}
	\left. \frac{d}{dt} \right|_{t=0} \exp ( r^k_{K +_1 t \cdot N}) \geq \left. \frac{d}{dt} \right|_{t=0}
	\( \exp( r^k_K) + t \exp \( r^k_N\)\), \quad  k=1, \ldots, n,
\end{equation*}
where $N$ is regarded as the origin of $\mathbb{H}^{n+1}$. In the $L_p$ Brunn-Minkowski theory, the $L_p$ Minkowski inequalities are generated from the variation of the $L_p$ Brunn-Minkowski inequalities. Therefore by using the hyperbolic $p$-sum in Definition \ref{def-p sum}, we can regard the above Theorem C and Theorem E as two special types of Minkowski inequalities in the horospherical $p$-Brunn-Minkowski theory, see the statements of Conjecture \ref{conj-Min ineq} and Conjecture \ref{conj-Min ineq-weak} for general horospherical $p$-Minkowski inequalities in $\mathbb{H}^{n+1}$.
 
Motivated by the above observation, by using the curvature flows and geometric inequalities in \cite{HL21, HLW20, SX19} , we will prove some special cases of Conjectures \ref{conj-BM}, \ref{conj-Min ineq}, and \ref{conj-Min ineq-weak}. The results are summarized in Theorems \ref{thm-sum-BM}, \ref{thm-sum-Min ineq-strong}, and \ref{thm-sum-Min ineq-weak}.

\subsection{Weighted horospherical $p$-Brunn-Minkowski theory} $\ $

The weight $\cosh r$ often appears in geometric inequalities in $\mathbb{H}^{n+1}$; for example, it occurred in the Minkowski inequalities proved by Brendle, Hung, and Wang \cite{BHW16}, and Chao Xia \cite{Xia16}, see also \cite{SX19}. Furthermore, Hu and the first author \cite{HL21}, Hu, Wei, and the first author \cite{HLW20} proved a class of Alexandrov-Fenchel type inequalities for curvature integrals with weight $\cosh r$. Now we call the integral of $\cosh r$ on domain $\Omega \subset \mathbb{H}^{n+1}$ the weighted volume of $\Omega$.

So what is the behavior of the weighted volume of domains along with the hyperbolic $p$-sum? To answer that question, we propose Conjecture \ref{conj-weighted-h-BM}, which is called the weighted horospherical $p$-Brunn-Minkowski inequality. Based on Conjecture \ref{conj-weighted-h-BM}, we can propose several Minkowski type inequalities and isoperimetric type inequalities with weight $\cosh r$ in Conjectures \ref{conj-weighted-horo-Mink-ineq}--\ref{conj-last of weighted isoperi type}.
Then we find that the weighted isoperimetric inequality proved by Scheuer and Xia \cite{SX19} can be viewed as a corollary of Conjecture \ref{conj-weighted-h-BM}. Beyond that observation, we prove a class of weighted isoperimetric type inequalities in Theorem \ref{thm-p=1 case in conj last of weighted iso type}, which can be considered as corollaries of the aforementioned Conjecture \ref{conj-weigh-Min-ineq}. 

In the last part of this paper, we study the relationship between the Minkowski-Firey $L_p$-addition for convex bodies in $\mathbb{R}^{n+1}$ and the hyperbolic $p$-sum in Definition \ref{def-p sum}. We will construct a map $\pi$ from h-convex bounded domains in $\mathbb{H}^{n+1}$ to convex bodies in $\mathbb{R}^{n+1}$ that contains the origin in their interiors. Geometrically, $-\pi(\Omega) \subset \mathbb{R}^{n+1}$ is the domain enclosed by the intersection of the null hypersurface generated by $\partial \Omega$ with $\mathbb{R}^{n+1} := \mathbb{R}^{n+1} \times \{0\}$, where $\Omega$ is a h-convex bounded domain in $\mathbb{H}^{n+1}$. We will prove in Proposition \ref{prop-lc proj map-supp func} that the Euclidean support function of $\pi (\Omega) \subset \mathbb{R}^{n+1}$ is precisely $e^{u_{\Omega}(z)}$, where $u_{\Omega}(z)$ is the horospherical support function of $\Omega \subset \mathbb{H}^{n+1}$. Combining that fact with Definition \ref{def-p sum}, we will obtain the relationship between the hyperbolic $p$-sum and the Firey's $p$-sum in Theorem \ref{thm-rel p-sums in Eucl and Hyperbolic}. Consequently, we can use the map $\pi$ to study the Brunn-Minkowski type inequality and Minkowski type inequalities associated with the hyperbolic $p$-sum, see Theorem \ref{thm-weighted p-BM ineq-lc proj} and Theorem \ref{thm-weighted p-Min ineq-lc proj}. Besides, we will prove geometric inequalities for domains in $\mathbb{H}^{n+1}$, where equality holds if and only if the domains are geodesic balls in $\mathbb{H}^{n+1}$, see Theorem \ref{thm-weighted p-iso ineq-lc proj}.
 
\subsection{Organization of the paper} $ \ $

In Sections \ref{sec-h domain} and Section \ref{sec-useful lemmas}, we collect basic concepts and study properties of h-convex domains and h-convex hypersurfaces in $\mathbb{H}^{n+1}$. In Section \ref{sec-p-sum}, we prove that  the hyperbolic $p$-sum in Definition \ref{def-p sum} is well-defined and prove the compatibility of the hyperbolic $p$-sums in Definition \ref{def-p sum} and Definition \ref{def-p sum-p>0-new}. In Section \ref{sec-p,k-surface area measure}, through calculating the variation of the modified quermassintegrals of h-convex domains by use of our hyperbolic $p$-sum, we introduce the horospherical $p$-surface area measures of h-convex bounded domains in the hyperbolic space. Then we introduce the horospherical $p$-Minkowski problem and the horospherical $p$-Christoffel-Minkowski problem, see Problem \ref{prob-Horospherical p-Minkowski problem} and Problem \ref{prob-Horospherical p-Christoffel-Minkowski problem}. In Section \ref{sec-Steiner formula}, we define the hyperbolic $p$-dilation and study the Steiner formulas for h-convex bounded domains. 

In Section \ref{sec-horosp-p-Christoffel Minkowski problem}, we study the existence results of the horospherical $p$-Minkowski problem and the horospherical $p$-Christoffel-Minkowski problem by designing and studying a new volume preserving flow \eqref{flow-HCMF}. For the horospherical $p$-Minkowski problem, we prove the existence of origin symmetric solutions for all $p \in (-\infty, \infty)$ when the given measure is smooth and even, see Theorem \ref{thm-exist-all p-k=0}. For the horospherical $p$-Christoffel-Minkowski problem, we prove the existence of origin symmetric solutions for $p \in (-n, +\infty)$ under appropriate assumption on the given measure, see Theorem \ref{thm-exist-all p}. In Section \ref{sec-uniqueness of horo-p-Chri-Min-constant function}, we prove for $p \geq -n$ that geodesic balls are the unique solutions to the horospherical $p$-Minkowski problem and the horospherical $p$-Christoffel-Minkowski problem when the given measure is a constant multiple of the standard measure on $\mathbb{S}^n$, see Theorem \ref{thm-f=c-CM-M-sphere}. In Section \ref{sec-Kazdan-Warner}, we use two different methods to prove the Kazdan-Warner type obstructions for the $p=-n$ case of the horospherical $p$-Minkowski problem and the horospherical $p$-Christoffel-Minkowski problem, see Theorem \ref{thm-necess-cond-p=-n}.

In Section \ref{sec: BM inequalities}, we propose and study several conjectures about horospherical $p$-Brunn-Minkowski inequalities and horospherical $p$-Minkowski inequalities. The results are summarized in Theorem \ref{thm-sum-BM}, Theorem \ref{thm-sum-Min ineq-strong}, and Theorem \ref{thm-sum-Min ineq-weak}. In Section \ref{sec: weighted h-Brunn-Minkowski}, we propose several conjectures of weighted geometric inequalities for h-convex bounded domains and prove a class of weighted isoperimetric type inequalities for hypersurfaces in $\mathbb{H}^{n+1}$. In Section \ref{sec-hyperbolic sum and Firey's sum}, we prove the relationship between the hyperbolic $p$-sum and the Minkowski-Firey $L_p$-addition in Theorem \ref{thm-rel p-sums in Eucl and Hyperbolic}. Then we prove a Brunn-Minkowski type inequality for h-convex domains along the hyperbolic $p$-sum ($1 \leq p \leq 2$), see Theorem \ref{thm-weighted p-BM ineq-lc proj}. We also prove geometric inequalities of Minkowski type for h-convex domains, see Theorem \ref{thm-weighted p-Min ineq-lc proj}.

In Section \ref{sec-appendix}, we give an intrinsic proof of Lemma \ref{lem-conf-Schouten} and another proof of Theorem \ref{thm-sum of balls 0.5--2} for $1\leq p \leq 2$. Besides, we show that the origin symmetric assumption in Conjecture \ref{conj-weighted-h-BM} can not be dropped.

\begin{ack}
	The research of the authors was supported by NSFC grant No. 11831005, NSFC grant No. 12126405 and NSFC-FWO Grant No. 11961131001. 
\end{ack}
\section{Horospherically convex domains}\label{sec-h domain}
\subsection{Minkowski space $\mathbb{R}^{n+1,1}$ and Hyperbolic space $\mathbb{H}^{n+1}$}$\ $

The Minkowski space $\mathbb{R}^{n+1,1}$ is an $(n+2)$-dimensional vector space with the Lorentzian metric
\begin{equation*}
	\metric{X}{Y} =\metric{(x_0,x_1, \ldots,x_{n+1})}{(y_0,y_1,\ldots, y_{n+1})}
	= \sum_{i=0}^n x_iy_i- x_{n+1}y_{n+1}.
\end{equation*}
A vector $X \in \mathbb{R}^{n+1,1}$ is called \emph{time-like}, \emph{light-like} or \emph{space-like} if $\metric{X}{X}<0$, $\metric{X}{X}=0$ or $\metric{X}{X}>0$ respectively. In the following text, we will write a vector $X \in \mathbb{R}^{n+1,1}$ as $(x,x_{n+1})$, where $x = (x_0, x_1, \ldots, x_n) \in \mathbb{R}^{n+1}$. A vector $X$ in $\mathbb{R}^{n+1,1}$ is called future directed or past directed if $x_{n+1}>0$ or $x_{n+1}<0$ respectively, and the \emph{future (or past) light cone} of $X$ is defined by the set of future (or past) directed light-like vectors issuing from $X$.
 
The hyperboloid model of the hyperbolic space is given by
\begin{equation*}
	\mathbb{H}^{n+1} = \lbrace X=(x,x_{n+1}) \in \mathbb{R}^{n+1,1}:  \metric{X}{X}=|x|^2-x_{n+1}^2= -1, \, x_{n+1}>0 \rbrace,
\end{equation*}
and the position vector $X \in \mathbb{H}^{n+1}$ is the normal of $\mathbb{H}^{n+1}$ in $\mathbb{R}^{n+1,1}$. Without additional hypotheses, we always regard the `north pole' $N=(0,1)$ as the origin of $\mathbb{H}^{n+1}$. Given two points $X$ and $Y$ on $\mathbb{H}^{n+1}$, the geodesic distance between $X$ and $Y$ is denoted by ${d}_{\mathbb{H}^{n+1}} \(X, Y\)$. Denote by $\mathbb{S}^n$ the unit sphere. For any $X \in \mathbb{H}^{n+1}$ with $r= {d}_{\mathbb{H}^{n+1}} \((0,1), X \)$, there exists $\theta \in \mathbb{S}^n$ such that $X = (\sinh r \theta, \cosh r)$, which is exactly the warped product structure of $\mathbb{H}^{n+1}$, see Figure \ref{Figure polar coordinate}. 
\begin{figure}[htbp]
	\centering
	\begin{tikzpicture} 
	\coordinate (X) at (1, {sqrt(2)});
	\shade[top color=white, bottom color= yellow, domain=-1:2, samples=100]
	(-1,3) -- plot(\x, {sqrt(\x*\x+1)}) -- (2,3);
	\draw[->] (-1.5,0) -- (2.5,0) node[right]{$x$};
	\draw[->] (0,-0.2)-- (0,2.5) node[above]{$x_{n+1}$};
	\filldraw[->, thick, red] (0,0) -- node[below]{$\sinh r \theta$} (1,0);
	\draw[domain= -1:2, samples=200, thick]
	plot(\x, {sqrt(\x*\x+1)});
	\draw[domain= 0:1, samples=200, very thick, blue]
	plot(\x, {sqrt(\x*\x+1)});
	\draw[->, dashed] (1,0) --  node[right]{$\cosh r$} (X);
	\node at (-1.4,1.5){$\mathbb{H}^{n+1}$};
	\node at (1.2, 1.214){$X$};
	\filldraw (X) circle(.05);
	\filldraw (0,0) circle(.05);
	\end{tikzpicture}
	\caption{Warped product structure of $\mathbb{H}^{n+1}$}
	\label{Figure polar coordinate}
\end{figure}
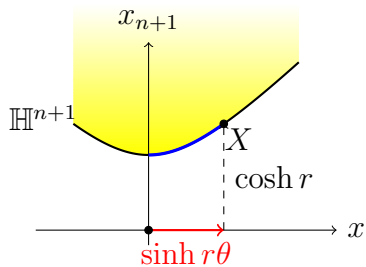

In this paper, we always assume that $M = \partial \Omega$ denotes the boundary of some bounded domain $\Omega \subset \mathbb{H}^{n+1}$. To distinguish the metrics and the Levi-Civita connections of spaces $\lbrace \mathbb{S}^n, M, \mathbb{H}^{n+1}, \mathbb{R}^{n+1,1} \rbrace$, we set the notations of their connections and metrics as $\(\mathbb{S}^n, D, \sigma\)$, $\(M, \nabla, g\)$, $\(\mathbb{H}^{n+1}, \overline{\nabla}, \bar{g}\)$ and $\(\mathbb{R}^{n+1,1}, \widetilde{\nabla}, \metric{\cdot}{\cdot}\)$. The volume elements of $\mathbb{S}^n$, $M$ and $\mathbb{H}^{n+1}$ are denoted by $d\sigma$, $d\mu$ and $dv$ respectively.
For any smooth function $f$ on $\mathbb{S}^n$ and any local orthonormal frame $\lbrace e_1, e_2, \ldots,e_n \rbrace$ on $\mathbb{S}^n$, we use the notations $f_i := D_i f$, $f_{ij} := D_j D_i f$ to denote the covariant derivatives of $f$ with respect to the standard metric of $\mathbb{S}^n$ in the context. For any hypersurface $M \subset \mathbb{H}^{n+1}$, we use $\nu$ to denote the outward unit normal of $M$ and use $h_{ij}$ to denote the second fundamental form of $M \subset \mathbb{H}^{n+1}$. The Weingarten matrix $\mathcal{W} =(h_i{}^j)$ is given by $h_i{}^j = g^{jk}h_{kj}$. At last, we remark that the notations with superscript $\widehat{\cdot}$ are stand for geometric objects and geometric quantities in Euclidean space $\mathbb{R}^{n+1}$, e.g. $\(\widehat{M}, \widehat{\nabla}, \widehat{g}\)$ denotes a hypersurface $\widehat{M}$ with connection $\widehat{\nabla}$ and metric $\widehat{g}$ in $\mathbb{R}^{n+1}$.

\subsection{Horospheres and Horo-balls}\label{subsec-horosphere,ball} $ \ $

Now we study the \emph{horospheres} in $\mathbb{H}^{n+1}$, which are the spheres centered at infinity $\mathbb{H}^{n+1}_{\infty}$ and with the constant principal curvatures equal to $1$ everywhere. Then $\nabla_i \(X-\nu \) = \(\delta_{i}{}^j-h_i{}^j \) \partial_j X=0$  on any horosphere $H$, which implies that $X-\nu$ is a constant vector on $H$. Since $\metric{X}{X} =-1$, $\metric{\nu}{\nu} =1$ and $\metric{X}{\nu} = 0$, we have $\metric{X-\nu}{X-\nu}=0$. Then there exist a constant $\lambda$ and a unit vector $z \in \mathbb{S}^n$ such that  $X -\nu = \lambda (z,1)$ on $H$ and hence
\begin{equation}\label{-1=lbd X-z,1}
	-1 = \metric{X-\nu}{X} = \lambda \metric{X}{(z,1)}. 
\end{equation}
Since $\metric{X}{(z,1)} = -x_{n+1} + \metric{x}{z} \leq -\sqrt{|x|^2+1}+|x|<0$, we have that $\lambda$ is positive. Let $\lambda = e^{-r}$. 
From now on, we denote this horosphere $H$ as $H_z(r)$. Then
\begin{equation}\label{horosphere-X-nu}
	X-\nu =e^{-r}(z,1)
\end{equation}
on $H_z(r)$. Now we explain the geometric interpretation of $r$. The above formula \eqref{-1=lbd X-z,1} implies that  $ H_z(r)$ is the intersection of the hyperplane $\lbrace X \in \mathbb{R}^{n+1,1}: \metric{X}{(z,1)} = -e^r \rbrace$ and the hyperboloid $\mathbb{H}^{n+1}$ in $\mathbb{R}^{n+1,1}$. Suppose that $X \in H_z(r)$ with
\begin{equation}\label{polar coord X}
	X = (\sinh s \theta, \cosh s),
\end{equation}
where $\theta \in \mathbb{S}^n$ and $s= d_{\mathbb{H}^{n+1}} \( (0,1),X \)$. If $r \geq 0$, then 
\begin{equation*}
	-e^r = \metric{X}{(z,1)} =\sinh s \metric{\theta}{z} - \cosh s \geq -e^s.
\end{equation*}
Therefore, $d_{\mathbb{H}^{n+1}} ((0,1), X) = s \geq r$ for any $X \in H_z(r)$, and equality holds if and only if  $X= (-\sinh rz, \cosh r) \in H_z(r)$. If $r \leq 0$, then
\begin{equation*}
	-e^r = \metric{X}{(z,1)} =\sinh s \metric{\theta}{z} - \cosh s \leq -e^{-s}.
\end{equation*}
Therefore $d_{\mathbb{H}^{n+1}} ((0,1), X) = s \geq -r$ for any $X \in H_z(r)$, and equality holds if and only if  $X= (-\sinh rz, \cosh r) \in H_z(r)$. Hence $r$ represents the signed geodesic distance from the `north pole' $N=(0,1)$ of $\mathbb{H}^{n+1}$ to $H_z(r)$. We conclude that the horosphere $H_z(r)$ is given by 
\begin{equation}\label{horosphere-support function}
	H_z(r) = \lbrace X \in \mathbb{H}^{n+1}: \log \(-\metric{X}{(z,1)} \) =r \rbrace,
\end{equation}
where $r$ is the signed geodesic distance from $N=(0,1)$ to $H_z(r)$. We define the \emph{horo-ball} $B_z(r)$ bounded by $H_z(r)$ by $\lbrace X \in \mathbb{H}^{n+1}:  \log (-\metric{X}{(z,1)}) <r \rbrace$, i.e.
\begin{equation}\label{Bz(r)}
	B_z(r) = \lbrace X \in \mathbb{H}^{n+1}: 0 > \metric{X}{(z,1)} > -e^r \rbrace.
\end{equation}
If we use the Poincar\'e ball model $\mathbb{B}^{n+1}$ of $\mathbb{H}^{n+1}$, then $B_z(r)$ corresponds to an $(n+1)$-dimensional ball which tangents to $\partial \mathbb{B}^{n+1}$ at $z$, see Figure \ref{Fig-horosphere in Poncare disc}. Furthermore, if $r>0$, then $B_z(r)$ contains the origin in its interior.
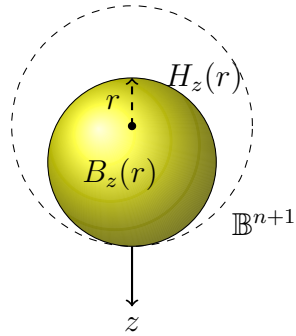
\begin{figure}[htbp]
	\centering
	\begin{tikzpicture}[scale=1.6]
	\coordinate (O) at (0,0);
	\shade[ball color=yellow, fill opacity=0.5] (0,-0.3) circle[radius=0.7];
	\draw (0,-0.3) circle[radius=0.7];
	\draw[dashed] (O) circle[radius=1];
 	\node at (1.1,-0.8){$\mathbb{B}^{n+1}$};
	\draw[->, thick] (0,-1) -- (0,-1.5) node[below]{$z$};
	\draw[->, dashed, thick] (O) -- node[left]{$r$}(0,0.4);
	\node at (-0.1, -0.4){$B_z(r)$};
	\filldraw (0,0) circle(0.03);
	\node at (0.6,0.4){$H_z(r)$};
	\end{tikzpicture}
	\caption{Horosphere $H_z(r)$ and Horo-ball $B_z(r)$}
	\label{Fig-horosphere in Poncare disc}
\end{figure} 
 
 \subsection{Horospherical Gauss map and Horospherical support function} $ \ $
 
Comparing the above arguments in Subsection \ref{subsec-horosphere,ball} with the objects in the classical convex geometry in Euclidean space, we have defined the ``hyperplane" $H_z(r)$ and the ``half space" $B_z(r)$ in the hyperbolic setting. In the next step, we will give the definitions of ``convex bodies" and ``Gauss map" in the hyperbolic setting.

A bounded domain $\Omega \subset \mathbb{H}^{n+1}$ (or its boundary $\partial \Omega$) is called \emph{horospherically convex} (or \emph{h-convex} for short) if for every $X \in \partial \Omega$ there exists a horo-ball $B$ such that $\Omega \subset B$ and $X \in \Omega \cap B$. When $\Omega$ is smooth, the horospherical convexity (h-convexity for short) of $\Omega$ implies that the principal curvatures are greater than or equal to $1$ on $\partial \Omega$. For a smooth bounded domain $\Omega$, we say $\Omega$ (or $\partial \Omega$) is \emph{uniformly h-convex} if the principal curvatures are greater than $1 +\delta$ for some $\delta>0$ everywhere on $\partial \Omega$. 

Now we define the \emph{horospherical Gauss map} $G: M \to \mathbb{S}^n$ for a smooth hypersurface $M = \partial \Omega$ in $\mathbb{H}^{n+1}$, see \cite{Bry87, Eps86}. For each $p \in M$, we know from \eqref{-1=lbd X-z,1} that $X(p) -\nu(p) = \lambda_p (z_p,1)$ for some $\lambda_p>0$ and $z_p\in \mathbb{S}^n$. Then we define the horospherical Gauss map $G$ at $p$ of $M$ by $G(p) = z_p$. From \eqref{horosphere-X-nu}, it is simple to see $G\(H_z(r)\) =z$. 

Note that the horospherical Gauss map defined in this paper is the same as that in \cite{ACW18}. In some other papers (see e.g., \cite{EGM09}), one may define the horospherical Gauss map ``$G(p)=z'_p$" by using ``$X(p) + \nu (p) = \lambda'_p (z'_p,1)$", which means that the horospherical Gauss map is defined locally on $M$, and $\nu$ need not be the ``outward" unit normal. However, since we deal with the domain which are intersections of horo-balls (not intersections of complements of horo-balls), the boundary of the domain is oriented, and hence we can choose $\nu$ to be the outward unit normal.

Let $\Omega$ be a bounded domain in $\mathbb{H}^{n+1}$. The \emph{horospherical support function} $u_\Omega : \mathbb{S}^n \to \mathbb{R}$ of $\Omega$ is defined by (see Figure \ref{Fig-h-convex bdd domain})
\begin{equation*}
	u_\Omega(z) := \inf  \lbrace s \in \mathbb{R}:  \Omega \subset \overline{B}_z(s)  \rbrace, \quad z\in \mathbb{S}^n.
\end{equation*}
By \eqref{Bz(r)}, that is equivalent to
\begin{equation}\label{horo supp funct-def}
	u_\Omega (z) = \sup \lbrace \log \(- \metric{X}{(z,1)} \):   X \in \Omega \rbrace.
\end{equation}
We will write $u(z)$ for $u_\Omega (z)$ if there is no confusion about  the choice of the domain. 

\begin{figure}[htbp]
	\centering
	\begin{tikzpicture}[scale =1.5]
	\draw (0,0) circle(1);
	\filldraw[rounded corners, color=gray, fill opacity =0.5](-0.182,0.359) -- (-0.231,0.087) --(-0.112, -0.13) -- (0.11, -0.234) --(0.374,-0.173) -- (0.537, 0.045) -- (0.553, 0.169) -- (0.524, 0.307) -- (0.487, 0.377) -- (0.414, 0.459) -- (0.353, 0.502) -- (0.233, 0.546) -- (0.127, 0.552) -- (-0.082, 0.472) -- cycle ;
	\draw (0.289, 0.297) circle(0.586);
	\draw (0.087, -0.212) circle(0.771);
	\draw[->, thick, red] (0,0) -- (0.378,-0.926) node[below]{$z$};
	\node at (0.3,0.3){$\Omega$};
	\draw[->, blue, dashed] (0,0) -- (-0.204, 0.502);
	\node[blue] at (-0.18, 0.18){$u(z)$};
	\filldraw (0,0) circle(.01);
	\end{tikzpicture}
	\caption{Horo-convex bounded domain $\Omega$}
	\label{Fig-h-convex bdd domain}
\end{figure}
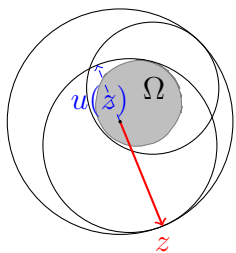

Assume that $\Omega$ is a h-convex bounded domain, and let $M =\partial \Omega$. If $X \in H_z(u(z)) \cap M$ for some $z \in \mathbb{S}^n$, then 
\begin{equation} \label{X-z,1}
	\metric{X}{(z,1)} = -e^{u(z)} = -\g(z), 
\end{equation}
where we defined $\g(z) := e^{u(z)}$. We call $ H_z(u(z))$ the \emph{supporting horosphere} of $\Omega$ at $X$. Furthermore, if $M$ is smooth, then \eqref{horosphere-X-nu} implies
\begin{equation}\label{X-nu}
	X -\nu =e^{-u(z)} (z,1) = \frac{1}{\g(z)} (z,1).  
\end{equation}
When $M = \partial \Omega$ is smooth and uniformly h-convex,  we can choose an orthonormal frame $\lbrace \partial_1, \ldots, \partial_n \rbrace$ at $p \in M$ such that the Weingarten matrix $\mathcal{W}=\( h_i{}^j \)$ is diagonal at $p$, i.e. $h_i{}^j = \kappa_i \delta_{i}{}^j$ and $\kappa_i>1$. Then we have $\partial_i \(X-\nu \)(p) = (1-\kappa_i) \partial_i \neq 0$. Hence $\ker T_p G \subset T_pM \cap \mathbb{R}(z_p,1) = 0$, which means that the horospherical Gauss map $G$ is a locally diffeomorphism. Therefore, we can parameterize $M$ by $\mathbb{S}^n$ locally via the inverse of horospherical Gauss map $G^{-1}$ when $M$ is smooth and uniformly h-convex, see \cite{ACW18}.

In the following lemma, we recover $\partial \Omega$ by using the horospherical support function $u(z)$ of the smooth uniformly h-convex bounded domain $\Omega$ (see \cite{ACW18, EGM09}).

\begin{lem}\label{lem-X,nu,di X}
	Assume that $M = \partial \Omega$ is a smooth uniformly h-convex hypersurface in $\mathbb{H}^{n+1}$ with horospherical support function $u(z)$. Then for each $z \in \mathbb{S}^n$, $M \cap H_z(u(z))$ only contains one point, which we denote it as $X(z)$. We let $\nu(z)$ denote the outward unit normal of $M \subset \mathbb{H}^{n+1}$ at $X(z)$ and choose normal coordinates for $\mathbb{S}^n$ around $z$ such that the coordinate frame $\lbrace e_1,\ldots,e_n \rbrace$ is an orthonormal basis at $z$. Then we have
	\begin{align}
		X(z) =& \frac{1}{2} \g(-z,1) + \frac{1}{2} \( \frac{|D\g|^2}{\g}+ \frac{1}{\g} \)(z,1) -(D \g,0), \label{X(z)}\\
		\nu(z)=& \frac{1}{2} \g(-z,1) + \frac{1}{2} \( \frac{|D\g|^2}{\g}- \frac{1}{\g} \)(z,1) -(D \g,0), \label{nu}\\
		\partial_i X(z) =& - \sum_{j=1}^n A_{ij}[\g] (e_j,0) +  \frac{\sum_{j=1}^n A_{ij} [\g] D_j \g}{\g}(z,1), \label{di X}
	\end{align}
	where we used $z$ to parameterize $M$ locally, $\g(z):= e^{u(z)}$ and  
	\begin{equation}\label{def A-phi}
		A_{ij}[\g] := \g_{ij} - \frac{1}{2} \frac{|D \g|^2}{\g} \sigma_{ij}+\frac{1}{2}\(\g-\frac{1}{\g}\)\sigma_{ij}.
	\end{equation}
\end{lem}

\begin{proof}
	For any $z \in \mathbb{S}^n$, vectors $\lbrace (-z,1), (z,1), (e_1, 0), \ldots, (e_n,0) \rbrace$ form a basis of $\mathbb{R}^{n+1,1}$. Then for any $X \in H_z(u(z)) \cap M$, we can assume that
	\begin{equation*}
		X =\alpha (-z,1) + \beta(z,1) + \sum_{j=1}^n \gamma_j (e_j,0),
	\end{equation*}
	where $\alpha$, $\beta$ and $\gamma_j$ are real numbers. By \eqref{X-z,1}, we have $-2\alpha =- e^u$, which induces $\alpha = \frac{1}{2}e^u$. Differentiating \eqref{X-z,1} with respect to $e_i$, we obtain
	\begin{equation*}
		\metric{\partial_i X}{(z,1)} + \metric{X}{(e_i,0)} = -D_i \g.
	\end{equation*}
	This together with the fact $\metric{\partial_i X}{X-\nu} =0$ and \eqref{X-nu} gives
	\begin{equation*}
		\metric{X}{(e_i,0)} =-D_i\g- \g\metric{\partial_i X}{X-\nu}= -D_i \g.
	\end{equation*}
	Therefore $\gamma_j = -D_j \g$. Now we arrive at
	\begin{equation*}
		X = \frac{1}{2} \g (-z,1) + \beta (z,1) - (D\g,0).
	\end{equation*}
	Because of $X \in \mathbb{H}^{n+1} \subset \mathbb{R}^{n+1,1}$, we have
	\begin{equation*}
		-1 = \metric{X}{X} = |D \g|^2 +\beta \g \metric{(z,1)}{(-z,1)}
		=|D \g|^2 - 2\beta \g,
	\end{equation*}
	which induces $\beta = \frac{1}{2} \frac{|D \g|^2}{\g} + \frac{1}{2\g}$. Then 
	\begin{equation*}
		X= \frac{1}{2} \g(-z,1) +\frac{1}{2} \( \frac{|D \g|^2}{\g}+ \frac{1}{\g} \)(z,1)- (D \g,0).
	\end{equation*}
	It follows that $X(z)$ is uniquely determined by $\g(z)$, and \eqref{X(z)} is proved. Formula \eqref{nu} follows by substituting \eqref{X(z)} into \eqref{X-nu}.
	
	Differentiating \eqref{X(z)} with respect to $e_i$, we have
	\begin{align}
		\partial_i X =& 
		\frac{1}{2}\partial_i \( \g(-z,1) \) + \frac{1}{2} \partial_i \(  \( \frac{|D \g|^2}{\g}+ \frac{1}{\g} \)(z,1)   \) -  \sum_{j=1}^n \partial_i \(\g_j e_j,0 \) \nonumber \\
		=&\frac{1}{2} \g_i (-z,1) - \frac{1}{2} \g(e_i,0)
		+ \frac{1}{2} \partial_i \(\frac{|D \g|^2}{\g} + \frac{1}{\g} \)(z,1) 
		+ \frac{1}{2} \( \frac{|D \g|^2}{\g}+ \frac{1}{\g} \)(e_i,0) \nonumber\\
		&- \sum_{j=1}^n \g_{ij}\(e_j, 0 \)+ \g_i (z,0) \nonumber\\
		=&- \sum_{j=1}^n \( \g_{ij} - \frac{1}{2} \frac{|D \g|^2}{\g} \delta_{ij}+ \frac{1}{2} \(\g -\frac{1}{\g}\)\delta_{ij} \)\( e_j,0\)+ \partial_i \left(\frac{1}{2}\frac{|D \g|^2}{\g} + \frac{1}{2}\(\g + \frac{1}{\g}\) \right)(z,1)  \nonumber\\
		=&- \sum_{j=1}^n A_{ij}[\g] (e_j,0) +  \partial_i \left(\frac{1}{2}\frac{|D \g|^2}{\g} + \frac{1}{2}\(\g + \frac{1}{\g}\) \right)(z,1),\label{di X-1}
	\end{align}
	where we used $\partial_i (e_j,0) = -\delta_{ij} (z,0)$ in the second equality. A direct calculation shows
	\begin{align}
		\partial_i \(\frac{1}{2}\frac{|D \g|^2}{\g} + \frac{1}{2}\(\g + \frac{1}{\g}\)  \)
		=&
		\frac{\sum_{j=1}^n \g_j \g_{ji}}{\g} - \frac{1}{2} \frac{|D \g|^2}{\g^2} \g_i + \frac{1}{2}\g_i
		-\frac{1}{2} \frac{\g_i}{\g^2} \nonumber\\
		=&\sum_{j=1}^n\frac{\g_j}{\g} \( \g_{ij} - \frac{1}{2} \frac{|D \g|^2}{\g} \delta_{ij} + \frac{1}{2} \(\g- \frac{1}{\g}\)\delta_{ij} \)  \nonumber\\
		=&\sum_{j=1}^n \frac{A_{ij} [\g ] \g_j }{\g}.\label{di X-2}
	\end{align}
	Then \eqref{di X} follows by substituting \eqref{di X-2} into \eqref{di X-1}. We complete the proof of Lemma \ref{lem-X,nu,di X}. 
\end{proof}

As we have proved that $G$ is bijective and locally diffeomorphic for the smooth closed and uniformly h-convex hypersurface $M$, we have that  $G : M \to \mathbb{S}^n$ is a diffeomorphism. In the context, we always assume that $M$ is parameterized by the unit sphere $\mathbb{S}^n$, i.e., the image of the horospherical Gauss map. 

In the following lemma, we show the geometric property of $A_{ij}[\g(z)]$ defined in \eqref{def A-phi}.

\begin{lem}\label{lem-shifted curvature-support function}
	Using the same notations and assumptions as in Lemma \ref{lem-X,nu,di X},
	it holds that
	\begin{equation}\label{shifted curvature-support function}
		(h_i{}^j - \delta_i{}^j ) A_{jk} [\g(z)] = \frac{1}{\g(z)} \sigma_{ik},
	\end{equation}
	where $\( h_i{}^j \)$ was evaluated at $X(z) \in M$. The above formula \eqref{shifted curvature-support function} shows that, if a smooth closed hypersurface $M$ is uniformly h-convex, then $A_{ij}[\g(z)] >0$ on $\mathbb{S}^n$.
\end{lem}

\begin{proof}
	Differentiating \eqref{X-nu} with respect to $e_i$, we have
	\begin{equation*}
		\partial_i X - \partial_i \nu = - \frac{\g_i}{\g^2}(z,1) + \frac{1}{\g} (e_i,0). 
	\end{equation*}
	Taking the inner product of the both sides with $(e_k,0)$ gives
	\begin{equation*}
		(h_i{}^j -\delta_i{}^j) \metric{-\partial_j X}{(e_k,0)} = \frac{1}{\g} \sigma_{ik}.
	\end{equation*} 
	Then \eqref{shifted curvature-support function} follows by substituting \eqref{di X} into the above formula. We complete the proof of Lemma \ref{lem-shifted curvature-support function}.
\end{proof}

\begin{cor}
	Let $M$ be a smooth uniformly h-convex hypersurface in $\mathbb{H}^{n+1}$. Then, via the horospherical Gauss map $G: M \to \mathbb{S}^n$, we have
	\begin{equation} \label{rel-area element}
		d\mu = \det A [\g(z)] d \sigma,
	\end{equation}
	where $d\mu$ denoted the area element of $M$ at $X(z)= G^{-1}(z)$, $d\sigma$ denoted the area element of $\mathbb{S}^n$ at $z$, and $A [\g(z)]$ was defined in \eqref{def A-phi}.
\end{cor}

\begin{proof}
	Since $\metric{(z,1)}{(z,1)} =0$ and $\metric{(z,1)}{(e_i,0)} =0$ for any $e_i \in T_z \mathbb{S}^n$, formula \eqref{di X} implies
	\begin{equation}\label{gij-A-phi}
		g_{ij} :=\metric{\partial_i X}{\partial_j X}
		= \sum_{k=1}^n A_{ik}[\g]A_{kj} [\g].
	\end{equation}
	Since we have proved that $A_{ij}[\g]$ is positive definite in Lemma \ref{lem-shifted curvature-support function}, the area element of $M$ can be given by
	\begin{equation*}
		d\mu = \sqrt{\det g_{ij}} d\sigma= \sqrt{\det A^2 [\g]} d\sigma = \det A[\g] d\sigma.
	\end{equation*}
	Thus we proved formula \eqref{rel-area element}.
\end{proof}

For the following Corollary \ref{cor-support-construct-domain}, readers may refer to \cite[Section 5.4]{ACW18}.

\begin{cor}\label{cor-support-construct-domain}
	Given a smooth function $u(z)$ defined on $\mathbb{S}^n$, we let $\g(z) = e^{u(z)}$. Then $u(z)$ is the horospherical support function of some smooth uniformly h-convex bounded domain $\Omega \subset \mathbb{H}^{n+1}$ if and only if
	$A_{ij}[\g(z)] >0$ on $\mathbb{S}^n$. Furthermore, domain $\Omega$ is uniquely determined by $u(z)$.
\end{cor}

\begin{proof}
	The uniqueness of $\Omega$ follows from Lemma \ref{lem-X,nu,di X}. Hence we only need to prove the existence of such a domain $\Omega$ with horospherical support function $u(z)$. We let $X(z)$ be defined as in \eqref{X(z)}. From \eqref{gij-A-phi}, the tensor $g_{ij} = \metric{\partial_i X}{\partial_j X}$ is positive definite provided that $A_{ij}[\g]$ is non-singular. Hence the map $X(z)$ is an immersion from $\mathbb{S}^n$ to $\mathbb{H}^{n+1}$ with normal vector field $\nu(z)$. Note that \eqref{shifted curvature-support function} implies that $\mathcal{W} - I= \(h_i{}^j - \delta_{i}{}^j\)$ is positive definite at $X(z)$ if and only if $A_{ij}[\g]$ is positive definite at $z$. Thus, it suffices to show that $X(z)$ is an embedding provided that $A_{ij}[\g(z)]>0$.
	
	When $n \geq 2$, the Gauss equation implies that the sectional curvatures of the range of $X(z)$ are positive provided that $\mathcal{W} - I>0$. Then we know that $X(z)$ is an embedding by a theorem of do Carmo and Warner \cite[Section 5]{dW70}.
	
	When $n=1$, we identify $z \in \mathbb{S}^1$ with $t \in [0,2\pi)$ by $z = (\cos t, \sin t)$. Let $\gamma$ be the closed curve generated by $X(t)$, and let $\widehat{\gamma}$ be the planar curve generated by $x(t)$, where $x(t) = (x_0(t), x_1(t))$ is the spacial component of $X(t)$. Define  $\widehat{\nu}(t)$ by 
	\begin{equation*}
		\widehat{\nu}(t) = \frac{-\cos t\(1, \frac{\g'}{\g}\) + \sin t \( \frac{\g'}{\g},-1 \)}{\(1+ \(\frac{\g'}{\g}\)^2\)^\frac{1}{2} }.
	\end{equation*}   
	Formula \eqref{di X} shows that
	\begin{align*}
		\frac{d}{dt} x(t) =& -A[\g] \frac{d}{dt} z+ A[\g]\frac{\g'}{\g}z\\
		=&-A[\g] (-\sin t, \cos t) +A[\g] \frac{\g'}{\g} (\cos t, \sin t)\\
		=&A[\g] \cos t \( \frac{\g'}{\g},-1\) + A[\g] \sin t \(1, \frac{\g'}{\g}\).
	\end{align*}
	Thus $\left| \frac{d}{dt} x(t)\right| = A[\g] \( 1+ (\frac{\g'}{\g})^2\)^{\frac{1}{2}}$, and $\widehat{\nu}(t)$ is the outward unit normal of $\widehat{\gamma}$ at $x(t)$. Then the curvature of $\widehat{\gamma}$ at $x(t)$ is given by
	\begin{equation}\label{kappa-hat-t}
		\widehat{\kappa}(t) = -\metric{\frac{d^2}{dt^2}x(t)}{\widehat{\nu}(t)}\left|\frac{d}{dt} x(t) \right|^{-2} 
		= \frac{A[\g]\(\(\frac{\g'}{\g}\)'+ \(\frac{\g'}{\g}\)^2+1 \)}{A[\g]^2 \(  1+ (\frac{\g'}{\g})^2 \)^{\frac{3}{2}}}
		= \frac{\g''+\g}{\g A[\g]\(1 + \(\frac{\g'}{\g}\)^2\)^\frac{3}{2}}. 
	\end{equation}
	By \eqref{def A-phi}, the assumption $A[\g]>0$ implies
	\begin{equation*}
		\g'' +\g > \frac{1}{2} \frac{\( \g'\)^2}{\g}+ \frac{1}{2} \(\g+\g^{-1}\) >0.
	\end{equation*} 
	Hence $\widehat{\kappa}(t)>0$ for all $t \in [0,2\pi)$. Then by using \eqref{kappa-hat-t}, we have 
	\begin{align*}
		\int_{\gamma} |\widehat{\kappa}| d\widehat{\mu}
		=& \int_0^{2\pi} \widehat{\kappa} \left| \frac{d}{dt} x(t) \right| dt\\
		=&  \int_0^{2\pi}\frac{\g''+\g}{ \g \(1+  \(\frac{\g'}{\g} \)^2\) } dt\\
		=& \int_0^{2\pi} \frac{1+ \( \frac{\g'}{\g} \)^2+ \( \frac{\g'}{\g} \)'}{ 1+  \(\frac{\g'}{\g} \)^2  } dt\\
		=&  2\pi + \left. \arctan \( \frac{\g'(t)}{\g(t)} \)\right\vert_{t=0}^{t=2\pi} = 2\pi. 
	\end{align*} 
	Then Fenchel's theorem \cite{Fen29} implies that $\widehat{\gamma}$ must be an embedded, closed convex curve in $\mathbb{R}^2$, in particular, $\widehat{\gamma}$ is a simple closed curve. Since the hyperboloid model of $\mathbb{H}^2$ can be viewed as a graph defined on $\mathbb{R}^2 \times \{0\}$ in $\mathbb{R}^{2,1}$, it is easy to see that the vertical projection gives a diffeomorphism from $\mathbb{H}^2$ to $\mathbb{R}^2$ and maps $\gamma$ to $\widehat{\gamma}$. Consequently, the curve $\gamma$ is an embedding from $\mathbb{S}^1$ to $\mathbb{H}^2$.
	 
	Therefore, we conclude that $X(z)$ is an embedding from $\mathbb{S}^n$ to the boundary of a uniformly h-convex bounded domain $\Omega$ if and only if $A_{ij}[\g] >0$ on $\mathbb{S}^n$. This completes the proof of Corollary \ref{cor-support-construct-domain}.
\end{proof}

\subsection{Horospherical Wulff shape} $\ $

In Euclidean space $\mathbb{R}^{n+1}$, the Wulff shape is defined by the intersection of a family of half spaces, which are determined by a given positive continuous function on $\mathbb{S}^n$. In hyperbolic space $\mathbb{H}^{n+1}$, by replacing the half spaces by horo-balls in the above argument, we can define the \emph{horospherical Wulff shape} of a continuous function $u(z)$ defined on $\mathbb{S}^n$. Remark that we do not require $u(z)$ to be positive in the following definition. Recall the definition of $B_z(r)$ in \eqref{Bz(r)}.
\begin{defn}\label{def-horo-Wulff shape}
	Given any continuous function $u(z)$ defined on $\mathbb{S}^n$, the horospherical Aleksandrov body (or horospherical Wulff shape) of $u(z)$ is defined by
	\begin{equation*}
		\Omega:= \bigcap_{z \in \mathbb{S}^n} \overline{B}_z(u(z)) = \bigcap_{z \in \mathbb{S}^n} \{ X \in \mathbb{H}^{n+1}:  \log \( -\metric{X}{z,1}\) \leq u(z) \}.
	\end{equation*}
\end{defn}

Now we use Definition \ref{def-horo-Wulff shape} to characterize points in $\mathbb{H}^{n+1}$.
\begin{lem}\label{lem-horospp of point}
	Given a smooth function $u(z)$ defined on $\mathbb{S}^n$, we let $\g(z) = e^{u(z)}$. Then $u(z)$ is the horospherical support function of a point $X$ in  $\mathbb{H}^{n+1}$ if and only if $A_{ij}[\g(z)] =0$. Furthermore,  $X$ is the horospherical Wulff shape of $u(z)$.
\end{lem}

\begin{proof}
	Assume that $u(z)$ is a smooth function that satisfies $A_{ij} [\g(z)] =0$ for all $z \in \mathbb{S}^n$. Let us prove that $\cap_{z \in \mathbb{S}^n} \overline{B}_z(u(z))$ is a point in $\mathbb{H}^{n+1}$.
	
	Let $X(z)$ be a vector-valued function on $\mathbb{S}^n$ defined by expression \eqref{X(z)}. Since $A_{ij}[\g] = 0$, formula \eqref{di X} shows that $X(z)$ is constant, and we denote it by $X= (x,x_{n+1})$. It is easy to check that 
	\begin{equation*}
		\metric{-X}{(z,1)} = \metric{-X(z)}{(z,1)} = \g(z).
	\end{equation*}
	Thus $u(z)$ is the horospherical support function of $X$, and then $X \in \cap_{z \in \mathbb{S}^n} \overline{B}_z(u(z))$.
	
	If there exists a point $X'= (x',x_{n+1}')$ in $\cap_{z \in \mathbb{S}^n} \overline{B}_z(u(z))$ with $X' \neq X$, then we let $z' = \frac{x-x'}{|x-x'|}$. Hence
	\begin{align*}
		\metric{-X'}{(z',1)}- \metric{-X}{(z',1)}
		=& \metric{X-X'}{(z',1)}\\
		=& \metric{x-x'}{z'}- (x_{n+1}-x_{n+1}')\\
		\geq&  \metric{x-x'}{z'} -|x_{n+1} - x_{n+1}'|\\
		=& |x-x'| - \left| \( |x|^2+1 \)^{\frac{1}{2}} - \(|x'|^2+1\)^{\frac{1}{2}} \right|\\
		\geq& \left||x|- |x'|\right|\(1- \frac{|x|+|x'|}{ \( |x|^2+1 \)^{\frac{1}{2}} + \(|x'|^2+1\)^{\frac{1}{2}}   }\)>0.
	\end{align*}
	Consequently, 
	\begin{equation*}
		\metric{-X'}{(z',1)} > \metric{-X}{(z',1)} = \g(z'),
	\end{equation*}
	which contradicts the assumption $X' \in \cap_{z \in \mathbb{S}^n} \overline{B}_z(u(z))$. Thus $X =  \cap_{z \in \mathbb{S}^n} \overline{B}_z(u(z))$.
	
	On the other hand, for any point $X= (x, x_{n+1}) \in \mathbb{H}^{n+1}$, let us prove that  $A_{ij}[\g_X] = 0$. 
	
	Using \eqref{X-z,1}, we have
	\begin{equation*}
		\g_X(z) = \metric{-X}{(z,1)} = -\metric{x}{z} + x_{n+1}.
	\end{equation*}
	Then
	\begin{equation}\label{phi-i,phi-ij-point}
		D_i \g_X = -\metric{x}{e_i}, \quad D_j D_i \g_X = -D_j \metric{x}{e_i} =\metric{x}{z} \delta_{ij}.
	\end{equation}
	Therefore
	\begin{align*}
		A_{ij} [\g_X(z)] =& D_j D_i \g_X- \frac{1}{2} \frac{|D \g_X|^2}{\g_X}\delta_{ij} + \frac{1}{2} \( \g_X - \frac{1}{\g_X} \) \delta_{ij}\\
		=&\metric{x}{z} \delta_{ij} - \frac{1}{2} \frac{\sum_{i=1}^n \metric{x}{e_i}^2}{-\metric{x}{z} +x_{n+1}} \delta_{ij}+ \frac{1}{2} \( \(-\metric{x}{z} +x_{n+1}\) - \frac{1}{-\metric{x}{z} +x_{n+1}} \) \delta_{ij}\\
		=& \metric{x}{z} \delta_{ij} - \frac{1}{2} \frac{|x|^2 -\metric{x}{z}^2 +1}{-\metric{x}{z} +x_{n+1}} \delta_{ij} + \frac{1}{2} \( -\metric{x}{z} +x_{n+1} \)\delta_{ij}\\
		=& \metric{x}{z} \delta_{ij} -\frac{1}{2} \( \metric{x}{z} + x_{n+1}\) \delta_{ij} +  \frac{1}{2} \( -\metric{x}{z} +x_{n+1} \)\delta_{ij}
		=0,
	\end{align*}
	where we used $\sum_{i=1}^n \metric{x}{e_i}^2 = |x|^2 - \metric{x}{z}^2$ in the third equality and used $|x|^2 +1 = x_{n+1}^2$ in the last equality. Hence we obtain the desired equality $A_{ij}[\g_X(z)] = 0$ for any point $X \in \mathbb{H}^{n+1}$.
	
	We complete the proof of Lemma \ref{lem-horospp of point}.
\end{proof}

Without the smooth assumption of the domain $\Omega$ in Lemma \ref{lem-X,nu,di X}, we will show in the following Lemma \ref{lem-bdy of h-convex domain-by X(z)} that the expression \eqref{X(z)} is also valid to describe $\partial \Omega$.

\begin{lem}\label{lem-bdy of h-convex domain-by X(z)}
	Let $\Omega$ be a h-convex bounded domain with smooth horospherical support function $u(z)$. 
	Then the expression \eqref{X(z)} gives a surjective map from $\mathbb{S}^n $ to $\partial \Omega$. 
\end{lem}

\begin{proof}
	By the h-convexity of $\Omega$, for any $Y \in \partial \Omega$, there exists $z_Y \in \mathbb{S}^n$ such that $Y \in \partial \Omega \cap H_{z_Y} \( u(z_Y)\)$. By use of the expression \eqref{X(z)}, let us prove $Y=X(z_Y)$.
		 
	By using \eqref{horosphere-support function} and \eqref{Bz(r)}, we have $\metric{-Y}{(z_Y,1 )} = \g(z_Y)$, and $\metric{-Y}{(z,1)} \leq \g(z)$ for all $z \in \mathbb{S}^n$. Hence 
	\begin{equation*}
		\g (z) \geq \g(z_Y) + \metric{-Y}{(z-z_Y, 0)}, \quad \forall \ z\in \mathbb{S}^n.
	\end{equation*}
	This together with the smoothness of $\g(z)$ deduces that the projection of $-Y$ on $T_{z_Y} \mathbb{S}^n$ is exactly $D \g (z_Y)$. Hence we can assume 
	\begin{equation}\label{express of Y-new}
		Y =(y, y_{n+1}) =  \(-D \g(z_Y) +\alpha z_Y, \beta\)
	\end{equation}
	for some constants $\alpha \in \mathbb{R}$ and $\beta \geq 1$. Since $Y \in  H_{z_Y} \( u(z_Y)\)$, we have from \eqref{X-z,1} that
	\begin{equation*}
		\g(z_Y) = -\metric{Y}{(z_Y,1)} = \beta- \alpha.
	\end{equation*}
	Besides, the fact $Y \in \mathbb{H}^{n+1}$ implies
	\begin{equation*}
		-1 = \metric{Y}{Y}=|D \g(z_Y)|^2+\alpha^2-\beta^2.
	\end{equation*}
	Solving the above two equations gives
	\begin{equation*}
		\alpha= \frac{1}{2} \frac{|D \g(z_Y)|^2}{\g(z_Y)}- \frac{1}{2} \(\g(z_Y)- \frac{1}{\g(z_Y)}\), \quad
		\beta = \frac{1}{2} \frac{|D \g(z_Y)|^2}{\g(z_Y)}+ \frac{1}{2} \(\g(z_Y)+ \frac{1}{\g(z_Y)}\).
	\end{equation*}
	Substituting the above $\alpha, \beta$ into \eqref{express of Y-new} and using expression \eqref{X(z)}, we have $Y =X(z_Y)$. Consequently, we obtain $\partial \Omega \subset X(\mathbb{S}^n)$.
	
	Now, it suffices to show $X(\mathbb{S}^n) \subset \partial \Omega$. For any $z_0 \in \mathbb{S}^n$, there exists $Y_0 \in \partial \Omega \cap H_{z_0}(u(z_0))$. The above argument also shows $Y_0 = X(z_0)$. Thus, we have that the image of expression \eqref{X(z)} lies on $\partial \Omega$.
	
	Putting the above facts together, we complete the proof of Lemma \ref{lem-bdy of h-convex domain-by X(z)}.
\end{proof}

 We are now in a position to show that the domain $\Omega$ in Corollary \ref{cor-support-construct-domain} is exactly the horospherical Wulff shape of $u(z)$. 

\begin{cor}\label{cor-Aij>0 Omega=Wulff shape}
	Let $\Omega$ be a h-convex bounded domain with smooth horospherical support function $u(z)$.
	Assume that $A_{ij} [\g(z)]>0$ for all $z \in \mathbb{S}^n$, where $\g(z) =e^{u(z)}$. Then $\Omega = \cap_{z \in \mathbb{S}^n} \overline{B}_z(u(z))$, and  expression \eqref{X(z)} gives a diffeomorphism from $\mathbb{S}^n$ to $\partial \Omega$.
\end{cor}

\begin{proof}
	By Corollary \ref{cor-support-construct-domain}, we can let $\widetilde{ \Omega}$ denote the unique smooth uniformly h-convex bounded domain with horospherical support function $u(z)$. By \eqref{horo supp funct-def}, we have  $\widetilde{ \Omega} \subset \cap_{z \in \mathbb{S}^n} \overline{B}_z(u(z))$. Consequently, the horospherical support function of $\cap_{z \in \mathbb{S}^n} \overline{B}_z(u(z))$ is exactly $u(z)$. So far, the domains $\Omega$, $\cap_{z \in \mathbb{S}^n} \overline{B}_z(u(z))$ and $\widetilde{ \Omega}$ are equipped with the same smooth horospherical support function $u(z)$. Then we can apply Lemma \ref{lem-bdy of h-convex domain-by X(z)} to get 
	\begin{equation*}
		\Omega = \cap_{z \in \mathbb{S}^n} \overline{B}_z(u(z))=\widetilde{ \Omega}.
	\end{equation*}
	Therefore, Corollary \ref{cor-Aij>0 Omega=Wulff shape} follows from Corollary \ref{cor-support-construct-domain}.
\end{proof}

In the following Proposition \ref{prop-Aij >=0}, we will extend the assumption  $A_{ij} [\g(z)]>0$ in Corollary \ref{cor-Aij>0 Omega=Wulff shape} to  $A_{ij} [\g(z)] \geq 0$, and we will prove the $1$-$1$ correspondence between smooth function $u(z)$ satisfying  $A_{ij} [\g(z)] \geq 0$ and the h-convex domain with smooth horospherical support function $u(z)$.

\begin{prop}\label{prop-Aij >=0}
	Let $u(z)$ be a smooth function on $\mathbb{S}^n$ and $\g(z) = e^{u(z)}$. Assume that $A_{ij} [\g(z)] \geq 0$ for all $z \in \mathbb{S}^n$. Then the horospherical Wulff shape $ \cap_{z \in \mathbb{S}^n} \overline{B}_z(u(z))$ is the unique h-convex domain with horospherical support function $u(z)$. Furthermore, expression \eqref{X(z)} gives a surjective map from $\mathbb{S}^n $ to $\partial \(\cap_{z \in \mathbb{S}^n} \overline{B}_z(u(z))\)$. 
\end{prop}

\begin{proof}
	For each integer $k \geq 1$, we define $u_k (z) := \log \(\g(z) + \frac{1}{k} \)$. Then $\g_k(z) := e^{u_k(z)}= \g(z) +\frac{1}{k}$ converges smoothly and uniformly to $\g(z)$ as $k \to +\infty$. By a direct calculation, we have
    \begin{align}
    	A_{ij} [\g_k(z)] =& \(\g_k\)_{ij} - \frac{1}{2} \frac{|D \g_k|^2}{\g_k} \delta_{ij} + \frac{1}{2}\(\g_k- \frac{1}{\g_k}\) \delta_{ij} \nonumber\\
    	=& \g_{ij}- \frac{1}{2}\frac{|D \g|^2}{\g +\frac{1}{k}}\delta_{ij} + \frac{1}{2} \( \g+ \frac{1}{k}- \frac{1}{\g+ \frac{1}{k}}\delta_{ij} \) \nonumber\\
    	>& \g_{ij} - \frac{1}{2} \frac{|D \g|^2}{\g} \delta_{ij} + \frac{1}{2} \( \g- \frac{1}{\g}\) \delta_{ij} \nonumber\\
    	=& A_{ij} [\g(z)] =0.\label{Aij-phik}
    \end{align}
    For any integer $k \geq 1$, we know from Corollary \ref{cor-support-construct-domain} that $u_k(z)$ uniquely determines a smooth uniformly h-convex bounded domain $\Omega_k$. At the same time, Corollary \ref{cor-Aij>0 Omega=Wulff shape} implies that $\Omega_k = \cap_{z\in \mathbb{S}^n} \overline{B}_z (u_k(z))$.
    	 
    Let $\Omega := \cap_{z\in \mathbb{S}^n} \overline{B}_z(u(z))$. We will show that the horospherical support function of $\Omega$ is $u(z)$.
    	 
    By the above argument, we have
    \begin{equation*}
    	\Omega = \bigcap_{z \in \mathbb{S}^n} \overline{B}_z (u(z)) 
    	= \bigcap_{z \in \mathbb{S}^n} \bigcap_{k=1}^{+ \infty}\overline{B}_z (u_k(z))
    	=\bigcap_{k=1}^{+ \infty} \bigcap_{z \in \mathbb{S}^n} \overline{B}_z (u_k(z))
    	= \bigcap_{k=1}^{+ \infty} \Omega_k.
    \end{equation*}
    Since $\Omega_i \subset \Omega_j$ for all $i>j \geq 1$,  $\Omega$ is nonempty. Replacing $\g(z)$ and $X(z)$ by $\g_k(z)$ and $X_k(z)$ respectively in expression \eqref{X(z)}, we know that $\partial \Omega_k$ is parameterized by $X_k (z)$. Taking $k \to +\infty$ and by the definition of $\g_k(z)$, we have that $X_k(z)$ converges smoothly and uniformly to $X(z)$ for all $z \in \mathbb{S}^n$, and hence $X(z) \in \Omega$. Besides, it is easy to see that $\metric{-X(z)}{(z,1)} = \g(z)$, which implies that the horospherical support function of $\Omega$ is $u(z)$. Note that Lemma \ref{lem-bdy of h-convex domain-by X(z)} implies that the above $\Omega$ is the unique h-convex domain with horospherical support function $u(z)$. This completes the proof of Proposition \ref{prop-Aij >=0}.
 \end{proof}

\section{Horospherically convex hypersurfaces}\label{sec-useful lemmas}

\subsection{Symmetric functions} $\ $

 For a symmetric function $s$ defined on $\mathbb{R}^n$, there is a $GL(n)$-invariant function $S$ defined on the space of $n \times n$ symmetric matrices ${\rm Sym}(n)$, such that $s \(\lambda\(A\)\) = S\(A\)$, where $\lambda\(A\) = \( \lambda_1,\ldots, \lambda_n\)$ are the eigenvalues of $A$.  Then we can define 
\begin{equation*}
\dot{S}^{pq} = \frac{\partial S}{\partial A_{pq}}, \quad \ddot{S}^{pq, rs} = \frac{\partial^2 S}{\partial A_{rs} \partial A_{pq}}. 
\end{equation*} 
At $A = {\rm diag} ( \lambda_1, \ldots, \lambda_n )$, we have 
\begin{equation}\label{S-pq s-p}
\dot{S}^{pq} = \frac{\partial s}{\partial \lambda_p} \delta_p{}^q.
\end{equation}
In the context, we will write $S$ for $s$ if there is no confusion.

Let $n \geq 1$ and $1 \leq m \leq n$ be integers, and let $\lambda = \(\lambda_1, \ldots, \lambda_n\)$ be a point in $\mathbb{R}^n$. The normalized $m$-th elementary symmetric function $p_m$ is defined by
\begin{equation*}
p_m(\lambda) := \frac{1}{C_n^m} \sigma_m (\lambda) = \frac{1}{C_n^m} \sum_{1 \leq i_1< \cdots< i_m \leq n} \lambda_{i_1} \cdots \lambda_{i_m}.
\end{equation*}
It is convenience to set $p_{0}(\lambda) = 1$. By the above argument, $p_m$ induces a function $p_m(A):=p_m (\lambda(A))$ defined on ${\rm Sym}(n)$. The following Lemma \ref{lem-pm-ij-sym matrix} can be found in \cite{HLW20}.

\begin{lem}\label{lem-pm-ij-sym matrix}
	If $1 \leq m \leq n$, then
	\begin{align}
		\dot{p}_m^{ij}\delta_{i}{}^j = m p_{m-1}, \label{pm-ij delt-ij}\\
		\dot{p}_m^{ij} A_{ij} = m p_m.\label{pm-ij A-ij}
	\end{align}
	If $0 \leq k<l \leq n$ and $G(A) := \(\frac{p_l \(\lambda(A)\)}{p_k \(\lambda(A)\)} \)^{\frac{1}{l-k}}$, then
	\begin{equation}\label{G=Gij Aij}
		G= \dot{G}^{ij} A_{ij}.
	\end{equation}
\end{lem}

Define the positive cone $\Gamma^+ \subset \mathbb{R}^n$ by
\begin{equation*}
	\Gamma^{+} : =\{ \lambda=\(\lambda_1, \dots, \lambda_n\):  \lambda_i>0, \, 1 \leq i \leq n   \}.
\end{equation*}
The following Newton-MacLaurin inequalities are well known, see e.g. \cite{WX14}.

\begin{lem}\label{lem-Newton-MacLaurin ineq}
	For $1 \leq k <l \leq n$ and $\lambda = (\lambda_1, \ldots, \lambda_n)\in \Gamma^+$, the following inequalities hold:
	\begin{align}
		p_k(\lambda) p_{l-1}(\lambda) \geq& p_{k-1}(\lambda) p_l(\lambda). \label{Newton ineq}\\
		p_k(\lambda) \geq& p_l^{\frac{k}{l}}(\lambda). \label{McLau ineq}
	\end{align}
	Equality  holds in the both \eqref{Newton ineq} and \eqref{McLau ineq} if and only if $\lambda_i = \lambda_j$ for all $1 \leq i,j \leq n$.
\end{lem}

\begin{lem}\label{lem-Gij delt-ij geq 1}
	Let  $0 \leq k <l \leq n $, $G(A) :=\(\frac{p_l(\lambda(A))}{p_k(\lambda(A))}\)^{\frac{1}{l-k}}$, and $\lambda(A) \in \Gamma^+$. Then
	\begin{equation}\label{Gij delt-ij in lemma}
		\dot{G}^{ij} \delta_i{}^j \geq 1,
	\end{equation}
	with equality if and only if $\lambda_1 (A) = \cdots =\lambda_n(A)>0$.
\end{lem}

\begin{proof}
	By \eqref{Newton ineq}, we have	$\frac{p_{k+1}}{p_k} \geq \cdots \geq \frac{p_l}{p_{l-1}}$. Then
	\begin{equation*}
		\frac{p_l}{p_k} = \prod_{i=k}^{l-1} \frac{p_{i+1}}{p_i} \geq \( \frac{p_l}{p_{l-1}}\)^{l-k},
	\end{equation*}
	which is equivalent to
	\begin{equation}\label{p_l-1 geq p_k p_l }
		p_kp_{l-1} \geq p_k^{1+\frac{1}{l-k}} p_{l}^{1- \frac{1}{l-k}}.
	\end{equation}
	Using \eqref{pm-ij delt-ij}, \eqref{Newton ineq} and  \eqref{p_l-1 geq p_k p_l }, we have
	\begin{align}
		\dot{G}^{ij} \delta_i{}^j =& \frac{1}{l-k}\( \frac{p_l}{p_k} \)^{\frac{1}{l-k} -1}\frac{p_k \dot{p}_l^{ij} \delta_{i}{}^j - p_l \dot{p}_k^{ij} \delta_i{}^j}{p_k^2} \nonumber\\
		=& \frac{1}{l-k}\( \frac{p_l}{p_k} \)^{\frac{1}{l-k} -1} \frac{l p_k p_{l-1} - k p_l p_{k-1}}{p_k^2} \nonumber\\
		\geq& \frac{1}{l-k}\( \frac{p_l}{p_k} \)^{\frac{1}{l-k} -1} \frac{\(l-k\)p_k p_{l-1} }{p_k^2} \nonumber\\
		\geq& \( \frac{p_l}{p_k} \)^{\frac{1}{l-k} -1} \frac{ p_k^{1+\frac{1}{l-k}} p_{l}^{1- \frac{1}{l-k}}}{p_k^2}=1. \label{Gij-delt-ij}
	\end{align}
	By Lemma \ref{lem-Newton-MacLaurin ineq}, equality holds in the first inequality of \eqref{Gij-delt-ij} if and only if $\lambda_1(A) = \cdots = \lambda_n(A)$. Conversely, equality holds in \eqref{Gij-delt-ij} when $\lambda_1(A) = \cdots = \lambda_n(A)$. This completes the proof of Lemma \ref{lem-Gij delt-ij geq 1}.
\end{proof}

Let $s=s(\lambda)$ be a symmetric function defined on the positive cone $\Gamma^+$. Then the dual function $s_*$ of $s$ is defined by $s_* (\lambda) := s^{-1} (\lambda_1^{-1}, \ldots, \lambda_n^{-1}   )$. For example, if $s(\lambda) = \(\frac{p_{l}(\lambda)}{p_k(\lambda)}\)^{\frac{1}{l-k}}$ for $0 \leq k <l \leq n$ and $\lambda \in \Gamma^+$, then
\begin{equation}\label{inverse function of pl/pk}
	s_{*}(\lambda) = \( \frac{p_l (\lambda_1^{-1}, \ldots , \lambda_n^{-1})}{p_k( \lambda_1^{-1}, \ldots, \lambda_n^{-1} )} \)^{-\frac{1}{l-k}}
	=\( \frac{p_{n-l}(\lambda)}{p_n(\lambda)} \cdot \frac{p_n(\lambda)}{p_{n-k}(\lambda)} \)^{-\frac{1}{l-k}}
	=\(\frac{p_{n-k}(\lambda)}{p_{n-l}(\lambda)}\)^{\frac{1}{l-k}}.
\end{equation}
We call $s$ inverse-concave if its dual function $s_*$ is concave on $\Gamma^+$.

Andrews \cite{And07} proved the following lemma.

\begin{lem}\label{lem-mono-increasing concavity}
	Let  $0 \leq k <l \leq n $, $s(\lambda) = \(\frac{p_l(\lambda)}{p_k(\lambda)}\)^{\frac{1}{l-k}}$ and $\lambda \in \Gamma^+$. Then $s(\lambda)$ is monotone increasing and concave on $\Gamma^+$. Let $G(A) :=\(\frac{p_l(\lambda(A))}{p_k(\lambda(A))}\)^{\frac{1}{l-k}}$ and $\lambda(A) \in \Gamma^+$. Then
	\begin{equation}\label{G-ij >0}
		\dot{G }^{ij}>0.
	\end{equation}
	For any $B \in {\rm Sym}(n)$, there holds
	\begin{equation}\label{G-ij,kl <0}
		\ddot{G}^{ij, kl} B_{ij} B_{kl} \leq 0.
	\end{equation}
\end{lem}

Let $\(M^n,g\)$ be a Riemannian manifold. A $GL(n)$-invariant function $S$ defined on ${\rm Sym}(n)$ can induce a function $\widehat{S}$ defined on symmetric $(0,2)$-tensors on $M$. Precisely, for a symmetric $(0,2)$-tensor $A$ on $M$, we define
\begin{equation*}
	\widehat{S} \(A, g\) :=S \(\lambda\(g^{-\frac{1}{2}} A g^{-\frac{1}{2}} \) \)=S(\lambda(g^{-1}A)).
\end{equation*}
Due to the $GL(n)$-invariance of $S$, we have $\widehat{S} \(A, g\) = \widehat{S}\(g^{-\frac{1}{2}} A g^{-\frac{1}{2}}, I\)$. For $\widehat{S} = \widehat{S}(A, g)$, we define
\begin{equation*}
	\dot{ \widehat{S} }^{pq} = \frac{\partial \widehat{S}}{\partial A_{ij}}, \quad
	\ddot{\widehat{S}}^{pq, rs} = \frac{\partial^2 \widehat{S}}{\partial A_{rs} \partial A_{pq}},
\end{equation*} 
which are tensors on $(M,g)$. If $g=I$, then
\begin{equation*}
	\widehat{S}(A,I) = S(A), \quad
	\dot{\widehat{S}}^{pq} = \dot{S}^{pq}, \quad
	\ddot{\widehat{S}}^{pq,rs} = \ddot{S}^{pq,rs}.
\end{equation*}
In the following text, we will write $S$ for $\widehat{S}$ if there is no confusion. For more detailed  conventions and properties about curvature functions,  we refer readers to \cite[Chapter 2]{Ger06} and \cite[Section 2]{AMZ13}.

\subsection{Hypersurfaces in hyperbolic space} $ \ $

Let $M$ be a smooth hypersurface in $\mathbb{H}^{n+1}$ with second fundamental form $\(h_{ij}\)$ and principal curvatures $\kappa = \( \kappa_1, \ldots, \kappa_n\)$. We call $\tilde{h}_{ij}:= h_{ij}- g_{ij}$ the \emph{shifted second fundamental form} of $M \subset \mathbb{H}^{n+1}$. The eigenvalues $\tilde{\kappa} = \( \tilde{\kappa}_1, \ldots, \tilde{\kappa}_n \)$ of $g^{-1}\tilde{h}$ are called the \emph{shifted principal curvatures} of $M \subset \mathbb{H}^{n+1}$. It is easy to see that $\tilde{\kappa}_i = \kappa_i-1$, $i=1, \ldots, n$. For a $GL(n)$-invariant function $G$, we can view $G= G(\tilde{\kappa}) = G(\tilde{h}_{ij}, g_{ij})$. By the above argument, the following Lemma \ref{lem-pm-shifted curvature} follows directly from Lemma \ref{lem-pm-ij-sym matrix} and Lemma \ref{lem-Gij delt-ij geq 1}. 

\begin{lem}\label{lem-pm-shifted curvature}
	Let $M$ be a smooth hypersurface in $\mathbb{H}^{n+1}$ and $\tilde{\kappa}_i = \kappa_i-1$ be the shifted principal curvatures of $M$, and let $1 \leq m \leq n$ be an integer. Then
	\begin{align}
		\dot{p}_{m}^{ij} \(\tilde{\kappa} \) g_{ij} =& m p_{m-1}(\tilde{\kappa}),\label{pm-ij g-ij}\\
		\dot{p}_{m}^{ij} \( \tilde{\kappa} \) \tilde{h}_{ij} =& m p_m \(\tilde{\kappa} \).\label{pm-ij h-ij}
	\end{align}
	Let $0 \leq k < l \leq n$ and $F (\tilde{\kappa}) = \( \frac{p_l(\tilde{\kappa})}{p_k (\tilde{\kappa})} \)^{\frac{1}{l-k}}$. Then
	\begin{equation}\label{F-ij h-ij}
		F = \dot{F}^{ij} \tilde{h}_{ij}.
	\end{equation}
	In addition, if $\tilde{\kappa} \in \Gamma^+$, then
	\begin{equation}\label{F-ij delt-ij geq 1}
		\dot{F}^{ij} g_{ij} \geq 1,
	\end{equation}
	with equality if and only if $\tilde{\kappa}_i = \tilde{\kappa}_j$ for all $1 \leq i, j \leq n$.
\end{lem}

Note that the shifted second fundamental form $\(\tilde{h}_{ij} \)$ of a hypersurface $M$ in $\mathbb{H}^{n+1}$ is a Codazzi tensor, i.e. $\nabla_k \tilde{h}_{ij} = \nabla_j \tilde{h}_{ik}$. This implies that $\dot{p}_m^{ij}(\tilde{\kappa})$ is divergence-free for $1 \leq m \leq n$, i.e. $\nabla_j \dot{p}_m^{ij} (\tilde{\kappa})=0$ on $M$, see \cite{Rei73}.

The following Lemmas \ref{lem-conf-vf-direvative}--\ref{lem-shifted Minkowski formula} are well-known and can be found in \cite{HLW20}.
\begin{lem}\label{lem-conf-vf-direvative}
	Let $V = \sinh r \partial_r$ be a conformal Killing vector field in $\mathbb{H}^{n+1}$. Then
	\begin{align}
		\metric{\overline{\nabla}_X V}{Y} =& \cosh r  \metric{X}{Y}, \label{conf-vf}\\
    	\divv_{\mathbb{H}^{n+1}} V =& (n+1) \cosh r,\label{div-conf-vf}
	\end{align}
	where $\divv_{\mathbb{H}^{n+1}} V$ denoted the divergence of vector field $V$ in $\mathbb{H}^{n+1}$.
\end{lem}

\begin{lem}\label{lem-Minkowski formula}
	Let $M$ be a smooth closed hypersurface in $\mathbb{H}^{n+1}$ and $\tilde{u} = \metric{V}{\nu}$ be the classical support function of $M$. Then
	\begin{equation}\label{eq-Minkowski formula}
		\int_M \cosh r p_m(\kappa) d \mu = \int_M \tilde{u} p_{m+1}(\kappa) d\mu, \quad m=0,1,\ldots,n-1.
	\end{equation}
\end{lem}

The following Lemma \ref{lem-shifted Minkowski formula} is a direct consequence of Lemma \ref{lem-Minkowski formula}, see \cite{HLW20}. However, we will give a direct proof of it.

\begin{lem}\label{lem-shifted Minkowski formula}
	Let $M$ be a smooth closed hypersurface in $\mathbb{H}^{n+1}$ and $\tilde{\kappa}_i = \kappa_i -1$ be the shifted principal curvatures.  Then
	\begin{equation}\label{eq-shifted Minkowski formula}
		\int_{M} (\cosh r -\tilde{u}) p_m( \tilde{\kappa}) d \mu =
		\int_{M} \tilde{u} p_{m+1}(\tilde{\kappa}) d \mu, \quad
		m =0,1,\ldots, n-1.
	\end{equation}
\end{lem}

\begin{proof}
	By $\nabla_i \cosh r = \metric{V}{\partial_i X}$ and \eqref{conf-vf}, we have
	\begin{equation*}
		\nabla_j \nabla_i \cosh r = \nabla_j \metric{V}{\partial_i X}
		= \cosh r g_{ij} - \tilde{u}h_{ij}
		= \(\cosh r-\tilde{u}\) g_{ij} - \tilde{u}\tilde{h}_{ij}.
	\end{equation*}
	This together with \eqref{pm-ij g-ij} and \eqref{pm-ij h-ij} deduces
	\begin{align}
		\dot{p}_{m+1}^{ij}(\tilde{\kappa}) 	\nabla_j \nabla_i \cosh r
		=& (\cosh r -\tilde{u})\dot{p}_{m+1}^{ij}(\tilde{\kappa}) g_{ij}
		- \tilde{u} \dot{p}_{m+1}^{ij}(\tilde{\kappa}) \tilde{h}_{ij} \nonumber\\
		=& (m+1) \( (\cosh r-\tilde{u})p_m(\tilde{\kappa})- \tilde{u} p_{m+1}(\tilde{\kappa})\). \label{shifted-Min-fml-diverg-formula}
	\end{align}
	Since $\dot{p}_{m+1}^{ij}(\tilde{\kappa})$ is divergence-free, integration by parts and \eqref{shifted-Min-fml-diverg-formula} imply 
	\begin{align*}
		0=& -\frac{1}{m+1}\int_M  \nabla_j \dot{p}_{m+1}^{ij}(\tilde{\kappa})  \nabla_i \cosh r d\mu\\
		=&
		\frac{1}{m+1}\int_M \dot{p}_{m+1}^{ij}(\tilde{\kappa}) 	\nabla_j \nabla_i \cosh r d\mu\\
		=& \int_{M} (\cosh r -\tilde{u}) p_m( \tilde{\kappa}) d \mu-
		\int_{M} \tilde{u} p_{m+1}(\tilde{\kappa}) d \mu.
	\end{align*}
	This completes the proof of Lemma \ref{lem-shifted Minkowski formula}.
\end{proof}

In the classical convex geometric theory, researchers would like to transform the geometric quantities of convex bodies to functions of the support function on $\mathbb{S}^n$ by using the Gauss map. In the next part of Section \ref{sec-useful lemmas}, we will do the corresponding work in the hyperbolic setting.

For the radial function $\widehat{r}$ of a smooth convex hypersurface in $\mathbb{R}^{n+1}$, one has ``$\widehat{r} = \sqrt{\widehat{u}^2+ |D \widehat{u}|^2}$" (see \cite[Section 1.7]{Sch14}), here $\widehat{u}$ denotes the Euclidean support function. Conversely, $\widehat{u}$ can be viewed as a function defined on a smooth convex hypersurface in $\mathbb{R}^{n+1}$ by ``$\widehat{u} = \metric{\widehat{X}}{\widehat{\nu}} = \widehat{r}\metric{\partial \widehat{r}}{\widehat{\nu}}$". Motivated by the above formulas, for the smooth uniformly h-convex hypersurfaces in $\mathbb{H}^{n+1}$, we have the following Lemma \ref{lem-cosh r, tilde u, 1/phi}.

\begin{lem}\label{lem-cosh r, tilde u, 1/phi}
	Let $M$ be a smooth closed and uniformly h-convex hypersurface $M$ in $\mathbb{H}^{n+1}$,  let $r$ denote the radial function of $M$, and let $\tilde{u} = \metric{V}{\nu}$ denote the classical support function of $M$. Then
	\begin{align}
		\cosh r  =& \frac{1}{2} \frac{|D \g|^2}{\g}(z) + \frac{1}{2}\(\g(z) + \frac{1}{\g(z)}\), \label{coshr}\\
		\tilde{u}  =& \frac{1}{2} \frac{|D \g|^2}{\g}(z) + \frac{1}{2}\(\g(z) - \frac{1}{\g(z)}\), \label{tilde u}
	\end{align}
	where $\cosh r$ and $\tilde{u}$ were evaluated at $X(z) \in M$, and $\g$ was evaluated at $z \in \mathbb{S}^n$. 
	Conversely, for the horospherical support function $u(z)$ of $M$ defined on $\mathbb{S}^n$, we have
	\begin{equation}\label{1/phi, coshr-u}
		e^{-u(z)} = \frac{1}{\g(z)} = \cosh r - \tilde{u}.
	\end{equation} 
\end{lem}

\begin{proof}
	For the position vector $X(z) = G^{-1}(z) \in M$,  there exists $\theta \in \mathbb{S}^n$ such that
	\begin{align*}
		X(z) =& (\sinh r \theta, \cosh r),\\
		X(z) =& \left( \left( \frac{1}{2} \frac{|D \g|^2}{\g} - \frac{1}{2}\(\g- \frac{1}{\g} \)\right)z
		-D \g, \frac{1}{2} \frac{|D \g|^2}{\g} + \frac{1}{2} \(\g + \frac{1}{\g} \) \right),
	\end{align*}
	where the second formula is \eqref{X(z)}. Comparing these two formulas, we have
	\begin{align}
		\cosh r =& \frac{1}{2} \frac{|D \g|^2}{\g} + \frac{1}{2} \(\g + \frac{1}{\g} \), \nonumber\\
		\sinh r \metric{\theta}{z} =& \frac{1}{2} \frac{|D \g|^2}{\g} - \frac{1}{2}\(\g - \frac{1}{\g}\). \label{coshr, sinhr<theta,z>}
	\end{align}
	Thus we obtain \eqref{coshr}.
	
	Since $\partial_r = \partial_r X = (\cosh r \theta, \sinh r)$, we know $\metric{\partial_r}{X} =0$. Then by \eqref{X-nu}, we have
	\begin{align}
		\tilde{u} =& \metric{V}{\nu} = \sinh r \metric{\partial_r}{\nu} 
		=-\sinh r \metric{\partial_r}{X-\nu}  \nonumber\\
		=& -\frac{\sinh r }{\g} \metric{\partial_r}{(z,1)}
		= \frac{1}{\g}(\sinh^2 r - \cosh r \sinh r\metric{ \theta}{z}). \label{u-r-theta}
	\end{align}
	Using \eqref{u-r-theta}, \eqref{coshr, sinhr<theta,z>} and \eqref{coshr}, we have
	\begin{equation*}
		\tilde{u} = \frac{1}{\g} \( \cosh^2 r-1- \cosh r \( \cosh r-\g \)  \)
		=\cosh r- \frac{1}{\g}
		=\frac{1}{2} \frac{|D \g|^2}{\g} + \frac{1}{2} \(\g - \frac{1}{\g} \).
	\end{equation*}
	Thus we obtain \eqref{tilde u}. Formula \eqref{1/phi, coshr-u} follows directly from \eqref{coshr} and \eqref{tilde u}. Then we complete the proof of Lemma \ref{lem-cosh r, tilde u, 1/phi}.
\end{proof}

\begin{cor}\label{cor-coshrnu-uX}
	Let $M$ be a smooth uniformly h-convex hypersurface in $\mathbb{H}^{n+1}$ with horospherical support function $u(z)$. Then
	\begin{align}
		\cosh r \nu - \widetilde{u} X =& - \( \frac{D \g}{\g}+z,0\), \label{coshr nu-uX}\\
		\frac{\widetilde{u} X - \cosh r \nu}{\cosh r -\widetilde{u}} =& \( D \g+ \g z,0 \) . \label{phi(coshr nu-uX)}                                                           
	\end{align}
\end{cor}

\begin{proof}
	Using \eqref{1/phi, coshr-u} and \eqref{X-nu}, we have
	\begin{equation}\label{coshrnu-uX-1}
		\cosh r \nu- \widetilde{u} X
		= \cosh r \( \nu- X\) + \( \cosh r- \widetilde{u} \)X
		= \frac{X - \cosh r (z,1)}{\g}. 
	\end{equation}
	Formulas \eqref{X(z)} and \eqref{coshr} yield that
	\begin{align}
		X - \cosh r (z,1) =&
		\frac{1}{2} \g \(-z,1\)+ \frac{1}{2} \( \frac{|D \g|^2}{\g} +\frac{1}{\g}  \)\(z,1\)- \( D \g,0 \) \nonumber\\
		&-\frac{1}{2}\( \frac{|D \g|^2}{\g}+ \g +\frac{1}{\g}   \)\(z,1\) \nonumber\\
		=&- \(D \g +\g z,0\). \label{X-coshr (z,1)}
	\end{align}
	Then formula \eqref{coshr nu-uX} follows by substituting \eqref{X-coshr (z,1)} into \eqref{coshrnu-uX-1}, and thus formula \eqref{phi(coshr nu-uX)} follows from \eqref{coshr nu-uX} and \eqref{1/phi, coshr-u}. We complete the proof of Corollary \ref{cor-coshrnu-uX}.
\end{proof}

We call a hypersurface  $M \subset \mathbb{H}^{n+1}$ \emph{star-shaped} if it can be represented as a graph of a positive radial function $r(\theta)$ defined on $\mathbb{S}^{n}$. In this case, for any $\theta \in \mathbb{S}^n$, there exist a unique point $X(\theta)$ on $M$ and a positive number $r$ such that $X(\theta) = \(  \sinh r \theta, \cosh r \)$.  It is well-known that a closed hypersurface $M \subset \mathbb{H}^{n+1}$ is star-shaped if and only if $\tilde{u} = \metric{V}{\nu}$ is positive on $M$. 
\begin{cor}\label{cor-star-shape}
	Let $M = \partial \Omega$ be a smooth closed and uniformly h-convex hypersurface in $\mathbb{H}^{n+1}$. Assume that $\Omega$ contains the origin in its interior. Then $M$ is star-shaped.
\end{cor}

\begin{proof}
	Since $\Omega$ contains the origin in its interior, we have $u(z) >0$ for all $z \in \mathbb{S}^n$, where $u(z)$ is the horospherical support function of $\Omega$. Hence $\g(z) = e^{u(z)} >1$, and then \eqref{tilde u} implies 
	\begin{equation*}
		\tilde{u}\(X(z) \) = \frac{1}{2} \frac{|D \g|^2}{\g}(z) + \frac{1}{2}\(\g(z) - \frac{1}{\g(z)}\)>0
	\end{equation*}
	for all $z \in \mathbb{S}^n$,
	which is equivalent to the fact that $M$ is star-shaped. 
\end{proof}

There are formulas $\widehat{\nabla}_i \widehat{\nu} = \widehat{h}_{i}{}^j \partial_j \widehat{X}$ and $\widehat{\nabla}_j \widehat{\nabla}_i \widehat{\nu} = \widehat{\nabla}^l \widehat{h}_{ij} \partial_l \widehat{X} - \widehat{h}_{i}{}^l \widehat{h}_{lj} \widehat{\nu}$ on any smooth hypersurfaces $\widehat{M}$ in $\mathbb{R}^{n+1}$ (see e.g., \cite[p. 9]{BIS20}), where $\widehat{X}$, $\widehat{\nu}$, $\widehat{\nabla}$ and $\widehat{h}_{ij}$ denote the position vector, outward unit normal, connection and second fundamental form of $\widehat{M} \subset \mathbb{R}^{n+1}$ respectively. In the following Lemma \ref{lem-zi-zij}, we derive the corresponding formulas in the hyperbolic setting.
\begin{lem}\label{lem-zi-zij}
	Let $M$ be a smooth hypersurface in $\mathbb{H}^{n+1} \subset \mathbb{R}^{n+1}$. In this lemma, we use the connection of the pullback bundle induced by the inclusion map $M \subset \mathbb{R}^{n+1,1}$, i.e.
	\begin{equation*}
		\nabla_i U := \widetilde{\nabla}_i U, \quad \nabla_j \nabla_i U : = \widetilde{\nabla}_j (\widetilde{\nabla}_i U)- \widetilde{\nabla}_{\nabla_j \partial_i} U,
	\end{equation*}
	where $\partial_i$, $\partial_j$ are tangential vector fields on $M$, and $U$ is a vector field in $\mathbb{R}^{n+1,1}$ defined on $M$. Then we have
	\begin{align}
		\nabla_i (X-\nu) =&-\tilde{h}_{i}{}^j \partial_j X, \label{di-X-nu} \\
		{\nabla}_i (z,1) =& - \g \tilde{h}_{i}{}^j \partial_j X + \g^2 \tilde{h}_{i}{}^j \metric{V}{\partial_j X}(X-\nu), \label{zi}\\
		{\nabla}_j {\nabla}_i (z,1) =& - \g \nabla^k \tilde{h}_{ij} \partial_k X +\g \tilde{h}_{i}{}^{k} \tilde{h}_{kj} \nu -\g^2 \tilde{h}_{i}{}^k \tilde{h}_{j}{}^{l} \metric{V}{\partial_l X} \partial_k X \nonumber\\
	 	&-\g^2 \tilde{h}_{i}{}^{k} \tilde{h}_{j}{}^l \metric{V}{\partial_k X} \partial_l X + T_{ij} (X-\nu), \label{zij}
	\end{align}
	where we viewed $\g$ as a function defined on $M$ by \eqref{1/phi, coshr-u} and set
	\begin{equation*}
		T_{ij} = -\g \tilde{h}_{ij}+ \g^2 \nabla_j (\tilde{h}_{i}{}^k \metric{V}{\partial_k X}) + 2 \g^3 \tilde{h}_{i}{}^k \tilde{h}_{j}{}^l \metric{V}{\partial_k X} \metric{V}{\partial_l X}.
	\end{equation*}
\end{lem}

\begin{proof}
	In the proof of Lemma \ref{lem-zi-zij}, we will use \eqref{1/phi, coshr-u} frequently without mention. Let us work in normal coordinates for $M$ about $X \in M$. By a direct calculation, we have
	\begin{equation*}
		\nabla_i (X-\nu) = \partial_i X -h_{i}{}^j \partial_j X = -\tilde{h}_{i}{}^j \partial_j X.
	\end{equation*}
	Thus we obtain \eqref{di-X-nu}. 
	
	By using \eqref{conf-vf}, we have
	\begin{equation}\label{1/phi i}
		\nabla_i (\cosh r -\tilde{u}) = \metric{V}{\partial_i X} - \nabla_i \metric{V}{\nu}
		=\metric{V}{\partial_i X} - h_{i}{}^j\metric{V}{ \partial_j X}
		= -\tilde{h}_{i}{}^j \metric{V}{\partial_j X}.
	\end{equation}
	By the definition of the horospherical Gauss map \eqref{X-nu} and \eqref{1/phi, coshr-u}, we have
	\begin{equation*}
		(z,1) = \frac{X-\nu}{\cosh r - \tilde{u}}.
	\end{equation*}
	This together with \eqref{di-X-nu} and \eqref{1/phi i} yields
	\begin{align*}
		\nabla_i (z,1) =& \nabla_i \(\frac{X-\nu}{\cosh r - \tilde{u}} \)
		= \g \nabla_i (X -\nu) - \g^2 \nabla_i (\cosh r - \tilde{u})(X -\nu)\\
		=&- \g \tilde{h}_{i}{}^j \partial_j X + \g^2 \tilde{h}_{i}{}^j \metric{V}{\partial_j X} (X- \nu),
	\end{align*}
	thus we obtain \eqref{zi}.
	
	Differentiating the above formula with respect to  $\partial_j$ gives
	\begin{equation}\label{(z,1)ij}
		\nabla_j \nabla_i (z,1) =- \nabla_j \left( \frac{\tilde{h}_{i}{}^k\partial_k X}{\cosh r -\tilde{u}}\right)
		+\nabla_j \left(\frac{\tilde{h}_{i}{}^k \metric{V}{\partial_k X}(X-\nu)}{(\cosh r - \tilde{u})^2} \right).
	\end{equation}
	 Since
	\begin{align*}
		\metric{\widetilde{\nabla}_j \partial_i X}{X}=& -\metric{\partial_i X}{\partial_j X}=-g_{ij},\\
		\metric{\widetilde{\nabla}_j \partial_i X}{\nu} =& -h_{j}{}^l \metric{\partial_i X}{\partial_l X} = -h_{ij},
	\end{align*}
	we have
	\begin{equation}\label{Xij}
		\widetilde{\nabla}_{j} \partial_i X = -h_{ij} \nu +g_{ij} X = -\tilde{h}_{ij} \nu +g_{ij} (X-\nu).
	\end{equation}
	For the first term in the right-hand side of \eqref{(z,1)ij}, by using the Codazzi equation, \eqref{1/phi i} and \eqref{Xij}, we have
	\begin{align}
		- \nabla_j \left( \frac{\tilde{h}_{i}{}^k\partial_k X}{\cosh r -\tilde{u}}\right)
		=&- \frac{\nabla^k \tilde{h}_{ij} \partial_k X}{\cosh r- \tilde{u}}
		- \frac{\tilde{h}_{i}{}^k (- \tilde{h}_{kj} \nu + g_{kj}(X-\nu))}{\cosh r - \tilde{u}}
		+ \frac{\tilde{h}_{i}{}^k \partial_k X (- \tilde{h}_{j}{}^l \metric{V}{\partial_l X})}{(\cosh r - \tilde{u})^2} \nonumber\\
		=& -\g \nabla^k \tilde{h}_{ij} \partial_k X + \g \tilde{h}_{i}{}^k \tilde{h}_{kj} \nu
		- \g \tilde{h}_{ij}(X- \nu) - \g^2 \tilde{h}_{i}{}^k \tilde{h}_{j}{}^l \metric{V}{\partial_l X} \partial_k X.\label{zij-p1}
	\end{align}
	For the second term in the right-hand side of \eqref{(z,1)ij}, by using \eqref{di-X-nu} and \eqref{1/phi i}, we have
	\begin{align}
		\nabla_j \left(\frac{\tilde{h}_{i}{}^k \metric{V}{\partial_k X}(X-\nu)}{(\cosh r - \tilde{u})^2} \right)
		=& \g^2 \nabla_j (\tilde{h}_{i}{}^k \metric{V}{\partial_k X})(X-\nu)
		+ \g^2 \tilde{h}_{i}{}^k \metric{V}{\partial_k X}(- \tilde{h}_{j}{}^l \partial_l X) \nonumber\\
		&+ 2 \g^3 \tilde{h}_{i}{}^k \metric{V}{\partial_k X} \tilde{h}_{j}{}^l \metric{V}{\partial_l X} (X-\nu) \nonumber\\
		=&\left( \g^2 \nabla_j (\tilde{h}_{i}{}^k \metric{V}{\partial_k X}) + 2 \g^3 \tilde{h}_{i}{}^k \tilde{h}_{j}{}^l \metric{V}{\partial_k X} \metric{V}{\partial_l X}\right) (X- \nu) \nonumber\\
	  	&-\g^2 \tilde{h}_{i}{}^k \tilde{h}_{j}{}^l \metric{V}{\partial_k X} \partial_l X.\label{zij-p2}
	\end{align}
	Therefore formula \eqref{zij} follows by inserting \eqref{zij-p1} and \eqref{zij-p2} into \eqref{(z,1)ij}.
	
	Then we complete the proof of Lemma \ref{lem-zi-zij}.
\end{proof}

For a smooth hypersurface $\widehat{M}$ in Euclidean space, one has (see e.g. \cite[Equation (3.7)]{SX19})
\begin{align*}
	\widehat{\nabla}_i \widehat{u} =& \widehat{h}_{i}{}^j \metric{\widehat{X}}{\partial_j \widehat{X}},\\
	\widehat{\nabla}_j \widehat{\nabla}_i \widehat{u} =& \widehat{\nabla}^k \widehat{h}_{ij} \metric{\widehat{X}}{\partial_k \widehat{X}} +\widehat{h}_{ij} -\widehat{h}_{i}{}^k \widehat{h}_{kj} \widehat{u},
\end{align*}
where $\widehat{u}:= \metric{\widehat{X}}{\widehat{\nu}}$ is the support function of $\widehat{M}$. The following Lemma \ref{lem-1/phi i, 1/phi ij} gives the corresponding formulas for smooth hypersurfaces in $\mathbb{H}^{n+1}$.

\begin{lem}\label{lem-1/phi i, 1/phi ij}
	Let $M$ be a smooth hypersurface in $\mathbb{H}^{n+1}$. Then
	\begin{align}
		\nabla_i \cosh r =& \metric{V}{\partial_i X}, \label{di cosh r}\\
		\nabla_i (\cosh r -\tilde{u})=& -\tilde{h}_{i}{}^j \metric{V}{\partial_j X}, \label{1/phi i new}\\
		\nabla_j \nabla_i (\cosh r - \tilde{u})=& -\nabla^k \tilde{h}_{ij} \metric{V}{\partial_k X} 
		-(\cosh r -\tilde{u})\tilde{h}_{ij} +\tilde{h}_{i}{}^k \tilde{h}_{kj} \tilde{u} . \label{1/phi ij}
	\end{align}
\end{lem}

\begin{proof}
	In the proof of Lemma \ref{lem-1/phi i, 1/phi ij}, we use local normal coordinates for $M$. Formula \eqref{di cosh r} follows from the definition of $V$, i.e. $V := \sinh r \partial_r =\overline{\nabla} \cosh r$. Formula \eqref{1/phi i new} was proved in \eqref{1/phi i}. Using \eqref{conf-vf} and the Codazzi equation, we have
	\begin{align*}
		\nabla_j \nabla_i (\cosh r - \tilde{u}) =& 
		-\nabla^k \tilde{h}_{ij} \metric{V}{\partial_k X}- \tilde{h}_{i}{}^k \nabla_j \metric{V}{\partial_k X}\\
		=& 	-\nabla^k \tilde{h}_{ij} \metric{V}{\partial_k X}- \cosh r \tilde{h}_{ij} + \tilde{h}_{i}{}^k (\tilde{h}_{kj} +g_{kj}) \tilde{u}\\
		=& -\nabla^k \tilde{h}_{ij} \metric{V}{\partial_k X} 
		-(\cosh r-\tilde{u})\tilde{h}_{ij} +\tilde{h}_{i}{}^k \tilde{h}_{kj} \tilde{u},
	\end{align*}
	thus we obtain \eqref{1/phi ij}. This completes the proof of Lemma \ref{lem-1/phi i, 1/phi ij}.
\end{proof}

\begin{lem}\label{lem-gij-hildlX-hilVdlX}
	By use of the horospherical Gauss map, we parameterize a smooth uniformly h-convex hypersurface $M$ about $X(z)$ by local coordinates for $\mathbb{S}^n$ about $z \in \mathbb{S}^n$. Let $\partial_i z= e_i$ at $z$. Then 
	\begin{align}
		\tilde{h}_i{}^j \partial_j X =& \frac{\g_i}{\g^2}(z,1) - \frac{1}{\g}(e_i,0),  \label{hil Xl}\\
    	\tilde{h}_i{}^j \metric{V}{\partial_j X} =& \frac{\g_i}{\g^2}. \label{hilVdlX}
	\end{align}
\end{lem}

\begin{proof}
	Differentiating \eqref{X-nu} with respect to $\partial_i$, we have
	\begin{equation}\label{di X-nu}
		\nabla_i (X-\nu) = -\frac{\g_i}{\g^2}(z,1) + \frac{1}{\g}(e_i,0).
	\end{equation}
	This together with $\nabla_i(X-\nu) = -\tilde{h}_i{}^j \partial_j X$ implies \eqref{hil Xl}.
	
	By \eqref{X-nu} and \eqref{zi}, we know
	\begin{align}
		(e_i,0) =& \nabla_i (z,1) =  - \g \tilde{h}_i{}^j \partial_j X + \g^2 \tilde{h}_i{}^j \metric{V}{\partial_j X}(X-\nu) \nonumber\\
		=& - \g \tilde{h}_i{}^j \partial_j X+ \g \tilde{h}_{i}{}^j \metric{V}{\partial_j X}(z,1).\label{ei,0}
	\end{align}
	Then \eqref{hilVdlX} follows by substituting \eqref{ei,0} into the right-hand side of \eqref{hil Xl}. We complete the proof of Lemma \ref{lem-gij-hildlX-hilVdlX}.
\end{proof}

\subsection{$C^0$ estimates of h-convex hypersurfaces} $\ $

Let $M = \partial \Omega$ be a smooth uniformly h-convex hypersurface such that $\Omega$ contains the origin in its interior. We have proved that $M$ is star-shaped in Corollary \ref{cor-star-shape}, i.e.
\begin{equation*}
	M = \lbrace \(\sinh r(\theta) \theta, \cosh r(\theta) \)\in \mathbb{H}^{n+1}:  \theta \in \mathbb{S}^n\rbrace
\end{equation*}
for some positive function $r(\theta)$. Here $r(\theta)$ is called the radial function of $M$. We define 
\begin{equation*}
	\begin{aligned}
	r_{\max} :=& \max_{\theta \in \mathbb{S}^n} r(\theta), \quad r_{\min} :=& \min_{\theta \in \mathbb{S}^n} r(\theta),\\
	u_{\max} :=& \max_{z \in \mathbb{S}^n} u(z), \quad u_{\min} :=& \min_{z \in \mathbb{S}^n} u(z),
	\end{aligned}
\end{equation*}
where $r(\theta)$ is the radial function, and $u(z)$ is the horospherical support function of $M$. Thanks to Corollary \ref{cor-support-construct-domain} and Corollary \ref{cor-star-shape}, we can regard $r$ and $u$  as functions defined on $M$.

When researchers study the $C^0$ estimate of a star-shaped hypersurface $\widehat{M}$ in Euclidean space, they often use the formulas (see e.g. \cite[Lemma 3.5]{LSW20})
\begin{equation*}
	\widehat{u}_{\max} = \widehat{r}_{\max}, \quad
	\widehat{u}_{\min} = \widehat{r}_{\min}.
\end{equation*}
In the following Lemma \ref{lem-r-min-u-min-r-max-u-max}, we derive the corresponding formulas for smooth uniformly h-convex hypersurfaces in $\mathbb{H}^{n+1}$.  

\begin{lem}\label{lem-r-min-u-min-r-max-u-max}
	Let $M = \partial \Omega$ be a smooth uniformly h-convex hypersurface in $\mathbb{H}^{n+1}$. If $\Omega$ contains the origin in its interior, then
	\begin{equation} \label{r-min-u-min-r-max-u-max}
		u_{\max}  =r_{\max} , \quad u_{\min}  = r_{\min}. 
	\end{equation}
\end{lem}

\begin{proof}
	By \eqref{coshr}, we have
	\begin{equation*}
		\cosh r(X(z)) = \frac{1}{2} \frac{|D \g|^2}{\g}(z) + \frac{1}{2} \(\g(z) + \frac{1}{\g(z)} \) \geq \cosh u(z),
	\end{equation*}
	and hence $u(z) \leq r(X(z))$ for all $z\in \mathbb{S}^n$. Then we have $u \leq r$ on $M$.
	
	Assume that $u(z)$ attains its minimum at $z_0 \in \mathbb{S}^n$. By \eqref{coshr}, we have $\cosh r(X(z_0)) =\cosh u(z_0)$, which means $u_{\min}=u =r$ at $X(z_0)$ by $u(z_0) > 0$. Hence
	\begin{equation*}
		u_{\min} = r_{\min}.
	\end{equation*}
	
	By \eqref{1/phi, coshr-u}, we have
	\begin{equation*}
		e^{-u}= \cosh r - \tilde{u} = \cosh r- \sinh r \metric{\partial_r}{\nu}.
	\end{equation*}
	Choose $\theta_0 \in \mathbb{S}^n$ such that $r(\theta_0) = r_{\max}$, and let $X_0 = (\sin r(\theta_0) \theta_0, \cosh r (\theta_0)) \in M$.
	By the above formula, we have $e^{-u} = e^{-r}= e^{-r_{\max}}$ at $X_0$. This together with the fact that $u \leq r$ on $M$ implies
	\begin{equation*}
		u_{\max} = r_{\max}.
	\end{equation*}
	Then we complete the proof of Lemma \ref{lem-r-min-u-min-r-max-u-max}.
\end{proof}

In \cite[Theorem 1]{BM99}, Borisenko and Miquel estimated the maximal distance from the center of an inball to the points on the boundary for h-convex bounded domains, which implies the following lemma. 
\begin{lem}\label{lem-u-min-r-max}
	Let $\Omega := \overline{B}_z(u) \cap \overline{B}_{-z}(u)$ be the intersection of two horo-balls. Then the radial function $r(\theta)$ of $\partial \Omega$ satisfies  
	\begin{equation}\label{u-min-r-max}
		\cosh r (\theta) \leq e^u, \quad \forall \ \theta \in \mathbb{S}^n,
	\end{equation}
	with equality if and only if $\theta$ satisfies $\metric{\theta}{z} =0$.
\end{lem}

\begin{proof}
	Here we give an elementary geometric proof of Lemma \ref{lem-u-min-r-max} by using the Euclidean norm of the Poincar\'e ball model $\mathbb{B}^{n+1}$, see Figure \ref{fig-intersec-2 balls}.
	Let $O$ be the center of $\mathbb{B}^{n+1}$ and $A$ be a point on $H_z(u) \cap H_{-z}(u)$. Then the direction $\theta_A$ of $A$ satisfies $\metric{\theta_A}{z} = 0$.  For any point $B \in H_z(u) \cap \overline{B}_{-z}(u)$, we can extend the ray $\overrightarrow{BO}$ such that it intersects with $H_{-z}(u)$ at point $C$. It is easy to see $|OC| \geq |OB|$. Note that the points $A$, $B$, $C$, $O$ lie on the boundary of a two-dimensional flat disc ${\rm span}\{OA,OB\} \cap \overline{B}_{-z}(u)$ from the Euclidean point of view. By the Intersecting Chords Theorem, we have 
	\begin{equation}\label{inter-chords}
		|OA|^2 = |OB|\cdot |OC| \geq |OB|^2,
	\end{equation}
	with equality if and only if $B \in H_z(u) \cap H_{-z}(u)$. Denote by $r_A$ the hyperbolic norm of segment $OA$. Then we have $r_A = \max_{\theta \in \mathbb{S}^n} r(\theta)$. 
	
	In the remainder of this proof, we use the hyperboloid model of $\mathbb{H}^{n+1}$ in $\mathbb{R}^{n+1,1}$. Let $X_A$ denote the position vector of $A$. Since $X_A \in  H_z(u) \cap H_{-z}(u)$, formula \eqref{horosphere-support function} implies
	\begin{equation}\label{position-A}
		\metric{-X_A}{(z,1)} = e^u, \quad \metric{-X_A}{(-z,1)} = e^u.
	\end{equation}
	By using \eqref{polar coord X}, the formulas in \eqref{position-A} imply $\cosh r_A=\metric{-X_A}{(0,1)}=e^u$. Then inequality \eqref{inter-chords} deduces that inequality $\cosh r (\theta) \leq e^u$ holds for all $\theta \in \mathbb{S}^n$, and equality holds if and only if $\metric{\theta}{z} =0$. This completes the proof of Lemma \ref{lem-u-min-r-max}.
\end{proof}

\begin{figure}[htbp]
	\centering
	\begin{tikzpicture}[scale=0.8]
	\coordinate (O) at (0,0);
	\coordinate (B) at  (1.2+2*cos{160}, 2*sin{160});
	\coordinate (Bp) at (-4.8-8*cos{160}, -8*sin{160});
	\shade[ball color =yellow, fill opacity=0.7] (0,1.6) arc(2.215 r: 4.068 r: 2) -- (0,-1.6) arc(-0.927r : 0.927 r:2);
	\draw  (1.2,0)  circle[radius=2];
	\node at (1.9,-0.3){$B_{-z}(u)$};
	\draw  (-1.2,0)  circle[radius=2];
	\node at (-1.9,-0.3){$B_z(u)$};
	\path [name path = rtcirc] (1.2,0) circle[radius=2];
	\path [name path = BO] (B) -- (Bp);
	\path [name intersections = {of = BO and rtcirc, by = C}];
	\draw[->, very thick] (-3.2,0) -- (-4.2, 0) node[left]{$z$};
	\draw[->, very thick] (3.2,0) -- (4.2,0) node[right]{$-z$};
	\draw[dashed] (O) circle[radius=3.2];
	\node at (3,-3){$\mathbb{B}^{n+1}$};
	\filldraw (0,0) circle(0.05) node[right=0.2]{$O$};
	\filldraw (0,1.6) circle (0.05) node[above]{$A$};
	\filldraw (B) circle (0.05) node[left]{$B$};
	\draw[blue, very thick] (B)--(C); 
	\filldraw (C) circle(0.05) node[above=0.6, right=0.1]{$C$};
	\draw[blue, very thick] (0, 1.6) -- (0,-1.6);
	\end{tikzpicture}
	\caption{Intersection of horo-balls}
	\label{fig-intersec-2 balls}
\end{figure}
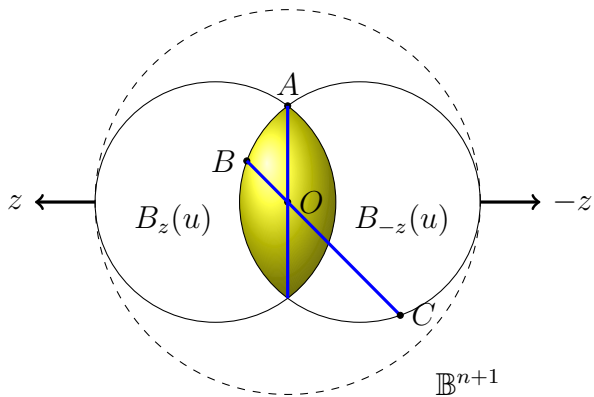

\section{Hyperbolic $p$-sum}\label{sec-p-sum}
\subsection{Hyperbolic $p$-sum for h-convex bounded domains} \label{subsec-p sum h convex domain} $ \ $

The main result in Subsection \ref{subsec-p sum h convex domain} is the following Theorem \ref{thm-def p sum-well defined}, and the proof of it will be given in the last part of Subsection \ref{subsec-p sum h convex domain}.
\begin{thm}\label{thm-def p sum-well defined}
	Let $p$, $a$, $b$, $K$ and $L$ satisfy the assumptions in Definition \ref{def-p sum}. Then there exists a unique h-convex bounded domain with horospherical support function
	\begin{equation*}
		u(z) = \frac{1}{p} \log  \(a e^{pu_K(z)} +b e^{pu_L(z)}\).
	\end{equation*}
	Consequently, Definition \ref{def-p sum} is well-defined. 
	
	Let $\Omega =a\cdot K +_p b \cdot L$ be the hyperbolic $p$-sum of $K$ and $L$ defined by Definition \ref{def-p sum}. Then for any point $X \in \partial \Omega$, there exists $z_X \in \mathbb{S}^n$ such that 
	\begin{equation*}
		X = X(z_X) = \frac{1}{2} \g(-z_X,1) + \frac{1}{2} \( \frac{|D\g|^2}{\g}+ \frac{1}{\g} \)(z_X,1) -(D \g,0),
	\end{equation*}
	where $\g(z) = \( a\g_K^p(z) +b\g_L^p(z) \)^{\frac{1}{p}}$, and $X(z_X)$ was defined by expression \eqref{X(z)}. Moreover, except the following special cases in Definition \ref{def-p sum}:
	\begin{enumerate}
		\item $\frac{1}{2} \leq p \leq 2$, $a=1$, $b=0$ and $K$ is a point;
		\item $\frac{1}{2} \leq p \leq 2$, $a=0$, $b=1$ and $L$ is a point;
		\item $p \in \{\frac{1}{2}, 2\}$, $a>0$, $b>0$, $a+b=1$, and both of $K$ and $L$ are points,
	\end{enumerate}
	we have that $\Omega$ is a smooth, non-degenerate and uniformly h-convex bounded domain.
\end{thm}

By Corollary \ref{cor-support-construct-domain}, in order to show that the smooth function $u(z)$ in Theorem \ref{thm-def p sum-well defined} determines a unique smooth uniformly h-convex bounded domain in $\mathbb{H}^{n+1}$, we only need to provide that $A_{ij} [e^{u(z)}] >0$ for all $z \in \mathbb{S}^n$.

\begin{prop}\label{prop-p-sum}
	Let $p>0$, $a>0$ and $b \geq 0$ be real numbers, and let $K$ and $L$ be two different smooth uniformly h-convex bounded domains in $\mathbb{H}^{n+1}$. Here $K$ and $L$ can degenerate to points. Define
	\begin{equation*}
		u(z) := \frac{1}{p} \log \(a e^{p u_{K}(z)}+b e^{p u_{L}(z)} \).
	\end{equation*}
	In addition, we require that $p$, $a$, $b$ ,$K$ and $L$ satisfy one of the following conditions,
	\begin{enumerate}
		\item $p>0$, $a> 1$, and $b =0$;
		\item $\frac{1}{2} \leq p \leq 2$, $a>0$, $b>0$, and $a+b >1$;
		\item $\frac{1}{2} \leq p \leq 2$,  $a>0$, $b>0$, $a+b =1$, and either $K$ or $L$ is non-degenerate;
		\item  $\frac{1}{2} < p < 2$,  $a>0$, $b>0$, $a+b =1$, and both $K$ and $L$ are degenerate.
	\end{enumerate}
	Then $A[e^{u(z)}]>0$, and $u(z)$ is the horospherical support function of a smooth, non-degenerate, uniformly h-convex bounded domain $\Omega$ in $\mathbb{H}^{n+1}$. 
\end{prop}

\begin{proof}
	Let $f(z):= \g^p(z)= e^{p u(z)}$, $f_K(z):=\g_K^p (z)= e^{p u_K(z)}$ and  $f_L (z):=\g_L^p (z)= e^{p u_L(z)}$. Then 
	\begin{equation*}
		f(z) =af_K(z)+bf_L(z).
	\end{equation*}
	We use normal coordinates for $\mathbb{S}^n$ around $z \in \mathbb{S}^n$ with $\sigma_{ij} = \delta_{ij}$ at $z$. Then
	\begin{align*}
		A_{ij}[f^{\frac{1}{p}}] =& 
		\( f^{\frac{1}{p}} \)_{ij} - \frac{1}{2}f^{-\frac{1}{p}}|D f^{\frac{1}{p}}|^2 \delta_{ij}
		+\frac{1}{2} \(f^{\frac{1}{p}}-f^{-\frac{1}{p}} \)\delta_{ij}\\
		=&\frac{1}{p}f^{\frac{1}{p}-1}f_{ij}+ \frac{1}{p} \( \frac{1}{p}-1 \) f^{\frac{1}{p}-2}f_if_j
		-\frac{1}{2p^2} f^{\frac{1}{p}-2}|Df|^2 \delta_{ij} \\
		&+ \frac{1}{2} \(f^{\frac{1}{p}}-f^{-\frac{1}{p}} \) \delta_{ij}.
	\end{align*}
	Since we assumed $p>0$, the inequality $A_{ij}[f^{\frac{1}{p}}] \geq0$ is equivalent to
	\begin{equation}\label{Bij}
		B_{ij}[f^{\frac{1}{p}}]:=f_{ij} + \frac{1-p}{p} \frac{f_i f_j}{f}-\frac{1}{2p}\frac{|D f|^2}{f} \delta_{ij}
		+\frac{p}{2} \( f -f^{1-\frac{2}{p}} \) \delta_{ij} \geq 0.
	\end{equation} 
	At this time, we have $B_{ij}[f_K^{\frac{1}{p}}] \geq 0$ and  $B_{ij}[f_L^{\frac{1}{p}}] \geq 0$ by the uniform h-convexity of $K$ and $L$, respectively. By Corollary \ref{cor-support-construct-domain} (see also \cite{ACW18}), it is enough to show 
	\begin{equation}\label{p-sum h-convex condition}
		B_{ij} [f^{\frac{1}{p}}(z)]> 0, \quad \forall z \in \mathbb{S}^n.
	\end{equation}
	A direct calculation yields
	\begin{equation}\label{Bij-Tij}
		B_{ij} [f^{\frac{1}{p}}]- aB_{ij}[f_K^{\frac{1}{p}}]-bB_{ij}[f_L^{\frac{1}{p}}]
		=(T_1)_{ij} + (T_2)_{ij} +(T_3)_{ij},
	\end{equation}
	where
	\begin{align*}
		(T_1)_{ij}:=& \frac{p-1}{p} \left( \frac{a (f_K)_i (f_K)_j}{f_K}
		+\frac{b (f_L)_i (f_L)_j}{f_L}
		- \frac{(af_K+bf_L)_i(af_K+bf_L)_j}{af_K+bf_L}
		 \right),\\
		(T_2)_{ij} :=& \frac{1}{2p} \left( a\frac{|D f_K|^2}{f_K} +b\frac{|D f_L|^2}{f_L} 
		-\frac{|D(af_K+bf_L)|^2}{af_K+bf_L}\right) \delta_{ij},\\
		(T_3)_{ij}:=&\frac{p}{2} \left( af_K^{1- \frac{2}{p}} +bf_L^{1- \frac{2}{p}} -(af_K+bf_L)^{1- \frac{2}{p}}  \right) \delta_{ij}
	\end{align*}
	are tensors on $\mathbb{S}^n$. Simplifying tensor $T_1$ gives
	\begin{align}
		\frac{p}{p-1} (T_1)_{ij}=& \frac{1}{f_K f_L (af_K+b f_L)} \left(  \left( af_L(f_K)_i(f_K)_j +bf_K(f_L)_i(f_L)_j \right)(af_K+bf_L)  \right. \nonumber\\
		&\left. -f_Kf_L \left( a^2(f_K)_i(f_K)_j +b^2 (f_L)_i(f_L)_j
		+ab(f_K)_i(f_L)_j +ab(f_K)_j(f_L)_i \right) \right) \nonumber\\
		=&\frac{1}{f_K f_L (af_K+b f_L)} \left( abf_K^2(f_L)_i(f_L)_j + ab f_L^2(f_K)_i(f_K)_j   
		 -abf_Kf_L(f_K)_i(f_L)_j  \right. \nonumber \\
		 & \left. -abf_Kf_L(f_K)_j(f_L)_i \right)  \nonumber\\
		=&\frac{ab}{f_K f_L (af_K+b f_L)} \mathcal{V}_i \mathcal{V}_j,\label{T_1}
	\end{align}
	where $\mathcal{V}:= f_LDf_K-f_KDf_L$. Analogously,
	\begin{equation}\label{T_2}
		(T_2)_{ij} =\frac{1}{2p}\frac{ab}{f_K f_L (af_K+b f_L)}|\mathcal{V}|^2 \delta_{ij}.
	\end{equation}
	The following argument is divided into four cases.
	
	\textbf{Case 1.} $p>0$, $a>1$ and $b=0$.
	Using \eqref{T_1} and \eqref{T_2}, we have $T_1 =0$ and $T_2 = 0$ in this case. Assumptions $a > 1$ and $b=0$ imply  
	\begin{equation}\label{p>0, T_3}
		(T_3)_{ij} = \frac{p}{2} \( a f_K^{1-\frac{2}{p}}- a^{1-\frac{2}{p}} f_K^{1-\frac{2}{p}}\) \delta_{ij}
		=\frac{p}{2} \( 1- a^{-\frac{2}{p}}\)a f_K^{1-\frac{2}{p}} \delta_{ij} > 0.
	\end{equation}
	Then we obtain the desired inequality \eqref{p-sum h-convex condition} by use of \eqref{Bij-Tij} and \eqref{p>0, T_3}. 
	
	For the follow-up three cases, let us begin to assume that 
	$\frac{1}{2} \leq p \leq 2$, $a>0$, $b>0$ and $a+b \geq 1$. The assumption $\frac{1}{2} \leq p \leq 2$ provides that $\zeta(t) = t^{1- \frac{2}{p}}$ is a convex function on $\mathbb{R}^+$. This together with the assumption $a+b \geq 1$ gives
	\begin{equation}\label{T3>0}
	af_K^{1- \frac{2}{p}}+	bf_L^{1- \frac{2}{p}} \geq
	(a+b)\( \frac{af_K+bf_L}{a+b} \)^{1- \frac{2}{p}} \geq (af_K+bf_L)^{1-\frac{2}{p}},
	\end{equation}
	which induces $T_3 \geq 0$. On the other hand, the assumption $ \frac{1}{2} \leq p \leq 2$ also implies $\frac{1}{2p}+1-\frac{1}{p}\geq 0$, and hence 
	\begin{equation}\label{T_1+T_2 >0}
		(T_1)_{ij}+(T_2)_{ij} = \frac{ab}{f_K f_L (af_K+b f_L)}
		\left(  \frac{1}{2p}|\mathcal{V}|^2 \delta_{ij} + \( 1-\frac{1}{p} \)\mathcal{V}_i\mathcal{V}_j \right) \geq 0.
	\end{equation}	
	
	\textbf{Case 2.}  $\frac{1}{2} \leq p \leq 2$, $a>0$, $b>0$ and $a+b > 1$.
	As $a+b >1$, the second inequality in \eqref{T3>0} is strict, which induces $T_3 >0$. The desired inequality \eqref{p-sum h-convex condition} therefore follows from \eqref{T_1+T_2 >0} and \eqref{Bij-Tij}. 
	
	\textbf{Case 3.} $\frac{1}{2} \leq p \leq 2$, $a>0$, $b>0$, $a+b =1$, and either $K$ or $L$ is non-degenerate. Since either $K$ or $L$ is non-degenerate, we have $aB_{ij} [f_K^{\frac{1}{p}}] + b B_{ij} [f_L^{\frac{1}{p}}]>0$. Then \eqref{T3>0}, \eqref{T_1+T_2 >0} and \eqref{Bij-Tij} imply  the desired inequality \eqref{p-sum h-convex condition}. 
	
	\textbf{Case 4.} $\frac{1}{2} < p < 2$, $a>0$, $b>0$, $a+b=1$, and $K, L$ are both degenerate. Let $K =X = (x,x_{n+1})$ and $L = Y = (y, y_{n+1})$ be two different points in $\mathbb{H}^{n+1}$. Using \eqref{X-z,1}, we have 
	\begin{align*}
		\g_K(z) =&\g_X(z) = \metric{-X}{(z,1)} = x_{n+1} - \metric{x}{z},\\
		\g_L (z)=& \g_Y(z) = \metric{-Y}{(z,1)} = y_{n+1} - \metric{y}{z}.
	\end{align*}
	If the inequality in \eqref{T3>0} is strict for all $z \in \mathbb{S}^n$, then the desired inequality \eqref{p-sum h-convex condition} follows from the argument in Case 2. Hence we can assume that equality holds in \eqref{T3>0} for some $z_0 \in \mathbb{S}^n$ without loss of generality.
	
	Since $p<2$, we have $f_K(z_0) =f_L(z_0)$ in this case, which is equivalent to $\g_X(z_0) = \g_Y(z_0)$. Then we have $x_{n+1} - \metric{x}{z_0} = y_{n+1} - \metric{y}{z_0}$, which deduces $y_{n+1} - x_{n+1} = \metric{y-x}{z_0}$. On the other hand, the Cauchy-Schwarz inequality yields
	\begin{align}
		\(y_{n+1} -x_{n+1} \)^2=& \( \sqrt{|y|^2 +1} - \sqrt{|x|^2 +1} \)^2
		= |y|^2 + |x|^2+2 - 2 \sqrt{ \(|y|^2 +1 \) \(|x|^2+1\)  }  \nonumber\\
		\leq& |y|^2 +|x|^2 +2 - 2\(|y||x| +1\)
		= \(|y| -|x|  \)^2
		\leq |y-x|^2,\label{Cau-Sch-x-y}
	\end{align}
	with equality if and only if $x=y$. Then the assumption $x \neq y$ implies that the vector $z_0$ is not parallel to $y-x$. Besides, at $z_0$ we have
	\begin{align*}
		\metric{\mathcal{V}}{e_i} =& f_L D_i f_K -f_K D_i f_L\\
		=& \g_Y^p D_i \(x_{n+1} - \metric{x}{z}  \)^p - 
		\g_X^p D_i \(y_{n+1} - \metric{y}{z}  \)^p\\
		=&
		p \g_Y^p \g_X^p\( \frac{-\metric{x}{e_i}}{\g_X} - 	\frac{-\metric{y}{e_i}}{\g_Y}  \)\\
		=& p \g_X^{2p-1}(z_0) \metric{y-x}{e_i}.
	\end{align*}
	Then $|\mathcal{V} | \neq 0$ at $z_0$, and hence inequality \eqref{T_1+T_2 >0} is strict at $z_0$ as $p> \frac{1}{2}$. Therefore we have $T_1 + T_2 + T_3 >0$ for all $z \in \mathbb{S}^n$, which implies the desired inequality \eqref{p-sum h-convex condition}.  
	 
	We complete the proof of Proposition \ref{prop-p-sum}.
\end{proof}

\begin{rem} $ \ $
	\begin{enumerate}
		\item  If $K = L$ in Definition \ref{def-p sum}, then we can view $a \cdot K +_p b \cdot L$ as $(a+b) \cdot_p K$. Hence we assumed that $K \neq L$ in Proposition \ref{prop-p-sum}. 
		\item Since $a \cdot K +_p b \cdot L =b \cdot L +_p a \cdot K$ and $ab \neq 0$ in Definition \ref{def-p sum}, we assumed that $a >0$ and $b \geq 0$ in Proposition \ref{prop-p-sum}.
		\item The case  $p>0$, $a=1$, and $b=0$ is trivial in Proposition \ref{prop-p-sum}, since $\Omega = 1 \cdot K = K$. Hence we assumed that $a>1$ in the first case of Proposition \ref{prop-p-sum}. 
	\end{enumerate}
\end{rem}

The above Proposition \ref{prop-p-sum} has already covered the generic cases in Definition \ref{def-p sum}. However, the reason that we can not assume that either $p = \frac{1}{2}$ or $p=2$ in Case (4) of Proposition \ref{prop-p-sum} will be given in the following Corollary \ref{cor-2pts-A>=0}, where we use the same notations as in the proof of Proposition \ref{prop-p-sum}.

\begin{cor}\label{cor-2pts-A>=0}
	Let $X=(x, x_{n+1})$ and $Y= (y, y_{n+1})$ be two different points in $\mathbb{H}^{n+1}$. 
	\begin{itemize}
		\item 	If $p=\frac{1}{2}$, $a>0$, $b>0$ and $a+b =1$, then $A_{ij}[\g(z)] \geq 0$ for all $z \in \mathbb{S}^{n}$.  Furthermore, $A_{ij} [\g(z_0)]$ has a $1$-dimensional null space for some $z_0 \in \mathbb{S}^n$ if and only if $z_0$ satisfies $ y_{n+1} - x_{n+1} = \metric{y-x}{z_0}$. 
		\item 	If $p=2$, $a>0$, $b>0$ and $a+b =1$, then $A_{ij} [\g(z)] \geq 0$ for all $z \in \mathbb{S}^n$. Furthermore, $A_{ij} [\g(z_0)]$ has a null space for some $z_0 \in \mathbb{S}^n$ if and only if $z_0$ is the critical point of the function $\frac{\metric{-X}{(z,1)}}{\metric{-Y}{(z,1)}}$ defined on $\mathbb{S}^n$. Indeed, $A_{ij} [\g(z_0)] = 0$ in this case, and the function $\frac{\metric{-X}{(z,1)}}{\metric{-Y}{(z,1)}}$ only has finite critical points. 
	\end{itemize}
\end{cor}

\begin{proof}
	Inserting \eqref{T3>0}, \eqref{T_1+T_2 >0}, $B_{ij} [\g_X(z)]=0$ and $B_{ij}[\g_Y(z)]=0$ into \eqref{Bij-Tij}, we have $B_{ij}[\g(z)] \geq 0$, which is equivalent to $A_{ij} [\g(z)] \geq 0$. Furthermore, $A_{ii}[\g(z_0)] = 0$ for some $z_0 \in \mathbb{S}^n$ and $e_i \in T_{z_0} \mathbb{S}^n$ if and only if both equality holds in \eqref{T3>0} at $z_0$, and equality holds in \eqref{T_1+T_2 >0} along direction $e_i$ at $z_0$. We divide the rest proof into two cases.
	
	\textbf{Case 1.} $p=\frac{1}{2}$, $a >0$, $b>0$ and $a+b =1$. Let $z_0 \in \mathbb{S}^n$ satisfy that $A_{ij} [\g(z_0)]$ has a null space. Since $\zeta(t) = t^{1-\frac{2}{p}} = t^{-3}$ is strictly convex on $\mathbb{R}^{+}$, equality holds in \eqref{T3>0} at $z_0$ if and only if $f_K(z_0) = f_L(z_0)$, which means $\metric{-X}{(z_0,1)}^p = \metric{-Y}{(z_0,1)}^p$. That is equivalent to $y_{n+1} - x_{n+1} = \metric{y-x}{z_0}$.
	
	The assumption $X \neq Y$ and inequality \eqref{Cau-Sch-x-y} show that $y-x$ is not parallel to $z_0$. Nevertheless, as $p=\frac{1}{2}$, we have from \eqref{T_1+T_2 >0} that
	\begin{equation*}
		(T_1)_{ii} + (T_2)_{ii} = \frac{ab}{f_K f_L \(af_K+b f_L\)} \( |\mathcal{V}| - (\mathcal{V}_i)^2 \) \geq 0,
	\end{equation*}
	and equality holds if and only if 
	\begin{equation*}
		e_i:= \pm  \frac{y-x- \metric{y-x}{z_0}}{|{y-x- \metric{y-x}{z_0}|}}.
	\end{equation*}
	Using these special $e_i$,  at $z_0$ we have $(T_1 + T_2+T_3)_{ii} =0$, and hence  $A_{ii} [\g(z_0)] =0$. This proves Corollary \ref{cor-2pts-A>=0} in Case 1.
	
	\textbf{Case 2.} $p=2$, $a >0$, $b>0$ and $a+b =1$.
	In this case, inequality \eqref{T3>0} is trivial. Thus, $A_{ij} [\g(z)]$ has a null space if and only if equality holds in \eqref{T_1+T_2 >0} along some unit vector in $T_z \mathbb{S}^n$.  Furthermore, $A_{ij} [\g(z)] = 0$ in this case. Since $\frac{1}{2p} +1 - \frac{1}{p} = \frac{3}{4}>0$, we have that $A_{ij} [\g(z)] = 0 $ is equivalent to $\mathcal{V} =0$ at $z$. Since $\mathcal{V}:=f_L D f_K - f_K D f_L$, we have 
	\begin{equation*}
		\frac{\mathcal{V}}{f_L^2} = D \(\frac{f_K}{f_L}\) = D \( \frac{\metric{-X}{(z,1)}^2}{\metric{-Y}{(z,1)}^2}  \).
	\end{equation*}
	Hence $A_{ij} [\g(z)] = 0$ if and only if $D \( \frac{\metric{-X}{(z,1)}}{\metric{-Y}{(z,1)}} \) = 0$ at $z$, equivalently, $z$ is a critical point of the function 
	\begin{equation*}
		\chi(z):=\log \metric{-X}{(z,1)}-\log \metric{-Y}{(z,1)}=\log \g_X(z) -\log \g_Y(z).
	\end{equation*}
	Let $z_0$ be a critical point of $\chi(z)$.
	By using \eqref{phi-i,phi-ij-point}, we have 
	\begin{equation}\label{critical eq-1-chi}
		D_i \chi(z_0) = \frac{D_i \g_X(z_0)}{\g_X(z_0)}-  \frac{D_i \g_X(z_0)}{\g_X(z_0)}
		= \frac{-\metric{x}{e_i}}{\g_X(z_0)} -  \frac{-\metric{y}{e_i}}{\g_Y(z_0)}=0,
	\end{equation}
	and thus
	\begin{align}
		D_j D_i \chi(z_0) =& \frac{D_j D_i \g_X}{\g_X}- \frac{D_j \g_X D_i \g_X}{\g_X^2}-\frac{D_j D_i \g_Y}{\g_Y} + \frac{D_j \g_Y D_i \g_Y}{\g_Y^2} \nonumber\\
		=& \frac{D_j D_i \g_X}{\g_X}-\frac{D_j D_i \g_Y}{\g_Y}
		=\(\frac{\metric{x}{z_0}}{\g_X(z_0)} - \frac{\metric{y}{z_0}}{\g_Y(z_0)}\) \delta_{ij}. \label{critical eq-2-chi}
	\end{align}
	Suppose for a contradiction $D^2 \chi(z_0)$ is neither positive definite nor negative definite. Then
	equations \eqref{critical eq-1-chi} and \eqref{critical eq-2-chi} imply that 
	\begin{equation}\label{phiY x- phiX y}
		\g_Y(z_0) x-\g_X(z_0) y = 0.
	\end{equation}
	Hence
	\begin{equation*}
		0 =\metric{\g_Y x-\g_X y}{z_0}
		=\g_Y (x_{n+1} - \g_X) - \g_X(y_{n+1} -\g_Y)
		=\g_Y(z_0) x_{n+1} - \g_X(z_0) y_{n+1}.
	\end{equation*}
	This together with \eqref{phiY x- phiX y} gives
	\begin{equation*}
		0= \g_Y^2 x_{n+1}^2-\g_X^2 y_{n+1}^2 = |\g_Y x|^2+ \g_Y^2-|\g_X y|^2- \g_X^2= \g_Y^2(z_0)- \g_X^2(z_0).
	\end{equation*}
	Then we get $\g_Y(z_0) = \g_X(z_0)$, and thus $x=y$. This contradicts the assumption $X \neq Y$. Therefore $D^2 \chi (z_0)$ is either positive definite or negative definite, which implies that the critical points of $\chi(z)$ are discrete and form a finite subset of $\mathbb{S}^n$. We complete the proof of Corollary \ref{cor-2pts-A>=0}.
\end{proof}

Now, we proceed to prove Theorem \ref{thm-def p sum-well defined}.

\begin{proof}[Proof of Theorem \ref{thm-def p sum-well defined}]
	Proposition \ref{prop-p-sum} and Corollary \ref{cor-2pts-A>=0} show that $A_{ij} [\g] \geq 0$ for all cases in Definition \ref{def-p sum}. Then Proposition \ref{prop-Aij >=0} asserts that
	$u(z)$ uniquely determines a h-convex bounded domain $\Omega$ with horospherical support function $u(z)$ and $\Omega = \cap_{z \in \mathbb{S}^n} \overline{B}_z\( u(z)\)$. Hence, Definition \ref{def-p sum} is well-defined. Furthermore, the boundary of $\Omega$ can be described by using expression \eqref{X(z)}.  Except the three cases we mentioned in Theorem \ref{thm-def p sum-well defined}, the regularity of $\Omega$ follows from Proposition \ref{prop-p-sum}. Then we complete the proof of Theorem \ref{thm-def p sum-well defined}. 
\end{proof}

\subsection{Hyperbolic $p$-sum for sets}$ \ $

In \cite{LYZ12}, Lutwak et al. found a pointwise addition for nonconvex sets in $\mathbb{R}^{n+1}$, which extended the Minkowski-Firey $L_p$ addition from convex bodies containing the origin to arbitrary subsets of $\mathbb{R}^{n+1}$. For real $p \geq 1$, $a \geq 0$, $b\geq 0$, and $\widetilde{K}, \widehat{L} \subset \mathbb{R}^{n+1}$, they defined
\begin{equation*}
	a \cdot \widehat{K} +_p b \cdot \widehat{L} =
	\{ a^{\frac{1}{p}}(1-t)^{\frac{1}{q}} x+ b^{\frac{1}{p}}t^{\frac{1}{q}} y: x \in \widehat{K}, \ y \in \widehat{L}, \ {\rm and} \ t \in [0,1]  \},
\end{equation*}
where $q$ is the H\"older conjugate of $p$. If $p =1$, then $q = +\infty$, and hence the above definition degenerates to the pointwise Minkowski sum $a\widehat{K}+b \widehat{L}$.
 
This subsection aims to show that the pointwise addition for general sets in Definition \ref{def-p sum-p>0-new} is compatible with Definition \ref{def-p sum} for h-convex bounded domains. 

The following Assumption \ref{assump-p-sum} will be needed.

\begin{assump}\label{assump-p-sum}
	Let $p>0$, $t \in [0,1]$, $a \geq 0$ and $b \geq 0$ be real numbers, and let $X = (x,x_{n+1})$, $Y =(y,y_{n+1})$ and $W =(w,w_{n+1})$ be points on $\mathbb{H}^{n+1} \subset \mathbb{R}^{n+1,1}$.  We suppose that 
	\begin{align*}
		(1-t)^{\frac{1}{q}} a^{\frac{1}{p}}x_{n+1} +t^{\frac{1}{q}}b^{\frac{1}{p}} y_{n+1}- w_{n+1} \geq& \left\vert (1-t)^{\frac{1}{q}} a^{\frac{1}{p}}x +t^{\frac{1}{q}}b^{\frac{1}{p}} y-w \right\vert, \quad &p \in (0,1) \cup (1, +\infty),\\
		a x_{n+1}+ by_{n+1}-w_{n+1} \geq& \left\vert ax+by-w \right\vert, \quad &p=1,
	\end{align*}
	where $q$ is the H\"older conjugate of $p$, i.e. $\frac{1}{p} + \frac{1}{q}=1$. Equivalently, it means that the vector $ (1-t)^{\frac{1}{q}} a^{\frac{1}{p}}X +t^{\frac{1}{q}}b^{\frac{1}{p}} Y- W$  lies inside the future light cone of $(0,0)$ when $p \in (0,1) \cup (1, +\infty)$, and the vector $X+Y-W$ lies inside the future light cone of $(0,0)$ when $p=1$.
\end{assump}

Now we proceed to give a geometric explanation of Assumption \ref{assump-p-sum}. The following Lemma \ref{lem-geo-dis on hyperbolic space} can be found in \cite[p. 49]{MB16}.

\begin{lem}\label{lem-geo-dis on hyperbolic space}
	Let $X$, $Y$ be two points on $\mathbb{H}^{n+1} \subset \mathbb{R}^{n+1,1}$ and $d_{\mathbb{H}^{n+1}}(X,Y)$ be the geodesic distance between $X$ and $Y$ on $\mathbb{H}^{n+1}$. Then
	\begin{equation}\label{geo-dis on hyperbolic space}
		\cosh \( d_{\mathbb{H}^{n+1}}(X,Y)\) = -\metric{X}{Y}.
	\end{equation}
\end{lem}

\begin{proof}
	It is well-known that the geodesic segment from $X$ to $Y$ is given by (see e.g., \cite[p. 48]{MB16})
	\begin{equation}\label{formula of geodesic segment}
		\cosh t X + \sinh t v, \quad t\in[0, d_{\mathbb{H}^{n+1}}(X,Y)],
	\end{equation}
	where $v$ is a unit vector in $T_X \mathbb{H}^{n+1}$. Since $\metric{X}{v} =0$, we have
	\begin{equation*}
		\metric{X}{Y} = \metric{X}{\cosh \(d (X,Y) \)X + \sinh \( d_{\mathbb{H}^{n+1}}(X,Y)\)v } = -\cosh \( d_{\mathbb{H}^{n+1}}(X,Y) \).  
	\end{equation*}
	We complete the proof of Lemma \ref{lem-geo-dis on hyperbolic space}.
\end{proof}

\begin{prop}\label{prop-Assump-geodesic ball}
	Let $p>0$, $t \in[0,1]$, $a \geq 0$ and $b \geq 0$ be real numbers, 
	let $X= (x, x_{n+1})$ and $Y= (y, y_{n+1})$ be two points in $\mathbb{H}^{n+1}$, and let $T = (1-t)^{\frac{1}{q}} a^{\frac{1}{p}}X + t^{\frac{1}{q}}b^{\frac{1}{p}} Y$.
	\begin{enumerate}
		\item If $\metric{T}{T} \leq -1$, then
		\begin{equation}\label{set-assump 4.1}
			\left\{ W= (w,w_{n+1}) \in \mathbb{H}^{n+1}:  p, \, t, \, a, \, X, \, b, \, Y, \, {\rm and} \, W  {\rm satisfy \ Assumption \  \ref{assump-p-sum}}  \right\}.
		\end{equation}
		is a geodesic ball of radius $ \log \sqrt{- \metric{T}{T}}$ centered at $T/ \sqrt{-\metric{T}{T}}$ . Moreover, $T/ \sqrt{-\metric{T}{T}}$ lies on the geodesic segment connecting $X$ and $Y$.
		\item If $\metric{T}{T}>-1$,  then the set \eqref{set-assump 4.1} is empty.
	\end{enumerate}
\end{prop}

\begin{proof}
	By Assumption \ref{assump-p-sum} and $\metric{W}{W} =-1$, we have
	\begin{equation}\label{T-w}
		0 \geq \metric{T-W}{T-W} = \metric{T}{T}-1- 2\metric{T}{W}.
	\end{equation}
	Using \eqref{geo-dis on hyperbolic space} and \eqref{T-w}, we have
	\begin{equation}\label{T'-w}
		\cosh d_{\mathbb{H}^{n+1} } (T', W) = -\metric{T'}{W} \leq  \frac{1- \metric{T}{T}}{2 \sqrt{-\metric{T}{T}}},
	\end{equation}
	where we set $T' = T/ \sqrt{-\metric{T}{T}} \in \mathbb{H}^{n+1}$. If $\metric{T}{T} >-1$, then the right-hand side of inequality \eqref{T'-w} is less than $1$, which implies that \eqref{set-assump 4.1} is empty.
	
	It is easy to see that
	\begin{equation}\label{radius ball T'}
		\cosh \(\log \sqrt{- \metric{T}{T}} \) =
		\frac{1- \metric{T}{T}}{  2 \sqrt{-\metric{T}{T}}} .
	\end{equation}
	If $\metric{T}{T} \leq -1$, then combining \eqref{T'-w}, \eqref{radius ball T'} with \eqref{geo-dis on hyperbolic space}, we have that $W$ lies in a geodesic ball of radius $\log \sqrt{- \metric{T}{T}}$ centered at $T'$.  Conversely, for any $W$ in this ball, one can easily check that $W$ satisfies Assumption \ref{assump-p-sum} by the same argument as above.
	
	From $(1-t)^{\frac{1}{q}} a^{\frac{1}{p}}\geq 0$ and $ t^{\frac{1}{q}}b^{\frac{1}{p}} \geq 0$, we have that $T$ and $T'$ both locate in the cone generated by the linear combinations of the vectors $X$ and $Y$ with non-negative coefficients. Besides, it is well-known that the geodesic lines in $\mathbb{H}^{n+1}$ are given by the intersection of $2$-dimensional vector subspaces of $\mathbb{R}^{n+1,1}$ with the hyperboloid model of $\mathbb{H}^{n+1}$, see e.g. \cite[Section 2.1]{MB16}. Consequently, $T'$ is on the geodesic segment connecting $X$ and $Y$. We complete the proof of Proposition \ref{prop-Assump-geodesic ball}.
\end{proof}

Here we provide Figure \ref{fig-new sum} ($p=2$ case) and Figure \ref{fig-new sum p=1} ($p=1$ case) to show the geometric explanation of  Assumption \ref{assump-p-sum}, as we proved in Proposition \ref{prop-Assump-geodesic ball}.
\begin{figure}[htbp!]
	\centering
	\begin{tikzpicture}[scale=1.5]
	\coordinate (X) at (0,1);
	\coordinate (Y) at (1, {sqrt(2)});
	\coordinate (T) at (0.8, 1.731);
	\shade[domain=-1:2, bottom color=yellow, top color= white, samples =100]
	(-1, {sqrt(5)}) -- plot(\x, {sqrt(\x*\x+1)});
	\shade[top color=green, bottom color= white] (-0.3, 0.631) --(T) --(1.9,0.631) --cycle;
	\filldraw (T) circle(.02);
	\node at (0.8,1.931){$T$};
	\draw[->] (-1.5,0) -- (2.5,0) node[right]{$x= (x_0, \ldots, x_n)$};
	\draw[->] (0,-0.2)-- (0,2.5) node[above]{$x_{n+1}$};
	\draw[domain= -1:2, samples=200, thick]
	plot(\x, {sqrt(\x*\x+1)});
	\node at (-1.4,2){$\mathbb{H}^{n+1}$};
	\node at (-0.1,1.1){$X$};
	\node at (1.1, 1.214){$Y$};
	\draw[domain=0:1, samples=200, very thick, blue]
	plot(\x, {sqrt(1-\x*\x)+sqrt(2)*\x});
	\coordinate (P) at (0.924,1.6893);
	\coordinate (Q) at (0.3828,1.465);
	\filldraw (P) circle(.01);
	\filldraw (Q) circle(.01);
	\filldraw[domain=0.072: 1.1, samples=100, color = red, fill opacity=1, rounded corners]
	plot(\x, {sqrt(\x*\x+1)})--(1, 1.45) arc(1.55 r:1.62 r:1) -- (0.5, 1.25) -- cycle;
	\draw[->, dashed] (0,0) --(T);
	\filldraw (X) circle(.02);
	\filldraw (Y) circle(.02);
	\node[red] at (0.78,1.1){$W$};
	\filldraw (0.521,1.128) circle(.02);
	\node at (0.39,0.97){$T'$};
	\end{tikzpicture}
	\caption{In this figure, we let $p=2$ (thus $q=2$), $t_0=\frac{16}{25}$, $a=1$, $b=1$, $X=(0,0,\ldots, 0, 1)$ and $Y= (1,0,\ldots, 0, \sqrt{2})$. The blue curve in $\mathbb{R}^{n+1,1}$ represents $\lbrace (1-t)^{\frac{1}{2}}X +t^{\frac{1}{2}} Y:   t \in [0,1] \rbrace$, and $T = (1-t_0)^{\frac{1}{q}} X + t_0^{\frac{1}{q}} Y= (\frac{4}{5}, 0, \ldots, 0, \frac{3+4 \sqrt{2}}{5})$ lies on the curve. The green shade represents the domain enclosed by the past light cone of $T$; the red domain represents the geodesic ball of radius $\log \sqrt{-\metric{T}{T}} $ centered at $T'=T/ \sqrt{-\metric{T}{T}}$. Proposition \ref{prop-Assump-geodesic ball} implies that, $W$ together with the chosen $p$, $t_0$, $a$ and $b$ satisfies Assumption \ref{assump-p-sum} if and only if $W$ lies in the red domain. }
	\label{fig-new sum}
\end{figure}
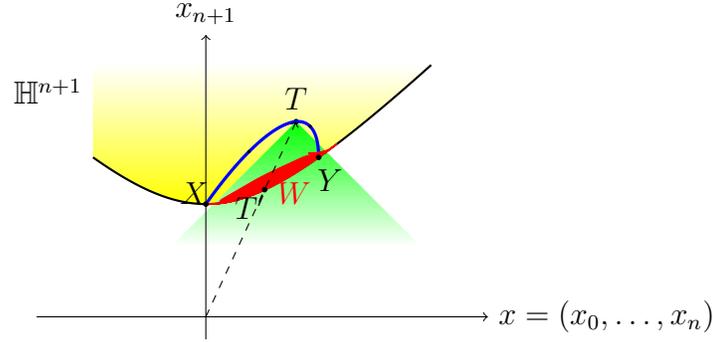
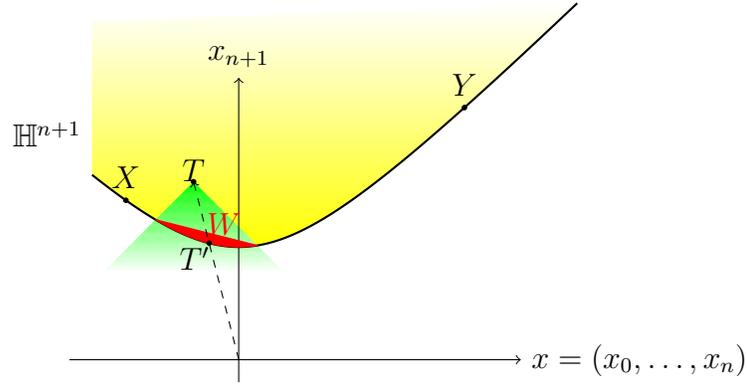
\begin{figure}[htbp!]
	\centering
	\begin{tikzpicture}[scale=1.5]
	\coordinate (X) at (-1,{sqrt(2)});
	\coordinate (Y) at (2, {sqrt(5)});
	\coordinate (T) at (-0.4, 1.5784);
	\coordinate (S)  at  (-0.262,1.034); 
	\node at (-0.262,1.034){$T'$};
	\shade[domain=-1.3:3, bottom color=yellow, top color= white, samples =100]
	(-1.3, 3) -- plot(\x, {sqrt(\x*\x+1)});
	\shade[top color=green, bottom color= white] (-1.2, 0.7784) --(T) --(0.4, 0.7784) --cycle;
	\filldraw (T) circle(.02);
	\node at (-0.4,1.68){$T$};
	\draw[->] (-1.5,0) -- (2.5,0) node[right]{$x= (x_0, \ldots, x_n)$};
	\draw[->] (0,-0.2)-- (0,2.5) node[above]{$x_{n+1}$};
	\draw[domain= -1.3:3, samples=200, thick]
	plot(\x, {sqrt(\x*\x+1)});
	\node at (-1.7,2){$\mathbb{H}^{n+1}$};
	\node at (-1, {sqrt(2)+0.2}){$X$};
	\node at (2, {sqrt(5)+0.2}){$Y$};
	\filldraw[domain= -0.736:0.165, samples=100, color = red, fill opacity=1]
	(-0.736, 1.242) -- plot(\x, {sqrt(\x*\x+1)}) -- (0.165, 1.014) -- cycle;	
	\draw[->, dashed] (0,0) --(T);
	\filldraw (X) circle(.02);
	\filldraw (Y) circle(.02);
	\node[red] at (-0.13,1.2){$W$};
	\filldraw (S) circle(.02);
	\node at (-0.39,0.9){$T'$};
	\end{tikzpicture}
	\caption{In this figure, we let $p=1$, $a=\frac{4}{5}$, $b=\frac{1}{5}$, $X=(-1, 0, \ldots,0,\sqrt{2})$ and $Y= (2,0,\ldots, 0, \sqrt{5})$. Then  $T = a X + b Y=(-0.4,0,\ldots, 0, \frac{4\sqrt{2}+\sqrt{5}}{5})$. The green shade represents the domain enclosed by the past light cone of $T$; the red domain represents the geodesic ball of radius $\log \sqrt{-\metric{T}{T}} $ centered at $T'=T/ \sqrt{-\metric{T}{T}}$.  Proposition \ref{prop-Assump-geodesic ball} implies that, $W$ together with the chosen $a$ and $b$ satisfies Assumption \ref{assump-p-sum} if and only if $W$ lies in the red domain. }
	\label{fig-new sum p=1}
\end{figure}
	
Now we can define a pointwise $p$-sum for general sets in $\mathbb{H}^{n+1}$.
 
\begin{defn}\label{def-p sum-p>0}
	Let $p>0$, $a \geq 0$ and $b \geq 0$ be real numbers, and let $K$ and $L$ be sets in $\mathbb{H}^{n+1}$. We define
	\begin{small}
		\begin{equation*}
			a \cdot K \widetilde{+}_p b \cdot L = \left\{
			\begin{aligned}
			&\underset{ X\in K, \ Y \in L}{\bigcup} \bigcup_{t \in [0,1]} \left\{ W \in \mathbb{H}^{n+1}:  p, \, t, \, a, \, X, \, b, \, Y \, {\rm and} \, W {\rm satisfy \ Assumption \ \ref{assump-p-sum}} \right\}, \quad p \geq 1,\\
			&\underset{ X\in K, \ Y \in L}{\bigcup} \bigcap_{t \in [0,1]} \left\{ W \in \mathbb{H}^{n+1}:  p,\, t, \, a, \, X, \, b, \, Y \, {\rm and} \, W {\rm satisfy \ Assumption \ \ref{assump-p-sum}}  \right\}, \quad  0<p<1.
			\end{aligned} \right.
		\end{equation*}
	\end{small}
\end{defn}

\begin{prop}\label{prop-point wise def-equivalent form}
	Let $p>0$, $a \geq 0$ and $b \geq 0$ be a real numbers, and let $K$ and $L$ be sets in $\mathbb{H}^{n+1}$. Then
	\begin{equation*}
		a \cdot K \widetilde{+_p} b \cdot L =
		\left\{ 
			\begin{aligned}
			&\bigcup_{X \in K, \, Y \in L} \bigcup_{t \in [0,1]} B(p,t; a,X,b, Y), \quad &p>1,\\
			&\bigcup_{X \in K, \, Y \in L}  B( a, X,b,Y), \quad &p=1,\\
			&\bigcup_{X \in K, \, Y \in L} \bigcap_{t \in [0,1]} B(p,t; a,X,b,Y), \quad &0 <p<1,
			\end{aligned}
		\right.
	\end{equation*}
	Consequently, Definition \ref{def-p sum-p>0} is equivalent to Definition \ref{def-p sum-p>0-new}.
\end{prop}

\begin{proof}
	Let $t \in [0,1]$ be a real number, and let $X \in K$ and $Y \in L$ be two points. For abbreviation, let
	\begin{equation*}
		E = \left\{ W \in \mathbb{H}^{n+1}:  p, \, t, \, a, \, X, \, b, \, Y \, {\rm and} \, W {\rm satisfy \ Assumption \ \ref{assump-p-sum}}  \right\}.
	\end{equation*}
	
	We first consider the case $p \in (0,1) \cup (1, +\infty)$.
	Let $T:= (1-t)^\frac{1}{q} a^{\frac{1}{p}} X+ t^{\frac{1}{q}} b^\frac{1}{p} Y$ be defined as in Proposition \ref{prop-Assump-geodesic ball}. 
	Then Proposition \ref{prop-Assump-geodesic ball} asserts that $E$ is a geodesic ball whose center $T'$ lies on the geodesic segment connecting $X$ and $Y$, and the radius of $E$ is given by 
	\begin{align*}
		\log \sqrt{-\metric{T}{T}}
		=&\log \(- \metric{(1-t)^\frac{1}{q} a^{\frac{1}{p}} X+ t^{\frac{1}{q}} b^\frac{1}{p} Y}{
		(1-t)^\frac{1}{q} a^{\frac{1}{p}} X+ t^{\frac{1}{q}} b^\frac{1}{p} Y}  \)^{\frac{1}{2}}\\
		=& \log  \(- (1-t)^\frac{2}{q} a^\frac{2}{p} \metric{X}{X} - t^\frac{2}{q} b^\frac{2}{p} \metric{Y}{Y}
		- 2 (1-t)^\frac{1}{q}t^\frac{1}{q} a^\frac{1}{p} b^\frac{1}{p} \metric{X}{Y}\)^\frac{1}{2}\\
		=&\log  \((1-t)^\frac{2}{q} a^\frac{2}{p}+ t^\frac{2}{q} b^\frac{2}{p} 
		+2 (1-t)^\frac{1}{q}t^\frac{1}{q} a^\frac{1}{p} b^\frac{1}{p} \cosh  d_{\mathbb{H}^{n+1}} (X,Y) \)^\frac{1}{2}\\
		=&\log R(p,t,a,X,b,Y),
	\end{align*}
	where we used \eqref{geo-dis on hyperbolic space} and \eqref{R-p-t-a-X-b-Y}. If $\metric{T}{T} >-1$, then Proposition \ref{prop-Assump-geodesic ball} implies that $E$ is an empty set. Hence we can assume $\metric{T}{T} \leq -1$. Using \eqref{geo-dis on hyperbolic space},  we have  that the distance between $T'$ and $X$ is given by
	\begin{align*}
		\cosh d_{\mathbb{H}^{n+1}} (T', X)
		=&-\metric{T'}{X} 
		=\frac{-\metric{T}{X}}{\sqrt{- \metric{T}{T}}}\\
		=&\frac{-  (1-t)^\frac{1}{q} a^{\frac{1}{p}} \metric{X}{X} - t^{\frac{1}{q}}b^\frac{1}{p} \metric{X}{Y}  }{ R(p,t,a,X,b,Y) 	 }\\
	=& \frac{(1-t)^\frac{1}{q}a^\frac{1}{p} + t^\frac{1}{q} b^\frac{1}{p}\cosh d_{\mathbb{H}^{n+1}} (X, Y) }{R(p,t,a,X,b,Y)}.
	\end{align*} 
	Hence $T' = P(p,t,a,X,b,Y)$ and $E = B(p,t,a,X,b,Y)$, where we used Definition \ref{def-B(p,t,a,X,b,Y)}.
	
	In the case $p=1$, we have $E = B(a,X,b,Y)$ by viewing $(1-t)^\frac{1}{q} = t^\frac{1}{q} =1$ in the above proof, where $B(a,X,b,Y)$ was defined in Definition \ref{def-B(p,t,a,X,b,Y)}. Then Proposition \ref{prop-point wise def-equivalent form} follows from Definition \ref{def-p sum-p>0}.
\end{proof}

\subsection{Compatibility of different definitions of hyperbolic $p$-sum} $ \ $

The following Theorem \ref{thm-new-old-sum-compatible} shows the relationship between Definition \ref{def-p sum-p>0} and Definition \ref{def-p sum}.
\begin{thm}\label{thm-new-old-sum-compatible}
	Let $p$, $a$, and $b$ be real numbers such that $\frac{1}{2} \leq p \leq 2$,  $a \geq 0$, $b \geq 0$ and  $a+b \geq 1$. Let $K$ and $L$ be two smooth uniformly h-convex bounded domains in $\mathbb{H}^{n+1}$. Then Definition \ref{def-p sum-p>0} is compatible with Definition \ref{def-p sum}, i.e.
	\begin{equation*}
		a \cdot K \widetilde{+}_p b \cdot L =a \cdot K +_p b \cdot  L.
	\end{equation*}
	Consequently, Definition \ref{def-p sum-p>0-new} is compatible with Definition \ref{def-p sum}.
\end{thm}
For the convenience of readers, we give an outline of the proof of Theorem \ref{thm-new-old-sum-compatible}.
\begin{itemize}
	\item \emph{Step 1}.  We will prove that $a \cdot K \widetilde{+}_p b \cdot L \subset a \cdot K+_p b \cdot L$ in Proposition \ref{prop-K+p L in h-convex hull}. Then, it suffices to show the opposite direction.
	\item \emph{Step 2}. Since any point $W$ on $\partial(a \cdot K+_p b \cdot L)$ can be given explicitly by expression \eqref{X(z)}, we intend to find suitable $X \in K$, $Y \in L$ and $t \in [0,1]$ that satisfies Assumption \ref{assump-p-sum}. Then we will prove that $\partial(a \cdot K +_p b \cdot L) \subset a \cdot K \widetilde{+}_p b \cdot L$ in Lemma \ref{lem-old-bdy in new sum}.
	\item \emph{Step 3}. In Proposition \ref{prop-sum of 2 pts new}, we will prove that $a \cdot X \widetilde{+}_p b \cdot Y = a \cdot X+_p b \cdot Y$ for any two points $X$, $Y$. 
	\item \emph{Step 4}. Let $X \in K$, $Y\in L$ and $W \in a \cdot K+_p b \cdot L$ be points. If $W \in a \cdot X+_p b \cdot Y$, then the result in Step 3 implies $W \in a \cdot X \widetilde{+}_p b \cdot Y$. If $W \in a \cdot K+_p b \cdot L \backslash \(a \cdot X+_p b\cdot Y\)$, then we will find two h-convex bounded domains $K_0 \subset K$ and $L_0 \subset L$ such that $W \in \partial( a \cdot K_0 +_p b \cdot L_0)$. Then the fact in Step 2 implies  $W \in a \cdot K \widetilde{+}_p b \cdot L$. Hence we will get $a \cdot K+_p b \cdot L \subset a \cdot K \widetilde{+}_p b \cdot L$.
	\item Theorem \ref{thm-new-old-sum-compatible} will follow from the results in Step 1 and Step 4.
\end{itemize}

\begin{lem}\label{lem-jesen-ineq}
	Let $p$, $A$, $B$ be positive numbers and $t \in [0,1]$. Then 
	\begin{align}
		(A^p + B^p)^{\frac{1}{p}} \geq&  (1-t)^{\frac{1}{q}} A +t^{\frac{1}{q}}B, &\quad p>1, \label{jesen-ineq-p>1}\\
		(A^p + B^p)^{\frac{1}{p}} \leq&  (1-t)^{\frac{1}{q}} A +t^{\frac{1}{q}}B, &\quad 0<p<1,
		\label{jesen-ineq-p<1}
	\end{align}
	where $q$ satisfies $\frac{1}{p} + \frac{1}{q}=1$. In \eqref{jesen-ineq-p>1} and \eqref{jesen-ineq-p<1}, equality holds if and only if $t = \frac{B^p}{A^p +B^p}$.
\end{lem}

\begin{proof}
	Let $m= \frac{B^p}{A^p+B^p}\in (0,1)$. Since $\zeta(s) = s^{\frac{1}{p}}$ is  strictly concave on $\mathbb{R}^{+}$ for $p>1$,  Jensen's inequality implies
	\begin{equation*}
		(1-t) \( \frac{1-m}{1-t} \)^{\frac{1}{p}} + t \( \frac{m}{t}\)^{\frac{1}{p}} \leq 1, 
	\end{equation*}
	with equality if and only if $t=m$. This fact shows the desired inequality \eqref{jesen-ineq-p>1}. The proof of inequality \eqref{jesen-ineq-p<1} follows similarly. This completes the proof of Lemma \ref{lem-jesen-ineq}.
\end{proof}

\begin{prop}\label{prop-K+p L in h-convex hull}
	Let $p>0$ be a real number, let $K$ and $L$ be sets in $\mathbb{H}^{n+1}$, and let $u(z) = \frac{1}{p} \log \(a \g_K^p(z)+ b\g_L^p(z)\)$ be a function on $\mathbb{S}^n$. Then
	\begin{equation}\label{K+p L in h-convex hull}
		a \cdot K\widetilde{+}_p b \cdot L \subset \bigcap_{z \in \mathbb{S}^n} \overline{B}_z(u(z)).
	\end{equation}
\end{prop}

\begin{proof}
	It is easy to see that the Assumption \ref{assump-p-sum} is equivalent to
	\begin{align}
		&\metric{-(1-t)^{\frac{1}{q}} X - t^{\frac{1}{q}} Y}{(z,1)}-\metric{-W}{(z,1)} \nonumber\\
		=& (1-t)^{\frac{1}{q}} x_{n+1} -t^{\frac{1}{q}}y_{n+1} -w_{n+1} -\metric{ (1-t)^{\frac{1}{q}} x -t^{\frac{1}{q}}y -w}{z} \geq 0, \quad \forall \ z \in \mathbb{S}^n, \label{equiv-assump}
	\end{align}
	where $q$ satisfies $\frac{1}{p} + \frac{1}{q}=1$. 
	
	We consider the case $p>1$ at first. Apparently, 
	\begin{equation*}
		\metric{-X}{z} = x_{n+1}- \metric{x}{z} \geq \(|x|^2+1\)^{\frac{1}{2}}-|x|>0.
	\end{equation*}
	 Similarly, $\metric{-Y}{z}>0$.  Taking $A=a^{\frac{1}{p}}\metric{-X}{(z,1)}$ and $B=b^{\frac{1}{q}}\metric{-Y}{(z,1)}$ in Lemma \ref{lem-jesen-ineq}, we have
	\begin{equation}\label{ineq-lemma}
		\( a\metric{-X}{(z,1)}^p + b\metric{-Y}{(z,1)}^p \)^\frac{1}{p}
		= \max_{t \in (0,1)} \metric{-(1-t)^{\frac{1}{q}} a^{\frac{1}{p}}X-t^{\frac{1}{q}} b^{\frac{1}{p}}Y }{(z,1)}.
	\end{equation}
	By Definition \ref{def-p sum-p>0}, we have
	\begin{align*}
		a\cdot K\widetilde{+}_p b\cdot  L =&
		\underset{ X\in K, \ Y \in L}{\bigcup} \bigcup_{t \in[0,1]}\bigcap_{z \in \mathbb{S}^n} \left\{W \in \mathbb{H}^{n+1}: \metric{-W}{(z,1)} \leq \metric{-(1-t)^{\frac{1}{q}} a^\frac{1}{p}X - t^{\frac{1}{q}}b^{\frac{1}{p}} Y}{(z,1)}  \right\}\\
		\subset& \underset{ X\in K, \ Y \in L}{\bigcup} \bigcap_{z \in \mathbb{S}^n}  \bigcup_{t \in[0,1]} \left\{W \in \mathbb{H}^{n+1}:  \metric{-W}{(z,1)} \leq \metric{-(1-t)^{\frac{1}{q}}  a^\frac{1}{p}X - t^{\frac{1}{q}}  b^\frac{1}{p}Y}{(z,1)} \right\}\\
		=&\underset{ X\in K, \ Y \in L}{\bigcup} \bigcap_{z \in \mathbb{S}^n}  \left\{W \in \mathbb{H}^{n+1}:  \metric{-W}{(z,1)} \leq \max_{t \in [0,1]}\metric{-(1-t)^{\frac{1}{q}} a^\frac{1}{p} X - t^{\frac{1}{q}}  b^\frac{1}{p}Y}{(z,1)} \right\}\\
		=&  \underset{ X\in K, \ Y \in L}{\bigcup} \bigcap_{z \in \mathbb{S}^n}  \left\{W \in \mathbb{H}^{n+1}: \metric{-W}{(z,1)} \leq \( a\metric{-X}{(z,1)}^p +b \metric{-Y}{(z,1)}^p \)^\frac{1}{p}  \right\}\\
		\subset& \bigcap_{z \in \mathbb{S}^n}\underset{ X\in K, \ Y \in L}{\bigcup}  \left\{W \in \mathbb{H}^{n+1}:  \metric{-W}{(z,1)} \leq \( a\metric{-X}{(z,1)}^p +b \metric{-Y}{(z,1)}^p \)^\frac{1}{p}  \right\}\\
		=& \bigcap_{z \in \mathbb{S}^n} \left\{W \in \mathbb{H}^{n+1}:  \metric{-W}{(z,1)} \leq \( \max_{X \in K} a\metric{-X}{(z,1)}^p + \max_{Y \in L} b\metric{-Y}{(z,1)}^p \)^\frac{1}{p} \right\}\\
		=& \bigcap_{z \in \mathbb{S}^n} \left\{W \in \mathbb{H}^{n+1}:  \metric{-W}{(z,1)} \leq \( a\g_K^p(z) + b\g_L^p(z) \)^\frac{1}{p}  \right\} \nonumber\\
		=& \bigcap_{z \in \mathbb{S}^n} \overline{B}_z \(u(z) \),
	\end{align*}
	where we used \eqref{equiv-assump} and \eqref{ineq-lemma} in the first and the third equality respectively, and the last two equalities came from \eqref{horo supp funct-def} and the definition of $u(z)$. We note that the above argument is also valid for the case $p=1$ by viewing $\(1-t\)^{\frac{1}{q}} = t^{\frac{1}{q}}=1$.
	
	Since the proof of the case $0 <p <1$ of Proposition \ref{prop-K+p L in h-convex hull} follows similarly, we write it down briefly.
	\begin{align}
		K\widetilde{+}_p L =&
		\underset{ X\in K, \ Y \in L}{\bigcup} \bigcap_{t \in[0,1]}\bigcap_{z \in \mathbb{S}^n} \left\{W \in \mathbb{H}^{n+1}:  \metric{-W}{(z,1)} \leq \metric{-(1-t)^{\frac{1}{q}} a^\frac{1}{p} X - t^{\frac{1}{q}} b^\frac{1}{p} Y}{(z,1)}  \right\} \nonumber\\
		=&\underset{ X\in K, \ Y \in L}{\bigcup} \bigcap_{z \in \mathbb{S}^n}  \left\{W \in \mathbb{H}^{n+1}: \metric{-W}{(z,1)} \leq \min_{t \in [0,1]}\metric{-(1-t)^{\frac{1}{q}} a^\frac{1}{p} X - t^{\frac{1}{q}} b^\frac{1}{p} Y}{(z,1)}  \right\} \nonumber\\
		=&  \underset{ X\in K, \ Y \in L}{\bigcup} \bigcap_{z \in \mathbb{S}^n}  \left\{W \in \mathbb{H}^{n+1}:  \metric{-W}{(z,1)} \leq \( a\metric{-X}{(z,1)}^p + b\metric{-Y}{(z,1)}^p \)^\frac{1}{p}  \right\} \nonumber\\
		\subset& \bigcap_{z \in \mathbb{S}^n} \overline{B}_z \(u(z) \). \label{old sum in new sum-p<1}
	\end{align}
	We complete the proof of Proposition \ref{prop-K+p L in h-convex hull}.
\end{proof}

\begin{lem}\label{lem-old-bdy in new sum}
	Let $\frac{1}{2} \leq p \leq 2$ be a real number, and let $a \geq 0$ and $b \geq 0$ be two real numbers with $a+b \geq 1$. Assume that $K$ and $L$ are two smooth uniformly h-convex bounded domains in $\mathbb{H}^{n+1}$, here the both $K$ and $L$ can degenerate to a single point. Let $\Omega =a\cdot K+_p b \cdot L$. Then
	\begin{equation}\label{old-bdy in new sum}
		\partial \Omega \subset a \cdot K \widetilde{+}_p b \cdot L. 
	\end{equation}
\end{lem}

\begin{proof}
	Let $u(z)= \frac{1}{p} \log \(a\g_K^p(z)+ b\g_L^p(z)\)$ be the horospherical support function of $\Omega$, and let $W$ be any point on $\partial \Omega$. Then there exists $z_0 \in \mathbb{S}^n$ such that 
	\begin{equation*}
		-\metric{W}{(z_0,1)} = \g(z_0),
	\end{equation*}
	where $\g(z) =e^{u(z)}$, and we used \eqref{X-z,1}. By Lemma \ref{lem-X,nu,di X}, there also exist unique $X \in \partial K$ and $Y \in \partial L$ such that
	\begin{equation} \label{take 2 pts}
		-\metric{X}{(z_0,1)} = \g_K(z_0), \quad  	-\metric{Y}{(z_0,1)} = \g_L(z_0).
	\end{equation}
	Note that
	\begin{equation*}
		\g^p(z_0) = a \g_K^p(z_0) +b \g_L^p(z_0).
	\end{equation*}
	Furthermore, Proposition \ref{prop-Aij >=0} asserts that the points $X$, $Y$ and $W$ can be given by use of expression \eqref{X(z)}.
	
	At first, we consider the case $1 \leq p \leq 2$. Let
	\begin{equation*}
		t = \frac{b \g_L^p(z_0)}{a \g_K^p(z_0) +b\g_L^p(z_0)} = \frac{b \g_L^p(z_0)}{\g^p (z_0)}.
	\end{equation*}
    Then
	\begin{align}
		(1-t)^{\frac{1}{q}} a^{\frac{1}{p}}\g_K(z_0) + t^{\frac{1}{q}} b^\frac{1}{p} \g_L(z_0) 
		=& \(\frac{a \g_K^p(z_0)}{ \g^p(z_0) } \)^{\frac{1}{q}} a^{\frac{1}{p}} \g_K(z_0)
		+ \(\frac{b \g_L^p(z_0)}{ \g^p(z_0) } \)^{\frac{1}{q}} b^{\frac{1}{p}} \g_L(z_0)  \nonumber\\
		=&  \frac{a \g_K^p (z_0) +b \g_L^p(z_0)}{\g^{p-1}(z_0)} = \g(z_0),\label{phiK-phiL-t}
	\end{align}
	where $q$ satisfies $\frac{1}{p} +\frac{1}{q} =1$. If $p=1$, then we view $ (1-t)^{\frac{1}{q}} = t^{\frac{1}{q}} =1$. Using \eqref{take 2 pts} and \eqref{phiK-phiL-t}, we obtain
	\begin{equation*}
		\metric{-W}{(z_0,1)}= \(1-t\)^{\frac{1}{q}} a^{\frac{1}{p}}\metric{-X}{(z_0,1)}+ t^{\frac{1}{q}}b^{\frac{1}{p}} \metric{-Y}{(z_0,1)},
	\end{equation*}
	 which is equivalent to
	\begin{equation}\label{eq-W on the bdry}
		(1-t)^{\frac{1}{q}}a^{\frac{1}{p}} x_{n+1} +t^{\frac{1}{q}}b^{\frac{1}{p}} y_{n+1}- w_{n+1} = \metric{(1-t)^{\frac{1}{q}} a^{\frac{1}{p}}x +t^{\frac{1}{q}}b^{\frac{1}{p}} y-w }{z_0}.
	\end{equation}
	The key idea of the remaining proof is to show 
	\begin{equation}\label{goal-abs=x+y-w,z}
		\metric{(1-t)^{\frac{1}{q}} a^{\frac{1}{p}}x +t^{\frac{1}{q}}b^{\frac{1}{p}} y-w }{z_0} =\left \vert(1-t)^{\frac{1}{q}}a^{\frac{1}{p}} x +t^{\frac{1}{q}}b^{\frac{1}{p}} y-w \right\vert.
	\end{equation}
	If \eqref{goal-abs=x+y-w,z} is true, then \eqref{eq-W on the bdry} and Definition \ref{def-p sum-p>0} imply that $W \in a \cdot K \widetilde{+}_p b \cdot L$.
	
	From the definitions of $\g(z)$ and $t$ in this proof, we have
	\begin{equation}
		(1-t)^\frac{1}{q} a^{\frac{1}{p}}D \g_K(z_0) + t^{\frac{1}{q}}  b^{\frac{1}{p}}D \g_L (z_0)  
		=\frac{ a\g_K^{p-1} D \g_K(z_0)}{  \g^{p-1} (z_0)}
		+\frac{ b\g_L^{p-1} D \g_L(z_0)}{  \g^{p-1} (z_0)}  
		 = D \g(z_0). \label{DphiK-DphiL-t}
	\end{equation}
	Combining \eqref{DphiK-DphiL-t} with \eqref{X(z)}, we deduce that the vector $(1-t)^{\frac{1}{q}}a^{\frac{1}{p}} x +t^{\frac{1}{q}} b^{\frac{1}{p}} y-w$ is parallel to $z_0$. From \eqref{eq-W on the bdry}, it remains to show 
	\begin{equation}\label{check ineq geq 0}
		(1-t)^{\frac{1}{q}}a^{\frac{1}{p}} x_{n+1} +t^{\frac{1}{q}}b^{\frac{1}{p}} y_{n+1}- w_{n+1} \geq 0.
	\end{equation}
	Since $1 \leq p \leq 2$, we know that $\zeta(s) =s^{\frac{p-2}{p}}$ is a convex function on $(0,+\infty)$, then (see also \eqref{T3>0})
	\begin{equation*}
		a\g_K^{p-2}(z_0) + b\g_L(z_0)^{p-2} \geq (a+b) \( \frac{a\g_K^p(z_0) +b\g_L^p(z_0)}{a+b} \)^{\frac{p-2}{p}} \geq  \g^{p-2}(z_0),
	\end{equation*}
	where we used the assumption $a+b \geq 1$ in the second inequality.
	Hence 
	\begin{equation}\label{inv-phiK-inv-phiL-t}
		(1-t)^{\frac{1}{q}}a^{\frac{1}{p}} \g^{-1}_K(z_0) + t^{\frac{1}{q}}b^{\frac{1}{p}}\g^{-1}_L(z_0)
		=\frac{a\g^{p-2}_K(z_0) +b\g^{p-2}_L(z_0)}{\g^{p-1}(z_0)} \geq \g^{-1}(z_0).
	\end{equation}
	Using \eqref{phiK-phiL-t}, the Cauchy-Schwarz inequality and \eqref{DphiK-DphiL-t}, we get
	\begin{align}\label{grad-phiK-phiL-t}
		&\g(z_0) \( (1-t)^\frac{1}{q} a^\frac{1}{p}\frac{|D \g_K(z_0)|^2}{\g_K(z_0)} + t^{\frac{1}{q}} b^{\frac{1}{p}} \frac{|D \g_L(z_0)|^2}{\g_L(z_0)}\) \nonumber \\
		=&\((1-t)^\frac{1}{q} a^{\frac{1}{p}}\g_K(z_0) + t^{\frac{1}{q}}b^{\frac{1}{p}} \g_L(z_0) \)  \( (1-t)^\frac{1}{q} a^{\frac{1}{p}}\frac{|D \g_K(z_0)|^2}{\g_K(z_0)} + t^{\frac{1}{q}} b^{\frac{1}{p}} \frac{|D \g_L(z_0)|^2}{\g_L(z_0)}\) \nonumber \\
		\geq& \left\vert(1-t)^\frac{1}{q}  a^{\frac{1}{p}}D \g_K(z_0) + t^{\frac{1}{q}} b^{\frac{1}{p}}D \g_L(z_0) \right\vert^2 = |D \g(z_0)|^2.
	\end{align}
	Using  \eqref{X(z)}, \eqref{DphiK-DphiL-t}, \eqref{phiK-phiL-t}, \eqref{inv-phiK-inv-phiL-t} and \eqref{grad-phiK-phiL-t}, we have
	\begin{align*}
		&(1-t)^{\frac{1}{q}}a^{\frac{1}{p}} x_{n+1} +t^{\frac{1}{q}}b^{\frac{1}{p}} y_{n+1}- w_{n+1}\\
		=& (1-t)^{\frac{1}{q}}a^{\frac{1}{p}} \(\frac{1}{2} \frac{|D \g_K(z_0)|^2}{\g_K(z_0)} + \frac{1}{2} \( \g_K(z_0) + \frac{1}{\g_K(z_0)}\)  \)\\
		&+ t^{\frac{1}{q}} b^{\frac{1}{p}}\(\frac{1}{2} \frac{|D \g_L(z_0)|^2}{\g_L(z_0)} + \frac{1}{2} \( \g_L(z_0) + \frac{1}{\g_L(z_0)}\)  \)\\
		&-  \(\frac{1}{2} \frac{|D \g(z_0)|^2}{\g(z_0)} + \frac{1}{2} \( \g(z_0) + \frac{1}{\g(z_0)}\)  \)\\
		=& \frac{1}{2} \(	(1-t)^{\frac{1}{q}}a^{\frac{1}{p}} \frac{|D \g_K(z_0)|^2}{\g_K(z_0)} + t^{\frac{1}{q}} b^{\frac{1}{p}} \frac{|D \g_L(z_0)|^2}{\g_L(z_0)} - \frac{|D \g(z_0)|^2}{\g(z_0)} \)\\
		&+ \frac{1}{2} \( (1-t)^\frac{1}{q}a^{\frac{1}{p}} \g_K(z_0) + t^{\frac{1}{q}}b^{\frac{1}{p}} \g_L(z_0) - \g(z_0)   \)\\
		&+ \frac{1}{2} \( (1-t)^{\frac{1}{q}}a^{\frac{1}{p}} \g^{-1}_K(z_0) + t^{\frac{1}{q}} b^{\frac{1}{p}} \g^{-1}_L(z_0)
		-\g^{-1}(z_0) \) \geq 0.
	\end{align*}
	Thus we obtain the desired inequality \eqref{check ineq geq 0}.
	Therefore, in the case $1 \leq p \leq 2$, we conclude that $W \in a \cdot K \widetilde{+}_p b \cdot L$ for any $W \in \partial \Omega$, and hence $\partial \Omega \subset  a \cdot K \widetilde{+}_p b \cdot L$. 
	
	Now we consider the case $\frac{1}{2} \leq p<1$. The third equality in \eqref{old sum in new sum-p<1} implies that
	\begin{equation*}
		a \cdot X \widetilde{+}_p b \cdot Y =a \cdot X +_p b \cdot Y,
	\end{equation*}
	where $X$ and $Y$ are the same as those in \eqref{take 2 pts}. Hence we only need to prove $W \in a \cdot X+_p b \cdot Y$. Since $X \in \partial K$, we have that $\g_K(z)-\metric{-X}{(z,1)}$ is nonnegative for all $z \in \mathbb{S}^n$ and attains its minimum $0$ at $z_0$.  Then $\g_K(z_0) = \g_X(z_0)$ and $D\g_K(z_0) = D\g_X(z_0)$. Similarly, $\g_L(z_0) = \g_Y(z_0)$ and $D\g_L(z_0) = D \g_Y(z_0)$. Define 
	\begin{equation*}
		\tilde{\g}(z) := \(a \g_X^p(z) +b \g_Y^p(z)\)^{\frac{1}{p}}.
	\end{equation*}
	It is easy to show that $\tilde{\g}(z_0) = \g(z_0)$ and $D \tilde{\g}(z_0) = D\g(z_0)$. On the other hand, we have proved in Theorem \ref{thm-def p sum-well defined} that $\partial \(a \cdot X+_p b \cdot Y\)$ can be characterized by use of  $\tilde{\g}(z)$ and expression \eqref{X(z)}.  Besides, as $W \in \partial \Omega$, we can write $W$ explicitly by use of $\g(z_0)$, $D \g(z_0)$, $z_0$ and expression \eqref{X(z)}. Putting the above facts together, we have $W \in a \cdot X+_p b \cdot Y$. This completes the proof of Lemma \ref{lem-old-bdy in new sum}.
\end{proof}

\begin{prop}\label{prop-sum of 2 pts new}
	Let $n \geq 1$, $\frac{1}{2} \leq p \leq 2$, $a \geq 0$, $b \geq 0$ with $a+b \geq 1$. Let $X=(x, x_{n+1})$ and $Y=(y, y_{n+1})$ be two points in $\mathbb{H}^{n+1}$. Then
	\begin{equation*}
		a \cdot X \widetilde{+}_p b \cdot Y = a \cdot X+_p b \cdot Y.
	\end{equation*}
\end{prop}

\begin{proof}
	Since Proposition \ref{prop-sum of 2 pts new} follows from the third equality in \eqref{old sum in new sum-p<1} when $\frac{1}{2} \leq p<1$, we can assume $1 \leq p \leq 2$ in the remaining proof.
		
	Let $u(z) = \frac{1}{p} \log \(a \g_X^p(z) +b \g_Y^p(z)\)$.  Theorem \ref{thm-def p sum-well defined} asserts that $a\cdot  X+_p b \cdot Y$ is a  h-convex bounded domain with horospherical support function $u(z)$.
	Then \eqref{K+p L in h-convex hull} implies that 
	\begin{equation*}
		a \cdot X \widetilde{+}_p b \cdot Y \subset  \cap_{z \in \mathbb{S}^n} \overline{B}_z \(u(z) \)=a \cdot X+_p b \cdot Y.
	\end{equation*}
	However, we have from \eqref{old-bdy in new sum} that $\partial \(a \cdot X+_p b \cdot Y\) \subset a \cdot X \widetilde{+}_p b \cdot Y$. Therefore, it is enough to show that  $ a \cdot X \widetilde{+}_p b \cdot Y$ is contractible.
		
	Now we use the upper half-space model of $\mathbb{H}^{n+1}$, i.e. 
	\begin{equation*}
		\mathbb{H}^{n+1} = \mathbb{R}^{n+1}_{\xi_n >0}= \{\Xi =(\xi, \xi_{n}) = (\xi_0, \ldots, \xi_{n-1}, \xi_{n} ):  \xi_{n}>0  \}.
	\end{equation*}
	By Proposition \ref{prop-point wise def-equivalent form}, we can assume that $X$ and $Y$ both locate on the geodesic line $\Gamma :=\{\Xi \in \mathbb{H}^{n+1}:  \xi = 0  \}$ without loss of generality. Denote by $\overline{XY}$ the geodesic segment connecting $X$ and $Y$. Proposition \ref{prop-point wise def-equivalent form} also implies that  $a \cdot X \widetilde{+}_p b \cdot Y$ is the union of a continuous family of geodesic balls. Furthermore, each of the balls is symmetric with respect to $\overline{XY}$. Therefore, for any point $W \in a \cdot X \widetilde{+}_p b \cdot Y$, there exists a geodesic ball $\mathscr{B} \subset a \cdot X \widetilde{+}_p b \cdot Y$ such that $W \in \mathscr{B}$, and the center of $\mathscr{B}$ lies on  $\overline{XY}$. Denote $W = (\xi^W, \xi^W_{n}) \in \mathbb{R}^{n+1}_{\xi_n >0}$. Since $\mathscr{B}$ is also a Euclidean ball in the upper half-space model, we know that the ``segment" connecting $(0, \xi^W_n)$ and $W$ lies in $\mathscr{B}$; hence it lies in  $a \cdot X \widetilde{+}_p b \cdot Y$.  This shows that  $a \cdot X \widetilde{+}_p b \cdot Y$ can shrink continuously to $\(a \cdot X \widetilde{+}_p b \cdot Y \) \cap \Gamma$.
		 
	We claim that $\(a \cdot X \widetilde{+}_p b \cdot Y \) \cap \Gamma$ is connected for $1 \leq p \leq 2$. If $p=1$, then the claim follows directly from Definition \ref{def-p sum-p>0-new}. If $1< p \leq 2$, then Proposition \ref{prop-Assump-geodesic ball} implies that $B(p,t;a,X,b,Y)$ is non-empty if and only if 
	\begin{align*}
		-1 \geq& \metric{(1-t)^\frac{1}{q} a^{\frac{1}{p}} X+ t^{\frac{1}{q}} b^{\frac{1}{p}} Y}{(1-t)^\frac{1}{q} a^{\frac{1}{p}} X+ t^{\frac{1}{q}} b^{\frac{1}{p}} Y}\\
		=&-(1-t)^{\frac{2}{q}}a^{\frac{2}{p}} - t^{\frac{2}{q}} b^{\frac{2}{p}}- 2\(1-t\)^{\frac{1}{q}} t^{\frac{1}{q}} a^{\frac{1}{p}} b^{\frac{1}{p}}\cosh d_{\mathbb{H}^{n+1}}(X,Y):= \chi(t),
	\end{align*}
	where $q$ satisfies $\frac{1}{p} + \frac{1}{q} =1$, and we used \eqref{geo-dis on hyperbolic space}. Since $1<p \leq 2$, we have $q \geq 2$. Then the functions $-(1-t)^\frac{2}{q}$, $-t^{\frac{2}{q}}$ and $-2(1-t)^\frac{1}{q} t^{\frac{1}{q}}$ are convex on $[0,1]$. Hence $\chi(t)$ is convex on $[0,1]$, and thus $I:=\{t \in [0,1]:  \chi(t) \leq -1 \}$ is a connected interval. Then by \eqref{P-p-t-a-X-b-Y}, Definition \ref{def-B(p,t,a,X,b,Y)} and Definition \ref{def-p sum-p>0-new}, we have that $\(a \cdot X \widetilde{+}_p b \cdot Y \) \cap \Gamma$ is a segment. Therefore $a \cdot X \widetilde{+}_p b \cdot Y$ is contractible.
		
	We complete the proof of Proposition \ref{prop-sum of 2 pts new}.
\end{proof}

We now turn to prove Theorem \ref{thm-new-old-sum-compatible}.

\begin{proof}[Proof of Theorem \ref{thm-new-old-sum-compatible}]
	Let $X \in K$ and $Y \in L$ be two points. Let $K_\alpha = \alpha \cdot K+_p(1-\alpha) \cdot X$ and $L_\alpha = \alpha \cdot L +_p (1-\alpha) \cdot Y$ be h-convex bounded domains, where $0 \leq \alpha \leq 1$. Since $\g_K(z) \geq \g_X(z)$ and $\g_L(z) \geq \g_Y(z)$, we know $K_\alpha \subset K$ and $L_\alpha \subset L$. 
	According to Theorem \ref{thm-def p sum-well defined}, we have the following inclusions of h-convex bounded domains, i.e.
	\begin{equation}\label{inclusion h-convex body}
		a \cdot X+_p b \cdot Y \subset a \cdot K_{\alpha_1} +_p b \cdot L_{\alpha_1}\subset a \cdot K_{\alpha_2} +_p b \cdot  L_{\alpha_2 } \subset a \cdot K+_p b \cdot L
	\end{equation}
	for all $ 0 \leq \alpha_1 \leq \alpha_2 \leq 1$. Note that the horospherical support functions of the above h-convex domains are smooth. Let $W$ be an arbitrary point in  $a \cdot K+_p b \cdot L$.
	By Proposition \ref{prop-K+p L in h-convex hull} and Proposition \ref{prop-sum of 2 pts new}, we only need to show that if $W \in a \cdot K+_p b \cdot L \backslash \(a \cdot  X+_p b \cdot Y \)$, then $W \in a \cdot K \widetilde{+}_p b \cdot L$. 	
	
	The above inclusions in \eqref{inclusion h-convex body} imply that $W \in \partial \(a \cdot K_{\alpha_0}+_p b \cdot L_{\alpha_0} \)$ for some $0 < \alpha_0 \leq 1$. Then by Lemma \ref{lem-old-bdy in new sum}, we have $W \in a \cdot K_{\alpha_0} \widetilde{+}_p b \cdot L_{\alpha_0} \subset a \cdot K \widetilde{+}_p b \cdot L$. This completes the proof of Theorem \ref{thm-new-old-sum-compatible}.
\end{proof}	

\begin{prop}\label{prop-sums not rely on origin} $ \ $
	Fix real numbers $p$, $a$, $b$ that satisfy the assumptions in Definition \ref{def-p sum} (Definition \ref{def-p sum-p>0} resp.). Then the hyperbolic $p$-sum in  Definition \ref{def-p sum} (Definition \ref{def-p sum-p>0} resp.) only depends on the relative position of the smooth uniformly h-convex bounded domains $K$ and $L$. 
\end{prop}

\begin{proof}
	From  Definition \ref{def-B(p,t,a,X,b,Y)}, we have that the geodesic balls $B (p,t;a,X,b,Y)$ and $B(a,X,b,Y)$ only depend on the relative position of $X$ and $Y$ for fixed $p$, $t$, $a$ and $b$. Then Proposition \ref{prop-sums not rely on origin} follows from Proposition \ref{prop-point wise def-equivalent form} and Theorem \ref{thm-new-old-sum-compatible}
\end{proof}

\subsection{Hyperbolic $p$-sum of geodesic balls} $ \ $

This subsection studies the geometric properties of the hyperbolic $p$-sum of geodesic balls.

Let $X$ be a point in  $\mathbb{H}^{n+1}$. Denote by $B(X, r)$ the geodesic ball of radius $r$ centered at $X$ in $\mathbb{H}^{n+1}$. If $X$ is treated as the origin of $\mathbb{H}^{n+1}$, e.g., $X= (0,1)$ in the hyperboloid model, then we will write $B(r)$ for $B(X, r)$.

We first derive the horospherical support functions of geodesic balls in $\mathbb{H}^{n+1}$.

\begin{lem}\label{lem-horo supp of geodesic ball}
	Let $X = (x, x_{n+1})$ and $\Omega= B(X, r) \subset \mathbb{H}^{n+1}$. Then the horospherical support function $u_{\Omega} (z)$ of $\Omega$ is given by
	\begin{equation}\label{horo supp of geodesic ball}
		\g_{\Omega} (z) =e^{u_{\Omega} (z)} = e^r \(x_{n+1}- \metric{x}{z}\).
	\end{equation}
	If we set $T =e^r X$, then $\g_{\Omega}(z)=- \metric{T}{(z,1)}$.
\end{lem}

\begin{proof}
	By \eqref{formula of geodesic segment}, we know that the geodesic ball $\Omega \subset \mathbb{H}^{n+1}$ is given by
	\begin{equation*}
		\Omega= \left\{ Y = \cosh t X + \sinh t v: t \in [0,r], \, v \in T_X \mathbb{H}^{n+1}, \, \metric{v}{v} =1  \right\}.
	\end{equation*}
	For any light-like vector $(z_0,1)$, there exists a unit space-like vector $v_0 \in T_X \mathbb{H}^{n+1}$ such that
	\begin{equation*}
		\(z_0,1\)= -\metric{X}{ \(z_0,1\) } X+ \metric{X}{\(z_0,1\)}v_0.
	\end{equation*}
	Note that $-\metric{X}{(z_0,1)} = x_{n+1}-\metric{x}{z_0}\geq (|x|^2+1)^{\frac{1}{2}}-|x|>0$. Using \eqref{horo supp funct-def}, we have
	\begin{align*}
		\g_{\Omega}(z_0)=& \max_{Y \in \Omega} -\metric{Y}{(z_0,1)}\\
		=& \max_{t \in [0,r]} \, \max_{v \in T_X \mathbb{H}^{n+1}, \, |v|=1} \(-\cosh t \metric{X}{(z_0,1)} -\sinh t \metric{v}{(z_0,1)}\)\\
		=&\max_{t \in [0,r]}\(-\cosh t \metric{X}{(z_0,1)} -\sinh t \min_{v \in T_X \mathbb{H}^{n+1}, \, |v|=1} \(\metric{X}{\(z_0,1\)}\metric{v}{v_0}\)\)\\
		=&\max_{t \in [0,r]}\( -e^t \metric{X}{(z_0,1)}\)\\
		=& e^r \(x_{n+1} - \metric{x}{z_0}\)\\
		=&- \metric{T}{(z_0,1)}.
	\end{align*}
	Thus we obtain the desired formula \eqref{horo supp of geodesic ball}. This completes the proof of Lemma \ref{lem-horo supp of geodesic ball}.
\end{proof}

\begin{rem} $ \ $
	\begin{enumerate}
		\item Let $T =e^r X$, $r \geq 0$ and $X \in \mathbb{H}^{n+1}$.
		Geometrically, the domain enclosed by the intersection of $\mathbb{H}^{n+1}$ and the past light cone of $T$  is exactly $B(X, r)$. This fact was proved in Proposition \ref{prop-Assump-geodesic ball}.
		\item By treating the geodesic ball $\Omega=B(X,r)$ as the outer parallel set of the point $X \in \mathbb{H}^{n+1}$ with distance $r$, a more general form of Lemma \ref{lem-horo supp of geodesic ball} will be proved in Proposition \ref{prop-geo-hyper-p-dilation} below. 
	\end{enumerate}
\end{rem}

Recall that the \emph{Minkowski norm} of a future time-like vector $X\in \mathbb{R}^{n+1,1}$ is defined by
\begin{equation*}
	N(X) = \sqrt{-\metric{X}{X}}.
\end{equation*}
Combining formula \eqref{horo supp of geodesic ball} with Definition \ref{def-p sum}, it is straightforward to see that the hyperbolic $1$-sum of geodesic balls can be identified with the sum of future time-like vectors lying above the hyperboloid model of $\mathbb{H}^{n+1}$ in $\mathbb{R}^{n+1,1}$, and thus we obtain the following Corollary \ref{cor-1 sum of balls}. We note that the case $a=b =1$ in Corollary \ref{cor-1 sum of balls} was also studied in \cite[Proposition 5]{GST13}.

\begin{cor}\label{cor-1 sum of balls}
	Let $K=B(X, r_1)$ and $L=B(Y, r_2)$ be geodesic balls in $\mathbb{H}^{n+1}$, let $T_K = e^{r_1} X$ and $T_L =e^{r_2} Y$ be vectors in $\mathbb{R}^{n+1,1}$, and let $a \geq 0$ and $b \geq 0$ be real numbers with $a+b \geq 1$. Define $T =a T_K+b T_L$ and $\Omega =a \cdot K+_1 b \cdot L$. Then
	\begin{equation*}
		\Omega = B\(\frac{T}{N(T)}, \, \log N(T) \).
	\end{equation*}
	Consequently, the hyperbolic $1$-sum of geodesic balls are geodesic balls.
\end{cor}

We next estimate the hyperbolic $p$-sum of geodesic balls for $\frac{1}{2} \leq p \leq 2$.  Let $K$ and $L$ be two geodesic balls centered at the same point in $\mathbb{H}^{n+1}$. We will show that the hyperbolic $p$-sum of $K$ and $L$ will become larger if we separate their centers.
 
\begin{thm}\label{thm-sum of balls 0.5--2}
	Let $\frac{1}{2} \leq p \leq 2$, let $a>0$ and $b>0$ such that $a+b \geq 1$, and let $K= B(X, r_1)$ and $L = B(Y, r_2)$ be two geodesic balls in $\mathbb{H}^{n+1}$. Let  $\Omega = a \cdot K+_p b \cdot L$. Then there is a geodesic ball of radius $\frac{1}{p}\log \(a e^{p r_1+ b e^{p r_2}}\)$ lying inside $\Omega$. Moreover, the above inclusion is proper if and only if $X \neq Y$.
\end{thm}

\begin{proof}
	Define $t \in (0,1)$ by
	\begin{equation}\label{const t}
		t= \frac{b e^{pr_2}}{ae^{pr_1} +b e^{pr_2}}.
	\end{equation}
	Theorem \ref{thm-def p sum-well defined} asserts that
	\begin{equation*}
		\widetilde{\Omega}:= (1-t) \cdot X +_p t \cdot Y
	\end{equation*}
	is well-defined and non-empty. Then we can choose a point $T \in \widetilde{\Omega}$, and hence
	\begin{equation*}
		\g_T(z) \leq  \g_{\widetilde{\Omega}}(z) = \((1-t) \g_X(z)^p + t \g_Y(z)^p\)^{\frac{1}{p}}, \quad \forall \ z \in \mathbb{S}^n.
	\end{equation*}
	Substituting \eqref{const t} into the above inequality, we arrive at
	\begin{equation}\label{ineq-p-sum balls 0.5--2}
		\(a e^{pr_1} + b e^{pr_2} \)^{\frac{1}{p}} \g_T 
		\leq  \( a \(e^{r_1} \g_X\)^p +b \(e^{r_2} \g_Y \)^p \)^{\frac{1}{p}}.
	\end{equation}
	Let $\widetilde{B} = B\(T, \frac{1}{p}\log \( ae^{pr_1}+b e^{pr_2}\)\)$. Then Lemma \ref{lem-horo supp of geodesic ball} implies 
	\begin{equation*}
		\g_K = e^{r_1} \g_X, \quad \g_L =e^{r_2} \g_Y, \quad 
		\g_{\widetilde{B}} = \( ae^{pr_1}+b e^{pr_2}\)^{\frac{1}{p}} \g_T.
	\end{equation*}
	Putting these formulas into \eqref{ineq-p-sum balls 0.5--2}, we have
	\begin{equation*}
		\g_{\widetilde{B}} \leq \(a \g_K^p+ b\g_L^p\)^{\frac{1}{p}} = \g_{\Omega},
	\end{equation*}
	which means that the geodesic ball $\widetilde{B}$ lies inside $\Omega$.
	
	Now we explore the case $\widetilde{B} =\Omega$. We first assume that $X \neq Y$, i.e. $K$ and $L$ are not centered at the same point. If $\widetilde{\Omega}$ is not a single point, then the choice of $T$ is not unique, and hence $\widetilde{B}$ is a proper subset of $ \Omega$. Otherwise, if $\widetilde{ \Omega}$ is a single point, then we infer from Lemma \ref{lem-horospp of point} that $A [\g_{\widetilde{\Omega}}(z)] = 0$ for all $z \in \mathbb{S}^n$. However, by the fourth case of Proposition \ref{prop-p-sum} and Corollary \ref{cor-2pts-A>=0}, this can not happen whenever $\frac{1}{2} \leq  p \leq 2$ and $X \neq Y$. Therefore we conclude that $\widetilde{B} \subset \neq \Omega$ when $X \neq Y$.
	
	If $X =Y$, then we will prove that $\Omega$ is a geodesic ball of radius $\frac{1}{p} \log  \(a e^{pr_1} + b e^{p r_2}\)$. Proposition \ref{prop-sums not rely on origin} implies that we can assume $X=Y=(0,1)$ without loss of generality. Consequently, $\g_K(z) =e^{r_1}$, $\g_L(z) =e^{r_2}$, and hence $\g_\Omega = \(a e^{p r_1} +b e^{p r_2} \)^{\frac{1}{p}}$. Then Lemma \ref{lem-horo supp of geodesic ball} implies that $\Omega$ is a geodesic ball of radius $\frac{1}{p} \log  \(a e^{pr_1} + b e^{p r_2}\)$ centered at the origin $(0,1)$. This completes the proof of Theorem \ref{thm-sum of balls 0.5--2}.
\end{proof}

\section{Horospherical $p$-surface area measures, Horospherical $p$-Minkowski problem and Horospherical $p$-Christoffel-Minkowski problem}\label{sec-p,k-surface area measure}
The \emph{quermassintegral} $W_k(\Omega)$ for a (geodesically) convex domain $\Omega \subset \mathbb{H}^{n+1}$ is defined by (see e.g., \cite{San04})
\begin{equation*}
	W_k(\Omega) = \frac{\omega_{k-1} \cdots \omega_0}{\omega_{n-1} \cdots \omega_{n-k}}
	\int_{\mathcal{L}_k} \chi(L_k \cap \Omega) d L_k, \quad k=1,\ldots,n,
\end{equation*}
where $\mathcal{L}_k$ is the space of $k$-dimensional totally geodesic subspaces $L_k$ in $\mathbb{H}^{n+1}$, and $\omega_i$ denotes the area of the unit sphere $\mathbb{S}^i$ in Euclidean space. If $L_k \cap \Omega \neq \emptyset$, then $\chi(L_k) =1$, otherwise $\chi(L_k) =0$.

This definition of $W_k(\Omega)$ seems difficult to be used in the calculations, hence we will give an equivalent definition of $W_k(\Omega)$ for smooth convex $\Omega$ by induction. We set 
\begin{equation*}
	W_0(\Omega) = \Vol(\Omega), \quad W_{n+1}(\Omega) = \frac{\omega_n}{n+1}.
\end{equation*} 
Furthermore, we let the curvature integrals be defined by 
\begin{equation*}
	V_{n-k} = \int_{\partial \Omega} p_k (\kappa) d\mu, \quad k=0,1,\ldots,n,
\end{equation*}
where $\kappa = (\kappa_1, \ldots, \kappa_n)$ are the principal curvatures of $\partial \Omega$. Then we can define the quermassintegral $W_k(\Omega)$ by the following relations (see \cite{HL21, HLW20})
\begin{equation*}
	V_{n-k}(\Omega) = (n-k)W_{k+1}(\Omega) + kW_{k-1}(\Omega), \quad k=1,2, \ldots,n.
\end{equation*}
Since we focus on the geometry of h-convex domains, we would like to consider the following \emph{modified quermassintegrals} of a h-convex domain $\Omega \subset \mathbb{H}^{n+1}$, which were introduced by Andrews et al. \cite{ACW18},
\begin{equation}\label{def-modified quermassintegral}
	\widetilde{W}_k(\Omega):= \sum_{i=0}^k (-1)^{k-i}C_k^i W_i(\Omega), \quad k=0,1,\ldots,n.
\end{equation} 
Equivalently, that means
\begin{equation*}
	{W}_k (\Omega) = \sum_{i=0}^k C_k^i \widetilde{W}_i (\Omega).
\end{equation*}
If the domain $\Omega$ is smooth, then the modified quermassintegrals can also be defined by induction, i.e.
\begin{equation}\label{induction def of modi-quer-intg}
	\begin{aligned}
	\widetilde{W}_0(\Omega) =& \Vol(\Omega), \quad \widetilde{W}_1 (\Omega) = \frac{1}{n} |\partial \Omega| -\Vol (\Omega),\\
	\widetilde{W}_{k+1}(\Omega)=& \frac{1}{n-k} \int_{\partial \Omega} p_k (\tilde{\kappa}) d\mu 
	- \frac{n-2k}{n-k} \widetilde{W}_k (\Omega), \quad k=1,\ldots,n-1,
	\end{aligned}
\end{equation}
where $\tilde{\kappa} = \kappa-1 = (\kappa_1-1, \ldots, \kappa_n-1)$ are the shifted principal curvatures of $\partial \Omega$.

Andrews et al. \cite{ACW18} showed that the modified quermassintegral $\widetilde{W}_k(\Omega)$ is monotone with respect to the inclusion of h-convex domains.

\begin{prop}[\cite{ACW18}]\label{prop-domain monotone modified quermass}
	If $\Omega_0$ and $\Omega_1$ are h-convex domains with $\Omega_0 \subset \Omega_1$, then $\widetilde{W}_k (\Omega_0) \leq \widetilde{W}_k(\Omega_1)$.
\end{prop}
Furthermore, they gave the following evolution equations (see \cite[Lemma 2.4]{ACW18}) of $\widetilde{W}_k(\Omega)$ along general flows.

\begin{lem}[\cite{ACW18}]\label{lem-variation of modified quermass}
	Along the flow $\partial_t X =\mathscr{F} \nu$, the modified quermassintegral $\widetilde{W}_k$ of the evolving domain $\Omega_t$ satisfies
	\begin{equation*}
	\frac{d}{dt} \widetilde{W}_k(\Omega_t) = \int_{M_t} p_k(\tilde{\kappa}) \mathscr{F} d \mu_t, \quad
	k=0, \ldots, n.
	\end{equation*}
\end{lem}

Now we introduce the mixed modified quermassintegrals which are defined by the variation of modified quermassintegrals along the hyperbolic $p$-sum in Definition \ref{def-p sum}.

\begin{defn}\label{def-p mixed k-th modified qmintegral}
	Let $0 \leq k \leq n$ be an integer and $\frac{1}{2} \leq p \leq 2$ be a real number. We define the $p$-mixed $k$-th modified quermassintegral of two smooth uniformly h-convex bounded domains $K$, $L \subset \mathbb{H}^{n+1}$ by
	\begin{equation*}
		\widetilde{W}_{p,k} (K,L) = \lim_{t \to 0^+} \frac{\widetilde{W}_{k} (K +_p t \cdot L) - \widetilde{W}_{k}(K)}{t}.
	\end{equation*} 
\end{defn}

To calculate the integral representation of $\widetilde{W}_{p,k} (K,L)$, we should study the scalar evolution equation of the horospherical support function $u(z)$ along a general flow $\partial_t X = \mathscr{F} \nu$. Now we assume that, on a small time interval, the solution to the above flow is given by a vector-valued function $X(z,t)$ defined on $\mathbb{S}^n \times [0, \varepsilon)$,  where $X(z, t)$ is determined by function $\g(z,t)$ and expression \eqref{X(z)}. By \eqref{X-nu}, \eqref{nu} and DeTurck's trick, we have
\begin{align}
	\mathscr{F} =&\metric{\partial_t X(z,t) }{\nu}
	=\metric{\partial_t \(X(z,t)-\nu(z,t)\)}{\nu}
	= \metric{\partial_t  \(\frac{ (z,1)}{\g(z,t)} \)}{ \nu} \nonumber\\
	=& -\frac{1}{\g} \partial_t u(z,t) \metric{(z,1)}{\nu}
	= \partial_t u(z,t).\label{scalr-flow DT's trick}
\end{align}
Hence, the scalar evolution equation of $u(z,t)$ is given by $\partial_t u(z,t) = \mathscr{F}$. Equivalently,
\begin{equation*}
	\partial_t \g(z,t) = \g \mathscr{F}.
\end{equation*}
We also refer readers to \cite[Proposition 5.3]{ACW18} for details. 

Denote by $\tilde{\lambda} = \tilde{\kappa}^{-1} = (\frac{1}{\kappa_1-1}, \ldots, \frac{1}{\kappa_n-1})$ the \emph{shifted principal radii of curvature} of a smooth uniformly h-convex hypersurface $M \subset \mathbb{H}^{n+1}$. 

\begin{lem}\label{lem-formula-Wpk-K-L}
	Let $k$,  $p$, $K$ and $L$ satisfy the assumptions in Definition \ref{def-p mixed k-th modified qmintegral}. Then the integral representation of the $p$-mixed $k$-th modified quermassintegral $\widetilde{W}_{p,k} (K,L)$ is given by
	\begin{equation*}
		\widetilde{W}_{p,k} (K,L)=
		\frac{1}{p} \int_{\mathbb{S}^n} \g_L^p(z) \g_K^{-p-k}(z) p_{n-k} (A[\g_K(z)]) d \sigma.
	\end{equation*}
\end{lem}

\begin{proof}
	Let $\Omega_t = K+_p t \cdot L$. Then the horospherical support function of $\Omega_t$ is given by
	\begin{equation*}
		u(z,t) = \frac{1}{p} \log (\g_K^p(z) +t \g_L^p(z) ).
	\end{equation*}
	By Definition \ref{def-p mixed k-th modified qmintegral}, we know 
	\begin{equation*}
		\widetilde{W}_{p,k} (K,L) = \left. \frac{d}{dt} \right|_{t=0} \widetilde{W}_{p,k} (\Omega_t).
	\end{equation*}
	Using \eqref{scalr-flow DT's trick}, we have that the initial speed of the above variation is given by 
	\begin{equation*}
		\frac{\partial }{\partial t} u (z,0)= \frac{1}{p} \frac{\g_L^p(z)}{\g_K^p(z)}.
	\end{equation*} 
	Then Lemma \ref{lem-variation of modified quermass} implies
	\begin{equation}\label{dt W-pk-1}
		\left. \frac{d}{dt} \right|_{t=0} \widetilde{W}_{k} (\Omega_t)
		= \frac{1}{p} \int_{\partial K} \g_L^p \g_K^{-p} p_k (\tilde{\kappa}) d\mu.
	\end{equation}
	Using \eqref{rel-area element} and \eqref{shifted curvature-support function}, we have
	\begin{equation}\label{dt W-pk-2}
		\g_K^{-p} p_k (\tilde{\kappa}) d\mu
		= \g_K^{-p} \frac{p_{n-k} (\tilde{\lambda})}{p_n (\tilde{\lambda})}  \det A[\g_K] d \sigma
		=\g_K^{-p-k} p_{n-k} (A[\g_K]) d \sigma.
	\end{equation} 
	Then Lemma \ref{lem-formula-Wpk-K-L} follows by substituting \eqref{dt W-pk-2} into  \eqref{dt W-pk-1}.
\end{proof}

\begin{defn}\label{def-(p,k)surface area measure}
	Let $0 \leq k \leq n$ be an integer and $p$ be a real number. The \emph{$k$-th horospherical $p$-surface area measure} of a smooth uniformly h-convex bounded domain $K \subset \mathbb{H}^{n+1}$ is defined by
	\begin{equation*}
		d S_{p,k}(K,z) = \g_K^{-p-k}p_{n-k} (A [\g_K]) d \sigma.
	\end{equation*}
	Particularly, we call $dS(K,z) := dS_{0,0} (K,z)$ the \emph{horospherical surface area measure} of $K$.
\end{defn}

By Lemma \ref{lem-formula-Wpk-K-L} and Definition \ref{def-(p,k)surface area measure}, we have
\begin{equation*}
	\widetilde{W}_{p,k}(K,L) = \frac{1}{p}\int_{\mathbb{S}^n} \g_L^p d S_{p,k}(K,z).
\end{equation*}
This motivates us to propose the following two prescribed measure problems by using Definition \ref{def-(p,k)surface area measure}, which are nonlinear elliptic PDEs on $\mathbb{S}^n$. 

\begin{prob}[Horospherical $p$-Minkowski problem]\label{prob-Horospherical p-Minkowski problem}
	Let $n \geq 1$ be an integer and $p$ be a real number. Given a smooth positive function $f(z)$ defined on $\mathbb{S}^n$, what are necessary and sufficient conditions for $f(z)$, such that there exists a smooth uniformly h-convex bounded domain $K \subset \mathbb{H}^{n+1}$ satisfying 
	\begin{equation*}
		d S_{p,0}\(K,z \) = f(z)d\sigma. 
	\end{equation*} 
	That is, finding a smooth positive solution $\g(z)$ to
	\begin{equation}\label{Horo-p-Min-phi}
		\g^{-p}(z) p_n \(A[\g(z)] \) =f(z),
	\end{equation}
	such that $A[\g(z)]>0$ for all $z\in \mathbb{S}^n$.
\end{prob}

\begin{rem}
	Even though we assumed the smoothness of $f(z)$ in Problem \ref{prob-Horospherical p-Minkowski problem}, the polytopal version of Problem \ref{prob-Horospherical p-Minkowski problem} is still interesting. Recall that the classical Minkowski problem for polytopes was proposed and completely solved by Minkowski \cite{Min1897}, and the $L_p$ Minkowski problem for polytopes was also well studied, see e.g. \cites{HLYZ05,BHZ16, Zhu15,Zhu17}. The counterpart of Problem \ref{prob-Horospherical p-Minkowski problem} for h-convex polytopes (intersections of finite horo-balls) has been called the discrete horospherical $p$-Minkowski problem and has been studied in a recent joint work of the authors with Yao Wan \cite{LWX}. 
\end{rem}

\begin{prob}[Horospherical $p$-Christoffel-Minkowski problem]\label{prob-Horospherical p-Christoffel-Minkowski problem}
	Let $n \geq 2$ and $1 \leq k \leq n-1$ be integers, and let $p$ be a real number. For a given smooth positive function $f(z)$ defined on $\mathbb{S}^n$, what are necessary and sufficient conditions for $f(z)$, such that there exists a smooth uniformly h-convex bounded domain $K \subset \mathbb{H}^{n+1}$ satisfying
	\begin{equation*}
		d S_{p,k}\(K,z\) = f(z)d\sigma.
	\end{equation*} 
	That is, finding a smooth positive solution $\g(z)$ to
	\begin{equation}\label{Horo-p-Ch-Min-phi}
		\g^{-p-k}(z)p_{n-k} \(A[\g(z)] \) = f(z),
	\end{equation}
	such that $A[\g(z)]>0$ for all $z\in \mathbb{S}^n$.
\end{prob}

\section{Hyperbolic $p$-dilation and Steiner formulas}\label{sec-Steiner formula}
In Section \ref{sec-Steiner formula}, we study the (weighted) Steiner formulas for modified quermassintegrals (weighted volume) of smooth h-convex bounded domains.  

\subsection{Hyperbolic $p$-dilation}$ \ $

\begin{defn}\label{def-p dilation}
	Let $p$ and $a$ be real numbers such that $p>0$ and $a \geq 1$, and let $K$ be a smooth uniformly h-convex bounded domain in $\mathbb{H}^{n+1}$ $(n \geq 1)$ with horospherical support function $u_K(z)$. We allow $K$ to degenerate to a single point. We define the hyperbolic $p$-dilation $\Omega:= a \cdot_p K$ of $K$ by the smooth uniformly h-convex bounded domain with horospherical support function 
	\begin{equation*}
		u_\Omega(z) : = u_K(z) + \frac{1}{p} \log a.
	\end{equation*} 
	Equivalently, 
	\begin{equation*}
		\Omega = \bigcap_{z \in \mathbb{S}^n} \overline{B}_z (u_\Omega(z)) : = 
		\bigcap_{z \in \mathbb{S}^n}  \{ X \in \mathbb{H}^{n+1}:  \log ( - \metric{X}{(z,1)}) \leq u_\Omega(z)  \}.
	\end{equation*}
\end{defn}

By the first case in Proposition \ref{prop-p-sum}  and Proposition \ref{prop-Aij >=0}, the above Definition \ref{def-p dilation} is well-defined.

\begin{defn}\label{def-hyperbolic dilates}
	Let $K$ and $L$ be smooth uniformly h-convex bounded domains in $\mathbb{H}^{n+1}$. Then $K$ and $L$ are called hyperbolic dilates if there is a constant $c \geq 1$ such that either $K = c \cdot_1 L$ or $L = c \cdot_1 K$.
\end{defn}

From Definition \ref{def-p dilation}, it is easy to see that $K$ and $L$ are hyperbolic dilates if and only if $\g_K^{-1}(z) \g_L(z) = e^{u_L(z) - u_K(z)}$ is constant on $\mathbb{S}^n$.

In the context, we write $a \cdot K$ for $a \cdot_p K$ if there is no confusion about the choice of $p$. The following Proposition \ref{prop-regu-p-dilation} follows directly from Case (1) of Proposition \ref{prop-p-sum}. 

\begin{prop}\label{prop-regu-p-dilation}
	Let $n$, $p$, $a$, $K$ satisfy the same assumptions as in Definition \ref{def-p dilation}. Then Definition \ref{def-p dilation} is well-defined. Furthermore, except the case that $a=1$ and $K$ is a single point, $\Omega = a \cdot_p K$ is a smooth, non-degenerate and uniformly h-convex bounded domain. 
\end{prop}

Now we give a geometric description of the hyperbolic $p$-dilation. 

\begin{prop}\label{prop-geo-hyper-p-dilation}
	Definition \ref{def-p dilation} is compatible with Definition \ref{def-p sum-p>0-new}. Let $n \geq 1$ be an integer and $p>0$ be a real number. Let $K$ be a set in $\mathbb{H}^{n+1}$ and $\Omega := a \cdot_p K$. Then $\Omega$ is the outer parallel set of $K$ with distance $\frac{1}{p} \log a$, i.e.
	\begin{equation*}
	\Omega = \{ X \in \mathbb{H}^{n+1}:  d_{\mathbb{H}^{n+1}} \( X, K\) \leq \frac{1}{p} \log a   \}.
	\end{equation*}
\end{prop}

\begin{proof}
	In this proof, we use $a \ \widetilde{\cdot}_p \  K$ to denote the dilation we defined in Definition \ref{def-p sum-p>0-new}. At first, we prove that $a  \ \widetilde{\cdot}_p \  K = a^{\frac{1}{p}}  \ \widetilde{\cdot}_1 \  K$.  Using \eqref{R-p-t-a-X-b-Y} and \eqref{P-p-t-a-X-b-Y}, we have $P(p,t; a,X,0,\cdot) = X$ and $R(p,t; a,X,0,\cdot) = (1-t)^\frac{1}{q} a^{\frac{1}{p}}$, where $q$ satisfies $\frac{1}{p} + \frac{1}{q} =1$. According to Definition \ref{def-B(p,t,a,X,b,Y)}, it is easy to show that
	\begin{align*}
		\bigcup_{t \in [0,1]} B(p,t; a, X, 0, \cdot)=& B( a^{\frac{1}{p}},X,0,\cdot), \quad p>1,\\
		\bigcap_{t \in [0,1]} B(p,t; a, X, 0, \cdot)=&  B( a^{\frac{1}{p}},X,0,\cdot), \quad 0<p<1,
	\end{align*}
	and $B( a^{\frac{1}{p}},X,0,\cdot)= B(X, \frac{1}{p} \log a)$ is the geodesic ball of radius $\frac{1}{p} \log a$ centered at $X$. Using Definition \ref{def-p sum-p>0-new}, this implies 
	\begin{align}
		a  \ \widetilde{\cdot}_p \  K =& a^{\frac{1}{p}}  \ \widetilde{\cdot}_1 \  K, \label{poinw-dila-rel}\\
		a  \ \widetilde{\cdot}_p \  K =& \{ X \in \mathbb{H}^{n+1}: d_{\mathbb{H}^{n+1}} \( X, K\) \leq \frac{1}{p} \log a   \}. \label{pw-dila-geom}
	\end{align}
	
	Now we assume that $K$ is a smooth uniformly h-convex bounded domain.
	By Definition \ref{def-p dilation}, it is easy to see that $\Omega=a \cdot_p K = a^{\frac{1}{p}} \cdot_1 K$. By Theorem \ref{thm-new-old-sum-compatible} and \eqref{poinw-dila-rel}, we know $\Omega =a \cdot_p K= a^\frac{1}{p} \ \widetilde{\cdot}_1 \ K = a \ \widetilde{\cdot}_p \ K$. Hence Definition \ref{def-p dilation} is compatible with Definition  \ref{def-p sum-p>0-new}, and $\Omega$ can be described by using \eqref{pw-dila-geom}. We complete the proof of Proposition \ref{prop-geo-hyper-p-dilation}.
\end{proof}

\subsection{Steiner formulas for modified quermassintegrals in $\mathbb{H}^{n+1}$}$ \ $

For a fixed smooth h-convex bounded domain $\Omega=\Omega_0$ with horospherical support function $u(z)$, we let $\Omega_t = e^t \cdot_1 \Omega$. Then  $\Omega_t$ is the h-convex domain with horospherical support function $u_t(z) = u(z)+t$. Define $\g_t(z) =e^{u_t(z)}$, then $\g_t(z) =e^t \g(z)$. The h-convexity of $\Omega_t$ follows from Proposition \ref{prop-regu-p-dilation}. According to Proposition \ref{prop-geo-hyper-p-dilation}, $\Omega_\rho = \left\{ X \in \mathbb{H}^{n+1}:  d_{\mathbb{H}^{n+1}} (X, \Omega) \leq \rho \right\}$.

\begin{thm}\label{thm-Stein-modi-quemass}
	The Steiner formulas for modified quermassintegrals of a smooth uniformly h-convex bounded domain $\Omega \subset \mathbb{H}^{n+1}$ $(n \geq 1) $ are
	\begin{equation*}
		\widetilde{W}_k(\Omega_\rho) - \widetilde{W}_k(\Omega) = \sum_{i=k}^n C_{n-k}^{i-k} \left( \int_{\partial \Omega} p_i (\tilde{\kappa}) d \mu \right)  \int_0^\rho e^{(n-k-i)t}  \sinh^{-k+i} t dt, \quad k=0,1,\ldots, n.
	\end{equation*}
\end{thm}

\begin{proof}
	By using \eqref{rel-area element}, \eqref{shifted curvature-support function} and Lemma \ref{lem-variation of modified quermass}, we have
	\begin{align*}
		\frac{d}{dt} \widetilde{W}_k(\Omega_t) =& \int_{\partial \Omega_t} p_k(\tilde{\kappa}) d \mu_t
		=\int_{\mathbb{S}^n} p_k (\tilde{\kappa}) \det(A[\g_t]) d \sigma\\
		=&\int_{\mathbb{S}^n} \g_t^{-n} p_{n-k}(\tilde{\lambda}) d\sigma
		=\int_{\mathbb{S}^n} \g_{t}^{-k} p_{n-k}(A[\g_t]) d\sigma.
	\end{align*}
	It follows that
	\begin{equation}\label{dt Wk Omega}
		C_n^k\frac{d}{dt} \widetilde{W}_k(\Omega_t) = \int_{\mathbb{S}^n} \g_t^{-k}\sigma_{n-k} (A[\g_t]) d\sigma
		=\int_{\mathbb{S}^n} e^{-kt} \g^{-k} \sigma_{n-k} (A[\g_t]) d\sigma.
	\end{equation}
	By \eqref{def A-phi}, we have
	\begin{align}
		\s_{n-k}(A [\g_t(z)]) =& \s_{n-k}\left((\g_{t})_{ij}  -\frac{1}{2}\frac{|D \g_t|^2}{\g_t} \delta_{ij} +\frac{1}{2} \( \g_t -\frac{1}{\g_t} \) \delta_{ij} \right) \nonumber\\
		=&\s_{n-k} \( e^t A_{ij} [\g] +\sinh t \frac{1}{\g} \delta_{ij} \) \nonumber\\
		=&\sum_{l=0}^{n-k} C_{n-l}^{n-k-l} \g^{k+l-n} \s_l(A[\g]) e^{lt}\sinh^{n-k-l} t.\label{sigma-k-phi-t}
	\end{align}
	Inserting \eqref{sigma-k-phi-t} into \eqref{dt Wk Omega} and using \eqref{rel-area element}, \eqref{shifted curvature-support function}, we have
	\begin{align*}
		C_n^k\frac{d}{dt} \widetilde{W}_k(\Omega_t)
		=& \sum_{l=0}^{n-k} C_{n-l}^{n-k-l} \int_{\mathbb{S}^n} \g^{l-n}\s_l (A[\g]) e^{(l-k) t} 	\sinh^{n-k-l} t d \sigma\\
		=& \sum_{l=0}^{n-k} \left( \int_{\partial \Omega} \s_{n-l} (\tilde{\kappa}) d \mu \right) C_{n-l}^{n-k-l} e^{(l-k)t} \sinh^{n-k-l} t\\
		=&\sum_{i=k}^{n}  \left( \int_{\partial \Omega} \s_i (\tilde{\kappa}) d \mu \right)  C_{i}^{i-k} e^{(n-i-k) t} \sinh^{i-k} t.
	\end{align*}
	Integrating both sides of the above formula from $0$ to $\rho$, we have
	\begin{equation*}
		\widetilde{W}_k(\Omega_\rho) - \widetilde{W}_k(\Omega) = \sum_{i=k}^n C_{n-k}^{i-k} \left( \int_{\partial \Omega} p_i (\tilde{\kappa}) d \mu \right)  \int_0^\rho e^{(n-k-i)t} \sinh^{i-k} t dt.
	\end{equation*}
	This completes the proof of Theorem \ref{thm-Stein-modi-quemass}.
\end{proof}

For the sake of completeness, let us check that the  $k=0$ case of Theorem \ref{thm-Stein-modi-quemass} corresponds to the following well-known Steiner formula, see \cite[\uppercase\expandafter{\romannumeral4} 18.4]{San04} and \cite{WX14}.

\begin{cor}\label{cor-Steiner-k=0}
	Let $\Omega \subset \mathbb{H}^{n+1}$ $(n \geq 1)$ be a smooth uniformly h-convex bounded domain. Then 
	\begin{equation*}
		\Vol(\Omega_\rho) - \Vol(\Omega) = \sum_{i=0}^n  \left( \int_{\partial \Omega} \s_i ({\kappa}) d \mu \right)
		\int_0^\rho \cosh^{n-i} t \sinh^{i} t dt.
	\end{equation*}
\end{cor}

\begin{proof}
	A direct calculation yields
	\begin{equation}\label{sigma-shifted-no shifted}
		\sigma_i(\tilde{\kappa}) = \sum_{s=0}^i (-1)^{i-s} C_n^i C_i^s p_s(\kappa)
		=\sum_{s=0}^i (-1)^{i-s}C_{n-s}^{i-s} \sigma_s (\kappa), \quad i=0,\ldots,n.
	\end{equation}
	Taking $k=0$ in Theorem \ref{thm-Stein-modi-quemass} and using \eqref{sigma-shifted-no shifted}, we have
	\begin{align}
		\Vol(\Omega_\rho) - \Vol(\Omega)
		=&\sum_{i=0}^n \( \int_{\partial \Omega} \sigma_i (\tilde{\kappa}) d\mu \)
		\int_{0}^\rho e^{(n-i) t} \sinh^i t dt  \nonumber\\
		=&\sum_{i=0}^n \( \sum_{s=0}^i (-1)^{i-s} C_{n-s}^{i-s} \int_{\partial \Omega} \sigma_s(\kappa)  d\mu \)
		\int_0^\rho e^{(n-i) t} \sinh^i t dt \nonumber\\
		=& \sum_{s=0}^n \sum_{i=s}^n \( (-1)^{i-s}C_{n-s}^{i-s} \int_0^\rho e^{(n-i) t} \sinh^i t dt \)\int_{\partial \Omega} \sigma_s(\kappa) d\mu \nonumber\\
		=& \sum_{s=0}^n \( \sum_{j=0}^{n-s} (-1)^{n-j-s} C_{n-s}^j \int_0^\rho e^{jt} \sinh^{n-j} t dt \) \int_{\partial \Omega} \sigma_s(\kappa) d\mu. \label{Steiner-k=0-1}
	\end{align}
	A direct calculation shows 
	\begin{equation}\label{cosh n-s t}
		\cosh^{n-s} t = (e^t - \sinh t)^{n-s} = \sum_{j=0}^{n-s} C_{n-s}^j e^{jt} (-1)^{n-s-j} \sinh^{n-s-j} t. 
	\end{equation}
	Inserting \eqref{cosh n-s t} into the right-hand side of \eqref{Steiner-k=0-1}, we obtain
	\begin{equation*}
		\Vol(\Omega_\rho) - \Vol(\Omega)
		=\sum_{s=0}^n \( \int_0^\rho \cosh^{n-s} t \sinh^s t dt \)\int_{\partial \Omega} \sigma_s(\kappa) d\mu.
	\end{equation*}
	We complete the proof of Corollary \ref{cor-Steiner-k=0}.
\end{proof}

\subsection{Steiner formula for weighted volume in $\mathbb{H}^{n+1}$}$ \ $

We define the \emph{weighted volume} of a domain $\Omega \subset \mathbb{H}^{n+1}$ by $\Vol_w(\Omega) = \int_\Omega \cosh r dv$. If the domain $\Omega$ is smooth, then by the divergence theorem and \eqref{div-conf-vf}, we have 
\begin{equation}\label{weighted volume}
	\int_{\partial \Omega} \tilde{u} d\mu := \int_{\partial \Omega} \metric{V}{\nu} d\mu = \int_{\Omega} \divv V dv = (n+1) \int_\Omega \cosh r dv= (n+1) \Vol_w (\Omega).
\end{equation}

\begin{thm}\label{thm-weighted Steiner formula}
	The weighted Steiner formula of a smooth uniformly h-convex bounded domain $\Omega \subset \mathbb{H}^{n+1}$ $(n\geq 1)$ is
	\begin{equation}\label{weighted Steiner formula}
		\begin{aligned}
		\Vol_w(\Omega_\rho) - \Vol_w(\Omega)
		=& \sum_{k=0}^{n} \left( \int_{\partial \Omega} \cosh r \sigma_k (\tilde{\kappa}) d \mu \right) \int_0^\rho  e^{(n-k+1)t} \sinh^{k}t dt\\
		&-\sum_{k=0}^{n} \left( \int_{\partial \Omega} (\cosh r - \tilde{u}) \sigma_k (\tilde{\kappa}) d \mu \right) \int_0^\rho  e^{(n-k)t} \sinh^{k+1}t dt.
		\end{aligned}
	\end{equation}
	Equivalently, it holds that
	\begin{equation*}
		\Vol_w(\Omega_\rho) = \Vol_w(\Omega) e^{(n+1)\rho} 
		+\sum_{k=0}^n \frac{1}{k+1} \left(\int_{\partial \Omega} (\cosh r -\tilde{u}) \sigma_k (\tilde{\kappa}) d\mu \right) e^{(n-k)\rho} \sinh^{k+1} \rho.
	\end{equation*}
\end{thm}

\begin{proof}
	Denote by $r_t$ the radial function of $\Omega_t$ defined on $\partial \Omega_t$. By the co-area formula and \eqref{rel-area element}, we have
	\begin{equation}\label{weighted Steiner-co-area}
		\frac{d}{dt} \Vol_w(\Omega_t) = \int_{\partial \Omega_t} \cosh r_t d\mu_t
		=\int_{\mathbb{S}^n} \cosh r_t \det (A[\g_t]) d\sigma.
	\end{equation}
	By $\g_t(z) = e^t \g(z)$ and \eqref{coshr}, we obtain 
	\begin{align}
		\cosh r_t =& \frac{1}{2} \frac{|D\g_t|^2}{\g_t} +\frac{1}{2} \( \g_t +\frac{1}{\g_t} \) \nonumber\\
		=& e^t \left( \frac{1}{2} \frac{|D \g|^2}{\g}+\frac{1}{2} \(\g +\frac{1}{\g} \) \right)
		-e^t \frac{1}{2\g} +e^{-t}\frac{1}{2\g} \nonumber\\
		=&e^t \cosh r -\sinh t \frac{1}{\g}.\label{cosh-r-t}
	\end{align}
	Inserting \eqref{cosh-r-t} and \eqref{sigma-k-phi-t} (take $k=0$ in \eqref{sigma-k-phi-t}) into \eqref{weighted Steiner-co-area}, we have
	\begin{align}
		\frac{d}{dt} \Vol_w(\Omega_t) =& \int_{\mathbb{S}^n} \(e^t \cosh r -\sinh t \frac{1}{\g} \) \(\sum_{l=0}^{n} \frac{1}{\g^{n-l}} \sigma_l (A[\g] ) e^{lt} \sinh ^{n-l} t \) d \sigma \nonumber\\
		=& \sum_{l=0}^{n}  \(\int_{\mathbb{S}^n} \cosh r\frac{1}{\g^{n-l}} \sigma_l (A[\g]) d \sigma \) e^{(l+1)t} \sinh^{n-l}t \nonumber\\
		&- \sum_{l=0}^{n} \(\int_{\mathbb{S}^n} \frac{1}{\g^{n-l+1}} \sigma_l (A[\g]) d \sigma\) e^{lt} \sinh^{n-l+1}t.\label{derivative of wei vol}
	\end{align}
	Integrating the both sides of \eqref{derivative of wei vol} from $0$ to $\rho$ yields
	\begin{align}
		\Vol_w(\Omega_\rho) - \Vol_w(\Omega_0)
		=& \sum_{l=0}^{n} \left( \int_{\mathbb{S}^n} \cosh r\frac{1}{\g^{n-l}} \sigma_l (A[\g]) d \sigma \right) \int_0^\rho  e^{(l+1)t} \sinh^{n-l}t dt \nonumber\\
		&-\sum_{l=0}^{n} \left( \int_{\mathbb{S}^n} \frac{1}{\g^{n-l+1}} \sigma_l (A[\g]) d \sigma \right) \int_0^\rho  e^{lt} \sinh^{n-l+1}t dt. \label{weighted Steiner-mid}
	\end{align}
	Taking $l =n-k$ in \eqref{weighted Steiner-mid}, and using  \eqref{rel-area element} and \eqref{shifted curvature-support function}, we obtain
	\begin{align*}
		\Vol_w(\Omega_\rho) - \Vol_w(\Omega)
		=& \sum_{k=0}^{n} \left( \int_{\partial \Omega} \cosh r \sigma_k (\tilde{\kappa}) d \mu \right) \int_0^\rho  e^{(n-k+1)t} \sinh^{k}t dt\\
		&-\sum_{k=0}^{n} \left( \int_{\partial \Omega} \frac{1}{\g} \sigma_k (\tilde{\kappa}) d \mu \right) \int_0^\rho  e^{(n-k)t} \sinh^{k+1}t dt,
	\end{align*}
	thus we obtain \eqref{weighted Steiner formula}. Furthermore, by using \eqref{1/phi, coshr-u}, we have
	\begin{align}
		&\Vol_w(\Omega_\rho) - \Vol_w(\Omega) \nonumber\\
		=& \sum_{k=0}^n \left( \int_{\partial \Omega} \frac{1}{\g} \sigma_k (\tilde{\kappa}) d \mu \right)
 		\int_0^\rho  e^{(n-k+1)t} \sinh^{k}t dt
 		+\sum_{k=0}^n \left( \int_{\partial \Omega} \tilde{u} \sigma_k (\tilde{\kappa}) d \mu \right)
		\int_0^\rho  e^{(n-k+1)t} \sinh^{k}t dt \nonumber\\
		& -\sum_{k=0}^{n} \left( \int_{\partial \Omega} \frac{1}{\g} \sigma_k (\tilde{\kappa}) d \mu \right) \int_0^\rho  e^{(n-k)t} \sinh^{k+1}t dt \nonumber\\
		=& \sum_{k=0}^n \left( \int_{\partial \Omega} \frac{1}{\g} \sigma_k (\tilde{\kappa}) d \mu \right)
		\int_0^\rho  e^{(n-k)t} \sinh^{k}t \cosh t dt
		+\sum_{k=0}^n \left( \int_{\partial \Omega} \tilde{u} \sigma_k (\tilde{\kappa}) d \mu \right)
		\int_0^\rho  e^{(n-k+1)t} \sinh^{k}t dt.\label{new form WSF}
	\end{align}
	For $0 \leq k \leq n$, by integration by parts, we have
	\begin{equation}\label{integrat by parts}
		\int_0^\rho  e^{(n-k)t} \sinh^{k}t \cosh t dt
		= \frac{1}{k+1} e^{(n-k)\rho} \sinh^{k+1}\rho - \frac{n-k}{k+1}
		\int_0^\rho  e^{(n-k)t} \sinh^{k+1}t dt.
	\end{equation}
	By \eqref{eq-shifted Minkowski formula} and \eqref{1/phi, coshr-u}, we know
	\begin{equation}\label{use shif Min for}
		\int_{\partial \Omega} \tilde{u} \sigma_k (\tilde{\kappa}) d \mu
		= \frac{n-k+1}{k} \int_{\partial \Omega} \frac{1}{\g} \sigma_{k-1} (\tilde{\kappa}) d \mu, \quad k=1,\ldots,n.
	\end{equation}
	Inserting \eqref{integrat by parts} and \eqref{use shif Min for} into \eqref{new form WSF},  we have by \eqref{1/phi, coshr-u}
	\begin{align*}
		\Vol_w(\Omega_\rho) =&  \Vol_w(\Omega) +\sum_{k=0}^{n} \frac{1}{k+1} \left(\int_{\partial \Omega} \frac{1}{\g} \sigma_k (\tilde{\kappa}) d\mu \right) e^{(n-k)\rho} \sinh^{k+1} \rho\\
		&-\sum_{k=0}^{n} \frac{n-k}{k+1} \left(\int_{\partial \Omega} \frac{1}{\g} \sigma_k (\tilde{\kappa}) d\mu \right) \int_0^\rho e^{(n-k)t} \sinh^{k+1} t dt\\
		&+\sum_{k=0}^{n-1} \frac{n-k}{k+1} \left(\int_{\partial \Omega} \frac{1}{\g} \sigma_k (\tilde{\kappa}) d\mu \right) \int_0^\rho e^{(n-k)t} \sinh^{k+1} t dt +(n+1)\Vol_w(\Omega) \int_0^\rho e^{(n+1)t} dt\\
		=&\Vol_w(\Omega) e^{(n+1)\rho} + \sum_{k=0}^n \frac{1}{k+1} \left(\int_{\partial \Omega} (\cosh r -\tilde{u}) \sigma_k (\tilde{\kappa}) d\mu \right) e^{(n-k)\rho} \sinh^{k+1} \rho.
	\end{align*}
\end{proof}

\begin{rem}\label{rem-WSF-convex}
	In Theorem \ref{thm-weighted Steiner formula}, the uniformly h-convex assumption of $\Omega$ is not necessary, see Proposition \ref{prop-new-pf-weighted Steiner formula}.
\end{rem}

\subsection{Application}$ \ $

In \cite[Equation (25)]{Xia16}, Chao Xia obtained the following geometric inequality for h-convex bounded domains,
\begin{equation}\label{ineq-Xia16}
	\lim_{\rho \to +\infty} \frac{\Vol_w(\Omega_\rho)^{\frac{1}{n+1}}}{\sinh \rho} \leq
	\Vol_w(\Omega)^{\frac{1}{n+1}}+ \frac{1}{n+1}\Vol_w(\Omega)^{-\frac{n}{n+1}} \int_{\partial \Omega} \cosh r d\mu.
\end{equation}
Combing \eqref{ineq-Xia16} with the weighted Steiner formula \eqref{weighted Steiner formula}, we prove the following new weighted Alexandrov-Fenchel type inequality for h-convex domains. 

\begin{thm}\label{thm-new-weighted-inequality}
	Let $\Omega$ be a smooth uniformly h-convex bounded domain in $\mathbb{H}^{n+1}$ $(n \geq 1)$. Then 
	\begin{equation*}
		\Vol_w(\Omega) + \sum_{k=0}^n \frac{1}{k+1} \int_{\partial \Omega} \cosh r \sigma_k(\kappa) d\mu
		\leq
		\left( 	\Vol_w(\Omega)^{\frac{1}{n+1}} +\frac{1}{n+1} 	\Vol_w(\Omega)^{-\frac{n}{n+1}}\int_{\partial \Omega} \cosh r d\mu \right)^{n+1}.
	\end{equation*}
\end{thm}

\begin{proof}
	We will prove this theorem by calculating the left-hand side of \eqref{ineq-Xia16} directly. Using the L' Hospital's rule, we have
	\begin{equation*}
		\lim_{\rho \to +\infty} \frac{\int_0^\rho e^{(n-k+1)t} \sinh^k t dt}{\sinh^{n+1} \rho}
		= \lim_{\rho \to +\infty} \frac{e^{(n-k+1)\rho} \sinh^k \rho}{(n+1) \sinh^n \rho \cosh \rho}
		=\frac{2^{n-k+1}}{n+1}.
	\end{equation*} 
	Dividing the both sides of \eqref{weighted Steiner formula} by $\sinh^{n+1} \rho$ and taking $\rho \to +\infty$, we have
	\begin{align}
		\lim_{\rho \to +\infty} \frac{\Vol_w(\Omega_\rho)}{\sinh^{n+1} \rho}
		=& \sum_{k=0}^n \left(\int_{\partial \Omega} \cosh r \sigma_k(\tilde{\kappa})d\mu \right) \frac{2^{n-k+1}}{n+1}
		-\sum_{k=0}^n  \left(\int_{\partial \Omega} (\cosh r -\tilde{u}) \sigma_k(\tilde{\kappa})d\mu \right)\frac{2^{n-k}}{n+1} \nonumber\\
		=&\sum_{k=0}^n \left(\int_{\partial \Omega} \cosh r \sigma_k(\tilde{\kappa})d\mu \right) \frac{2^{n-k}}{n+1}
		+\sum_{k=0}^n  \left(\int_{\partial \Omega} \tilde{u}\sigma_k(\tilde{\kappa})d\mu \right)\frac{2^{n-k}}{n+1} \nonumber\\
		:=&I_1 +I_2,\label{lim-Xia=I1,I2}
	\end{align}
	where $I_1$ is the first term and $I_2$ is the second term in the second line of \eqref{lim-Xia=I1,I2}. Using \eqref{sigma-shifted-no shifted}, we have
	\begin{align}
		I_1 =&
		\sum_{k=0}^n \sum_{i=0}^k \left( \int_{\partial \Omega} \cosh r \sigma_i(\kappa) d\mu \right) 
		\frac{(-1)^{k-i}C_{n-i}^{k-i} 2^{n-k}}{n+1} \nonumber\\
		=&
		\sum_{i=0}^n \left( \int_{\partial \Omega} \cosh r \sigma_i(\kappa) d\mu \right) \left( \sum_{k=i}^n \frac{(-1)^{k-i}C_{n-i}^{k-i} 2^{n-k}}{n+1}\right) \nonumber\\
		=& \frac{1}{n+1} \sum_{i=0}^n  \int_{\partial \Omega} \cosh r \sigma_i(\kappa) d\mu.\label{I1}
	\end{align}
	By using \eqref{eq-Minkowski formula}, \eqref{weighted volume} and a similar calculation as above, we have
	\begin{align}
		I_2 =&\sum_{k=0}^n  \left(\int_{\partial \Omega} \tilde{u}\sigma_k(\tilde{\kappa})d\mu \right)\frac{2^{n-k}}{n+1}
		=\frac{1}{n+1} \sum_{k=0}^n \int_{\partial \Omega} \tilde{u} \sigma_k(\kappa) d\mu \nonumber\\
		=&\Vol_w(\Omega)+\frac{1}{n+1} \frac{C_n^k}{C_n^{k-1}}\sum_{k=1}^n \int_{\partial \Omega} \cosh r \sigma_{k-1}(\kappa) d\mu \nonumber\\
		=&\Vol_w(\Omega) +\sum_{k=0}^{n-1} \frac{n-k}{(n+1)(k+1)} \int_{\partial \Omega} \cosh r \sigma_k(\kappa) d\mu.\label{I2}
	\end{align}
	Substituting \eqref{I1} and \eqref{I2} into the right-hand side of \eqref{lim-Xia=I1,I2}, we obtain
	\begin{equation*}
		\lim_{\rho \to +\infty} \frac{\Vol_w(\Omega_\rho)}{\sinh^{n+1} \rho}
		=\Vol_w(\Omega) + \sum_{k=0}^n \frac{1}{k+1}\int_{\partial \Omega} \cosh r \sigma_k(\kappa) d\mu.
	\end{equation*}
	Then Theorem \ref{thm-new-weighted-inequality} follows by substituting the above formula into \eqref{ineq-Xia16}.
\end{proof}

Here we give an explicit example of the result in Theorem \ref{thm-new-weighted-inequality}. Letting $n=2$ and $\Omega \subset \mathbb{H}^3$ be a smooth uniformly h-convex bounded domain, we have
\begin{equation*}
	\lim_{\rho \to +\infty} \frac{\Vol_w(\Omega_\rho)}{\sinh^3 \rho}
	= \frac{1}{3} \int_{\partial \Omega} \cosh r \sigma_2(\kappa) d\mu
	+\frac{1}{2} \int_{\partial \Omega} \cosh r H d\mu+\int_{\partial \Omega} \cosh r d\mu + \Vol_w(\Omega).
\end{equation*}
Then Theorem \ref{thm-new-weighted-inequality} reduces to
\begin{align*}
	&\frac{1}{3} \int_{\partial \Omega} \cosh r \sigma_2(\kappa) d\mu
	+\frac{1}{2} \int_{\partial \Omega} \cosh r H d\mu \\
	\leq& \frac{1}{3}\Vol_w(\Omega)^{-1} \(\int_{\partial \Omega} \cosh r  d\mu \)^2
	+\frac{1}{27}\Vol_w(\Omega)^{-2} \(\int_{\partial \Omega} \cosh r  d\mu \)^3.
\end{align*}
For other Alexandrov-Fenchel type inequalities proved by using Steiner formulas in space forms, readers can refer to e.g. \cites{Nat15, WX15}.

\section{Existence of solutions to Horospherical $p$-Minkowski problem and Horospherical $p$-Christoffel-Minkowski problem}\label{sec-horosp-p-Christoffel Minkowski problem}

In Section \ref{sec-horosp-p-Christoffel Minkowski problem}, we study Problem \ref{prob-Horospherical p-Minkowski problem} and Problem \ref{prob-Horospherical p-Christoffel-Minkowski problem}, i.e.
\begin{equation}\label{eq-p-CM problem}
	\g^{-p-k}(z) p_{n-k} \(A[\g(z)] \) =f(z), \quad k=0,1,\ldots,n-1,
\end{equation}
where $u(z)$ is the horospherical support function of the solution and $\g(z) := e^{u(z)}$. By \eqref{shifted curvature-support function}, equation \eqref{eq-p-CM problem} is equivalent to
\begin{equation}\label{eq-p-CM problem-with-lambda}
	\g^{-p-n}(z)  p_{n-k} \(\tilde{\lambda}\(G^{-1}(z) \)\) = f(z), \quad  k=0,1,\ldots,n-1,
\end{equation}
where $G$ is the horospherical Gauss map of the solution, and $\tilde{\lambda}$ are the shifted principal radii of curvature at $G^{-1}(z)$. When $n \geq 1$, $k=0$ and $p =0$ in equation \eqref{eq-p-CM problem},
\begin{equation}\label{pre-horo-surf-area-problem}
	p_{n}\(A[\g(z)]\) =f(z)
\end{equation}
is the equation for the prescribed horospherical surface area measure problem by  \eqref{rel-area element}. When $n \geq 1$, $k=0$ and $p=-n$ in equation \eqref{eq-p-CM problem},
\begin{equation}\label{pre-shif-Gauss-problem}
	\g^{n}(z) p_{n} \(A[\g(z)] \)=f(z)
\end{equation}
is the equation for prescribed shifted Gauss curvature problem. 

To obtain the existence of solutions to the Christoffel-Minkowski problem and the $L_p$ Christoffel-Minkowski problem in Euclidean space $\mathbb{R}^{n+1}$, researchers always need to make convex assumptions on the prescribed measures, see e.g. \cites{GM03, GX18, BIS20}. To obtain the existence of solutions to equation \eqref{eq-p-CM problem}, we need to make the following assumptions on $f(z)$ when $1 \leq k \leq n-1$.

\begin{assump}\label{assump-h}
	Let $n \geq 2$ and $1 \leq k \leq n-1$ be integers. Let $p \geq -n$ be a real number and $f(z)$ be a smooth positive function on $\mathbb{S}^n$. For abbreviation, let $h(z) :=  f^{-\frac{1}{n-k}}(z)$.
	\begin{enumerate}
		\item If $p=-n$, then we assume that $h(z)$ is a positive constant function on $\mathbb{S}^n$.
		\item If $-n <p \leq -\frac{n+k}{2}$, then we assume that $h(z)$ satisfies
		\begin{equation*}
			D^2 h- \(\frac{n-3k-2p}{n-k}\) |Dh|I + \( \frac{n+p}{n-k}  \)^2 h I \geq 0.
		\end{equation*}
		\item If $-\frac{n+k}{2}<p <-k$, then we assume that $h(z)$ satisfies
		\begin{equation*}
			D^2 h- \frac{(n-3k-2p)^2}{2(n+p)(n+k+2p)} \frac{|D h|^2}{h} I +  \frac{1}{2} \frac{n+p}{n-k}h I\geq 0.
		\end{equation*}
		\item If $-k \leq p \leq n-2k$, then we assume that $h(z)$ satisfies
		\begin{equation*}
			D^2 \( h^{\frac{n-k}{n+p}} \)
			-\frac{1}{2} \frac{\left|D \(h^{\frac{n-k}{n+p}} \)\right|^2}{h^{\frac{n-k}{n+p}}} I
			+\frac{1}{2} h^{\frac{n-k}{n+p}} I \geq 0.
		\end{equation*}
		\item If $p >n-2k$, then we assume that $h(z)$ satisfies
		\begin{equation*}
			D^2 \( h^{\frac{n-k}{n+p}} \)
			-\frac{1}{2} \frac{\left|D \(h^{\frac{n-k}{n+p}} \)\right|^2}{h^{\frac{n-k}{n+p}}} I
			+\frac{n-k}{n+p} h^{\frac{n-k}{n+p}} I \geq 0.
		\end{equation*}
	\end{enumerate}
\end{assump}

We state the main results in Section \ref{sec-horosp-p-Christoffel Minkowski problem} and their proofs will be given in Subsection \ref{Subsec-Long time existence and convergence}.
 
\begin{thm}[Horospherical $p$-Minkowski problem]\label{thm-exist-all p-k=0}
	Let $n \geq 1$ be an integer and $- \infty<p<+ \infty$ be a real number. Given a smooth positive even function $f(z)$ defined on $\mathbb{S}^n$, the following equation has a smooth even and uniformly h-convex solution, i.e.
 	\begin{equation}\label{eq-lim-hyp-k=0}
 		\g^{-p}(z) p_n \(A[\g(z)] \) = \gamma f(z)
 	\end{equation}
 	for some $\gamma >0$.
 \end{thm}

 \begin{thm}[Horospherical $p$-Christoffel-Minkowski problem]\label{thm-exist-all p}
	Let $n \geq 2$ and $1 \leq k \leq n-1$ be integers, and let $p \geq -n$ be a real number. Given a smooth positive even function $f(z)$ defined on $\mathbb{S}^n$ that satisfies the Assumption \ref{assump-h}, the following equation has a smooth even and uniformly h-convex solution, i.e.
	\begin{equation}\label{eq-limit-hypersurface}
		\g^{-p-k} p_{n-k}(A[\g(z)]) =\gamma f(z)
	\end{equation}
	for some $\gamma>0$.
\end{thm}

Our following curvature flow \eqref{flow-HCMF} will be used to study the existence of solutions to equation \eqref{eq-p-CM problem}. In the case that $n\geq 2$, $p=-n$ and $f(z)$ is a positive constant, the flow \eqref{flow-HCMF} was used to obtain a class of Alexandrov-Fenchel inequalities in hyperbolic space by Andrews et al. \cite[Corollary 1.9]{ACW18}.

Let $n \geq 1$ and $0 \leq k \leq n-1$ be integers, let $p$ be a real number, and let $f(z)$ be a smooth positive function on $\mathbb{S}^n$. Suppose that  $ X: M \times [0,T) \to \mathbb{H}^{n+1}$ is a family of smooth uniformly h-convex hypersurfaces, satisfying
\begin{equation}\label{flow-HCMF}
	\left\{\begin{aligned}
		\frac{\partial}{\partial t}X(x,t)=&\left( \Phi(t)  \(\cosh r- \widetilde{u}\)^{\frac{n+p}{n-k}}(x,t)f^{- \frac{1}{n-k}}(z(x,t)) - p_{n-k}^{-\frac{1}{n-k}}( \tilde{\lambda}(x,t)) \right) \nu , \\
		X(\cdot,0)=& M_0= \partial \Omega_0.
	\end{aligned}\right.
\end{equation}
Here $z(x,t)$ is the horospherical Gauss map at $X(x,t)$ of $M_t = X(M,t) =\partial \Omega_t$, $\tilde{\lambda}$ are the shifted principal  radii of curvature of $M_t$, and $\Phi(t)$ is a scalar function to be specified immediately.  Let $G_t: M_t \to \mathbb{S}^n$ be the horospherical Gauss map of $M_t$ and  $u(z,t) := u_{\Omega_t} (z)$ be the horospherical support function of $\Omega_t$, and $\g(z,t):=e^{u(z,t)}$. In this paper, the global term $\Phi(t)$ is defined by
\begin{equation}\label{def of Phi}
	\Phi(t) := \frac{\int_{\mathbb{S}^n} \g^{-n}(z,t) p_{n-k}^{1- \frac{1}{n-k}}\(\tilde{\lambda}(x,t)\) d \sigma}{\int_{\mathbb{S}^n} f^{-\frac{1}{n-k}}(z) \g^{-n-\frac{n+p}{n-k}}(z,t) p_{n-k} \(\tilde{\lambda}(x,t)\) d \sigma},
\end{equation}
where $\tilde{\lambda}(x,t)$ are the shifted principal  radii of curvature at $X(x,t):=G^{-1}_t(z)$ of $M_t$. This choice of $\Phi(t)$ ensures that $\widetilde{W}_k(\Omega_t)$ remains constant along the flow \eqref{flow-HCMF}. We denote the speed function of flow \eqref{flow-HCMF} as
\begin{equation}\label{speed HCMF}
	\mathscr{F} = \Phi(t) \g^{- \frac{n+p}{n-k}}(z,t)h(z)-F(\tilde{\kappa}(x,t)),
\end{equation}
where $h(z) := f^{-\frac{1}{n-k}}(z)$, $F(\tilde{\kappa}) := \(\frac{p_n(\tilde{\kappa})}{p_k(\tilde{\kappa})} \)^{\frac{1}{n-k}}$, and we used \eqref{1/phi, coshr-u}.
 
We will prove the following long time existence of the flow \eqref{flow-HCMF}, and the proof will be given in Subsection \ref{Subsec-Long time existence and convergence}.

\begin{thm}\label{thm-long time existence}
	Let $f(z)$ be a smooth positive even function on $\mathbb{S}^n$ and $M_0 = \partial \Omega_0$ be a smooth, origin symmetric and uniformly h-convex hypersurface in $\mathbb{H}^{n+1}$. 
 	\begin{enumerate}
		\item If $k=0$, $-\infty<p<+\infty$ and $n \geq 1$, then the flow \eqref{flow-HCMF} has a smooth uniformly h-convex solution $M_t$ for all time $t > 0$.
 	 	\item If $1 \leq k \leq n-1$, $p \geq -n$, $n \geq 2$, and $f(z)$ satisfies the Assumption \ref{assump-h}, then the flow \eqref{flow-HCMF} has a smooth uniformly h-convex solution $M_t$ for all time $t > 0$.
 	 \end{enumerate}
\end{thm}

\begin{rem}
	Let $n\geq 1$, $0 \leq k \leq n-1$, $p \geq -n$, and $f(z)$ be a positive constant. Let $M_0$ be a smooth, origin symmetric and uniformly h-convex hypersurface in $\mathbb{H}^{n+1}$. In Theorem \ref{thm-f=c-convergence}, we will prove that flow  \eqref{flow-HCMF} converges to a geodesic sphere centered at the origin.
\end{rem}

The flow \eqref{flow-HCMF} is equivalent to a parabolic equation of $u(z,t)$ by a similar argument as \eqref{scalr-flow DT's trick}, i.e.
\begin{align}
	\frac{\partial}{\partial t} u(z,t) =& \Phi \g^{-\frac{n+p}{n-k}}f^{-\frac{1}{n-k}} -  p_{n-k}^{- \frac{1}{n-k}}(\tilde{\lambda}) \nonumber\\
	=& \Phi \g^{-\frac{n+p}{n-k}}f^{-\frac{1}{n-k}} -  \g^{-1} p_{n-k}^{- \frac{1}{n-k}} (A [\g]).\label{scalr eq u-HCMF}
\end{align}
Then the scalar evolution equation of $\g(z,t)$ is
\begin{equation}\label{scalar eq phi-HCMF}
	\frac{\partial}{\partial t} \g (z,t)= \Phi \g^{-\frac{n+p}{n-k}+1}f^{-\frac{1}{n-k}}-  p_{n-k}^{- \frac{1}{n-k}} (A [\g]).
\end{equation}
If $\Omega$ is an origin symmetric domain in $\mathbb{H}^{n+1}$, then $\g_{\Omega}(z)$ is even on $\mathbb{S}^n$. Suppose that $f(z)$ is a positive even function on $\mathbb{S}^n$, $M_0$ is a smooth, origin symmetric hypersurface, and $\{M_t\}$ is a solution to the flow \eqref{flow-HCMF} for $t \in [0,T)$, then \eqref{scalar eq phi-HCMF} implies that $M_t$ is origin symmetric.  

To avoid confusion, when we do calculations on a uniformly h-convex hypersurface $M = \partial \Omega$, we view $z$ as a vector field defined on $M$ by $z = G_{\Omega}(X)$ for $X \in M$. Similarly, we view $X = G_{\Omega}^{-1}(z)$ when we do calculations on $\mathbb{S}^n$.

\subsection{Monotone quantities}$ \ $

Let $p$ be a real number and $f(z)$ be a smooth positive function on $\mathbb{S}^n$. Let $\Omega$ be a h-convex domain with horospherical support function $u(z)$, and $\g(z):= e^{u(z)}$. We define
\begin{equation}\label{def-J_p}
	J_p(\Omega) :=\left\{
	\begin{aligned}
		\frac{1}{p} \int_{\mathbb{S}^n} f(z) \g^p(z) d \sigma, \quad p\neq 0,\\
		\int_{\mathbb{S}^n} f(z) u(z) d\sigma, \quad p=0.
	\end{aligned}
\right.
\end{equation}

\begin{lem}\label{lem-mono-quantitirs}
	Let $M_t =\partial \Omega_t$ be a smooth uniformly h-convex solution to the flow \eqref{flow-HCMF}. Then
	\begin{align}
		\frac{d}{dt} \widetilde{W}_k(\Omega_t) =&0 , \quad \forall \ t\in [0,T), \label{mono quan-eq}\\
		 \frac{d}{dt} J_p(\Omega_t)\leq& 0, \quad \forall \ t\in [0,T).\label{mono quan-ineq}
	\end{align}
	Moreover, equality holds in \eqref{mono quan-ineq} if and only if $f^{-1}(z) \g^{-p-n}(z,t)p_{n-k}(\tilde{\lambda}(x,t))$ is constant.
\end{lem}

\begin{proof}
	From Lemma \ref{lem-variation of modified quermass}, we have by \eqref{rel-area element} and \eqref{shifted curvature-support function} that
	\begin{align*}
		\frac{d}{dt} \widetilde{W}_k(\Omega_t) =& \int_{M_t} p_k (\tilde{\kappa}) \mathscr{F} d \mu_t
		=\int_{\mathbb{S}^n} \g^{-n}\frac{p_{n-k} (\tilde{\lambda})}{p_n (\tilde{\lambda})}  p_{n}(\tilde{\lambda}) \mathscr{F}d \sigma\\
		=&\int_{\mathbb{S}^n} \g^{-n} p_{n-k}(\tilde{\lambda}) \mathscr{F} d \sigma.
	\end{align*}
	Substituting \eqref{speed HCMF} into the right-hand side and using \eqref{def of Phi}, we obtain
	\begin{equation*}
		\frac{d}{dt} \widetilde{W}_k(\Omega_t) 
		= \Phi(t) \int_{\mathbb{S}^n} f^{-\frac{1}{n-k}} \g^{-n-\frac{n+p}{n-k}} p_{n-k}(\tilde{\lambda}) d \sigma
		-\int_{\mathbb{S}^n} \g^{-n}p_{n-k}^{1- \frac{1}{n-k}}(\tilde{\lambda}) d \sigma
		=0.
	\end{equation*}
    When $p \neq 0$, we have from \eqref{scalr eq u-HCMF} that
	\begin{align*}
		\frac{d}{dt} J_p(\Omega_t) =& \frac{d}{dt} \( \frac{1}{p} \int_{\mathbb{S}^n} f(z) \g^p(z,t) d \sigma \)
		= \int_{\mathbb{S}^n} f \g^p \mathscr{F} d \sigma\\
		=& \Phi(t) \int_{\mathbb{S}^n} f^{1-\frac{1}{n-k}} \g^{p-\frac{n+p}{n-k}} d \sigma 
		-\int_{\mathbb{S}^n} f \g^{p} p_{n-k}^{-\frac{1}{n-k}}(\tilde{\lambda}) d \sigma.
	\end{align*}
	When $p=0$, we have by using \eqref{scalr eq u-HCMF}
	\begin{align*}
		\frac{d}{dt} J_0(\Omega_t) =&\frac{d}{dt} \(  \int_{\mathbb{S}^n} f(z) u(z,t) d\sigma\)= \int_{\mathbb{S}^n} f \mathscr{F} d\sigma\\
		=& \Phi(t)  \int_{\mathbb{S}^n} f^{1-\frac{1}{n-k}} \g^{-\frac{n}{n-k}} d \sigma 
		-\int_{\mathbb{S}^n} f  p_{n-k}^{-\frac{1}{n-k}}(\tilde{\lambda}) d \sigma.
	\end{align*}
	By the definition of $\Phi(t)$ in \eqref{def of Phi}, it suffices to prove the following inequality for all $p \in (-\infty,+\infty)$,
	\begin{equation}\label{goal ineq}
		\int_{\mathbb{S}^n} f^{1-\frac{1}{n-k}} \g^{p-\frac{n+p}{n-k}} d \sigma 
		\int_{\mathbb{S}^n} \g^{-n} p_{n-k}^{1- \frac{1}{n-k}}(\tilde{\lambda}) d \sigma
		\leq 
		\int_{\mathbb{S}^n} f \g^{p} p_{n-k}^{-\frac{1}{n-k}}(\tilde{\lambda}) d \sigma
		\int_{\mathbb{S}^n} f^{-\frac{1}{n-k}} \g^{-n-\frac{n+p}{n-k}} p_{n-k} (\tilde{\lambda}) d \sigma.
	\end{equation}
	In the case $k = n-1$, the inequality \eqref{goal ineq} follows from the Cauchy-Schwarz inequality, and the equality holds if and only if 
	\begin{equation*}
		f^{-1}(z)\g^{-n-p}(z,t)p_1(\tilde{\lambda}(x,t))
	\end{equation*}
	is constant on $\mathbb{S}^n$. In the case $0 \leq k \leq n-2$, we let $dm(z,t) = f^{1- \frac{1}{n-k}}(z) \g^{p-\frac{n+p}{n-k}}(z,t) d \sigma(z)$ and $\eta(z,t) = f^{\frac{1}{n-k}-1}(z) \g^{\frac{n+p}{n-k}-n-p}(z,t) p_{n-k}^{1- \frac{1}{n-k}}(\tilde{\lambda}(x,t))$. Then \eqref{goal ineq} becomes
	\begin{equation}\label{goal ineq-new}
		\int_{\mathbb{S}^n} dm \int_{\mathbb{S}^n} \eta dm - \int_{\mathbb{S}^n} \eta^{-\frac{1}{n-k-1}} dm 
		\int_{\mathbb{S}^n} \eta^{\frac{n-k}{n-k-1}} dm \leq 0.
	\end{equation} 
	The above inequality \eqref{goal ineq-new} follows from
	\begin{equation*}
		\int_{\mathbb{S}^n} \int_{\mathbb{S}^n} 
		\(\eta^{-\frac{1}{n-k-1}}(z_1,t) -\eta^{-\frac{1}{n-k-1}}(z_2,t) \) \(\eta^{\frac{n-k}{n-k-1}}(z_1,t)-\eta^{\frac{n-k}{n-k-1}}(z_2,t)\) dm(z_1,t)dm(z_2,t) \leq 0, 
	\end{equation*}
	with equality if and only if $\eta(z,t)$ is constant on $\mathbb{S}^n$, i.e.
	\begin{equation*}
		f^{-1}(z)\g^{-n-p}(z,t)p_{n-k}(\tilde{\lambda}(x,t))
	\end{equation*}
	is constant. Hence we complete the proof of Lemma \ref{lem-mono-quantitirs}.
\end{proof}
\subsection{A priori estimates}\label{subsec-a priori est} $ \ $

In order to achieve the long time existence of flow \eqref{flow-HCMF}, we need to do the related a priori estimates.
\subsubsection{$C^0$ estimate}
\begin{lem}\label{lem-C0-est}
	Let $M_t = \partial \Omega$ be a smooth, origin symmetric and uniformly h-convex solution to the flow \eqref{flow-HCMF}. Let
	\begin{equation*}
		r_{\min}(t) := \min_{\theta \in \mathbb{S}^n} r(\theta , t), \quad 
		r_{\max}(t) := \max_{\theta \in \mathbb{S}^n} r(\theta , t),
	\end{equation*}
	and 
	\begin{equation*}
		u_{\min}(t) := \min_{z \in \mathbb{S}^n} u(z , t), \quad 
		u_{\max}(t) := \max_{z \in \mathbb{S}^n} u(z , t).
	\end{equation*}
	Then along the flow \eqref{flow-HCMF}, there exist two positive constants $c$, $C$ such that
	\begin{equation}\label{C0-est}
		c \leq r_{\min}(t) = u_{\min} (t) \leq u_{\max}(t)= r_{\max}(t) \leq C, \quad \forall \ t\in [0,T),
	\end{equation}
	where $c$ and $C$ only depend on $n$, $k$ and $M_0$.
\end{lem}

\begin{proof}
	In this proof, we fix a time $t \in [0,T)$. By \eqref{mono quan-eq}, we have
	\begin{equation*}
		\widetilde{W}_k(\Omega_t)=\widetilde{W}_k(\Omega_0).
	\end{equation*}
	Now we choose $z_0 \in \mathbb{S}^n$ such that $u(z_0,t)=u_{\min}(t)$. Since $\Omega_t$ is origin symmetric, we know $\Omega_t \subset \overline{B}_{z_0}(u(z_0,t)) \cap \overline{B}_{-z_0}(u(z_0,t))$. When $u(z_0,t)$ goes to $0$, we have that $\widetilde{W}_k \left( \overline{B}_{z_0}(u(z_0,t)) \cap \overline{B}_{-z_0}(u(z_0,t))\right) $ also goes to $0$ by Proposition \ref{prop-domain monotone modified quermass}. Hence 
	\begin{equation}\label{u min >c}
		u_{\min}(t) \geq c
	\end{equation}
	 for some constant $c$ along the flow \eqref{flow-HCMF}. Since
	\begin{equation*}
		\Omega_t = \bigcap_{z \in \mathbb{S}^n} \overline{B}_z(u(z,t)), 
	\end{equation*}
	we have $B(u_{\min}(t)) \subset \Omega_t$. Then
	\begin{equation*}
		\widetilde{W}_k \left( B(u_{\min}(t)) \right) \leq \widetilde{W}_k (\Omega_t)=\widetilde{W}_k(\Omega_0).
	\end{equation*}
	Consequently, $u_{\min}(t) \leq C$ for some constant $C$ along the flow \eqref{flow-HCMF}. Now we let $z_1 \in \mathbb{S}^n$ such that $u(z_1, t) =u_{\max}(t)$ and set $\g_{\max} (t) = e^{u_{\max} (t)}$. Let $\overline{X} = X(z_1, t)$ as in \eqref{X(z)}. Since the horo-ball $\overline{B}_z(r)$ is given by
	\begin{equation*}
		\log(- \metric{X}{(z,1)}) \leq r,
	\end{equation*}
	 we have
	\begin{equation*}
		-\metric{\overline{X}}{(z,1)}  \leq \g(z,t), \quad \forall \ z \in \mathbb{S}^n,
	\end{equation*}
	which means $\overline{X} \in \overline{B}_z(u(z))$. By \eqref{X(z)}, we have 
	\begin{equation}\label{phi-max-phi}
		\frac{1}{2} \g_{\max}(t) (1 + \metric{z_1}{z}) + \frac{1}{2} \(\frac{|D \g (z_1, t)|^2}{\g_{\max}(t)} + \frac{1}{\g_{\max}(t)} \) (1- \metric{z_1}{z}) + \metric{D \g(z_1, t)}{z} \leq \g(z,t)
	\end{equation}
	for any $z \in \mathbb{S}^n$. Now we set
	\begin{equation*}
		E = \lbrace z \in \mathbb{S}^n:  \metric{z_1}{z} \geq 0 \rbrace.
	\end{equation*}
	Then by $D \g(z_1,t) =0$ and $\g_{\max}(t) \geq e^{u_{\min}(t)} >1 $, we have $\g(z,t) \geq \frac{1}{2} \(\g_{\max}(t) +\frac{1}{\g_{\max}(t)} \)= \cosh u_{\max}(t)$ for $ z \in E$. Since $\Omega_t$ is origin symmetric, we can assume $z_0 \in E$ without loss of generality. Then we have
	\begin{equation*}
		\cosh u_{\max}(t) = \frac{1}{2} \(\g_{\max}(t) +\frac{1}{\g_{\max}(t)} \) \leq e^{u_{\min}(t)} \leq C,
	\end{equation*} 
	which means 
	\begin{equation}\label{u max <C}
		u_{\max}(t) \leq C
	\end{equation}
	for some constant $C$. Finally, Lemma \ref{lem-C0-est} follows from the estimates in \eqref{u min >c}, \eqref{u max <C} and formula \eqref{r-min-u-min-r-max-u-max}.
\end{proof}

\subsubsection{$C^1$ estimates}
\begin{lem}\label{lem-Dphi<phi}
	For a smooth, origin symmetric and uniformly h-convex hypersurface $M = \partial \Omega$ in $\mathbb{H}^{n+1}$, we have 
	\begin{equation}\label{Dphi<phi}
		|D \g(z)| < \g(z), \quad \forall \ z \in \mathbb{S}^{n}.
	\end{equation}
\end{lem}

\begin{proof}
	Fix a point $z_0 \in \mathbb{S}^n$ such that $\min_{z \in \mathbb{S}^n} \g(z) = \g(z_0)$. Since $\Omega \subset \overline{B}_{z_0}(u(z_0)) \cap \overline{B}_{-z_0}(u(z_0))$, inequality \eqref{u-min-r-max} implies 
	\begin{equation*}
		\cosh r(\theta) \leq \g(z_0), \quad \forall \ \theta \in \mathbb{S}^n.
	\end{equation*}
	 For any $z \in \mathbb{S}^n$, by using \eqref{coshr} and $\g(z) \geq \g(z_0)$, we have
	\begin{equation*}
		\frac{|D \g|^2}{\g^2}(z) = \frac{2(\cosh r- \cosh u(z))}{\g(z)} \leq \frac{2 \sinh u(z_0)}{\g(z_0)} 
		=1- \frac{1}{\g^2(z_0)}<1.
	\end{equation*}
	 Hence, inequality $|D \g(z)| < \g(z)$ holds for all  $z \in \mathbb{S}^n$.
\end{proof}

\begin{cor}\label{cor-C1-est}
	Let $M_t = \partial \Omega$ be a smooth, origin symmetric and uniformly h-convex solution to the flow \eqref{flow-HCMF}.	Then there is a positive constant $C$ depending only on $n$, $k$ and $M_0$ such that
	\begin{equation}\label{C1-est}
		|D \g(z,t)| \leq C, \quad  \forall \ (z,t) \in \mathbb{S}^n \times [0,T).
	\end{equation}
\end{cor}

\begin{proof}
	Using \eqref{C0-est} and \eqref{Dphi<phi},  we have $|D\g(z,t)| < \g(z,t) \leq C$ along the flow, which induces Corollary \ref{cor-C1-est}.
\end{proof}

\begin{lem}\label{lem-c < Phi <C}
	Let $M_t = \partial \Omega$ be a smooth, origin symmetric and uniformly h-convex solution to the flow \eqref{flow-HCMF}. Along the flow \eqref{flow-HCMF}, it holds that
	\begin{equation}\label{Phi<C}
		\Phi(t) \leq C, \quad \forall \ t\in [0,T),
	\end{equation}
	where $C$ is a positive constant that only depends on $n$, $k$, $p$, $||f(z)||_{C^0(\mathbb{S}^n)}$ and $M_0$. Furthermore, if we assume $\max_{1 \leq i, j \leq n} \(\tilde{\kappa}_i / \tilde{\kappa}_j\) \leq C'$ on $M \times [0,T)$ along the flow for some positive constant $C'$, then
	\begin{equation}\label{Phi>c}
		\Phi(t) \geq c, \quad \forall \ t\in[0,T)
	\end{equation}
	for some positive constant $c$ that only depends on $C'$, $n$, $k$, $p$, $\min_{z \in \mathbb{S}^n} f(z)$ and $M_0$.  
\end{lem}

\begin{proof}
	We say $a \sim b$ if and only if $a \leq Cb$, $b \leq C a$ for some constant $C>0$ depending only on $n$, $k$ and $M_0$. Combining the a priori estimates in \eqref{C0-est} and \eqref{C1-est} with formulas \eqref{1/phi, coshr-u} and \eqref{tilde u}, we have 
	\begin{align}
		&\cosh r(x,t) -\tilde{u}(x,t) = \frac{1}{\g(z,t)} \sim 1, \label{sim1-cosh r- u}\\
		 &\tilde{u}(x,t)  =\frac{1}{2} \frac{|D \g|^2}{\g}(z,t) + \frac{1}{2} \(\g(z,t)- \frac{1}{\g(z,t)} \) \sim 1 \label{sim1-u}
	\end{align}
	along the flow \eqref{flow-HCMF}. By \eqref{shifted curvature-support function} and \eqref{rel-area element}, we have
	$\int_{M_t} \g^n(z,t) p_n(\tilde{\kappa}) d\mu_t =\int_{\mathbb{S}^n} d\sigma= \omega_n$, which induces $\int_{M_t} p_n (\tilde{\kappa}) d \mu_t \sim 1$. Using \eqref{sim1-cosh r- u}, \eqref{sim1-u} and \eqref{eq-shifted Minkowski formula}, we know $\int_{M_t} p_m (\tilde{\kappa}) d \mu_t \sim \int_{M_t} p_{m+1}(\tilde{\kappa}) d \mu_t$ for $m =0,1, \ldots, n-1$. Hence we have
	\begin{equation}\label{sim1-pm kappa}
		\int_{M_t} p_m (\tilde{\kappa}) d \mu_t \sim 1, \quad m=0,1,\ldots,n
	\end{equation}
	along the flow. Using \eqref{rel-area element}, \eqref{shifted curvature-support function} and \eqref{sim1-cosh r- u} in \eqref{sim1-pm kappa}, we get
	\begin{align}
		1 \sim& 	\int_{M_t} p_m (\tilde{\kappa}) d \mu_t
		= \int_{\mathbb{S}^n}  p_m (\tilde{\kappa}) \det A[\g] d\sigma \nonumber \\
		=& \int_{\mathbb{S}^n} p_m(\tilde{\kappa}) \g^{-n} p_n^{-1} (\tilde{\kappa}) d\sigma
		=\int_{\mathbb{S}^n} p_{n-m}(\tilde{\lambda}) \g^{-n} d\sigma \nonumber \\
		\sim& \int_{\mathbb{S}^n} p_{n-m}(\tilde{\lambda}(x,t)) d\sigma , \quad m=0,1, \ldots,n, \label{sim1-p_n-m lambda}
	\end{align}
	where $X(x,t) = G_t^{-1}(z)$. Formulas \eqref{sim1-cosh r- u} and \eqref{sim1-p_n-m lambda} imply
	\begin{equation}\label{Phi-1}
		\frac{1}{C_1} \leq \int_{\mathbb{S}^n} f^{-\frac{1}{n-k}}(z) \g^{-n-\frac{n+p}{n-k}}(z,t) p_{n-k} (\tilde{\lambda}(x,t)) d \sigma \leq C_2
	\end{equation}
	along the flow \eqref{flow-HCMF} for some constants $C_1$ and $C_2$. Here $C_1$ only depends on $n$, $k$, $p$, $||f(z)||_{C^0(\mathbb{S}^n)}$ and $M_0$, and $C_2$ only depends on $n$, $k$, $p$, $\min_{z \in \mathbb{S}^n} f(z)$  and $M_0$. If $n \geq 2$ and $k \leq n-2$, then \eqref{McLau ineq} yields
	\begin{equation}\label{Mac-lambda}
		p_{n-k}^{ \frac{n-k-1}{n-k}}(\tilde{\lambda}(x,t)) \leq p_{n-k-1}(\tilde{\lambda}(x,t)).
	\end{equation}
	When $n \geq 1$ and $k=n-1$, we note that \eqref{Mac-lambda} is trivial. Using \eqref{Mac-lambda}, \eqref{sim1-cosh r- u} and \eqref{sim1-p_n-m lambda}, we have
	\begin{equation}\label{Phi-3}
		\int_{\mathbb{S}^n} \g^{-n}(z,t) p_{n-k}^{1- \frac{1}{n-k}}(\tilde{\lambda}(x,t)) d\sigma 
		\leq \int_{\mathbb{S}^n} \g^{-n}(z,t) p_{n-k-1}(\tilde{\lambda}(x,t) ) d\sigma
		\leq C
	\end{equation}
	along the flow \eqref{flow-HCMF} for some constant $C$. Combining \eqref{Phi-1}, \eqref{Phi-3} with \eqref{def of Phi}, we obtain \eqref{Phi<C}.  
	
	Now we estimate $\Phi(t)$ from below. For convenience, we assume $\tilde{\kappa}_1(x,t) \leq \cdots \leq \tilde{\kappa}_n(x,t)$ at $X(x,t) \in M_t$, which means $\tilde{\lambda}_1(x,t) \geq \cdots \geq \tilde{\lambda}_n(x,t)$. By the assumption in Lemma \ref{lem-c < Phi <C}, we have
	\begin{equation}\label{pinch-assm-pk-lambda}
		p_{n-k}^{ \frac{n-k-1}{n-k}}(\tilde{\lambda}(x,t)) \geq \(\tilde{\lambda}_n(x,t) \)^{n-k-1} \geq \(\frac{1}{C'} \tilde{\lambda}_1(x,t)\)^{n-k-1} \geq \(C'\)^{-(n-k-1)} p_{n-k-1} (\tilde{\lambda}(x,t)).
	\end{equation}
	Then by using \eqref{Phi-3}, \eqref{pinch-assm-pk-lambda}, \eqref{sim1-cosh r- u} and \eqref{sim1-p_n-m lambda}, there exists a large positive constant $C$ such that
	\begin{align}
		C \geq&
		\int_{\mathbb{S}^n} \g^{-n}(z,t) p_{n-k}^{1- \frac{1}{n-k}}(\tilde{\lambda}(x,t))d\sigma \nonumber\\
		\geq& 
		\(C'\)^{-(n-k-1)}\int_{\mathbb{S}^n} \g^{-n}(z,t)  p_{n-k-1} (\tilde{\lambda}(x,t)) d\sigma
		\geq \frac{1}{C}.    \label{Phi-2}
	\end{align}
	Then \eqref{Phi>c} follows from \eqref{Phi-1}, \eqref{Phi-2} and the definition of $\Phi(t)$ in \eqref{def of Phi}. This completes the proof of Lemma \ref{lem-c < Phi <C}.
\end{proof}

\subsubsection{Pinching estimates}
Along a general flow
\begin{equation*}
	\frac{\partial}{\partial t} X =\mathscr{F} \nu,
\end{equation*}
the following evolution equations for hypersurfaces in $\mathbb{H}^{n+1}$ are well known, see e.g. \cite[Equations (2.7) and (2,9)]{ACW18},
\begin{align}
	\frac{\partial}{\partial t} g_{ij} =& 2 \mathscr{F} h_{ij}, \label{evl-gij}\\
	\frac{\partial}{\partial t} h_i{}^j =& - \nabla^j \nabla_i \mathscr{F} - \mathscr{F} (h_i{}^l h_l{}^j - \delta_i{}^j). \nonumber
\end{align}
Hence 
\begin{equation}\label{evl-thij}
	\frac{\partial}{\partial t} \tilde{h}_i{}^j = -\nabla^j \nabla_i \mathscr{F} - \mathscr{F} \tilde{h}_i{}^l \tilde{h}_l{}^j -2 \mathscr{F} \tilde{h}_i{}^j,
\end{equation}
where $\tilde{h}_i{}^j = h_i{}^j-\delta_i{}^j$.

\begin{lem}\label{lem-evl-thij}
	Along the flow \eqref{flow-HCMF}, the shifted Weingarten matrix $\tilde{h}_i{}^j$ evolves as
	\begin{align}
		\frac{\partial}{\partial t} \tilde{h}_i{}^j =&
		\dot{F}^{lm} \nabla_m \nabla_l \tilde{h}_i{}^j + \ddot{F}^{lm,qs} \nabla_i \tilde{h}_{lm} \nabla^j \tilde{h}_{qs} +\dot{F}^{lm} \tilde{h}_l{}^q \tilde{h}_{qm} \delta_i{}^j \nonumber\\
		&+\dot{F}^{lm} \tilde{h}_l{}^q \tilde{h}_{qm} \tilde{h}_i{}^j
		-\dot{F}^{lm}g_{lm} \tilde{h}_i{}^q \tilde{h}_q{}^j
		-F \tilde{h}_i{}^l \tilde{h}_l{}^j \nonumber\\
		&-\mathscr{F} \tilde{h}_i{}^l \tilde{h}_l{}^j
		-2\mathscr{F} \tilde{h}_i{}^j -\Phi \nabla^j \nabla_i\( \g^{-\frac{n+p}{n-k}} h\).\label{evl-thij-new}
	\end{align}
\end{lem}

\begin{proof}
	Recall that $F \(\tilde{\kappa}\):= \( \frac{p_n(\tilde{\kappa})}{p_k \(\tilde{\kappa}\)}  \)^{\frac{1}{n-k}}$. By the Simons' identity (see e.g., \cite[Equation (2-7)]{And94}, \cite[Equation (3.5)]{HLW20}), we have
	\begin{align*}
		\nabla^j\nabla_iF= & \nabla^j(\dot{F}^{lm}\nabla_i\tilde{h}_{lm}) \\
		=&\dot{F}^{lm}\nabla^j\nabla_i\tilde{h}_{lm}+\ddot{F}^{lm,qs}\nabla_i\tilde{h}_{lm}\nabla^j\tilde{h}_{qs}\nonumber\\
		=& \dot{F}^{lm}\nabla_m\nabla_l\tilde{h}_{i}{}^j + \dot{F}^{lm}((h^2)_{lm}+g_{lm})h_{i}{}^j-\dot{F}^{lm}h_{lm}((h^2)_{i}{}^j+\delta_i{}^j)\\
		&+\ddot{F}^{lm,qs}\nabla_i\tilde{h}_{lm}\nabla^j\tilde{h}_{qs}.
	\end{align*}
	This together with \eqref{F-ij h-ij} gives
	\begin{align}
		\nabla^j \nabla_i F =& 
		\dot{F}^{lm}\nabla_m\nabla_l\tilde{h}_i{}^j
		+\dot{F}^{lm} \( (\tilde{h}_l{}^q + \delta_l{}^q)(\tilde{h}_{qm} + g_{qm})+ g_{lm} \)\( \tilde{h}_i{}^j + \delta_i{}^j \) \nonumber\\
		&- \dot{F}^{lm}\( \tilde{h}_{lm}+g_{lm} \)\(  (\tilde{h}_{i}{}^q+ \delta_{i}{}^q  )(\tilde{h}_q{}^j + \delta_q{}^j)+\delta_i{}^j \)
		+\ddot{F}^{lm,qs}\nabla_i\tilde{h}_{lm}\nabla^j\tilde{h}_{qs}  \nonumber\\
		=&\dot{F}^{lm}\nabla_m\nabla_l\tilde{h}_i{}^j+\dot{F}^{lm} \tilde{h}_l{}^q \tilde{h}_{qm} \tilde{h}_i{}^j
		+\dot{F}^{lm} \tilde{h}_l{}^q \tilde{h}_{qm} \delta_i{}^j
	 	-\dot{F}^{lm}g_{lm} \tilde{h}_i{}^q \tilde{h}_q{}^j  \nonumber\\
	 	&-F \tilde{h}_i{}^l \tilde{h}_l{}^j
	 	+\ddot{F}^{lm,qs}\nabla_i\tilde{h}_{lm}\nabla^j\tilde{h}_{qs}.\label{sims-id}
	\end{align}
	By \eqref{evl-thij}, we have
	\begin{align}
		\frac{\partial}{\partial t} \tilde{h}_i{}^j =&
		-\nabla^j \nabla_i \( \Phi \g^{-\frac{n+p}{n-k}} h-F \)
		- \mathscr{F}\tilde{h}_i{}^l \tilde{h}_l{}^j - 2\mathscr{F} \tilde{h}_i{}^j \nonumber\\
		=&\nabla^j \nabla_i F -\Phi \nabla^j \nabla_i \( \g^{-\frac{n+p}{n-k}} h \)
		- \mathscr{F}\tilde{h}_i{}^l \tilde{h}_l{}^j - 2\mathscr{F} \tilde{h}_i{}^j. \label{dt-thij}
	\end{align}
	Then we obtain \eqref{evl-thij-new} by inserting \eqref{sims-id} into \eqref{dt-thij}. This completes the proof of Lemma \ref{lem-evl-thij}.
\end{proof}

The following Lemma \ref{lem-from assmp of f to desired est} shows how Assumption \ref{assump-h} works in the pinching estimates, i.e. Proposition \ref{prop-Pinching-est} below.

\begin{lem}\label{lem-from assmp of f to desired est}
	Let $n \geq 2$ and $1 \leq k \leq n-1$ be integers, and $p>-n$ be a real number. Let $f(z)$ be a smooth positive function on $\mathbb{S}^n$ satisfying Assumption \ref{assump-h}. Suppose that $M = \partial \Omega$ is a smooth uniformly h-convex hypersurface in $\mathbb{H}^{n+1}$ with horospherical support function $u(z)$. Let $\g(z) = e^{u(z)}$ and $h(z) = f^{-\frac{1}{n-k}}(z)$. Let $\{e_1, \ldots, e_n\}$ be a local orthonormal frame for $\mathbb{S}^n$. 
	\begin{enumerate}
		\item  If $-n<p\leq-\frac{n+k}{2}$, and in addition that $\Omega$ is origin symmetric, then
	 	\begin{align}
	 	&  \g^2 D^2 h( e_i,e_i)
	 	- \g \metric{D h}{D \g} 
	 	- \frac{2(k+p)}{n-k} \g \metric{Dh}{e_i} \metric{D \g}{e_i}  
	 	+\frac{n+p}{2(n-k)}h|D \g|^2 \nonumber\\
	 	&+ \frac{n+p}{2(n-k)}\g^2 h  
	 	+ \frac{(n+p)(k+p)}{(n-k)^2}h\metric{D \g}{e_i}^2
	 	+ \frac{n-2k-p}{2(n-k)}h \geq 0, \quad  \forall 1 \leq i \leq n, \label{1-suffi-cond-h-p>-n-new}
	 	\end{align}
	 	on $\mathbb{S}^n$.
 		\item If $-\frac{n+k}{2}<p<-k$, then \eqref{1-suffi-cond-h-p>-n-new} holds on $\mathbb{S}^n$.
 		\item If $-k \leq p \leq n-2k$, then \eqref{1-suffi-cond-h-p>-n-new} holds on $\mathbb{S}^n$.
		\item If $p>n-2k$,  and in addition that $\Omega$ contains the origin, then \eqref{1-suffi-cond-h-p>-n-new} holds on $\mathbb{S}^n$.
	\end{enumerate}
\end{lem}

\begin{proof}
	If $-n <p < -k$, then $n+p>0$, $k+p <0$, and $n-2k-p >n-k>0$. Hence 
	\begin{align}
		&\g^2 D^2 h( e_i,e_i)
		- \g \metric{D h}{D \g} 
		- \frac{2(k+p)}{n-k} \g \metric{Dh}{e_i} \metric{D \g}{e_i}  
		+\frac{n+p}{2(n-k)}h|D \g|^2 \nonumber\\
		&+ \frac{n+p}{2(n-k)}\g^2 h  
		+ \frac{(n+p)(k+p)}{(n-k)^2}h\metric{D \g}{e_i}^2
		+ \frac{n-2k-p}{2(n-k)}h \nonumber\\
		\geq& \g^2 D^2 h( e_i,e_i)- \g |Dh|\cdot|D \g|+\frac{2(k+p)}{n-k}\g |D h|\cdot|D \g|+ \frac{n+p}{2(n-k)}h|D \g|^2 \nonumber\\
		&+\frac{n+p}{2(n-k)}\g^2 h
		+ \frac{(n+p)(k+p)}{(n-k)^2}h|D \g|^2 \nonumber\\
		=&\g^2 D^2h(e_i,e_i) -\mathcal{A} \g|D \g|\cdot |Dh|+ \mathcal{B} |D \g|^2 h+\mathcal{C} \g^2 h \label{suff cond case 1}
	\end{align}
	holds for all $-n <p <-k$, where
	\begin{equation*}
		\mathcal{A} := 1- \frac{2(k+p)}{n-k}>0, \quad \mathcal{B} := \frac{n+p}{n-k} \( \frac{n+p}{n-k}- \frac{1}{2} \), \quad \mathcal{C}:= \frac{1}{2} \frac{n+p}{n-k}.
	\end{equation*}
	Now we divide the remaining proof of Lemma \ref{lem-from assmp of f to desired est} into four cases.
	
	\textbf{Case 1.} $-n <p \leq -\frac{n+k}{2}$, and $\Omega$ is in addition origin symmetric.
	In this case, we have $\mathcal{B} \leq 0$. As we assumed $\Omega$ is origin symmetric, the estimate \eqref{Dphi<phi} and Assumption \ref{assump-h} imply that
	\begin{align}
		&\g^2 D^2h -\mathcal{A} \g|D \g|\cdot |Dh|I+ \mathcal{B} |D \g|^2 h I+\mathcal{C} \g^2 hI  \nonumber\\
		\geq& \g^2 D^2 h- \mathcal{A} \g^2 |D h| I + \mathcal{B} \g^2 h I + \mathcal{C} \g^2 h I \nonumber\\
		=&\g^2 \(  D^2 h- \(\frac{n-3k-2p}{n-k}\) |Dh|I + \( \frac{n+p}{n-k}  \)^2 h I \) \geq 0.
		\label{est suff cond case 1.1}
	\end{align}
	Combining \eqref{suff cond case 1} with \eqref{est suff cond case 1.1}, \eqref{1-suffi-cond-h-p>-n-new} is proved for Case 1.
	
	\textbf{Case 2.} $-\frac{n+k}{2} <p <-k $. 
	It is direct to see that $\mathcal{B}>0$ in Case 2. By the arithmetic-geometric mean inequality, we have
	\begin{equation*}
		\mathcal{A} \g|D \g|\cdot|D h| \leq \mathcal{B}|D \g|^2h + \frac{\mathcal{A}^2}{4\mathcal{B}} \frac{|D h|^2}{h} \g^2. 
	\end{equation*}
	This together with Assumption \ref{assump-h} implies that
	\begin{align}
		&\g^2 D^2h -\mathcal{A} \g|D \g|\cdot |Dh|I+ \mathcal{B} |D \g|^2 h I+\mathcal{C} \g^2 hI \nonumber\\
		\geq& \g^2 \( D^2 h - \frac{\mathcal{A}^2}{4\mathcal{B}} \frac{|D h|^2}{h} I + \mathcal{C}h I \) \nonumber\\
		=&\g^2 \(D^2 h- \frac{(n-3k-2p)^2}{2(n+p)(n+k+2p)} \frac{|D h|^2}{h} I +  \frac{1}{2} \frac{n+p}{n-k}h I \)\geq 0. \label{est suff cond case 1.2}
	\end{align}
	Combining \eqref{suff cond case 1} with \eqref{est suff cond case 1.2}, we obtain \eqref{1-suffi-cond-h-p>-n-new} for Case 2.

	\textbf{Case 3.} $-k \leq p \leq n-2k$.
	Using $p \geq -k$ and the arithmetic-geometric mean inequality, we have
	\begin{align*}
		-\frac{2(k+p)}{n-k} \g \metric{D \g}{e_i} \metric{D h}{e_i}
		\geq& -\frac{(n+p)(k+p)}{(n-k)^2} h \metric{D\g}{e_i}^2 - \frac{k+p}{n+p} \g^2 \frac{\metric{Dh}{e_i}^2}{h},\\
		-\g \metric{D h}{D \g} \geq&  -\frac{n+p}{2(n-k)} h |D \g|^2 -\frac{n-k}{2(n+p)} \g^2\frac{|D h|^2}{h}. 
	\end{align*}
	Using the above inequalities, we obtain
	\begin{align}
		&\g^2 D^2 h( e_i,e_i)
		- \g \metric{D h}{D \g} 
		- \frac{2(k+p)}{n-k} \g \metric{Dh}{e_i} \metric{D \g}{e_i}  
		+\frac{n+p}{2(n-k)}h|D \g|^2 \nonumber\\
		&+ \frac{n+p}{2(n-k)}\g^2 h  
		+ \frac{(n+p)(k+p)}{(n-k)^2}h\metric{D \g}{e_i}^2
		+ \frac{n-2k-p}{2(n-k)}h  \nonumber\\
		\geq& \g^2 D^2 h( e_i,e_i) -\frac{n-k}{2(n+p)} \g^2\frac{|D h|^2}{h} - \frac{k+p}{n+p} \g^2 \frac{\metric{Dh}{e_i}^2}{h}+ \frac{n+p}{2(n-k)}\g^2 h  	+ \frac{n-2k-p}{2(n-k)}h \nonumber\\
		=& \g^2 \(D^2 h( e_i,e_i)  - \frac{k+p}{n+p}\frac{\metric{Dh}{e_i}^2}{h}- \frac{n-k}{2(n+p)}\frac{|D h|^2}{h}+\frac{n+p}{2(n-k)} h	+ \frac{n-2k-p}{2(n-k)} \frac{1}{\g^2}h  \). \label{1-est-k<p}
	\end{align}
	Applying the assumption $p \leq n-2k$ to \eqref{1-est-k<p}, we have
	\begin{align}
		&  \g^2 D^2 h( e_i,e_i)
		- \g \metric{D h}{D \g} 
		- \frac{2(k+p)}{n-k} \g \metric{Dh}{e_i} \metric{D \g}{e_i}  
		+\frac{n+p}{2(n-k)}h|D \g|^2 \nonumber\\
		&+ \frac{n+p}{2(n-k)}\g^2 h  
		+ \frac{(n+p)(k+p)}{(n-k)^2}h\metric{D \g}{e_i}^2
		+ \frac{n-2k-p}{2(n-k)}h  \nonumber\\
		\geq& \g^2 \(D^2 h( e_i,e_i)  - \frac{k+p}{n+p}\frac{\metric{Dh}{e_i}^2}{h}- \frac{n-k}{2(n+p)}\frac{|D h|^2}{h}+\frac{n+p}{2(n-k)} h\). \label{1-est -k<p<n-2k}
	\end{align}
	On the other hand, a direct calculation yields
	\begin{align*}
		D_i \(h^{\frac{n-k}{n+p}}\) = \frac{n-k}{n+p}h^{-\frac{k+p}{n+p}} D_i h, \quad
		D_i D_i \( h^{\frac{n-k}{n+p}} \) = \frac{n-k}{n+p} h^{-\frac{k+p}{n+p}}\( D_iD_i h -\frac{k+p}{n+p} \frac{\(D_i h\)^2}{h} \),\\
		\frac{1}{2} \frac{\left|D \(h^{\frac{n-k}{n+p}} \)\right|^2}{h^{\frac{n-k}{n+p}}}
		= \frac{1}{2} h^{-\frac{n-k}{n+p}} \(\frac{n-k}{n+p}\)^2 h^{-\frac{2(k+p)}{n+p}} |D h|^2
		= \frac{n-k}{n+p} h^{-\frac{k+p}{n+p}}\( \frac{n-k}{2(n+p)} \frac{|D h|^2}{h} \),
	\end{align*}
	and hence
	\begin{align}
		&D_i D_i \( h^{\frac{n-k}{n+p}} \)
		-\frac{1}{2} \frac{\left|D \(h^{\frac{n-k}{n+p}} \)\right|^2}{h^{\frac{n-k}{n+p}}}
		+\frac{1}{2} h^{\frac{n-k}{n+p}} \nonumber\\
		=&\frac{n-k}{n+p} h^{-\frac{k+p}{n+p}}\( D_iD_i h -\frac{k+p}{n+p} \frac{\(D_i h\)^2}{h} \)
		-\frac{n-k}{n+p} h^{-\frac{k+p}{n+p}}\( \frac{n-k}{2(n+p)} \frac{|D h|^2}{h} \)
		+\frac{1}{2} h^{\frac{n-k}{n+p}} \nonumber\\
		=&\frac{n-k}{n+p} h^{-\frac{k+p}{n+p}} \( D^2 h(e_i, e_i) -\frac{k+p}{n+p} \frac{\metric{D h}{e_i}^2}{h} - \frac{n-k}{2(n+p)} \frac{|D h|^2}{h}
		+ \frac{n+p}{2(n-k)} h\). \label{identity -k<p<n-2k}
	\end{align}
	Combining \eqref{1-est -k<p<n-2k} with \eqref{identity -k<p<n-2k} and using Assumption \ref{assump-h}, we have
	\begin{align*}
		&  \g^2 D^2 h( e_i,e_i)
		- \g \metric{D h}{D \g} 
		- \frac{2(k+p)}{n-k} \g \metric{Dh}{e_i} \metric{D \g}{e_i}  
		+\frac{n+p}{2(n-k)}h|D \g|^2 \nonumber\\
		&+ \frac{n+p}{2(n-k)}\g^2 h  
		+ \frac{(n+p)(k+p)}{(n-k)^2}h\metric{D \g}{e_i}^2
		+ \frac{n-2k-p}{2(n-k)}h  \nonumber\\
		\geq& \frac{n+p}{n-k} \g^2 h^{\frac{k+p}{n+p}} \(  D_i D_i \( h^{\frac{n-k}{n+p}} \)
		-\frac{1}{2} \frac{\left|D \(h^{\frac{n-k}{n+p}} \)\right|^2}{h^{\frac{n-k}{n+p}}}
		+\frac{1}{2} h^{\frac{n-k}{n+p}} \) \nonumber\\
		\geq& 0.
	\end{align*}
	Thus we obtain \eqref{1-suffi-cond-h-p>-n-new} for Case 3.
	
	\textbf{Case 4.} $p>n-2k$, and $\Omega$ contains the origin in addition. In this case, we know $\g(z) = e^{u(z)} \geq 1$ and hence
	\begin{equation*}
		\frac{1}{2} + \frac{n-2k-p}{2(n+p)} \frac{1}{\g^2} \geq \frac{1}{2} + \frac{n-2k-p}{2(n+p)} = \frac{n-k}{n+p}.
	\end{equation*}
	Using \eqref{1-est-k<p}, \eqref{identity -k<p<n-2k} and Assumption \ref{assump-h}, we have 
	\begin{align}
		 &\g^2 D^2 h( e_i,e_i)
		- \g \metric{D h}{D \g} 
		- \frac{2(k+p)}{n-k} \g \metric{Dh}{e_i} \metric{D \g}{e_i}  
		+\frac{n+p}{2(n-k)}h|D \g|^2 \nonumber\\
		&+ \frac{n+p}{2(n-k)}\g^2 h  
		+ \frac{(n+p)(k+p)}{(n-k)^2}h\metric{D \g}{e_i}^2
		+ \frac{n-2k-p}{2(n-k)}h  \nonumber\\
		\geq&  \g^2 \(D^2 h( e_i,e_i)  - \frac{k+p}{n+p}\frac{\metric{Dh}{e_i}^2}{h}- \frac{n-k}{2(n+p)}\frac{|D h|^2}{h}+\frac{n+p}{2(n-k)} h	+ \frac{n-2k-p}{2(n-k)} \frac{1}{\g^2}h  \) \nonumber\\
		=& \g^2 \(D^2 h( e_i,e_i)  - \frac{k+p}{n+p}\frac{\metric{Dh}{e_i}^2}{h}- \frac{n-k}{2(n+p)}\frac{|D h|^2}{h}+ \frac{n+p}{n-k} \(\frac{1}{2} + \frac{n-2k-p}{2(n+p)} \frac{1}{\g^2}  \)h  \) \nonumber\\
		\geq& \g^2 \( D^2 h( e_i,e_i)  - \frac{k+p}{n+p}\frac{\metric{Dh}{e_i}^2}{h}- \frac{n-k}{2(n+p)}\frac{|D h|^2}{h} \) +\g^2 h \nonumber\\
		=&\frac{n+p}{n-k} \g^2 h^{\frac{n-k}{n+p}} \(  D_i D_i \( h^{\frac{n-k}{n+p}} \)
		-\frac{1}{2} \frac{\left|D \(h^{\frac{n-k}{n+p}} \)\right|^2}{h^{\frac{n-k}{n+p}}} \)+ \g^2 h  \nonumber\\
		=&\frac{n+p}{n-k} \g^2 h^{\frac{k+p}{n+p}} \(D_i D_i \( h^{\frac{n-k}{n+p}} \)
		-\frac{1}{2}\frac{\left|D \(h^{\frac{n-k}{n+p}} \)\right|^2}{h^{\frac{n-k}{n+p}}} +  \frac{n-k}{n+p}h^{\frac{n-k}{n+p}}    \) \nonumber\\
		\geq& 0.
	\end{align}
	Thus we obtain \eqref{1-suffi-cond-h-p>-n-new} for Case 4.
	
	We complete the proof of Lemma \ref{lem-from assmp of f to desired est}.
\end{proof}

Now we recall the tensor maximum principle proved by Andrews \cite[Theorem 3.2]{And07}.

\begin{thm}[\cite{And07}]\label{And-tensor-max-principle}
	Let $S_{ij}$ be a smooth time-varying symmetric tensor field on a compact manifold $M$ satisfying
	\begin{equation*}
		\frac{\partial}{\partial t} S_{ij} =
		a^{kl} \nabla_k \nabla_l S_{ij} + b^k \nabla_k S_{ij} +N_{ij},
	\end{equation*}
	where $a^{kl}$ and $b^k$ are smooth, $\nabla$ is a (possibly time-dependent) smooth symmetric connection, and $a^{kl}$ is positive definite everywhere. Suppose that
	\begin{equation*}
		N_{ij} v^iv^j + \sup_{\Gamma} 2a^{kl} \( 2\Gamma_k{}^p \nabla_l S_{ip} v^i - \Gamma_k{}^p \Gamma_l{}^q S_{pq} \) \geq 0,
	\end{equation*}
	whenever $S_{ij} \geq 0$ and $S_{ij} v^j=0$ and $\Gamma$ is an $n \times n$-matrix. If $S_{ij}$ is positive definite everywhere on $M$ at $t=0$ and on $\partial M$ for $0 \leq t \leq T$, then it is positive on $M \times [0,T]$. 
\end{thm}

\begin{prop}\label{prop-Pinching-est}
	Let $n \geq 2$ be an integer. Suppose that one of the following assumptions of $k$, $p$, $f(z)$, and $M_t$ holds:
	\begin{itemize}
		\item  $1 \leq k \leq n-1$, $p > -n$,  $f(z)$ is a smooth positive even function on $\mathbb{S}^n$ satisfying Assumption \ref{assump-h}, and $M_t= \partial \Omega_t$ is a smooth, origin symmetric uniformly h-convex solution to the flow \eqref{flow-HCMF};
		\item  $0 \leq k \leq n-1$, $p=-n$, $f(z)$ is a positive constant function on $\mathbb{S}^n$, and $M_t$ is a smooth uniformly h-convex solution to the flow \eqref{flow-HCMF};
		\item $k=0$, $-\infty<p<+\infty$, $f(z)$ is a smooth positive even function on $\mathbb{S}^n$, and $M_t$ is a smooth, origin symmetric uniformly h-convex solution to the flow \eqref{flow-HCMF}.
	\end{itemize}
	Then there exists a positive constant $\varepsilon_0$ depending only on $M_0$, $k$, $p$ and $||f(z)||_{C^2(\mathbb{S}^n)}$, such that along the flow \eqref{flow-HCMF},
	\begin{equation*}
		\tilde{h}_{ij} \geq \varepsilon_0 \tilde{H} g_{ij},
	\end{equation*}
	where $\tilde{H} = g^{lm} \tilde{h}_{lm}$ is the shifted mean curvature. Therefore, along the flow \eqref{flow-HCMF},
	\begin{equation}\label{Pinching-est-2}
		\max_{1 \leq i, j \leq n} \frac{\tilde{\kappa}_i}{ \tilde{\kappa}_j}(x,t) \leq C(\varepsilon_0), \quad \forall \ (x,t) \in M \times [0,T),
	\end{equation}
	where $\tilde{\kappa}_i =\kappa_i-1$, $1 \leq i \leq n$, are the shifted principal curvatures, and $C(\varepsilon_0) = \(1- (n-1) \varepsilon_0\)/ \varepsilon_0$ is a positive constant.
\end{prop}

\begin{proof}
	Let $S_{ij} = \tilde{h}_{ij} - \varepsilon \tilde{H} g_{ij}$. By \eqref{evl-gij}, we have
	\begin{equation*}
		\frac{\partial}{\partial t} S_{ij} =\partial_t \( \( \tilde{h}_i{}^l -\varepsilon \tilde{H} \delta_i{}^l \)g_{lj}\)
		= \(\partial_t \tilde{h}_i{}^l - \varepsilon \partial_t \tilde{H} \delta_i{}^l \)g_{lj} + 2 \mathscr{F} h_j{}^l  S_{il}.
	\end{equation*}
	We set a parabolic operator $\mathcal{L} = \frac{\partial}{\partial t} - \dot{F}^{lm} \nabla_m \nabla_l$. Using \eqref{evl-thij-new}, we have
	\begin{align*}
		\mathcal{L} S_{ij} = &\ddot{F}^{lm,qs} \nabla_i \tilde{h}_{lm} \nabla_j \tilde{h}_{qs} 
		-\varepsilon  \ddot{F}^{lm,qs} \nabla^\iota \tilde{h}_{lm} \nabla_\iota \tilde{h}_{qs} g_{ij}\\
		& +\dot{F}^{lm} \tilde{h}_l{}^q \tilde{h}_{qm}(1- \varepsilon n ) g_{ij}
		 +\dot{F}^{lm} \tilde{h}_l{}^q \tilde{h}_{qm} S_{ij}
 		-\dot{F}^{lm}g_{lm} (\tilde{h}_i{}^q \tilde{h}_{qj} - \varepsilon |\tilde{A}|^2 g_{ij})\\
		&-F (\tilde{h}_i{}^l \tilde{h}_{lj} - \varepsilon |\tilde{A}|^2 g_{ij})
		-\mathscr{F}(\tilde{h}_i{}^l \tilde{h}_{lj} - \varepsilon |\tilde{A}|^2 g_{ij})
		-2\mathscr{F}S_{ij}\\
		&-\Phi\left( \nabla_j \nabla_i \(\g^{- \frac{n+p}{n-k}}h \) - \varepsilon \Delta \(\g^{- \frac{n+p}{n-k}}h \) g_{ij} \right) +2 \mathscr{F} h_j{}^l S_{il}\\
		:=&N_{ij} + \tilde{N}_{ij},
	\end{align*}
	where $|\tilde{A}|^2 := \tilde{h}_i{}^j \tilde{h}_j {}^i$ and  
	\begin{equation*}
		\tilde{N}_{ij} = \ddot{F}^{lm,qs} \nabla_i \tilde{h}_{lm} \nabla_j \tilde{h}_{qs} 
		-\varepsilon  \ddot{F}^{lm,qs} \nabla^\iota \tilde{h}_{lm} \nabla_\iota \tilde{h}_{qs} g_{ij}.
	\end{equation*}
	According to Theorem \ref{And-tensor-max-principle}, we only need to show
	\begin{equation*}
		N_{ij}v^i v^j +\tilde{N}_{ij} v^iv^j 
		+2 \sup_{\Gamma} \dot{F}^{lm} \(2 \Gamma_l{}^q \nabla_m S_{iq} v^i -\Gamma_l{}^q \Gamma_m{}^s S_{qs} \)
		\geq 0,
	\end{equation*}
	whenever $S_{ij} \geq 0$ and $S_{ij}v^j =0$, $|v|=1$. Let $(x_0, t_0) $ be the point where $S_{ij}$ has a null vector $v = v^i \partial_i X$. 

	By Lemma \ref{lem-mono-increasing concavity} and \eqref{inverse function of pl/pk}, $F = \(\frac{p_n}{p_k}\)^{\frac{1}{n-k}}$ is concave and inverse-concave on the positive cone $\Gamma^+$. Recall that $\nabla_l \tilde{h}_{ij}$ is a Codazzi tensor. Hence $F(\tilde{\kappa})$ satisfies all the assumptions to apply \cite[Theorem 4.1]{And07}. Consequently, we only need to prove $N_{ij}v^i v^j \geq 0$ at $(x_0,t_0)$.  

	Remark that $S_{ij}v^j =0$ is equivalent to $\tilde{h}_i{}^l v_l =\varepsilon \tilde{H} v_i$, where $v_i = v^j g_{ij}$. Then by the definition of $N_{ij}$, we have
	\begin{align}
		N_{ij}v^i v^j =& \dot{F}^{lm} \tilde{h}_l{}^q \tilde{h}_{qm}(1- \varepsilon n )
		+\varepsilon(\dot{F}^{lm}g_{lm} + F +\mathscr{F})(|\tilde{A}|^2 - \varepsilon \tilde{H}^2) \nonumber\\
		&-\Phi\left( \nabla_j \nabla_i \(\g^{- \frac{n+p}{n-k}}h \) - \varepsilon \Delta \(\g^{- \frac{n+p}{n-k}}h \) g_{ij} \right)v^i v^j \nonumber\\
		=& \dot{F}^{lm} \tilde{h}_l{}^q \tilde{h}_{qm}(1- \varepsilon n ) 
		+\varepsilon \dot{F}^{lm}g_{lm}(|\tilde{A}|^2 - \varepsilon \tilde{H}^2)
		+\varepsilon \Phi \g^{- \frac{n+p}{n-k}} h (|\tilde{A}|^2 - \varepsilon \tilde{H}^2)  \nonumber\\
		&-\Phi\left( \nabla_j \nabla_i \(\g^{- \frac{n+p}{n-k}}h \) - \varepsilon \Delta \(\g^{- \frac{n+p}{n-k}}h \) g_{ij} \right)v^i v^j.\label{Nij-vi-vj}
	\end{align}
	For the last two terms in the right-hand side of \eqref{Nij-vi-vj}, we have
	\begin{align}
		-\Phi\left( \nabla_j \nabla_i \( \g^{- \frac{n+p}{n-k}}h \) - \varepsilon \Delta \(\g^{- \frac{n+p}{n-k}}h \) g_{ij} \right)v^i v^j 
		=\( \uppercase\expandafter{\romannumeral1} \) + \(\uppercase\expandafter{\romannumeral2}\)
		+\( \uppercase\expandafter{\romannumeral3} \), \label{last-term-Nijvivj}
	\end{align}
	where
	\begin{align}
		\( \uppercase\expandafter{\romannumeral1} \) :=&-\Phi \g^{- \frac{n+p}{n-k}} \(\nabla_j \nabla_i h - \varepsilon \Delta h g_{ij} \)v^iv^j, \label{last-term-Nijvivj-1}\\
		\(\uppercase\expandafter{\romannumeral2}\):=& -2\Phi \(\nabla_i \(\g^{- \frac{n+p}{n-k}}\) \nabla_j h - \varepsilon \metric{\nabla \(\g^{- \frac{n+p}{n-k}}\)}{\nabla h}g_{ij}\)v^iv^j, 
		\label{last-term-Nijvivj-2}\\
		\( \uppercase\expandafter{\romannumeral3} \):=& -\Phi h \(\nabla_j \nabla_i \(\g^{- \frac{n+p}{n-k}}\) - \varepsilon \Delta \(\g^{- \frac{n+p}{n-k}}\) g_{ij}\)v^iv^j.
		\label{last-term-Nijvivj-3}
	\end{align}
	Now we parameterize $M_{t_0}$ around $X(x_0, t_0)$ by normal coordinates for $\mathbb{S}^n$ about $z_0 = G_{t_0}(X(x_0, t_0))$, such that $A_{ij} [\g(z_0, t_0)]$ is diagonal, and the corresponding coordinate vector fields on $\mathbb{S}^n$ are $\lbrace e_1, \ldots, e_n \rbrace$. By the above convention, we have
	\begin{equation}\label{partial i X--e_i}
		\(d G_{t_0} \)_{X(x,t_0)} \( \partial_i X \) = e_i, \quad 1 \leq i \leq n.
	\end{equation}
	We write $\nabla_i$ for $\nabla_{\partial_i X}$ on $M_{t_0}$ and $D_i$ for $D_{e_i}$ on $\mathbb{S}^n$. The calculations of \eqref{last-term-Nijvivj} are divided into three parts.

	\textbf{Part 1}. In this part, we calculate \eqref{last-term-Nijvivj-1}. Recall that $\widetilde{ \nabla}$ denotes the connection of $\(\mathbb{R}^{n+1,1}, \metric{\cdot}{\cdot} \)$. By identifying $\mathbb{S}^n$ with $\mathbb{S}^n \times \{1\} \subset \mathbb{R}^{n+1,1}$, we extend the definition of $h(z)$ from $\mathbb{S}^n$ to $\mathbb{R}^{n+1,1} \backslash \(\{0\} \times \mathbb{R}\)$ by $h((y,y_{n+1})) := h \(\frac{y}{|y|} \)$, where $y \in \mathbb{R}^{n+1} \backslash \{0\}$ and $y_{n+1} \in \mathbb{R}$. Therefore
	\begin{equation*}
 		\metric{\widetilde{\nabla} h}{ (z,0) } = \metric{\widetilde{\nabla} h}{ (0,1) } = 0
	\end{equation*}
	on $\mathbb{S}^n \times \{1\}$, which implies $\widetilde{\nabla} h = (D h,0) $ on $\mathbb{S}^n \times \{1\}$. Hence,
	\begin{equation}\label{rel-Hess-h}
		\widetilde{\nabla}^2 h ((v,0),(w,0)) = D^2 h(v,w), \quad \forall v, w \in T_z \mathbb{S}^n,
	\end{equation}
	where the left-hand side was evaluated at $(z,1)$, and the right-hand side was evaluated at $z$. Using \eqref{rel-Hess-h} and \eqref{partial i X--e_i}, we have
	\begin{align}
		\nabla_j \nabla_i h =& 
		\nabla_j \metric{\widetilde{\nabla} h}{ \nabla_i (z,1)} \nonumber\\
		=&\widetilde{\nabla}^2 h( \nabla_i (z,1), \nabla_j (z,1)) + \metric{\widetilde{\nabla} h}{\nabla_j \nabla_i (z,1)} \nonumber\\
		=&D^2 h(\nabla_i z, \nabla_j z) + \metric{(D h,0)}{\nabla_j \nabla_i  (z,1)}  \nonumber\\
		=&D^2 h(e_i, e_j) + \metric{(D h,0)}{\nabla_j \nabla_i  (z,1)}. \label{nblj-nbli-h}
	\end{align}
	Formula \eqref{X-nu} and $\metric{(Dh, 0)}{(z,1)}=0$ imply $\metric{\(Dh, 0\) }{X-\nu} = 0$. Inserting \eqref{zij} into \eqref{nblj-nbli-h} therefore yields
	\begin{align}
		\nabla_j \nabla_i h 
		=&D^2 h \(e_i, e_j\)  + \big\langle (Dh,0), - \g \nabla^l \tilde{h}_{ij} \partial_l X +\g \tilde{h}_i{}^l \tilde{h}_{lj} \nu \big\rangle \nonumber\\
		&-\g^2 \tilde{h}_i{}^l \tilde{h}_j{}^m \metric{V}{\partial_m X} \metric{(Dh,0)}{\partial_l X}
		-\g^2 \tilde{h}_i{}^m \tilde{h}_j{}^l \metric{V}{\partial_m X} \metric{(Dh,0)}{\partial_l X}. \label{Nbl-i-Nbl-j-h}
	\end{align}
	By using \eqref{di X} and \eqref{nu}, we have
	\begin{align}
		\metric{(Dh,0)}{- \g \nabla^l \tilde{h}_{ij} \partial_l X}=& \g \nabla^l \tilde{h}_{ij} A_{lm}[\g] D_m h, \label{Dh,hijlXl}\\
		\metric{(D h,0)}{\g \tilde{h}_i{}^l \tilde{h}_{lj} \nu} =& -\g \tilde{h}_i{}^l \tilde{h}_{lj} \metric{D h}{D \g} \label{Dh,nu}.
	\end{align}
	By using \eqref{hil Xl} and \eqref{hilVdlX}, we have
	\begin{equation}\label{Dh,VXsXl}
		-\g^2 \tilde{h}_i{}^l \tilde{h}_j{}^m \metric{V}{\partial_m X}\metric{(D h,0)}{ \partial_l X}
		= \frac{1}{\g} D_i h D_j \g.
	\end{equation}
	Inserting \eqref{Dh,hijlXl}, \eqref{Dh,nu} and \eqref{Dh,VXsXl} into \eqref{Nbl-i-Nbl-j-h}, we have
	\begin{align}
		\nabla_j \nabla_i h =& D^2h \( e_i, e_j\) +\g \nabla^l \tilde{h}_{ij} A_{lm}[\g] D_m h
		-\g \tilde{h}_i{}^l \tilde{h}_{lj} \metric{D h}{D \g} \nonumber\\
		&+\frac{1}{\g}\metric{D h}{e_i} \metric{D \g}{e_j}
		+\frac{1}{\g}\metric{D h}{e_j} \metric{D \g}{e_i}.\label{nabla-ij-h}
	\end{align}
	Then
	\begin{align}
		\nabla_j \nabla_i h v^i v^j =& D^2h \(\sum_{i=1}^n v^i e_i, \sum_{j=1}^n v^je_j\) +\g \nabla^l \tilde{h}_{ij} A_{lm}[\g] \(D_m h\)v^iv^j \nonumber\\
		&-\g \tilde{h}_i{}^l \tilde{h}_{lj} \metric{D h}{D \g}v^iv^j+\frac{2}{\g}\metric{D h}{\sum_{i=1}^n v^ie_i} \metric{D \g}{\sum_{j=1}^n v^je_j}. \label{nbli nblj h vi vj}
	\end{align}
	By \eqref{hil Xl} and the assumption $ \tilde{h}_i{}^jv^i =\varepsilon \tilde{H}v^j$ at $X(x_0, t_0)$, we have
	\begin{equation}\label{viei}
		\sum_{i=1}^n v^i(e_i,0) = -\g v^i \tilde{h}_i{}^j \partial_j X + \frac{v^i D_i \g}{\g}(z,1)
		=-\varepsilon \g \tilde{H} v^j \partial_j X+ \frac{v^i D_i \g}{\g}(z,1). 
	\end{equation}
	Note that formula \eqref{di X} implies $\metric{\partial_j X}{(z,1)} =0$. Therefore, by using \eqref{viei}, $\metric{(z,1)}{(z,1)} = 0$ and the assumption $v^iv^j g_{ij} =1$, we have 
	\begin{equation}\label{norm-viei}
		\metric{\sum_{i=1}^n v^i(e_i,0)  }{  \sum_{j=1}^n v^j(e_j,0)}
		= \varepsilon^2 \g^2 \tilde{H}^2 v^i v^j \metric{\partial_i X}{\partial_j X}
		= \varepsilon^2 \g^2 \tilde{H}^2.
	\end{equation}
	Letting $\bar{v} = \sum_{i=1}^n v^ie_i/(\varepsilon \g \tilde{H})$ be a unit length vector on $T_{z_0} \mathbb{S}^n$, we have from \eqref{nbli nblj h vi vj} that
	\begin{align}
		\nabla_j \nabla_i h v^i v^j =& \varepsilon^2 \g^2 \tilde{H}^2 D^2h(\bar{v}, \bar{v})
		+ \g \nabla^l \tilde{h}_{ij} v^iv^j A_{lm}[\g] D_m h \nonumber\\
		&-\varepsilon^2 \g \tilde{H}^2 \metric{Dh}{D \g} + 2\varepsilon^2 \g \tilde{H}^2 \metric{D h}{\bar{v}} \metric{D \g}{\bar{v}}.\label{hijvivj}
	\end{align}
	Using \eqref{gij-A-phi} and  \eqref{shifted curvature-support function}, we have
	\begin{equation}\label{g^ij-[hi-kappa]}
		g^{ij} = \( \sum_{l=1}^n A_{il}[\g] A_{lj}[\g] \)^{-1} = \g^2 \tilde{\kappa}_i^2 \delta_{ij} .
	\end{equation}
	Then by \eqref{nabla-ij-h}, we know
	\begin{align}
		\varepsilon \Delta h=
		\varepsilon g^{ij} \nabla_j \nabla_i h =&
		\varepsilon \g^2 \sum_{i=1}^n \tilde{\kappa}_i^2 D^2 h( e_i,e_i) 
		+\varepsilon \g \nabla^l \tilde{H}  A_{lm}[\g] D_m h \nonumber\\
		&- \varepsilon \g \metric{D h}{D \g} |\tilde{A}|^2 
		+\varepsilon \g^2 \frac{2}{\g}  \sum_{i=1}^n \tilde{\kappa}_i^2\metric{Dh}{e_i} \metric{D \g}{e_i} \nonumber\\
		=& \varepsilon \g^2 |\tilde{A}|^2 \sum_{i=1}^n \frac{\tilde{\kappa}_i^2}{|\tilde{A}|^2} D^2 h( e_i,e_i) 
		+\varepsilon \g \nabla^l \tilde{H}  A_{lm}[\g] D_m h \nonumber\\
		&- \varepsilon \g \metric{D h}{D \g} |\tilde{A}|^2 
		+2\varepsilon \g   |\tilde{A}|^2\sum_{i=1}^n \frac{\tilde{\kappa}_i^2}{|\tilde{A}|^2}\metric{Dh}{e_i} \metric{D \g}{e_i}. \label{hijgij}
	\end{align}
	Since $S_{ij} \geq 0$ on $M_{t_0}$, and $S_{ij}v^j=0$ at $X(x_0,t_0)$, we have $\nabla_l S_{ij} v^iv^j =0$ at the point, and hence $\nabla_l \tilde{h}_{ij} v^iv^j = \varepsilon \nabla_l \tilde{H}$ at $X(x_0, t_0)$ of $M_{t_0}$.  Inserting \eqref{hijvivj} and \eqref{hijgij} into \eqref{last-term-Nijvivj-1}, we have
	\begin{align}
		\( \uppercase\expandafter{\romannumeral1} \):= &-\Phi \g^{- \frac{n+p}{n-k}}(\nabla_j \nabla_i h - \varepsilon \Delta h g_{ij})v^iv^j \nonumber\\
		=&-\Phi \g^{- \frac{n+p}{n-k}}\( \(\nabla_j \nabla_i h\) v^iv^j - \varepsilon \Delta h \) \nonumber\\
 		=& \varepsilon \Phi \g^{- \frac{n+p}{n-k}} \(\g^2 \(\sum_{i=1}^n \frac{\tilde{\kappa}_i^2}{|\tilde{A}|^2} D^2 h( e_i,e_i)\)|\tilde{A}|^2 -\varepsilon \g^2  D^2h(\bar{v}, \bar{v}) \tilde{H}^2\)  \nonumber\\
		 &- \varepsilon \Phi \g^{- \frac{n+p}{n-k}+1} \metric{Dh}{D \g}(|\tilde{A}|^2 - \varepsilon \tilde{H}^2)  \nonumber\\
		& +2 \varepsilon \Phi \g^{- \frac{n+p}{n-k}+1}  \(\(\sum_{i=1}^n \frac{\tilde{\kappa}_i^2}{|\tilde{A}|^2}\metric{Dh}{e_i} \metric{D \g}{e_i}\) |\tilde{A}|^2
		-\varepsilon \metric{D h}{\bar{v}} \metric{D \g}{\bar{v}}\tilde{H}^2  \). \label{last-term-Nijvivj-P1}
	\end{align}

	\textbf{Part 2}. In this part, we calculate \eqref{last-term-Nijvivj-2}. A straightforward calculation yields that
	\begin{equation}\label{phi-i-h-j}
		\nabla_i \(\g^{- \frac{n+p}{n-k}}\) \nabla_j h
		= -\frac{n+p}{n-k} \g^{- \frac{n+p}{n-k}-1} \nabla_i \g \nabla_j h
		=-\frac{n+p}{n-k} \g^{- \frac{n+p}{n-k}-1} D_i \g D_j h.
	\end{equation}
	The following calculations are similar to those in \eqref{hijvivj} and \eqref{hijgij}. Using \eqref{phi-i-h-j} and \eqref{norm-viei}, we have
	\begin{align}
		-2\Phi \nabla_i \( \g^{- \frac{n+p}{n-k}}\) \(\nabla_j h\) v^iv^j
		=&
		2 \frac{n+p}{n-k}\Phi   \g^{- \frac{n+p}{n-k}-1} \metric{D \g}{e_i} \metric{D h}{e_j}v^iv^j \nonumber\\
		=& 2  \frac{n+p}{n-k} \Phi \g^{- \frac{n+p}{n-k}-1}  \metric{D \g}{\sum_{i=1}^n v^ie_i} \metric{Dh}{\sum_{j=1}^n v^j e_j} \nonumber\\
		=&  2  \frac{n+p}{n-k} \Phi \g^{- \frac{n+p}{n-k}-1} \varepsilon^2 \g^2 \tilde{H}^2 \metric{D\g}{\bar{v}} \metric{D h}{\bar{v}} \nonumber\\
		=& 2 \varepsilon^2\frac{n+p}{n-k} \Phi \g^{- \frac{n+p}{n-k}+1}\metric{D\g}{\bar{v}} \metric{D h}{\bar{v}}
		\tilde{H}^2.\label{Part2-1}
	\end{align}
	Using the assumption $v^iv^jg_{ij}=1$ and \eqref{g^ij-[hi-kappa]}, we obtain
	\begin{align}
		2 \varepsilon \Phi \metric{\nabla \(\g^{- \frac{n+p}{n-k}}\)}{\nabla h}g_{ij} v^iv^j
		=&
		2 \varepsilon \Phi \(-\frac{n+p}{n-k}   \g^{- \frac{n+p}{n-k}-1} D_i \g D_j h g^{ij}  \) \nonumber\\
		=&
		-2\varepsilon \frac{n+p}{n-k} \Phi  \g^{- \frac{n+p}{n-k}-1} D_i \g D_j h \( \g^2 \tilde{\kappa}_i^2 \delta_{ij} \) \nonumber\\
		=& 
		-2\varepsilon \frac{n+p}{n-k} \Phi  \g^{- \frac{n+p}{n-k}+1} \(\sum_{i=1}^n \tilde{\kappa}_i^2 D_i \g D_i h  \) \nonumber\\
		=&-2\varepsilon \frac{n+p}{n-k} \Phi  \g^{- \frac{n+p}{n-k}+1} \( \sum_{i=1}^n \frac{\tilde{\kappa}_i^2}{|\tilde{A}|^2} \metric{D\g}{e_i} \metric{D h}{e_i} \) |\tilde{A}|^2. \label{Par2-2}
	\end{align}
	Substituting \eqref{Part2-1} and \eqref{Par2-2} into \eqref{last-term-Nijvivj-2}, we have
	\begin{align}
		\( \uppercase\expandafter{\romannumeral2} \):=&-2\Phi \(\nabla_i \( \g^{- \frac{n+p}{n-k}}\) \nabla_j h - \varepsilon \metric{\nabla \(\g^{- \frac{n+p}{n-k}}\)}{\nabla h}g_{ij}  \)v^iv^j  \nonumber\\
		=& - 2 \varepsilon \frac{n+p}{n-k} \Phi \g^{- \frac{n+p}{n-k}+1}  \(\(\sum_{i=1}^n \frac{\tilde{\kappa}_i^2}{|\tilde{A}|^2}\metric{Dh}{e_i} \metric{D \g}{e_i}\) |\tilde{A}|^2
		-\varepsilon \metric{D h}{\bar{v}} \metric{D \g}{\bar{v}}\tilde{H}^2  \).\label{last-term-Nijvivj-P2}
	\end{align}

	\textbf{Part 3}. In this part, we calculate \eqref{last-term-Nijvivj-3}. For the term $\nabla_j \nabla_i \(\g^{- \frac{n+p}{n-k}} \)$, we have
	\begin{align*}
		&\nabla_j \nabla_i \(\g^{- \frac{n+p}{n-k}} \) 
		=\nabla_j \nabla_i  \( \(\frac{1}{\g}\)^{ \frac{n+p}{n-k}}  \) \\
		=& \frac{n+p}{n-k} \g^{- \frac{n+p}{n-k}+1} \nabla_j \nabla_i \(\frac{1}{\g}\)
		+ \frac{n+p}{n-k} \( \frac{n+p}{n-k}-1 \) \g^{- \frac{n+p}{n-k}+2} \nabla_i \(\frac{1}{\g}\) \nabla_j \(\frac{1}{\g}\).
	\end{align*}
	This together with \eqref{1/phi, coshr-u}, \eqref{1/phi i new} and \eqref{1/phi ij} yields
	\begin{align*}
		\nabla_j \nabla_i \(\g^{- \frac{n+p}{n-k}} \) 
		=& \frac{n+p}{n-k} \g^{- \frac{n+p}{n-k}+1} (-\nabla^l \tilde{h}_{ij} \metric{V}{\partial_l X} 
		-\frac{1}{\g}\tilde{h}_{ij} +\tilde{h}_i{}^l \tilde{h}_{lj} \tilde{u})\\
		&+\frac{n+p}{n-k}  \(\frac{n+p}{n-k}-1 \) \g^{- \frac{n+p}{n-k}-2} D_i \g D_j \g.
	\end{align*}
	Hence, by using \eqref{g^ij-[hi-kappa]} and the assumption that $\tilde{h}_i{}^j v_j = \varepsilon \tilde{H}v_i$ at $X(x_0, t_0)$, we have
	\begin{align}
		\nabla_j \nabla_i \(\g^{- \frac{n+p}{n-k}}\) v^i v^j
		=& -  \frac{n+p}{n-k} \g^{- \frac{n+p}{n-k}+1}  \(\nabla^l \tilde{h}_{ij}\)  v^iv^j \metric{V}{\partial_l X} 
		-\varepsilon \frac{n+p}{n-k} \g^{- \frac{n+p}{n-k}} \tilde{H} \nonumber\\
		&+ \varepsilon^2  \frac{n+p}{n-k} \g^{- \frac{n+p}{n-k}+1} \tilde{u} \tilde{H}^2
		+ \varepsilon^2 \frac{n+p}{n-k} \( \frac{n+p}{n-k}-1 \) \g^{- \frac{n+p}{n-k}} \metric{D \g}{\bar{v}}^2 \tilde{H}^2 \label{nbli-nblj-phi}
	\end{align}
	and
	\begin{align}
		\varepsilon \Delta \(\g^{- \frac{n+p}{n-k}}\) 
		=&- \varepsilon \frac{n+p}{n-k} \g^{- \frac{n+p}{n-k}+1}  \nabla^l \tilde{H} \metric{V}{\partial_l X} \nonumber\\
		&-\varepsilon \frac{n+p}{n-k} \g^{- \frac{n+p}{n-k}} \tilde{H}
		+\varepsilon  \frac{n+p}{n-k} \g^{- \frac{n+p}{n-k}+1} \tilde{u} \tilde{|A|}^2 \nonumber\\
		&+\varepsilon \frac{n+p}{n-k} \( \frac{n+p}{n-k}-1 \) \g^{- \frac{n+p}{n-k}} \(\sum_{i=1}^n \frac{\tilde{\kappa}_i^2}{|\tilde{A}|^2}\metric{D \g}{e_i}^2\) \tilde{|A|}^2. \label{epsl-Lap-phi}
	\end{align}
	Substituting \eqref{nbli-nblj-phi} and \eqref{epsl-Lap-phi} into \eqref{last-term-Nijvivj-3} and using $\nabla^l \tilde{h}_{ij} v^i v^j = \varepsilon \nabla^l \tilde{H}$ at $X(x_0, t_0)$, we have
	\begin{align}
		\( \uppercase\expandafter{\romannumeral3} \):=&-\Phi h \(\nabla_j \nabla_i \(\g^{- \frac{n+p}{n-k}}\) - \varepsilon \Delta \(\g^{- \frac{n+p}{n-k}}\) g_{ij} \)v^iv^j \nonumber\\
		=&\varepsilon  \frac{n+p}{n-k} \Phi \g^{- \frac{n+p}{n-k}+1}  h \tilde{u} (\tilde{|A|}^2 -\varepsilon \tilde{H}^2) \nonumber\\
		&+\varepsilon \frac{n+p}{n-k} \( \frac{n+p}{n-k}-1 \) \Phi \g^{- \frac{n+p}{n-k}} h \( \(\sum_{i=1}^n \frac{\tilde{\kappa}_i^2}{|\tilde{A}|^2}\metric{D \g}{e_i}^2\) \tilde{|A|}^2 -  \varepsilon \metric{D \g}{\bar{v}}^2 \tilde{H}^2   \).\label{last-term-Nijvivj-P3}
	\end{align}

	Putting \eqref{last-term-Nijvivj-P1}, \eqref{last-term-Nijvivj-P2} and \eqref{last-term-Nijvivj-P3} into \eqref{last-term-Nijvivj} and using \eqref{Nij-vi-vj}, we have
	\begin{align}
		N_{ij}v^iv^j 
		=& \varepsilon \Phi \g^{- \frac{n+p}{n-k}} \(\g^2 \(\sum_{i=1}^n \frac{\tilde{\kappa}_i^2}{|\tilde{A}|^2} D^2 h( e_i,e_i)\)|\tilde{A}|^2 -\varepsilon \g^2  D^2h(\bar{v}, \bar{v}) \tilde{H}^2\)   \nonumber\\
		&- \varepsilon \Phi \g^{- \frac{n+p}{n-k}+1} \metric{Dh}{D \g}(|\tilde{A}|^2 - \varepsilon \tilde{H}^2)  \nonumber\\
		& -2 \varepsilon (\frac{n+p}{n-k} -1)\Phi \g^{- \frac{n+p}{n-k}+1}
		\( 
		 \(\sum_{i=1}^n \frac{\tilde{\kappa}_i^2}{|\tilde{A}|^2}\metric{Dh}{e_i} \metric{D \g}{e_i}\) |\tilde{A}|^2
		-\varepsilon \metric{D h}{\bar{v}} \metric{D \g}{\bar{v}}\tilde{H}^2  
		\)  \nonumber\\
		&+\varepsilon  \frac{n+p}{n-k} \Phi \g^{- \frac{n+p}{n-k}+1}  h \tilde{u} (\tilde{|A|}^2 -\varepsilon \tilde{H}^2)  \nonumber\\
		&+\varepsilon \frac{n+p}{n-k} \( \frac{n+p}{n-k}-1 \) \Phi \g^{- \frac{n+p}{n-k}} h 
		\(
		\(\sum_{i=1}^n \frac{\tilde{\kappa}_i^2}{|\tilde{A}|^2}\metric{D \g}{e_i}^2\) \tilde{|A|}^2 -  \varepsilon \metric{D \g}{\bar{v}}^2 \tilde{H}^2
		\)  \nonumber\\
		&+\dot{F}^{lm} \tilde{h}_l{}^q \tilde{h}_{qm}(1- \varepsilon n ) 
		+\varepsilon \dot{F}^{lm}g_{lm}(|\tilde{A}|^2 - \varepsilon \tilde{H}^2)
		+\varepsilon \Phi \g^{- \frac{n+p}{n-k}} h (|\tilde{A}|^2 - \varepsilon \tilde{H}^2).\label{Nijvivj-final-version}
	\end{align}

	If $M_t = \partial \Omega_t$ is origin symmetric, then we recall the previous estimates of  $||\g||_{C^1(\mathbb{S}^n)}$ and $\sup_{t \in [0,T)} \Phi(t)$ in \eqref{C0-est}, \eqref{C1-est} and \eqref{Phi<C}. Since $\tilde{H}^2 \leq n |\tilde{A}|^2$, the terms containing $\tilde{H}^2$ (except $\varepsilon^2 \dot{F}^{lm}g_{lm} \tilde{H}^2$) in the right-hand side of \eqref{Nijvivj-final-version} can be bounded from below by $-\varepsilon^2 C |\tilde{A}|^2$, where the constant $C$ only depends on $n$, $k$, $p$, $\sup_{t \in [0,T)} \Phi(t)$, $||\g||_{C^1(\mathbb{S}^n)}$ and $||f||_{C^2(\mathbb{S}^n)}$; however, since \eqref{F-ij delt-ij geq 1} showed  $\dot{F}^{lm}g_{lm} \geq 1$, the quantity $-\varepsilon^2 C |\tilde{A}|^2$ can be controlled by the positive term $\varepsilon\dot{F}^{lm}g_{lm}(|\tilde{A}|^2 - \varepsilon \tilde{H}^2)$ when  $\varepsilon$ is sufficiently small. Hence, when $M_t = \partial \Omega_t$ is origin symmetric, in order to prove $N_{ij} v^i v^j \geq 0$ for small $\varepsilon$, we only need to focus on the coefficient of $|\tilde{A}|^2$ in \eqref{Nijvivj-final-version}.

\textbf{Case 1.} $1 \leq k \leq n-1$ and $p >-n$. Recall the assumption that $M_t= \partial \Omega_t$ is origin symmetric in this case. By the above argument, we only need to ensure the following inequality to be true, i.e.
\begin{align}
& \g^2 \(\sum_{i=1}^n \frac{\tilde{\kappa}_i^2}{|\tilde{A}|^2} D^2 h( e_i,e_i)\)
 - \g \metric{D h}{D \g} 
-2 \(\frac{n+p}{n-k} -1 \) \g  \(\sum_{i=1}^n \frac{\tilde{\kappa}_i^2}{|\tilde{A}|^2}\metric{Dh}{e_i} \metric{D \g}{e_i}\)  \nonumber\\
&+\frac{n+p}{n-k} \g h \tilde{u}
+ \frac{n+p}{n-k} \( \frac{n+p}{n-k}-1 \)h\(\sum_{i=1}^n \frac{\tilde{\kappa}_i^2}{|\tilde{A}|^2}\metric{D \g}{e_i}^2\)
+h \geq 0.\label{0-suffi-cond-h-p>-n}
\end{align}
Formula \eqref{tilde u} yields 
\begin{equation*}
\g \tilde{u} = \frac{1}{2} |D \g|^2 + \frac{1}{2}(\g^2-1).
\end{equation*}
Hence the inequality \eqref{0-suffi-cond-h-p>-n} is equivalent to
\begin{align}
0 \leq &  \sum_{i=1}^n \frac{\tilde{\kappa}_i^2}{|\tilde{A}|^2} \g^2 D^2 h( e_i,e_i)
- \g \metric{D h}{D \g} 
- \sum_{i=1}^n \frac{\tilde{\kappa}_i^2}{|\tilde{A}|^2} \frac{2(k+p)}{n-k} \g \metric{Dh}{e_i} \metric{D \g}{e_i}  
+\frac{n+p}{2(n-k)}h|D \g|^2 \nonumber\\
&+ \frac{n+p}{2(n-k)}\g^2 h  
+ \sum_{i=1}^n \frac{\tilde{\kappa}_i^2}{|\tilde{A}|^2} \frac{(n+p)(k+p)}{(n-k)^2}h\metric{D \g}{e_i}^2
+ \frac{n-2k-p}{2(n-k)}h \nonumber\\
=& \sum_{i=1}^n  \frac{\tilde{\kappa}_i^2}{|\tilde{A}|^2} \(\g^2 D^2 h( e_i,e_i)
-\g \metric{D h}{D \g}
- \frac{2(k+p)}{n-k} \g \metric{Dh}{e_i} \metric{D \g}{e_i} 
+\frac{n+p}{2(n-k)}h|D \g|^2 \right. \nonumber \\
&\left. +\frac{n+p}{2(n-k)}\g^2 h 
+\frac{(n+p)(k+p)}{(n-k)^2}h\metric{D \g}{e_i}^2
+ \frac{n-2k-p}{2(n-k)}h \). \label{1-suffi-cond-h-p>-n}
\end{align}
Since $\Omega_t$ is origin symmetric, and $f(z)$ satisfies Assumption \ref{assump-h}, we know that \eqref{1-suffi-cond-h-p>-n} follows from Lemma \ref{lem-from assmp of f to desired est}.
 
\textbf{Case 2.} $0 \leq  k \leq n-1$ and $p=-n$. In this case, we assumed that $f(z)$ is a positive constant function on $\mathbb{S}^n$. Then \eqref{Nijvivj-final-version} implies 
\begin{equation*}
N_{ij} v^iv^j = 
\dot{F}^{lm} \tilde{h}_l{}^q \tilde{h}_{qm}(1- \varepsilon n ) 
+\varepsilon \dot{F}^{lm}g_{lm}(|\tilde{A}|^2 - \varepsilon \tilde{H}^2)
+\varepsilon \Phi  h (|\tilde{A}|^2 - \varepsilon \tilde{H}^2) \geq 0
\end{equation*}
for $\varepsilon$ satisfying $0 <\varepsilon \leq \frac{1}{n}$.

\textbf{Case 3.} $k=0$ and $-\infty<p<+\infty$. From $S_{ij} v^i = 0$, we can easily get $\tilde{\kappa}_1(x_0, t_0) = \varepsilon \tilde{H}(x_0, t_0)$, where we assumed $\tilde{\kappa}_1 \leq  \cdots \leq \tilde{\kappa}_n$. Then we know $\tilde{\kappa}_1 \leq  \frac{(n-1) \varepsilon}{1-\varepsilon} \tilde{\kappa}_n$, which implies
\begin{equation*}
\dot{F}^{lm}g_{lm} \geq \frac{\partial p_n^{\frac{1}{n}}(\tilde{\kappa})}{\partial \tilde{\kappa}_1}
=\frac{p_n^{\frac{1}{n}}(\tilde{\kappa})}{ n \tilde{\kappa}_1}
\geq \frac{\tilde{\kappa}_1^{\frac{n-1}{n}} \tilde{\kappa}_n^{\frac{1}{n}} }{n \tilde{\kappa}_1} 
= \frac{1}{n}\(\frac{\tilde{\kappa}_n}{\tilde{\kappa}_1} \)^{\frac{1}{n}} \geq \frac{1}{n (n-1)^{\frac{1}{n}}}\(\frac{1-\varepsilon}{\varepsilon} \)^{\frac{1}{n}}. 
\end{equation*}
Hence we have that $\dot{F}^{lm}g_{lm} \to +\infty$ as $\varepsilon \to 0^+$. Recall that $M_t =\partial \Omega_t$ is origin symmetric in Case 3.
By using \eqref{Nijvivj-final-version} and the a priori estimates in  Lemma \ref{lem-C0-est}, Corollary \ref{cor-C1-est} and Lemma \ref{lem-c < Phi <C}, we can choose a small $\varepsilon_0 = \varepsilon_0(k, p, M_0, ||f||_{C^2(\mathbb{S}^n)}) >0$ such that $N_{ij}v^iv^j \geq 0$. 

In summary, with the assumptions of $n$, $k$, $p$, $f(z)$ in Proposition \ref{prop-Pinching-est}, there exists a small positive constant $\varepsilon_0>0$ such that $N_{ij} v^i v^j \geq 0$ at $X(x_0,t_0)$. Hence we have $\tilde{h}_{ij} \geq \varepsilon_0 \tilde{H} g_{ij}$ along the flow \eqref{flow-HCMF}, which implies
\begin{equation*}
\tilde{\kappa}_1(x,t) \geq \varepsilon_0 \tilde{H}(x,t) \geq \varepsilon_0(n-1) \tilde{\kappa}_1(x,t) + \varepsilon_0 \tilde{\kappa}_n(x,t), \quad \forall \ (x,t) \in M \times [0,T),
\end{equation*} 
where we assumed $\tilde{\kappa}_1 \leq \cdots \leq \tilde{\kappa}_n$ at $X(x,t)$ of $M_t$. This implies the desired inequality \eqref{Pinching-est-2}. We complete the proof of Proposition \ref{prop-Pinching-est}.
\end{proof}
\subsubsection{Estimates of shifted principal curvatures}
For abbreviation, let $\mathcal{F}(z,t) = p_{n-k}^{\frac{1}{n-k}} (A [\g(z,t)])$ and $P(z,t) =\Phi(t) h(z) \g^{-\frac{n+p}{n-k}}(z,t)$. If $A[\g]>0$, then \eqref{G=Gij Aij}, \eqref{Gij delt-ij in lemma} and \eqref{G-ij >0} yield
\begin{align}
\mathcal{F}(A [\g]) =& \dot{\mathcal{F}}^{ij} A_{ij}[\g], \label{mathcal F-ij A-ij}\\
\sum_{i=1}^n \mathcal{F}^{ii} \geq& 1, \label{sum mathcal F-ii}\\
\mathcal{F}^{ij} >&0. \label{mathcal F-ij >0}
\end{align}
Recall that $\sigma_{ij}$ denotes the metric on the unit sphere $\mathbb{S}^n$.
In the following lemmas, for a fixed pair $(z_0,t_0) \in \mathbb{S}^n \times [0,T)$, we use normal coordinates for $\mathbb{S}^n$ around $z_0$ such that $\sigma_{ij}(z_0, t) =\delta_{ij}$ for $t \in [0,T)$ and  $A_{ij}[\g(z_0,t_0)]$ is diagonal.

\begin{lem}\label{lem-dt-Aij}
	Along the scalar evolution equation \eqref{scalar eq phi-HCMF} of $\g(z,t)$, $A_{ij}[\g(z,t)]$ evolves as
	\begin{align}
	\frac{\partial}{\partial t} A_{ij} [\g(z,t)] =&
	\frac{1}{\mathcal{F}^2} D_j D_i \mathcal{F} - 2\frac{\dot{\mathcal{F}}^{ij}}{\mathcal{F}^3} D_i\mathcal{F} D_j \mathcal{F} 
	+\Phi D_i h D_j \(\g^{-\frac{n+p}{n-k}+1}\)  \nonumber\\
	&+\Phi D_j h D_i \(\g^{-\frac{n+p}{n-k}+1}\)
	+\Phi \g^{-\frac{n+p}{n-k}+1} D_j D_i h
	-\Phi \metric{D \g}{D h}\g^{-\frac{n+p}{n-k}} \delta_{ij} \nonumber\\
	&-\frac{\metric{D \g}{D \mathcal{F}}}{\g \mathcal{F}^2} \delta_{ij}
	-\frac{\cosh r}{\g \mathcal{F}}\delta_{ij}
	- \( \frac{n+p}{n-k}-1 \) P A_{ij}[\g] \nonumber \\
	&+\frac{n+p}{n-k} \( \frac{n+p}{n-k}-1 \)\frac{P \g_i \g_j}{\g}
	+\frac{n+p}{n-k}P \tilde{u}\delta_{ij}
	+\frac{P}{\g}\delta_{ij}. \label{dt-Aij}
	\end{align}	
\end{lem}
\begin{proof}
	From \eqref{scalar eq phi-HCMF} we have
	\begin{equation*}
	\frac{\partial}{\partial t} \g(z,t) = \Phi(t) \g^{-\frac{n+p}{n-k}+1}(z,t)h(z)  - \mathcal{F}^{-1}(z,t).
	\end{equation*}
	This together with \eqref{def A-phi} gives
	\begin{align}
	\frac{\partial}{\partial t} A_{ij} [\g(z,t)] =&
	(\partial_t\g(z,t))_{ij} - \frac{\metric{D \g}{D \partial_t\g(z,t)}}{\g} \delta_{ij} + \frac{1}{2} \frac{|D \g|^2}{\g^2} \partial_t\g(z,t) \delta_{ij} \nonumber\\
	&+ \frac{1}{2}\(1 + \frac{1}{\g^2} \) \partial_t\g(z,t) \delta_{ij} \nonumber\\
	=& \(\Phi \g^{-\frac{n+p}{n-k}+1} h - \frac{1}{\mathcal{F}} \)_{ij} - \frac{\metric{D \g}{D \(\Phi \g^{-\frac{n+p}{n-k}+1} h - \frac{1}{\mathcal{F}} \)} }{\g} \delta_{ij} \nonumber\\
	&+ \frac{1}{2} \frac{|D \g|^2}{\g^2} \(\Phi \g^{-\frac{n+p}{n-k}+1} h - \frac{1}{\mathcal{F}} \)\delta_{ij} 
	 + \frac{1}{2} \(1 + \frac{1}{\g^2} \) \(\Phi \g^{-\frac{n+p}{n-k}+1} h - \frac{1}{\mathcal{F}} \) \delta_{ij} \nonumber\\
	=& \Phi  h D_j D_i \(\g^{-\frac{n+p}{n-k}+1}\) + \Phi D_i  h D_j \(\g^{-\frac{n+p}{n-k}+1}\) + \Phi  D_j h D_i \(\g^{-\frac{n+p}{n-k}+1}\) \nonumber\\
	&+\Phi \g^{-\frac{n+p}{n-k}+1} D_j D_i h 
	+ \frac{D_j D_i \mathcal{F}}{\mathcal{F}^2}- 2 \frac{D_i \mathcal{F} D_j \mathcal{F}}{\mathcal{F}^3} 
	- \Phi  \g^{-\frac{n+p}{n-k}}\metric{D \g}{D  h} \delta_{ij} \nonumber\\
	&- \Phi h \frac{\metric{D \g}{D \( \g^{-\frac{n+p}{n-k}+1}\)}}{\g} \delta_{ij} - \frac{\metric{D \g}{D \mathcal{F}}}{\g \mathcal{F}^2} \delta_{ij}
	+ \frac{1}{2} \Phi h \g^{-\frac{n+p}{n-k}}\frac{|D \g|^2}{\g} \delta_{ij} \nonumber\\
	&- \frac{1}{2} \frac{1}{\g \mathcal{F}} \frac{|D \g|^2}{\g} \delta_{ij} 
	+ \frac{1}{2} \(\g + \frac{1}{\g} \) \Phi  h \g^{-\frac{n+p}{n-k}} \delta_{ij}
	- \frac{1}{2}\(\g + \frac{1}{\g} \) \frac{1}{\g \mathcal{F}} \delta_{ij}. \label{dt-Aij-1}
	\end{align}	
	 A direct calculation shows	
	\begin{align}
	\Phi  h D_j D_i \(\g^{-\frac{n+p}{n-k}+1}\)
	=& \Phi h D_j \left( \( 1-\frac{n+p}{n-k} \) \g^{-\frac{n+p}{n-k}} D_i \g \right) \nonumber\\
	=&- \( \frac{n+p}{n-k}-1 \) \Phi h \g^{-\frac{n+p}{n-k}} \g_{ij}
	+\frac{n+p}{n-k} \( \frac{n+p}{n-k}-1 \)\Phi h \g^{-\frac{n+p}{n-k}-1}\g_i \g_j \label{Phih-phi-ij}
	\end{align}
	and
	\begin{align}\label{Phih-Dphi-Dphi}
	- \Phi h \frac{\metric{D \g}{D \(\g^{-\frac{n+p}{n-k}+1}\) }  }{\g} \delta_{ij}
	=\( \frac{n+p}{n-k}-1 \)\Phi h \g^{-\frac{n+p}{n-k}} \frac{|D \g|^2}{\g} \delta_{ij}.
	\end{align}
	Since $P(z,t) =\Phi(t) h(z) \g^{-\frac{n+p}{n-k}}(z,t)$, by using \eqref{Phih-phi-ij} and \eqref{Phih-Dphi-Dphi}, we have 
	\begin{align}
	&\Phi  h D_j D_i \(\g^{-\frac{n+p}{n-k}+1}\) -\Phi h \frac{\metric{D \g}{D \(\g^{-\frac{n+p}{n-k}+1} \)}}{\g} \delta_{ij}+ \frac{1}{2} \Phi h \g^{-\frac{n+p}{n-k}} \frac{|D \g|^2}{\g}\delta_{ij} \nonumber\\
	&+\frac{1}{2} \(\g + \frac{1}{\g}\) \Phi  h \g^{-\frac{n+p}{n-k}} \delta_{ij} \nonumber\\
	=&- \(\frac{n+p}{n-k}-1 \) P \g_{ij}
	+\frac{n+p}{n-k} \( \frac{n+p}{n-k}-1 \) \frac{P \g_i \g_j}{\g}
	+\( \frac{n+p}{n-k}-\frac{1}{2} \) P \frac{|D \g|^2}{\g} \delta_{ij} \nonumber\\
	&+\frac{1}{2}P \(\g +\frac{1}{\g} \)\delta_{ij} \nonumber\\
	=& - \( \frac{n+p}{n-k}-1 \) P \left(\g_{ij}- \frac{1}{2} \frac{|D \g|^2}{\g} \delta_{ij} + \frac{1}{2} \(\g -\frac{1}{\g}\) \delta_{ij} \right)
	+\frac{n+p}{n-k} \( \frac{n+p}{n-k}-1 \) \frac{P\g_i \g_j}{\g} \nonumber\\
	&+\frac{1}{2} \frac{n+p}{n-k}P \frac{|D \g|^2}{\g} \delta_{ij} 
	+\frac{1}{2} \frac{n+p}{n-k}P \(\g-\frac{1}{\g} \) \delta_{ij}
	+ \frac{P}{\g}\delta_{ij}. \label{dt-Aij-main part}
	\end{align}
	Finally, plugging \eqref{dt-Aij-main part} into \eqref{dt-Aij-1}, we have
	\begin{align*}
	\frac{\partial}{\partial t} A_{ij} [\g] =&
	\frac{1}{\mathcal{F}^2} D_j D_i \mathcal{F} - \frac{2}{\mathcal{F}^3} D_i\mathcal{F} D_j \mathcal{F} 
	+\Phi D_i h D_j \(\g^{-\frac{n+p}{n-k}+1}\)\\
	&+\Phi D_j h D_i \(\g^{-\frac{n+p}{n-k}+1}\)
	+\Phi \g^{-\frac{n+p}{n-k}+1} D_j D_i h
	-\Phi \metric{D \g}{D h}\g^{-\frac{n+p}{n-k}} \delta_{ij}\\
	&-\frac{\metric{D \g}{D \mathcal{F}}}{\g \mathcal{F}^2} \delta_{ij}
	-\frac{\cosh r}{\g \mathcal{F}}\delta_{ij}
	- \( \frac{n+p}{n-k}-1 \) P A_{ij}[\g]\\
	&+\frac{n+p}{n-k} \( \frac{n+p}{n-k}-1 \) \frac{P\g_i \g_j}{\g}
	+\frac{n+p}{n-k} P\tilde{u}\delta_{ij}
	+\frac{P}{\g}\delta_{ij},
	\end{align*}	
	where we used \eqref{coshr} and \eqref{tilde u}. We complete the proof of Lemma \ref{lem-dt-Aij}.
\end{proof}

\begin{lem} \label{lem-dt log F-dt log phi-a}
	Along the scalar evolution equation \eqref{scalar eq phi-HCMF} of $\g(z,t)$, we have
	\begin{align}
	\frac{\partial }{\partial t} \log \mathcal{F}(z,t)
	=&\frac{\dot{\mathcal{F}}^{ij}}{\mathcal{F}^2} D_j D_i \log \mathcal{F} 
	- \frac{\dot{\mathcal{F}}^{ij}}{\mathcal{F}^2} D_i \log \mathcal{F} D_j \log \mathcal{F} 
	+2\frac{\dot{\mathcal{F}}^{ij}}{\mathcal{F}}\Phi D_i h D_j \(\g^{-\frac{n+p}{n-k}+1}\) \nonumber\\
	&+\Phi \g^{-\frac{n+p}{n-k}+1} \frac{\dot{\mathcal{F}}^{ij}}{\mathcal{F}}D_j D_i h
	-\Phi \metric{D \g}{D h}\g^{-\frac{n+p}{n-k}} \frac{ \sum_{i=1}^n\dot{\mathcal{F}}^{ii}}{\mathcal{F}}
	-\frac{\metric{D \g}{D \mathcal{F}}}{\g \mathcal{F}^3} \sum_{i=1}^n \dot{\mathcal{F}}^{ii} \nonumber\\
	&-\frac{\cosh r}{\g \mathcal{F}^2} \sum_{i=1}^n \dot{\mathcal{F}}^{ii}
	- \( \frac{n+p}{n-k}-1 \) P
	+\frac{n+p}{n-k} \( \frac{n+p}{n-k}-1 \) \frac{P}{\g\mathcal{F}}\mathcal{F}^{ij}\g_i \g_j  \nonumber\\
	&+\frac{n+p}{n-k} P\tilde{u}\frac{ \sum_{i=1}^n \dot{\mathcal{F}}^{ii}}{\mathcal{F}}
	+\frac{P}{\g}\frac{ \sum_{i=1}^n\dot{\mathcal{F}}^{ii}}{\mathcal{F}}, \label{dt-logF} 
	\end{align}
	and
	\begin{equation}\label{dt-log(phi-a)}
	\frac{\partial}{\partial t} \log (\g(z,t)-a)
	= P \frac{\g}{\g-a}- \frac{1}{(\g-a)\mathcal{F}},
	\end{equation}
	where $a$ is a constant satisfying $\g(z,t)- a>0$.
\end{lem}

\begin{proof}
		Since $\partial_t \mathcal{F}(z,t) = \dot{\mathcal{F}}^{ij} \partial_t A_{ij}[\g(z,t)]$, formulas \eqref{dt-Aij} and \eqref{mathcal F-ij A-ij} imply
	\begin{align}
	\frac{\partial }{\partial t} \log \mathcal{F}(z,t)
	=& \frac{\dot{\mathcal{F}}^{ij}}{\mathcal{F}} \partial_t A_{ij}[\g(z,t)] \nonumber\\
	=&\frac{\dot{\mathcal{F}}^{ij}}{\mathcal{F}^3} D_j D_i \mathcal{F} 
	- 2\frac{\dot{\mathcal{F}}^{ij}}{\mathcal{F}^4} D_i\mathcal{F} D_j \mathcal{F} 
	+2\frac{\dot{\mathcal{F}}^{ij}}{\mathcal{F}}\Phi D_i h D_j \(\g^{-\frac{n+p}{n-k}+1}\) \nonumber\\
	&+\Phi \g^{-\frac{n+p}{n-k}+1} \frac{\dot{\mathcal{F}}^{ij}}{\mathcal{F}}D_j D_i h
	-\Phi \metric{D \g}{D h}\g^{-\frac{n+p}{n-k}} \frac{ \sum_{i=1}^n\dot{\mathcal{F}}^{ii}}{\mathcal{F}}
	-\frac{\metric{D \g}{D \mathcal{F}}}{\g \mathcal{F}^3} \sum_{i=1}^n \dot{\mathcal{F}}^{ii} \nonumber\\
	&-\frac{\cosh r}{\g \mathcal{F}^2} \sum_{i=1}^n \dot{\mathcal{F}}^{ii}
	- \( \frac{n+p}{n-k}-1 \) P
	+\frac{n+p}{n-k} \( \frac{n+p}{n-k}-1 \) \frac{P}{\g\mathcal{F}}\mathcal{F}^{ij}\g_i \g_j  \nonumber\\
	&+\frac{n+p}{n-k} P\tilde{u}\frac{ \sum_{i=1}^n \dot{\mathcal{F}}^{ii}}{\mathcal{F}}
	+\frac{P}{\g}\frac{ \sum_{i=1}^n\dot{\mathcal{F}}^{ii}}{\mathcal{F}}. \label{dt-logF-old} 
	\end{align}
	A direct calculation yields
	\begin{align}\label{Fij-logFij}
	\frac{\dot{\mathcal{F}}^{ij}}{\mathcal{F}^3} D_j D_i \mathcal{F} 
	- 2\frac{\dot{\mathcal{F}}^{ij}}{\mathcal{F}^4} D_i\mathcal{F} D_j \mathcal{F} 
	=
	\frac{\dot{\mathcal{F}}^{ij}}{\mathcal{F}^2} D_j D_i \log \mathcal{F} 
	- \frac{\dot{\mathcal{F}}^{ij}}{\mathcal{F}^2} D_i \log \mathcal{F} D_j \log \mathcal{F}.
	\end{align}
	Inserting \eqref{Fij-logFij} into the right-hand side of \eqref{dt-logF-old}, we obtain \eqref{dt-logF}.
	
	Recall that $P(z,t) = \Phi(t) h(z) \g^{-\frac{n+p}{n-k}}(z,t)$.
	Then equation \eqref{scalar eq phi-HCMF} implies 
	\begin{equation*}
	\frac{\partial}{\partial t} \log (\g(z,t)-a)
	= \frac{\partial_t \g(z,t)}{\g(z,t)-a}
	= P \frac{\g}{\g-a}- \frac{1}{(\g-a)\mathcal{F}}.
	\end{equation*}
	This completes the proof of Lemma \ref{lem-dt log F-dt log phi-a}.
\end{proof}

\begin{lem}\label{lem-F<C}
	Suppose that one of the following assumptions of $n$, $k$, $p$, $f(z)$, and $M_t$ holds:
	\begin{itemize}
		\item $k=0$, $-\infty<p<+\infty$, $n \geq 1$, $f(z)$ is a smooth positive even function on $\mathbb{S}^n$, and $M_t$ is a smooth, origin symmetric uniformly h-convex solution to flow $\eqref{flow-HCMF}$;
		\item $1 \leq  k \leq n-1$, $p \geq -n$, $n \geq 2$, $f(z)$ is a smooth positive even function on $\mathbb{S}^n$ satisfying Assumption \ref{assump-h}, and $M_t$ is a smooth, origin symmetric uniformly h-convex solution to the flow \eqref{flow-HCMF}.
	\end{itemize}
	Then there is a positive constant $C$ depending only on $n$, $k$, $p$, $\min_{z \in \mathbb{S}^n} f(z)$, $||f(z)||_{C^2(\mathbb{S}^n)}$ and $M_0$ such that
	\begin{equation}\label{F<C}
	\mathcal{F}(z,t) \leq C, \quad \forall (z,t) \in \mathbb{S}^n \times [0,T).
	\end{equation}
\end{lem}
\begin{proof}
	At first, we explain the assumptions of $n$, $k$, $p$, $f(z)$, and $M_t$ in Lemma \ref{lem-F<C}. By Lemma \ref{lem-C0-est}, Corollary \ref{cor-C1-est} and Lemma \ref{lem-c < Phi <C}, the origin-symmetry of $M_t$ implies the $C^0$, $C^1$ estimates of $\g(z,t)$ and the uniform upper bound of $\Phi(t)$ along the flow \eqref{flow-HCMF}. If $n=1$ and $k=0$, 
	then Lemma \ref{lem-c < Phi <C} implies that $\Phi(t)$ 
	stays uniformly away from zero
	 along the flow \eqref{flow-HCMF}. If $n \geq 2$ and $0 \leq k \leq n-1$, then by Proposition \ref{prop-Pinching-est} and the assumptions in Lemma \ref{lem-F<C}, we can get the pinching estimates along the flow \eqref{flow-HCMF}. Then we get the lower bound of $\Phi(t)$ by use of Lemma \ref{lem-c < Phi <C}. Hence $P(z,t):= \Phi(t) h(z) \g^{-\frac{n+p}{n-k}}(z,t)$ has a uniform positive lower bound along the flow \eqref{flow-HCMF}. By the above arguments, under the assumptions of $n$, $k$, $p$, $f(z)$, and $M_t$ in Lemma \ref{lem-F<C}, there exists a positive constant $C$ such that
	\begin{align}
	1+\frac{1}{C} \leq& \g(z,t) \leq C, \label{C0 in F<C}\\
	\quad |D\g(z,t)| \leq& C, \label{C1 in F<C}\\
	\quad \frac{1}{C} \leq& \Phi(t) \leq C, \label{Phi in F<C}\\
    \frac{1}{C} \leq&  P(z,t) \leq C, \label{P in F<C}\\
	\quad \frac{1}{C} \leq& \frac{\tilde{\kappa}_i}{\tilde{\kappa}_j} \leq C, \quad 1 \leq i, j \leq n \label{Pinching in F<C}
	\end{align} 
   along the flow \eqref{flow-HCMF}, and $C$ only depends on $n$, $k$, $p$, $\min_{z \in \mathbb{S}^n} f(z)$, $||f(z)||_{C^2(\mathbb{S}^n)}$ and $M_0$.
	
	Now we define an auxiliary function
	\begin{equation}\label{aux func Q1}
	Q(z,t)= \log \mathcal{F}(z,t) - b\log(\g(z,t)-a),
	\end{equation}
	where $a \in [0,1]$ and $b>0$ are constants that will be chosen later in this proof. By \eqref{C0 in F<C} and the assumption $a \in [0,1]$, we have $\g-a$ is bounded between two uniform positive constants along the flow \eqref{flow-HCMF}. Hence, in order for \eqref{F<C}, we only need to show $Q(z,t) \leq C$.
    Using \eqref{shifted curvature-support function} and \eqref{Pinching in F<C}, we have
	\begin{align}
	\sum_{i=1}^n \dot{\mathcal{F}}^{ii}(z,t) =& \frac{1}{n-k} p_{n-k}^{\frac{-n+k+1}{n-k}}(A[\g]) \sum_{i=1}^n \frac{\partial p_{n-k}(A[\g])}{\partial A_{ii}[\g]} \nonumber\\
	\leq&  \min_{1 \leq i \leq n} (A_{ii}[\g])^{-n+k+1} \max_{1 \leq j \leq n} (A_{jj}[\g])^{n-k-1} \nonumber\\
	=&  \max_{1 \leq i, j \leq n} \(\frac{\tilde{\kappa}_j}{\tilde{\kappa}_i} \)^{n-k-1}(x,t)
	\leq  C.\label{F-ii leq C}
	\end{align}
	 Applying \eqref{C0 in F<C}, \eqref{C1 in F<C}, \eqref{Phi in F<C} and \eqref{F-ii leq C} to the right-hand side of \eqref{dt-logF}, we have
	\begin{align}
	\frac{\partial }{\partial t} \log \mathcal{F}(z,t) \leq&
	\frac{\dot{\mathcal{F}}^{ij}}{\mathcal{F}^2} D_j D_i \log \mathcal{F} 
	- \frac{\dot{\mathcal{F}}^{ij}}{\mathcal{F}^2} D_i \log \mathcal{F} D_j \log \mathcal{F}
	-\frac{\metric{D \g}{D \log \mathcal{F}}}{\g \mathcal{F}^2}  \sum_{i=1}^n \dot{\mathcal{F}}^{ii} \nonumber\\
	&- \( \frac{n+p}{n-k}-1 \) P+ C\frac{1}{\mathcal{F}}+ C \frac{1}{\mathcal{F}^2} \label{dt log F-ineq}
	\end{align}
	for some constant $C$.
	Inserting \eqref{dt-log(phi-a)} into \eqref{aux func Q1}, we obtain
	\begin{equation}\label{dtQ-F<C}
	\frac{\partial}{\partial t} Q(z,t) = \partial_t \log \mathcal{F} -bP \frac{\g}{\g-a}+ \frac{b}{(\g-a)\mathcal{F}}.
	\end{equation}
	For a fixed time $t_0 \in[0,T)$, we have the following critical equations at the spacial maximum point $z_0$ of $Q(\cdot,t_0)$,
	\begin{equation}\label{critical Q_1}
	D_i \log \mathcal{F} = bD_i \log(\g-a), \quad D_j D_i \log \mathcal{F} \leq bD_j D_i \log(\g-a).
	\end{equation}
	Hence, by using \eqref{critical Q_1}, \eqref{C0 in F<C}, \eqref{C1 in F<C} and \eqref{F-ii leq C}, at $(z_0,t_0)$ we have 
	\begin{align}
	&\frac{\dot{\mathcal{F}}^{ij}}{\mathcal{F}^2} D_j D_i \log \mathcal{F} 
	- \frac{\dot{\mathcal{F}}^{ij}}{\mathcal{F}^2} D_i \log \mathcal{F} D_j \log \mathcal{F} \nonumber\\ 
	\leq& \frac{\dot{\mathcal{F}}^{ij}}{\mathcal{F}^2}\left( b D_j D_i \log(\g-a) -b^2 D_i \log (\g-a) D_j \log(\g-a) \right) \nonumber\\
	=&\frac{\dot{\mathcal{F}}^{ij}}{\mathcal{F}^2} \left( \frac{b\g_{ij}}{\g-a} -\frac{(b^2+b)\g_i \g_j}{(\g-a)^2}\right) \nonumber\\
	=&\frac{b\dot{\mathcal{F}}^{ij}}{\mathcal{F}^2(\g-a)} \left( \g_{ij} -\frac{1}{2}\frac{|D\g|^2}{\g} \delta_{ij}+\frac{1}{2}\(\g-\frac{1}{\g}\) \delta_{ij} \right) \nonumber\\
	&+\frac{b\sum_{i=1}^n \dot{\mathcal{F}}^{ii}}{\mathcal{F}^2(\g-a)}\left( \frac{1}{2}\frac{|D\g|^2}{\g} - \frac{1}{2}\(\g-\frac{1}{\g}\) \right)
	-\frac{(b^2+b)\dot{\mathcal{F}}^{ij} \g_i \g_j}{\mathcal{F}^2(\g-a)^2} \nonumber\\
	=&\frac{b}{\mathcal{F}(\g-a)}+\frac{b\sum_{i=1}^n \dot{\mathcal{F}}^{ii}}{\mathcal{F}^2(\g-a)}\left( \frac{1}{2}\frac{|D\g|^2}{\g} - \frac{1}{2} \(\g-\frac{1}{\g} \) \right)
	-\frac{(b^2+b)\dot{\mathcal{F}}^{ij} \g_i \g_j}{\mathcal{F}^2(\g-a)^2} \nonumber\\
	\leq& C\frac{b}{\mathcal{F}}+ C\frac{b^2+b}{\mathcal{F}^2},\label{dt log F-first 2 terms}
	\end{align}
	where $C$ is a large positive constant depending only on $n$, $k$, $p$, $||f(z)||_{C^2(\mathbb{S}^n)}$ and $M_0$.
	Using \eqref{critical Q_1} and \eqref{sum mathcal F-ii}, at $(z_0,t_0)$ we have
	\begin{align}\label{Dphi-DlogF leq 0}
	-\frac{\metric{D \g}{D \log \mathcal{F}}}{\g \mathcal{F}^2}  \sum_{i=1}^n \dot{\mathcal{F}}^{ii}
	=-\frac{b|D \g|^2}{\g(\g-a) \mathcal{F}^2} \sum_{i=1}^n \dot{\mathcal{F}}^{ii} \leq 0.
	\end{align}
	Inserting \eqref{dt log F-first 2 terms} and \eqref{Dphi-DlogF leq 0} into \eqref{dt log F-ineq}, we have
	\begin{align*}
	\frac{\partial }{\partial t} \log \mathcal{F}(z,t) \leq&
	C\frac{b}{\mathcal{F}}+ C\frac{b^2+b}{\mathcal{F}^2}
    - \( \frac{n+p}{n-k}-1 \) P+ C\frac{1}{\mathcal{F}}+ C \frac{1}{\mathcal{F}^2}\\
    \leq&- \( \frac{n+p}{n-k}-1 \) P +C \frac{b+1}{\mathcal{F}} +C\frac{b^2+b+1}{\mathcal{F}^2}
	\end{align*}
	at $(z_0, t_0)$,
	where $C$ is a uniform positive constant. 
	This together with \eqref{dtQ-F<C} gives
	\begin{align*}
	\frac{\partial}{\partial t} Q(z,t)
	\leq 
	- \(\frac{n+p}{n-k}-1+ b\frac{\g}{\g-a} \) P + C\frac{b}{\mathcal{F}} + C\frac{b^2+b+1}{\mathcal{F}^2}
	\end{align*}
	at $(z_0, t_0)$.
	By \eqref{C0 in F<C}, we can choose $a=1$, then $\(\g-1\)^{-1} \g>c>0$ for some constant $c$ depending only on $n$, $k$ and $M_0$;
	hence, we can choose a large positive constant $b$ such that $\frac{n+p}{n-k} -1 +bc \geq 1$.
	By \eqref{P in F<C}, it is easy to see that $\partial_t Q < 0$ at $(z_0,t_0)$ for sufficiently large $\mathcal{F}$. By using the maximum principle and \eqref{C0 in F<C}, we have $Q(z,t) \leq C$ for some uniform positive constant $C$. Then we obtain \eqref{F<C}. This completes the proof of Lemma \ref{lem-F<C}.
\end{proof}

\begin{lem}\label{lem-F>C}
	Let $n \geq 1$, $0 \leq  k \leq n-1$, and $-\infty<p<+\infty$.
	Let $f(z)$ be a smooth positive even function on $\mathbb{S}^n$.
	Suppose that  $M_t = \partial \Omega_t$ is a smooth, origin symmetric uniformly h-convex solution to the flow \eqref{flow-HCMF}. 
	Then there is a  positive constant $C$ depending only on $n$, $k$, $p$, $||f(z)||_{C^2(\mathbb{S}^n)}$ and $M_0$ such that
	\begin{equation}\label{F>c}
	\mathcal{F}(z,t) \geq C^{-1}, \quad \forall (z,t) \in \mathbb{S}^n \times [0,T).
	\end{equation}
\end{lem}
\begin{proof}
	Since $M_t = \partial \Omega_t$ is origin symmetric, we have the $C^0$, $C^1$ estimates of $\g(z,t)= e^{u(z,t)}$ in \eqref{C0-est} and \eqref{C1-est}. Moreover, we know from Lemma \ref{lem-c < Phi <C} that the global term $\Phi(t)$ is bounded from above by a uniform positive constant along the flow \eqref{flow-HCMF}. Therefore, there exists a large constant $C$ depending only on $n$, $k$, $p$, $||f||_{C^2 ( \mathbb{S}^n )}$ such that
	\begin{align}
	1 + \frac{1}{C} \leq \g(z,t) \leq C, \label{C0-F>c}\\
	|D \g| \leq C, \label{C1-F>c}\\
	0<\Phi(t) \leq C, \label{Phi-F>c}\\
	P(z,t) \leq C.\label{P-F>c}
	\end{align}
	
	Now we define an auxiliary function 
	\begin{equation}\label{aux Q2}
	Q(z,t) = \log \mathcal{F}(z,t) +\log (\g(z,t)-a),
	\end{equation}
	where $a \in [0,1]$ is a positive constant that will be chosen later in this proof.
	By \eqref{C0-F>c}, in order for \eqref{F>c}, we only need to show $Q(z,t) \geq -C$ for some uniform positive constant $C$.
	For a fixed time $t_0 \in [0,T)$, at the spacial minimum point $z_0$ of $Q(\cdot,t_0)$, we have the critical equations
	\begin{equation}\label{critical eq-F>C}
	D_i \log \mathcal{F} =- D_i \log (\g-a), \quad D_j D_i \log \mathcal{F} \geq -D_jD_i \log(\g-a).
	\end{equation}
	Hence, at $(z_0,t_0)$ we have 
	\begin{align}
	&\frac{\dot{\mathcal{F}}^{ij}}{\mathcal{F}^2} D_j D_i \log \mathcal{F} 
	- \frac{\dot{\mathcal{F}}^{ij}}{\mathcal{F}^2} D_i \log \mathcal{F} D_j \log \mathcal{F} \nonumber\\
	\geq& \frac{\dot{\mathcal{F}}^{ij}}{\mathcal{F}^2}\left(-D_j D_i \log(\g-a) -D_i \log (\g-a) D_j \log(\g-a) \right) 
	=  -\frac{\dot{\mathcal{F}}^{ij}}{\mathcal{F}^2(\g-a)} \g_{ij}  \nonumber\\
	=&-\frac{\dot{\mathcal{F}}^{ij}}{\mathcal{F}^2(\g-a)}\left( \g_{ij} -\frac{1}{2}\frac{|D\g|^2}{\g} \delta_{ij}+\frac{1}{2} \(\g-\frac{1}{\g} \) \delta_{ij} \right)
	-\frac{\sum_{i=1}^n \dot{\mathcal{F}}^{ii}}{\mathcal{F}^2(\g-a)} \left( \frac{1}{2}\frac{|D\g|^2}{\g}- \frac{1}{2} \(\g-\frac{1}{\g} \)\right) \nonumber\\
	=&-\frac{1}{\mathcal{F}(\g-a)} -\frac{\sum_{i=1}^n \dot{\mathcal{F}}^{ii}}{\mathcal{F}^2(\g-a)} \left( \frac{1}{2}\frac{|D\g|^2}{\g}- \frac{1}{2} \(\g-\frac{1}{\g} \)\right).\label{F>C-high-ord-term}
	\end{align}
	Inserting \eqref{F>C-high-ord-term} into \eqref{dt-logF} and using \eqref{dt-log(phi-a)},  at $(z_0,t_0)$ we have
	\begin{align}
	\frac{\partial}{\partial t} Q(z,t)
	=&\partial_t \log \mathcal{F}(z,t) + \partial_t \log(\g(z,t)-a) \nonumber\\
	\geq& -\frac{\sum_{i=1}^n \dot{\mathcal{F}}^{ii}}{\mathcal{F}^2(\g-a)} \left( \frac{1}{2}\frac{|D\g|^2}{\g}- \frac{1}{2} \(\g-\frac{1}{\g} \)\right)
	-\frac{1}{\mathcal{F}(\g-a)}
	+2\frac{\dot{\mathcal{F}}^{ij}}{\mathcal{F}}\Phi D_i h D_j \(\g^{-\frac{n+p}{n-k}+1}\) \nonumber\\
	&+\Phi \g^{-\frac{n+p}{n-k}+1} \frac{\dot{\mathcal{F}}^{ij}}{\mathcal{F}}D_j D_i h
	-\Phi \metric{D \g}{D h}\g^{-\frac{n+p}{n-k}} \frac{ \sum_{i=1}^n\dot{\mathcal{F}}^{ii}}{\mathcal{F}}
	-\frac{\metric{D \g}{D \mathcal{F}}}{\g \mathcal{F}^3} \sum_{i=1}^n \dot{\mathcal{F}}^{ii} \nonumber\\
	&-\frac{\cosh r}{\g \mathcal{F}^2} \sum_{i=1}^n \dot{\mathcal{F}}^{ii}
	- \( \frac{n+p}{n-k}-1 \) P
	+\frac{n+p}{n-k} \( \frac{n+p}{n-k}-1 \) \frac{P}{\g\mathcal{F}}\mathcal{F}^{ij}\g_i \g_j  \nonumber\\
	&+\frac{n+p}{n-k} P\tilde{u}\frac{ \sum_{i=1}^n \dot{\mathcal{F}}^{ii}}{\mathcal{F}}
	+\frac{P}{\g}\frac{ \sum_{i=1}^n\dot{\mathcal{F}}^{ii}}{\mathcal{F}} \nonumber\\
	&+P \frac{\g}{\g-a}- \frac{1}{(\g-a)\mathcal{F}} \label{dt Q-F>C old}
	\end{align}
	Using the estimates \eqref{C0-F>c}--\eqref{P-F>c} in the right-hand side of \eqref{dt Q-F>C old}, at $(z_0,t_0)$ we have
	\begin{align}
		\frac{\partial}{\partial t} Q(z,t)
		\geq&  -\frac{\sum_{i=1}^n\dot{\mathcal{F}}^{ii}}{\mathcal{F}^2(\g-a)} \left( \frac{1}{2}\frac{|D\g|^2}{\g}- \frac{1}{2}\(\g-\frac{1}{\g} \)\right)
		-\frac{\metric{D \g}{D \log \mathcal{F}}}{\g \mathcal{F}^2} \sum_{i=1}^n \dot{\mathcal{F}}^{ii}  \nonumber\\
		&-\frac{\cosh r}{\g \mathcal{F}^2} \sum_{i=1}^n \dot{\mathcal{F}}^{ii}
		-C\frac{\sum_{i=1}^n \dot{\mathcal{F}}^{ii}}{\mathcal{F}}-\frac{C}{\mathcal{F}}- C. \label{dt Q-F>C}
	\end{align}
	Using the critical equation \eqref{critical eq-F>C} and \eqref{F-ij delt-ij geq 1}, we know
	\begin{equation}\label{Dphi,DlogF}
	-\frac{\metric{D \g}{D \log \mathcal{F}}}{\g \mathcal{F}^2} \sum_{i=1}^n \dot{\mathcal{F}}^{ii} = \frac{|D\g|^2}{\g(\g-a)} \frac{\sum_{i=1}^n\dot{\mathcal{F}}^{ii}}{\mathcal{F}^2} \geq 0.
	\end{equation}
	By \eqref{C0-F>c}, we can choose $a=1$.
   Plugging \eqref{Dphi,DlogF} into \eqref{dt Q-F>C}, and using \eqref{coshr}, we have 
   \begin{align}
   	\frac{\partial}{\partial t} Q(z,t) \geq&
   	\frac{\sum_{i=1}^n \dot{\mathcal{F}}^{ii}}{\mathcal{F}^2\g(\g-1)}\(  -\frac{1}{2}|D \g|^2+ \frac{1}{2} \g^2 -\frac{1}{2} + |D \g|^2 - \cosh r \(\g-1\)  \) \nonumber\\
   	&-C\frac{\sum_{i=1}^n \dot{\mathcal{F}}^{ii}}{\mathcal{F}}-\frac{C}{\mathcal{F}}- C \nonumber\\
   	=&\frac{\sum_{i=1}^n \dot{\mathcal{F}}^{ii}}{\mathcal{F}^2\g(\g-1)}\( \cosh r\g-1 - \cosh r \(\g-1 \)\)-C\frac{\sum_{i=1}^n \dot{\mathcal{F}}^{ii}}{\mathcal{F}}-\frac{C}{\mathcal{F}}- C\nonumber\\
   	=&\frac{\sum_{i=1}^n \dot{\mathcal{F}}^{ii}}{\mathcal{F}^2\g(\g-1)}(\cosh r-1) -C\frac{\sum_{i=1}^n \dot{\mathcal{F}}^{ii}}{\mathcal{F}}-\frac{C}{\mathcal{F}}- C \label{dtQ leq ..}
   \end{align}
   at $(z_0,t_0)$.
   Note that \eqref{coshr} and \eqref{C1-F>c} imply
   \begin{equation}\label{coshr-1>c}
   \cosh r-1 \geq  \frac{1}{2}\(\g + \frac{1}{\g}\)-1 >c>0
   \end{equation}
   for some constant $c$ depending only on $n$, $k$ and $M_0$.  Using \eqref{coshr-1>c} in the right-hand side of \eqref{dtQ leq ..}, we arrive at
   \begin{equation}\label{dtQ2 final}
   	\frac{\partial}{\partial t} Q(z,t)\geq c \frac{\sum_{i=1}^n \dot{\mathcal{F}}^{ii}}{\mathcal{F}^2} -C\frac{\sum_{i=1}^n \dot{\mathcal{F}}^{ii}}{\mathcal{F}}-\frac{C}{\mathcal{F}}- C
   \end{equation}
   at $(z_0,t_0)$,
   where  $C>0$ only depends on $n$, $k$, $p$, $||f(z)||_{C^2(\mathbb{S}^n)}$ and $M_0$. 
   Besides, we have $\sum_{i=1}^n\dot{\mathcal{F}}^{ii} \geq 1$ by \eqref{sum mathcal F-ii}.
   Hence, the estimate \eqref{dtQ2 final} implies $\partial_t Q(z,t) \geq 0$ for sufficiently small $\mathcal{F}>0$.
   We therefore prove that $Q(z,t) \geq -C$ for some constant $C$ along the flow \eqref{flow-HCMF} by using the maximum principle and \eqref{C0-F>c}. Consequently, we obtain the desired estimate \eqref{F>c}. This completes the proof of Lemma \ref{lem-F>C}. 
\end{proof}
\subsection{Long time existence and convergence}\label{Subsec-Long time existence and convergence}$ \ $

In this subsection, we give the proofs of Theorem \ref{thm-exist-all p-k=0}--\ref{thm-long time existence}.
\begin{proof}[Proof of Theorem \ref{thm-long time existence}]
	If $A_{ij}[\g]>0$, then \eqref{mathcal F-ij >0} implies
	\begin{equation*}
	\frac{\partial}{\partial \g_{pq}} \(-\mathcal{F}^{-1}(A[\g])\) =\mathcal{F}^{-2} \dot{\mathcal{F}}^{ij}>0 .
	\end{equation*}
	Hence, equation \eqref{scalar eq phi-HCMF} is parabolic at $t=0$ provided that  $A_{ij}[\g(z,0)]>0$ on $\mathbb{S}^n$.
	Consequently, as long as the initial smooth hypersurface $M_0$ is uniformly h-convex, the flow \eqref{flow-HCMF} exits for a short time. 
	Note that $M_t$ is origin symmetric along the flow \eqref{flow-HCMF}, which follows from
	 the evenness of $f(z)$ and equation \eqref{scalar eq phi-HCMF}.
	
		In order to prove the long time existence of the flow \eqref{flow-HCMF}, let us summarize the a priori estimates in Subsection \ref{subsec-a priori est}.
		In the case $n \geq 2$, by the pinching estimates in Proposition \ref{prop-Pinching-est} and the bounds  of $\mathcal{F}$ proved in Lemma \ref{lem-F<C} and Lemma \ref{lem-F>C}, we know that the shifted principal curvatures $\tilde{\kappa}_i =\kappa_i-1$ are bounded between two positive constants along the flow \eqref{flow-HCMF}. Here we remark that in the case $n=1$ and $k=0$, we have $C^{-1} \leq \tilde{\kappa} \leq C$ by Lemma \ref{lem-F<C} and Lemma \ref{lem-F>C} directly.  Combining this fact with the $C^0$ estimates in Lemma \ref{lem-C0-est} and the $C^1$ estimates in Corollary \ref{cor-C1-est}, we have
		\begin{equation*}
			|| \g(z,t) ||_{C_{z,t}^{2,1}(\mathbb{S}^n \times [0, T))} \leq C_{2,1}
		\end{equation*}
		for some constant $C_{2,1}$. Hence equation \eqref{scalar eq phi-HCMF} is uniformly parabolic on $\mathbb{S}^n \times [0, T)$. 
		By \eqref{G-ij,kl <0}, $\mathcal{F}(A[\g]) = p_{n-k}^{\frac{1}{n-k}}(A[\g])$ is concave for $A[\g]>0$,  then
		\begin{equation*}
		\frac{\partial^2}{\partial \g_{pq} \partial \g_{rs}}  \( -\mathcal{F}^{-1}(A[\g]) \) B_{pq}B_{rs}
		=\mathcal{F}^{-2} \ddot{\mathcal{F}}^{pq,rs}B_{pq}B_{rs} - 2\mathcal{F}^{-3} \dot{\mathcal{F}}^{pq}B_{pq} \dot{\mathcal{F}}^{rs}B_{rs}  \leq 0
		\end{equation*} 
		for $B_{pq} \in {\rm Sym}(n)$.
		We conclude that $- \mathcal{F}^{-1}$ is concave with respect to $\g_{pq}$.
		 In Lemma \ref{lem-c < Phi <C}, we proved that $\Phi(t)$ is bounded between two uniform positive constants. 
		 
		 By the above estimates, we can apply the a priori estimates in \cite{TW13} to equation \eqref{scalar eq phi-HCMF} and get the $C^{2+\alpha,1+\frac{\alpha}{2}}$ estimate of $\g(z,t)$, see also \cite{And04, BIS20}. Then by the parabolic Schauder theory, for each nonnegative integers $i$, $j$, there exists a uniform constant $C_{i,j}$ (independent of $T$) such that
		 \begin{equation}\label{higher deriv-phi}
		 || \g(z,t) ||_{C_{z,t}^{i,j}(\mathbb{S}^n \times [0, T))} \leq C_{i,j}.
		 \end{equation}
		 Consequently, we obtain the long time existence of the flow \eqref{flow-HCMF} by a standard continuation argument.
\end{proof}

\begin{proof}[Proofs of Theorem \ref{thm-exist-all p-k=0}--\ref{thm-exist-all p}]
	Given any smooth even and uniformly h-convex initial hypersurface $M_0$, we have $ J_p'(\Omega_t) \leq 0$ along the flow \eqref{flow-HCMF} by the estimate in \eqref{mono quan-ineq}. Besides, the $C^0$ estimate in Lemma \ref{lem-C0-est} yields that $|J_p(\Omega_t)|$ is uniformly bounded from above along the flow \eqref{flow-HCMF}. 
	From the estimates of higher order derivatives \eqref{higher deriv-phi}, we have that the both $|\frac{\partial}{\partial t} \g(z,t)|$ and $|\frac{\partial^2}{\partial t^2}\g(z,t)|$ are bounded from above by some constant $C$ along the flow. Hence  $J_p'(\Omega_t)$ is uniformly continuous and 
	\begin{align*}
	\lim_{t \to +\infty} J_p'(\Omega_t) =0.
	\end{align*}
	By the Arzel\`a-Ascoli theorem and a diagonal argument, there exists a time sequence $\lbrace t_{i} \rbrace \subset (0, +\infty)$ such that $\lbrace M_{t_i}\rbrace$ converge smoothly to a smooth, origin symmetric uniformly h-convex hypersurface $M = \partial \Omega$ as $i \to +\infty$.  Let $u(z) = \log \g(z)$ be the horospherical support function of $\Omega$.
	As 
	\begin{align*}
	\lim_{i \to +\infty} J_p'(\Omega_{t_i})=0 ,
	\end{align*}
	we have 
	\begin{equation*}
	J_p'(\Omega) = 0.
	\end{equation*}
	Then by Lemma \ref{lem-mono-quantitirs}, there exists a positive constant $\gamma$ such that $\g(z)$ satisfies
	\begin{align}\label{eq-limit-hyper}
	\gamma = f^{-1}(z) \g^{-p-n}(z)p_{n-k}(\tilde{\lambda}(x)).
	\end{align}
	Then we complete the proofs of Theorem \ref{thm-exist-all p-k=0} and Theorem \ref{thm-exist-all p} by using \eqref{shifted curvature-support function}. 
\end{proof}

\section{Uniqueness of solutions to Horospherical $p$-Minkowski problem and Horospherical $p$-Christoffel-Minkowski problem with constant function}
\label{sec-uniqueness of horo-p-Chri-Min-constant function}
\subsection{Results in Section \ref{sec-uniqueness of horo-p-Chri-Min-constant function}}
$ \ $

The main result in Section \ref{sec-uniqueness of horo-p-Chri-Min-constant function} is the following Proposition \ref{prop-uniq-sphere-higher dim}. 
\begin{prop}\label{prop-uniq-sphere-higher dim}
	Let $n \geq 1$ and $0 \leq k < l \leq n$ be integers. Assume that  $\chi(s)$ is a smooth positive and monotone non-decreasing function defined on $\mathbb{R}^+$.
	Let $\Omega$  be a smooth uniformly h-convex bounded domain in $\mathbb{H}^{n+1}$. If $n\geq 2$, $k=0$, and $1 \leq l \leq n$, then we assume in addition that $\Omega$ contains the origin in its interior. If $M = \partial \Omega$ satisfies the following equation
	\begin{equation}\label{eq-uniq-sphere}
	\frac{p_l (\tilde{\kappa})}{p_k (\tilde{\kappa})} = \chi(\psi(x)),  
	\end{equation}
	where $ \psi(x) =\cosh r -\tilde{u}$ for $x \in M$,
	then $M$ must be a geodesic sphere. In particular, if the function $\chi$ is strictly increasing, then $M$ is a geodesic sphere centered at the origin.
\end{prop}  
As an application of Proposition \ref{prop-uniq-sphere-higher dim}, we consider the case that $n \geq 1$, $p \geq -n$ and $f(z)$ is a positive constant function in equations \eqref{eq-lim-hyp-k=0} and \eqref{eq-limit-hypersurface}. In particular, we completely classify the solutions to these equations that satisfy $\g(z)>1$ and $A_{ij}[\g(z)] >0$. 

\begin{thm}\label{thm-f=c-CM-M-sphere}
	Let $n \geq 1$ and $0 \leq k \leq n-1$ be integers, let $p \geq -n$ be a real number, and let $\gamma$ be a positive constant. For $p>n-2k$, we let 
	\begin{equation*}
	\gamma_0 = \frac{(2k+p-n)^{\frac{2k+p-n}{2}} (n-k)^{n-k}}{(n+p)^\frac{n+p}{2}}.
	\end{equation*}
Consider equation
	\begin{equation}\label{f=c-sphere-CM-M}
	\g(z)^{-p-k} p_{n-k}\(A[\g(z)]\) =\gamma,
	\end{equation}
	where $\g (z)$ is required to satisfy $A_{ij} [\g(z)] >0$ for all $z \in \mathbb{S}^n$. Geometrically, the requirement of $\g(z)$ means that $\log \g(z)$ is the horospherical support function of a smooth uniformly h-convex bounded domain $\Omega$. If $n\geq 2$, $k=0$, and $p>-n$, then we require in addition that $\g(z)>1$ for any $z \in \mathbb{S}^n$.
	\begin{enumerate}
		\item If $p>n-2k$ and $\gamma \in (0, \gamma_0)$,
		then the solutions to equation \eqref{f=c-sphere-CM-M} are $\g_1(z) =c_1$ and $\g_2(z) =c_2$, where $c_1$ and $c_2$ are solutions to equation
		\begin{equation}\label{eq-c gamma}
		c^{-p-k} \( \frac{1}{2} (c-c^{-1}) \)^{n-k} = \gamma.
		\end{equation}
		Geometrically, the domains $\Omega_1$ and $\Omega_2$ are geodesic balls of radius $\log c_1$ and $\log c_2$ centered at the origin, respectively.
		\item If $p>n-2k$ and $\gamma = \gamma_0$,
		then 
		\begin{equation*}
		\g(z) = \( \frac{n+p}{2k+p-n} \)^{\frac{1}{2}} 
		\end{equation*}
		is the unique solution to equation \eqref{f=c-sphere-CM-M}, which means that $\Omega$ is a geodesic ball of radius $\frac{1}{2} \log \frac{n+p}{2k+p-n}$ centered at the origin. 
		\item If $p>n-2k$ and $\gamma > \gamma_0$, then equation \eqref{f=c-sphere-CM-M} admits no solution.
		\item If $p = n-2k$ and $\gamma \in (0, 2^{k-n})$, then there exists a unique solution $\g(z) =c$ to equation \eqref{f=c-sphere-CM-M}, where $c$ solves equation \eqref{eq-c gamma}. Equivalently, $\Omega$ is a geodesic ball of radius $\log c$ centered at the origin.
		\item If $p=n-2k$ and $\gamma \geq 2^{k-n}$, then equation \eqref{f=c-sphere-CM-M} admits no solution.
		\item If $-n<p< n-2k$, then there exists a unique solution $\g(z) =c$ to \eqref{f=c-sphere-CM-M} for each $\gamma \in (0, +\infty)$, where $c$ solves equation \eqref{eq-c gamma}. Equivalently, $\Omega$ is a geodesic ball of radius $\log c$ centered at the origin.
		\item If $p=-n$, then the solutions to \eqref{f=c-sphere-CM-M} are given by
		\begin{equation*}
		\g(z) = \(1+2 \gamma^{\frac{1}{n-k}}\)^{\frac{1}{2}}\( (|x|^2+1)^{\frac{1}{2}} - \metric{x}{z}\),
		\end{equation*}
		where $x \in \mathbb{R}^{n+1}$. Geometrically, $\Omega$ is a geodesic ball of radius $\frac{1}{2}\log \(1+ 2\gamma^{\frac{1}{n-k}}\)$ centered at $(x, (|x|^2+1)) \in \mathbb{H}^{n+1}$.
	\end{enumerate}
\end{thm}
By using Theorem \ref{thm-f=c-CM-M-sphere}, we can prove the convergence of the flow \eqref{flow-HCMF} for constant $f(z)$. 
\begin{thm}\label{thm-f=c-convergence}
	Let $n \geq 1$ and $0 \leq k \leq n-1$ be integers, let $p \geq -n$ be a real number,  let $f(z)$ be a positive constant function defined on $\mathbb{S}^n$, and let $M_0 = \partial \Omega_0$ be a smooth, origin symmetric and uniformly h-convex hypersurface in $\mathbb{H}^{n+1}$. Then the flow \eqref{flow-HCMF} converges smoothly to a 
	geodesic sphere $\partial B(r)$ centered at the origin satisfying $\widetilde{W}_k(B(r)) = \widetilde{W}_k(\Omega_0)$.
\end{thm}

\begin{rem}
	In the case that $n \geq 2$, $0 \leq k \leq n-1$, $p=-n$, and $f(z)$ is a positive constant, the long time existence and the exponential convergence of the flow \eqref{flow-HCMF} was studied by using Alexandrov reflection argument in \cite[Theorem 1.7]{ACW18}. 
\end{rem}

\subsection{Heintze-Karcher inequality and integral formulas}$ \ $

\begin{prop}\label{prop-HK-n=1}
	Let $\Omega$ be a smooth uniformly h-convex bounded domain in $\mathbb{H}^2$. Then
	\begin{align}\label{Heintze-Karcher-n=1}
	\int_{\partial \Omega} \(\frac{\cosh r- \tilde{u}}{\tilde{\kappa}} - \tilde{u}\) d\mu \geq 0.
	\end{align}
	Equality holds if and only if $\Omega$ is a geodesic ball.
\end{prop}

\begin{proof}
	Using \eqref{1/phi, coshr-u}, \eqref{shifted curvature-support function}, \eqref{tilde u}  and  \eqref{rel-area element}, we have
	\begin{align}
	\frac{\cosh r - \tilde{u}}{\tilde{\kappa}} =& \frac{1}{\g \tilde{\kappa}} = A[\g]
	=\g'' - \frac{1}{2} \frac{\(\g' \)^2 }{\g} + \frac{1}{2} \(\g-\frac{1}{\g}\)  ,  \label{HK-formu}\\
	\tilde{u} =& \frac{1}{2} \frac{\(\g'\)^2}{\g} + \frac{1}{2} \( \g- \frac{1}{\g} \),
	\quad 	d\mu = A[\g] d\sigma. \nonumber
	\end{align}
	Then 
	\begin{align}
	\int_{\partial \Omega} \(\frac{\cosh r- \tilde{u}}{\tilde{\kappa}} - \tilde{u}\) d\mu
	=& \int_{\mathbb{S}^1} \( A[\g] - \tilde{u}\) A[\g] d\sigma \nonumber\\
	=& \int_{\mathbb{S}^1} \( \g''-  \frac{\(\g'\)^2}{\g}\)
	\( \g''- \frac{1}{2}\frac{\(\g'\)^2}{\g} +\frac{1}{2} \(\g- \frac{1}{\g}\)  \) d\sigma. \label{HK-sphere}
	\end{align}
	By integration by parts, we have
	\begin{align*}
	\int_{\mathbb{S}^1} \frac{\g'' \(\g' \)^2}{\g} d\sigma
	=&- \int_{\mathbb{S}^1}  \( \frac{\( \g'\)^2 }{\g} \)' \g' d\sigma
	= -\int_{\mathbb{S}^1} \(2\frac{\g'' \g'}{\g}- \frac{ \(\g'\)^3 }{\g^2}  \) \g' d\sigma \nonumber\\
	=& -2 \int_{\mathbb{S}^1} \frac{\g'' \(\g'\)^2 }{\g} d\sigma + \int_{\mathbb{S}^1} \frac{ \(\g'\)^4 }{\g^2} d\sigma,  
	\end{align*}
	which implies
	\begin{align*}
	\int_{\mathbb{S}^1} \(- \frac{3}{2} \frac{\g'' \(\g'\)^2}{\g} + \frac{1}{2} \frac{ \(\g'\)^4 }{\g^2} \) d\sigma = 0. 
	\end{align*}
	Then we have
	\begin{align}
	&\int_{\mathbb{S}^1} \( \g'' - \frac{\(\g' \)^2}{\g} \) \(\g'' - \frac{1}{2} \frac{ \(\g'\)^2 }{\g}  \) d\sigma \nonumber\\
	=& \int_{\mathbb{S}^1} \(  \(\g''\)^2 - \frac{3}{2} \frac{\g''\(\g'\)^2 }{\g} + \frac{1}{2} \frac{\(\g'\)^4 }{\g^2}  \) d\sigma = \int_{\mathbb{S}^1} \( \g'' \)^2 d\sigma. \label{HK-part-1}
	\end{align}
	Using integration by parts, we have
	\begin{align}
	\int_{\mathbb{S}^1} \( \g'' -  \frac{ \(\g'\)^2 }{\g}   \)\( \frac{1}{2} \(\g- \frac{1}{\g}\) \) d\sigma
	=& \int_{\mathbb{S}^1} \( \log \g \)'' \( \frac{1}{2} \(\g^2- 1\) \) d\sigma \nonumber\\
	=& - \int_{\mathbb{S}^1} \( \log \g \)' \g' \g d\sigma
	= - \int_{\mathbb{S}^1} \( \g' \)^2 d\sigma. \label{HK-part-2}
	\end{align}
	Inserting \eqref{HK-part-1} and \eqref{HK-part-2} into \eqref{HK-sphere} and using the Wirtinger inequality, we have
	\begin{align*}
	\int_{\partial \Omega} \(\frac{\cosh r- \tilde{u}}{\tilde{\kappa}} - \tilde{u}\) d\mu
	= \int_{\mathbb{S}^1} (\g'')^2 d\sigma - \int_{\mathbb{S}^1} \( \g' \)^2 d\sigma \geq 0.
	\end{align*}
	Hence we obtain\eqref{Heintze-Karcher-n=1}.
	Besides, we have that equality holds in \eqref{Heintze-Karcher-n=1} only if $\g(z) = \alpha \sin t + \beta \cos t +\gamma$ for some constants $\a$, $\b$ and $\gamma$, where we set $z = (\cos t, \sin t)$ and $t \in [0, 2\pi]$. In this case, we know
	$\(\g'\)^2 + \(\g- \gamma\)^2 = \alpha^2 + \beta^2$ and $\g'' = -\g +\gamma$. Combining these facts with \eqref{HK-formu} yields
	\begin{align*}
	\frac{1}{\tilde{\kappa}} 
	=& \g A[\g]
	=\g'' \g - \frac{1}{2} \(\g'\)^2 + \frac{1}{2} \(\g^2-1\)
	=-\g^2 + \gamma\g - \frac{1}{2} \(\g'\)^2 + \frac{1}{2} \g^2- \frac{1}{2}\\
	=&-\frac{1}{2} \( \(\g'\)^2 + \( \g- \gamma \)^2 - \gamma^2 +1\) 
	=\frac{1}{2} \( \gamma^2 - \alpha^2 -\beta^2-1\).
	\end{align*}
	Hence equality holds in \eqref{Heintze-Karcher-n=1} only if $\Omega$ is a geodesic ball. Conversely, if $\Omega$ is a geodesic ball, then $\tilde{\kappa}$ is constant on $\partial \Omega$, and hence
	\begin{align*}
	\int_{\partial \Omega} \(\frac{\cosh r- \tilde{u}}{\tilde{\kappa}} - \tilde{u}\) d\mu
	= \frac{1}{\tilde{\kappa}} \int_M  \( \( \cosh r - \tilde{u}\) - \tilde{u} \tilde{\kappa}\) d\mu=0,
	\end{align*}
	where we used \eqref{eq-shifted Minkowski formula}. Therefore, equality holds in \eqref{Heintze-Karcher-n=1} if and only if $\Omega$ is a geodesic ball in $\mathbb{H}^2$. We complete the proof of Proposition \ref{prop-HK-n=1}.
\end{proof}

\begin{lem}\label{lem-integral formulas}
	Let $M$ be an $n$-dimensional $(n \geq 1)$ hypersurface in $\mathbb{H}^{n+1}$ and $\tilde{f}$ be a smooth function defined on $M$. For $0 \leq k \leq n-1$, we have
	\begin{equation}\label{generalize-Mink-formula-1}
	\int_{M} \(\( \cosh r-\tilde{u} \)p_k(\tilde{\kappa}) - \tilde{u} p_{k+1}(\tilde{\kappa}) \)\tilde{f} d\mu
	=- \frac{1}{k+1}\int_M \dot{p}_{k+1}^{ij}(\tilde{\kappa}) \nabla_i \cosh r \nabla_j \tilde{f} d\mu.
	\end{equation}
	For $1 \leq k ,\ l \leq n$, we have
	\begin{align}
	&\int_{M} \( \cosh r- \tilde{u}\) \tilde{f} (p_k(\tilde{\kappa}) p_{l-1} (\tilde{\kappa})- p_{k-1}(\tilde{\kappa})p_l(\tilde{\kappa}) )d\mu  \nonumber\\
	=&-\frac{1}{l}\int_M \dot{p}_{l}^{ij}(\tilde\kappa) \nabla_i \cosh r \nabla_j \( \tilde{f} p_k(\tilde{\kappa}) \) d\mu
	+\frac{1}{k}\int_M \dot{p}_k^{ij}(\tilde{\kappa}) \nabla_i \cosh r \nabla_j \( \tilde{f}p_l(\tilde{\kappa}) \) d\mu. \label{generalize-Mink-formula-2}
	\end{align}
\end{lem}
\begin{proof}
	In the case $0 \leq k \leq n-1$, by using \eqref{shifted-Min-fml-diverg-formula} and integration by parts, we have
	\begin{align*}
	&\int_{M} \(\( \cosh r-\tilde{u} \)p_k(\tilde{\kappa}) - \tilde{u} p_{k+1}(\tilde{\kappa}) \)\tilde{f} d\mu\\
	=& \frac{1}{k+1} \int_M \(\dot{p}^{ij}_{k+1}(\tilde{\kappa}) \nabla_j  \nabla_i \cosh r\) \tilde{f} d\mu
	=- \frac{1}{k+1} \int_M \dot{p}^{ij}_{k+1}(\tilde{\kappa}) \nabla_i \cosh r \nabla_j \tilde{f}  d\mu,
	\end{align*}
	where we also used the fact that $\dot{p}^{ij}_{k+1} (\tilde{\kappa})$ is divergence-free.
	Hence we obtain the integral formula \eqref{generalize-Mink-formula-1}.
	
	Similarly, for $1 \leq k ,\ l \leq n$, we have
	\begin{align}
	&-\int_{M} \(\( \cosh r-\tilde{u} \)p_{k-1}(\tilde{\kappa}) - \tilde{u} p_{k}(\tilde{\kappa}) \)\tilde{f}p_l (\tilde\kappa) d\mu \nonumber\\
	=& -\frac{1}{k} \int_M \(\dot{p}_k^{ij} (\tilde{\kappa}) \nabla_j \nabla_i \cosh r\)  \tilde{f}p_l (\tilde\kappa)  d\mu
	= \frac{1}{k} \int_M \dot{p}^{ij}_{k}(\tilde{\kappa}) \nabla_i \cosh r \nabla_j \(\tilde{f} p_l(\tilde{\kappa})\)  d\mu, \label{integrb-parts-1}
	\end{align}
	and
	\begin{align}
	&\int_{M} \(\( \cosh r-\tilde{u} \)p_{l-1}(\tilde{\kappa}) - \tilde{u} p_{l}(\tilde{\kappa}) \)\tilde{f}p_k (\tilde\kappa) d\mu  \nonumber\\
	=& \frac{1}{l} \int_M \(\dot{p}^{ij}_{l}(\tilde{\kappa}) \nabla_j \nabla_i \cosh r\)\tilde{f} p_k(\tilde{\kappa}) d\mu
	=- \frac{1}{l} \int_M \dot{p}^{ij}_{l}(\tilde{\kappa}) \nabla_i \cosh r \nabla_j \(\tilde{f} p_k(\tilde{\kappa})\)  d\mu. \label{integrb-parts-2}
	\end{align}
	Then formula \eqref{generalize-Mink-formula-2} follows from the sum of \eqref{integrb-parts-1} and \eqref{integrb-parts-2}. We complete the proof of Lemma \ref{lem-integral formulas}.
\end{proof}

\subsection{Proofs of the results}$ \ $

\begin{proof}[Proof of Proposition \ref{prop-uniq-sphere-higher dim}]
	The following proof is divided into four cases. 
	
	\textbf{Case 1.} $n =1$, $k=0$, and $l=1$. From Proposition \ref{prop-HK-n=1} and \eqref{eq-uniq-sphere}, we have
	\begin{align}\label{n=1,k=0,l=1-ineq 1}
	0 \leq  \int_M \frac{ \(\cosh r- \tilde{u}\)- \tilde{u}\tilde{\kappa} }{\tilde{\kappa}} d\mu
	=\int_M \frac{ \(\cosh r- \tilde{u}\)- \tilde{u}\tilde{\kappa} }{\chi} d\mu,
	\end{align}
	with equality if and only if $M$ is a geodesic circle.
	Taking $\tilde{f} = \frac{1}{\chi}$ in \eqref{generalize-Mink-formula-1} and using \eqref{1/phi i new}, we obtain 
	\begin{align}
	\int_M \frac{ \(\cosh r- \tilde{u}\)- \tilde{u}\tilde{\kappa} }{\chi} d\mu
	=\frac{1}{n} \int_M \frac{\chi'}{\chi^2} g^{ij} \nabla_i \cosh r \nabla_j \( \cosh r- \tilde{u} \) \nonumber\\
	=-\frac{1}{n} \int_M \frac{\chi'}{\chi^2} \tilde{h}^{ij} \metric{V}{\partial_i X} \metric{V}{\partial_j X} \leq 0 .  \label{n=1,k=0,l=1-ineq 2}
	\end{align}
	Combining \eqref{n=1,k=0,l=1-ineq 1} with \eqref{n=1,k=0,l=1-ineq 2}, we have $M$ is a geodesic circle, and $\chi$ is constant on $M$. If $\chi'>0$ on $\mathbb{R}^+$, then  \eqref{1/phi, coshr-u} and \eqref{eq-uniq-sphere} imply that $\cosh r- \tilde{u} = \frac{1}{\g}$ is constant, which means the horospherical support function of $M$ is constant. Hence $M$ is a geodesic circle centered at the origin.
	
	By Corollary \ref{cor-star-shape}, if $\Omega$ contains the origin in its interior, then $M= \partial \Omega$ is star-shaped and hence $\tilde{u}>0$ on $M$. This assumption will be used in the following Case 2 and Case 3.
			
	\textbf{Case 2.} $n \geq 2$, $k=0$, and $l=1$. By \eqref{McLau ineq}, \eqref{eq-shifted Minkowski formula} and the assumption $\tilde{u}>0$, we have
	\begin{equation*}
	\int_M \( (\cosh r- \tilde{u})- \tilde{u}p_1(\tilde{\kappa}) \)p_1(\tilde{\kappa}) d\mu
	\leq\int_M \( (\cosh r- \tilde{u})p_1(\tilde{\kappa}) - \tilde{u}p_2(\tilde{\kappa}) \) d\mu
	=0.
	\end{equation*}
	Equality holds if and only if $M$ is totally umbilical, which means $M$ is a geodesic sphere, and then $\chi$ is constant on $M$ by \eqref{eq-uniq-sphere}.
	Taking $\tilde{f} = p_1(\tilde{\kappa}) = \chi$ in \eqref{generalize-Mink-formula-1} and using \eqref{1/phi i new}, we have
	\begin{align*}
	\int_{M} \( (\cosh r- \tilde{u})- \tilde{u} p_1(\tilde{\kappa}) \)p_1(\tilde{\kappa})d\mu
	=&
	-\frac{1}{n}\int_M g^{ij}\nabla_i \cosh r \nabla_j \chi d\mu\\
	=&
	\frac{1}{n} \int_M \chi' \tilde{h}^{ij} \metric{V}{\partial_i X} \metric{V}{\partial_j X} d\mu \geq 0.
	\end{align*}
	Hence we have that $M$ is a geodesic sphere in Case 2. 
	Particularly, if $\chi' >0$ on $\mathbb{R}^+$, then $\cosh r- \tilde{u} =\frac{1}{\g}$ is constant, and hence $M$ is a geodesic sphere centered at the origin.

	\textbf{Case 3.} $n \geq 2$, $k=0$, and $2 \leq l \leq n$. Using \eqref{McLau ineq}, \eqref{eq-shifted Minkowski formula} and the assumption $\tilde{u}>0$, we have
	\begin{equation}\label{case2-l geq 2-A}
	\int_M \( \(\cosh r-\tilde{u}\)- \tilde{u}p_1(\tilde{\kappa})  \) p_l^{\frac{l-1}{l}}(\tilde{\kappa}) d\mu
	\leq \int_M \(\cosh r-\tilde{u} \)p_{l-1}(\tilde{\kappa}) - \tilde{u} p_l(\tilde{\kappa}) d\mu =0.
	\end{equation}
	Equality holds if and only if  $M$ is a geodesic sphere.
	Taking $\tilde{f} = p_l^{\frac{l-1}{l}}(\tilde{\kappa}) = \chi^{\frac{l-1}{l}}$ in \eqref{generalize-Mink-formula-1} and using Lemma \ref{lem-1/phi i, 1/phi ij}, we have
	\begin{align}
	&-\int_M \( \(\cosh r-\tilde{u}\)- \tilde{u}p_1(\tilde{\kappa})  \) \chi^{\frac{l-1}{l}} d\mu \nonumber\\
	=& \frac{1}{n}\int_{M} g^{ij} \nabla_i \cosh r\nabla_j \chi^{\frac{l-1}{l}} d\mu
	= \frac{l-1}{ln} \int_M \chi^{-\frac{1}{l}} \chi' \nabla_i \cosh r \nabla_j \( \cosh r-\tilde{u} \)d\mu \nonumber\\
	=&-\frac{l-1}{ln} \int_{M} \chi^{-\frac{1}{l}} \chi' \tilde{h}^{ij} \metric{V}{\partial_i X} \metric{V}{\partial_j X} d\mu  \leq 0.\label{case2-1 geq 2-B}
	\end{align}
	Combining \eqref{case2-l geq 2-A} with \eqref{case2-1 geq 2-B}, we have that $M$ is a geodesic sphere. Similar to Case 1, if $\chi'>0$, then $M$ is a geodesic sphere centered at the origin.
	
	\textbf{Case 4.} $n \geq 2$ and $1 \leq k <l \leq n$. Again, we have by use of \eqref{Newton ineq}
	\begin{equation*}
	p_k(\tilde{\kappa})p_{l-1}(\tilde{\kappa})- p_{k-1}(\tilde{\kappa}) p_l (\tilde{\kappa}) \geq 0,
	\end{equation*}
	with equality if and only if $\tilde{\kappa}_1=\tilde{\kappa}_2=\cdots= \tilde{\kappa}_n$. 
	Taking $\tilde{f} = p_k^{-1}(\tilde{\kappa}) = p_l^{-1}(\tilde{\kappa}) \chi$ in \eqref{generalize-Mink-formula-2} and using Lemma \ref{lem-1/phi i, 1/phi ij}, we have
	\begin{align*}
	0 \leq& \int_M \( \cosh r- \tilde{u}\) p_k^{-1}(\tilde{\kappa}) \(  (p_k(\tilde{\kappa}) p_{l-1} (\tilde{\kappa})- p_{k-1}(\tilde{\kappa})p_l(\tilde{\kappa}) ) \) d\mu\\
	=&
	\frac{1}{k} \int_M \dot{p}_{k}^{ij}(\tilde\kappa) \nabla_i \cosh r \nabla_j \( \frac{p_l (\tilde{\kappa})}{p_k (\tilde{\kappa})} \) d\mu
	= \frac{1}{k} \int_M \dot{p}_{k}^{ij}(\tilde\kappa) \nabla_i \cosh r \nabla_j \chi d\mu\\
	=& - \frac{1}{k} \int_M  \chi'\dot{p}_{k}^{ij}(\tilde\kappa) \tilde{h}_j{}^s \metric{V}{\partial_i X} \metric{V}{\partial_s X}d\mu \leq 0.
	\end{align*}
	Thus, $M$ is a geodesic sphere. In the case $\chi'>0$, we have that $M$ is centered at the origin.
	We complete the proof of Proposition \ref{prop-uniq-sphere-higher dim}.
\end{proof}

\begin{proof}[Proof of Theorem \ref{thm-f=c-CM-M-sphere}]
	Assume that $\g(z)$ is a solution to equation \eqref{f=c-sphere-CM-M} that satisfies the requirements in Theorem \ref{thm-f=c-CM-M-sphere}. By Corollary \ref{cor-support-construct-domain}, there exists a smooth uniformly h-convex bounded domain $\Omega$ with horospherical support function $\log \g(z)$. 
	By  \eqref{shifted curvature-support function} and \eqref{1/phi, coshr-u}, we have 
	\begin{equation}\label{pn/pk =const}
	\frac{p_n (\tilde{\kappa})}{p_k (\tilde{\kappa})} = \frac{1}{\gamma} (\cosh r-\tilde{u})^{n+p}
	\end{equation}
	on $\partial \Omega$.
	
	\textbf{Case 1.} $p>-n$.
	Taking $l=n$ and $\chi (s) = \frac{1}{\gamma} s^{n+p}$ in Proposition \ref{prop-uniq-sphere-higher dim}, we have 
	that $\Omega$ must be a geodesic ball centered at the origin. Then we can assume that $\g(z) = c>1$ for all $z \in \mathbb{S}^n$, and hence the radius of $\Omega$ is $\log c$. By the definition of $A[\g]$ in \eqref{def A-phi}, we have that $c$ is the solution to equation 
	\eqref{eq-c gamma}. Hence we only need to count the number of solutions to equation \eqref{eq-c gamma}.
	
	Define a function $\zeta(t)$ on $(1, +\infty)$ by
	\begin{equation*}
	\zeta(t) = t^{-p-k} \( \frac{1}{2} (t-t^{-1}) \)^{n-k}.
	\end{equation*}
	Then
	\begin{align*}
	\zeta'(t) =& -(p+k) t^{-p-k-1} \cdot 2^{k-n}(t- t^{-1})^{n-k} + t^{-p-k} \cdot (n-k) 2^{k-n}(t-t^{-1})^{n-k-1}(1+t^{-2})\\
	=&2^{k-n}t^{-p-k-2}(t-t^{-1}) \( -(2k+p-n) t^2+ n+p\).
	\end{align*}
	
	If $p>n-2k$, then
	 $\zeta(t)$ attains its maximum at $t_0= \(\frac{n+p}{2k+p-n}\)^{\frac{1}{2}}>1$, and $\zeta(t_0) = \gamma_0$.
	Furthermore, it is easy to see that $\zeta (1) =\zeta (+\infty) = 0$ provided that $p>n-2k$. Therefore, when $p>n-2k$, if $\gamma> \zeta(t_0)$, then \eqref{eq-c gamma} has no solution; if $\gamma = \zeta(t_0)$, then $c=t_0$ is the unique solution to \eqref{eq-c gamma}; if $0 < \gamma < \zeta(t_0)$, then \eqref{eq-c gamma} has exactly two different solutions. Thus the Cases (1)--(3) of Theorem \ref{thm-f=c-CM-M-sphere} are proved.
	
	If $p=n-2k$, then $\zeta(t)$ is strictly increasing on $(1,+\infty)$. Furthermore, it is easy to see that $\zeta(1) =0$ and $\zeta(+\infty)  =  2^{k-n}$. Then Case (4) and Case (5) of Theorem \ref{thm-f=c-CM-M-sphere} follows directly from the above argument. 
	 
	 If $-n<p<n-2k$, then $\zeta(t)$ is strictly increasing on $(1,+\infty)$ and $\zeta(+\infty) = +\infty$. Then Case (6) of Theorem \ref{thm-f=c-CM-M-sphere} follows in the same manner as above.
	 
	 \textbf{Case 2.} $p=-n$. By equation \eqref{pn/pk =const} and Proposition \ref{prop-uniq-sphere-higher dim}, we can assume that $\Omega = B(X, r)$, which is the geodesic ball of radius $r$ centered at $X$ in $\mathbb{H}^{n+1}$. It is direct to see that the shifted principal  radii of curvature are all equal to $\sinh r e^r$ on $\partial \Omega$. Hence 
	 \begin{equation*}
	 \sinh r  e^r= \gamma^{\frac{1}{n-k}}.
	 \end{equation*}
	 Consequently, we have $e^{r} = \(1+ 2\gamma^{\frac{1}{n-k}}\)^{\frac{1}{2}}$. This together with Lemma \ref{lem-horo supp of geodesic ball} implies the Case (7) of Theorem \ref{thm-f=c-CM-M-sphere}.

	Then we complete the proof of Theorem \ref{thm-f=c-CM-M-sphere}.
\end{proof}

\begin{proof}[Proof of Theorem \ref{thm-f=c-convergence}]
	With the assumptions of $n$, $k$ and $p$ in Theorem \ref{thm-f=c-convergence}, the long time existence of the flow \eqref{flow-HCMF} was proved in Theorem \ref{thm-long time existence}. The subsequential convergence of the flow \eqref{flow-HCMF} follows from the proofs of Theorem \ref{thm-exist-all p-k=0} and Theorem \ref{thm-exist-all p}.
	By \eqref{eq-limit-hyper}, each limiting  hypersurface $M$ satisfies $\g^{-p-k} p_{n-k}(A[\g]) = \gamma f(z)$, which is equivalent to
	\begin{equation*}
	\frac{p_n (\tilde{\kappa})}{p_k (\tilde{\kappa})} = \frac{1}{\gamma f(z)} \( \cosh r- \tilde{u}\)^{p+n}.
	\end{equation*}
	Since $f(z)$ is constant and $p \geq -n$, we obtain that $M$ must be a geodesic sphere $\partial B(r)$ by Theorem \ref{thm-f=c-CM-M-sphere}. Besides, formula \eqref{mono quan-eq} shows  that $\widetilde{W}_k(B(r)) = \widetilde{W}_k(\Omega_0)$, which implies the uniqueness of $M$. Hence we get the smooth convergence of the flow \eqref{flow-HCMF}.  
\end{proof}

\section{Kazdan-Warner type obstructions for Horospherical Minkowski problem and Horospherical Christoffel-Minkowski problem}\label{sec-Kazdan-Warner}
Letting $p=-n$, equation \eqref{eq-p-CM problem} becomes
\begin{equation}\label{eq-p=-n-CM problem}
\g^{n-k} p_{n-k} \(A[\g]\) =f(z).
\end{equation} 
By \eqref{shifted curvature-support function}, if $A[\g(z)]>0$ on $\mathbb{S}^n$, then equation \eqref{eq-p=-n-CM problem} is equivalent to
\begin{equation}\label{eq-p=-n-CM curvature problem}
p_{n-k} \(\tilde{\lambda} \) = f(z),
\end{equation}
where $\tilde{\lambda}$ are the shifted principal  radii of curvature. 
We call Problem \ref{prob-Horospherical p-Christoffel-Minkowski problem} the \emph{horospherical Christoffel problem} in the case  $p=-n$ and $k=n-1$.
By taking $k=n-1$ in both equation \eqref{eq-p=-n-CM problem} and equation \eqref{eq-p=-n-CM curvature problem}, we have that the equation for the horospherical Christoffel problem is 
\begin{align}
\g p_1(A[\g]) =& f(z),\label{eq-hyperbolic christoffel problem-support}
\end{align}
that is
\begin{align}
\frac{1}{n} \sum_{i=1}^n \frac{1}{\kappa_i-1} =& f(z). \label{eq-hyperbolic christoffel problem-curvature}
\end{align}
Similarly, we call Problem \ref{prob-Horospherical p-Minkowski problem} the \emph{horospherical Minkowski problem} in the case  $p=-n$ and $k=0$.
By taking $k=0$ in both equation \eqref{eq-p=-n-CM problem} and equation \eqref{eq-p=-n-CM curvature problem}, the equation for the horospherical Minkowski problem is 
\begin{equation}
\g^n p_n \(A[\g]\) = f(z), \label{eq-hyperbolic Minkowski problem-support} 
\end{equation}
that is
\begin{equation}
\prod_{i=1}^n \(\kappa_i -1 \) = f^{-1}(z).\label{eq-hyperbolic Minkowski problem-curvature}
\end{equation}
We call Problem \ref{prob-Horospherical p-Christoffel-Minkowski problem} the \emph{horospherical Christoffel-Minkowski problem} in the case $p=-n$ and $1 \leq k \leq n-1$.
Then \eqref{eq-p=-n-CM problem} (or \eqref{eq-p=-n-CM curvature problem}) is the equation for this problem.

In the following theorem we prove a Kazdan-Warner type obstruction for the horospherical Minkowski problem and the horospherical Christoffel-Minkowski problem.
\begin{thm}\label{thm-necess-cond-p=-n}
	Let $n \geq 2$  and $0 \leq k \leq n-1$ be integers. Assume that $\g(z)$ is a smooth uniformly h-convex solution to \eqref{eq-p=-n-CM problem}. Then 
	\begin{equation}\label{eq-necess-cond-p=-n-coordinate}
	\int_{\mathbb{S}^n} \g^{-n} \metric{D f(z)}{D x_i} d \sigma=0, \quad i=0,1,\ldots,n,
	\end{equation}
	where $x_0, x_1, \ldots, x_n$ are the coordinate functions of $\mathbb{S}^n \subset \mathbb{R}^{n+1}$. More generally, $f(z)$ satisfies
	\begin{equation}\label{eq-necess-cond-p=-n-conformal Killing}
	\int_{\mathbb{S}^n} \g^{-n} \overline{V}(f(z)) d \sigma =0
	\end{equation}
	for any conformal vector field $\overline{V}$ on the unit sphere $\mathbb{S}^n$.
\end{thm}

\begin{rem}$ \ $
	\begin{itemize}
		\item If $f(z)$ is monotone with respect to a coordinate $x_i$, then \eqref{eq-necess-cond-p=-n-coordinate} shows that there is no smooth uniformly h-convex solution to equation \eqref{eq-p=-n-CM problem}.
		\item If $f(z)$ and $\g(z)$ are even functions on $\mathbb{S}^n$, then they satisfy condition \eqref{eq-necess-cond-p=-n-coordinate}. Hence, there is no contradiction between the  existence result in Theorem \ref{thm-exist-all p-k=0} and the necessary condition in Theorem \ref{thm-necess-cond-p=-n}.
	\end{itemize}
\end{rem}

The remainder of Section \ref{sec-Kazdan-Warner} is organized as follows.
\begin{enumerate}
	\item In Subsection \ref{Subsec-p=-n-necess-cond}, we will give the proof of Theorem \ref{thm-necess-cond-p=-n}. Let $\widetilde{g} = \g(z)^{-2}g_0$ be a conformal metric on $\mathbb{S}^n$ and $M$ be a smooth uniformly h-convex hypersurface in $\mathbb{H}^{n+1}$ with horospherical support function $u(z) = \log \g(z)$.
	Motivated by the results in \cite{EGM09}, we show the relationship between the Schouten tensor of metric $\widetilde{g}$  and the shifted principal  radii of curvature on $M$. We remark that another proof of this relationship will be given in Section \ref{sec-appendix}.
	The  $\sigma_k$-Nirenberg problem studies prescribed $k$-th elementary symmetric polynomial of the eigenvalues of the Schouten tensor.
	 Then we prove Theorem \ref{thm-necess-cond-p=-n} by using several Kazdan-Warner type obstructions proved in \cite{Han06, LLL21, Via00} for the $\sigma_k$-Nirenberg problem.
	\item In Subsection \ref{subsec-direct pf of necc cond-p=-n}, we will give a direct proof of \eqref{eq-necess-cond-p=-n-coordinate}. The proof is self-contained and is also valid for the case $n=1$. In Lemma \ref{lem-KZ-non-shifted}, we will prove a sequence of identities for smooth closed hypersurfaces in $\mathbb{H}^{n+1}$. Then we prove \eqref{eq-necess-cond-p=-n-coordinate} by integration by parts.
\end{enumerate}

\subsection{Proof of Theorem \ref{thm-necess-cond-p=-n}}\label{Subsec-p=-n-necess-cond}$ \ $

Given a Riemannian manifold $(M^n, g)$ with $n \geq 3$, the Schouten tensor of metric $g$ is defined by
\begin{equation}\label{def-Schouten tensor}
S_g:= \frac{1}{n-2}\left(  \Ric_g - \frac{R_g}{2(n-1)}g \right),
\end{equation}
where $\Ric_g$ is the Ricci curvature tensor and $R_g$  is the scalar curvature of metric $g$. 

\begin{lem}\label{lem-2ff-M tilde}
	Let $M$ be a smooth closed uniformly h-convex hypersurface in $\mathbb{H}^{n+1}$ $(n \geq 1)$.
	Let $Y(z) =(X-\nu)(z)$, $\widetilde{M} = Y(\mathbb{S}^n)$. Fix a local normal coordinate system around $z \in \mathbb{S}^n$ such that $A_{ij}[\g(z)]$ is diagonal at $z$. Let $\uppercase\expandafter{\romannumeral2}(U,V):= (\widetilde{\nabla}_U V)^\perp$ denote the projection of $\widetilde{\nabla}_U V$ on the normal bundle of $\widetilde{M} \subset \mathbb{R}^{n+1,1}$.
	Then 
	\begin{equation}\label{2ff-M tilde}
	\uppercase\expandafter{\romannumeral2} (\partial_i Y, \partial_j Y) = \frac{1}{\g^2}(-\tilde{\lambda}_i X + (1+ \tilde{\lambda}_i)\nu  ) \delta_{ij}.
	\end{equation}
\end{lem}

\begin{proof}
	Since
	\begin{equation*}
	\metric{X}{X}=-1, \metric{X}{\nu}=0, \metric{\nu}{\nu}=1,
	\end{equation*}
	we have
	\begin{align*}
	\metric{\partial_i Y}{X} =& \metric{\partial_i (X-\nu)}{X}
	=\metric{\nu}{\partial_i X}=0,\\
	\metric{\partial_i Y}{\nu} =& \metric{\partial_i (X-\nu)}{\nu}
	= \metric{\partial_i X}{\nu}=0.
	\end{align*}
	Then we know that vectors $X$ and $\nu$ span the normal bundle of $\widetilde{M} \subset \mathbb{R}^{n+1,1}$.
	From the definition of 	$\uppercase\expandafter{\romannumeral2} (\partial_i Y, \partial_j Y)$ and \eqref{X-nu}, we have
	\begin{align}
	\uppercase\expandafter{\romannumeral2} (\partial_i Y, \partial_j Y)
	=& -\metric{\partial_j \partial_i (X-\nu)}{X} X 
	+\metric{\partial_j \partial_i (X-\nu)}{\nu} \nu \nonumber\\
	=&\metric{\partial_i (X-\nu)}{\partial_j X} X
	-\metric{\partial_i (X-\nu)}{\partial_j \nu} \nu \nonumber\\
	=& \frac{1}{\g}\metric{(e_i,0)}{\partial_j X} X
	+\metric{\partial_i (X-\nu)}{\partial_j (X-\nu) - \partial_j X}\nu  \nonumber\\
	=&-\frac{1}{\g} A_{ij}[\g] X +\(\frac{1}{\g^2} \delta_{ij} +\frac{1}{\g} A_{ij}[\g] \)\nu,\label{2-diY-djY}
	\end{align}
	where we used \eqref{di X} and \eqref{di X-nu} in the last equality. Since $A_{ij} [\g]$ is diagonal at $z$, by using \eqref{shifted curvature-support function} and \eqref{gij-A-phi}, we have
	\begin{equation}\label{Aij-thij}
	A_{ij} [\g]= \frac{1}{\g} \tilde{\lambda}_i \delta_{ij}.
	\end{equation}
	Therefore, we obtain Lemma \ref{lem-2ff-M tilde} by substituting \eqref{Aij-thij} into \eqref{2-diY-djY}. 
\end{proof}

\begin{lem}\label{lem-conf-Schouten}
	The induced Riemannian metric on $\widetilde{M}\subset \mathbb{R}^{n+1,1}$ is $\tilde{g} =\frac{1}{\g^2} g_0$, where $g_0$ is the canonical metric on the unit sphere $\mathbb{S}^n$. Then we have the eigenvalues of $\tilde{g}^{-1} \circ \widetilde{S}$ are $\lbrace \tilde{\lambda}_i + \frac{1}{2}\rbrace$, $i=1,\ldots,n$, where $\widetilde{S}$ is the Schouten tensor of $\tilde{g}$ and $\tilde{\lambda}_i$ are the shifted principal  radii of curvature of $M$ at $X(z)$.
\end{lem}

\begin{proof}
	By \eqref{di X-nu}, we have $\metric{\partial_i Y}{\partial_j Y} = \frac{1}{\g^2} \delta_{ij}$, which means $\tilde{g}= \frac{1}{\g^2} g_0$ is the induced metric on $\widetilde{M}$.
	Now we set a local orthonormal frame $\lbrace v_1, \ldots, v_n \rbrace$ for $\widetilde{M}$ such that $v_i := \g \partial_i (X-\nu)$, $i=1, \ldots, n$. By the Gauss equation of $\widetilde{M} \subset \mathbb{R}^{n+1,1}$ and \eqref{2ff-M tilde}, we have
	\begin{align*}
	\widetilde{{\rm Riem}} (v_i, v_j, v_k, v_l)
	=& \metric{\uppercase\expandafter{\romannumeral2} (v_i, v_k)}{\uppercase\expandafter{\romannumeral2}(v_j,v_l)}
	-\metric{\uppercase\expandafter{\romannumeral2}(v_i,v_l)}{\uppercase\expandafter{\romannumeral2}(v_j,v_k)}\\
	=& (\delta_{ik} \delta_{jl} - \delta_{il} \delta_{jk})
	\metric{-\tilde{\lambda}_i X +(1+ \tilde{\lambda}_i) \nu}{-\tilde{\lambda}_j X +(1+ \tilde{\lambda}_j) \nu}\\
	=&(\delta_{ik}\delta_{jl}- \delta_{il} \delta_{jk})(1+ \tilde{\lambda}_i + \tilde{\lambda}_j),
	\end{align*}
	where $\widetilde{{\rm Riem}}$ is the Riemannian curvature tensor of $\widetilde{M}$.
	Then the Ricci curvature of metric $\tilde{g}$ is
	\begin{equation}\label{Ric-g-conf}
	\widetilde{\Ric}(v_i, v_k) = \left( (n-1)+(n-2) \tilde{\lambda}_i + \sum_{j=1}^n \tilde{\lambda}_j \right) \delta_{ik},
	\end{equation}
	and the scalar curvature of $\tilde{g}$ is
	\begin{equation}\label{scal-g-conf}
	\widetilde{R} = n(n-1) +2(n-1) \sum_{i=1}^n \tilde{\lambda}_i.
	\end{equation}
	Therefore the matrix $\(\widetilde{S}(v_i, v_j)\)$ is diagonal. Substituting \eqref{Ric-g-conf} and \eqref{scal-g-conf} into the definition of $\widetilde{S}$ in \eqref{def-Schouten tensor}, we know
	\begin{equation*}
	\begin{aligned}
	\widetilde{S}(v_i,v_i) =& \frac{1}{n-2} \left( \widetilde{\Ric}(v_i,v_i) -\frac{\widetilde{R}}{2(n-1)} \right)\\
	=& \frac{1}{n-2} \left( n-1 +(n-2) \tilde{\lambda}_i +\sum_{j=1}^n \tilde{\lambda}_j - \frac{n(n-1) +2(n-1) \sum_{j=1}^n \tilde{\lambda}_j}{2(n-1)}  \right)\\
	=&\tilde{\lambda}_i + \frac{1}{2}.
	\end{aligned}
	\end{equation*}
	This completes the proof of Lemma \ref{lem-conf-Schouten}.
\end{proof}

Using Lemma \ref{lem-conf-Schouten} and \eqref{scal-g-conf}, we can get the following Corollary \ref{cor-rel-shif-prin-radii-Schouton}.
\begin{cor}\label{cor-rel-shif-prin-radii-Schouton}
	When $n \geq 3$, $0 \leq k \leq n-1$, we have
	\begin{align}\label{rel-shif-prin-radii-Schouton-1}
	p_{n-k}(\tilde{\lambda}) = 	p_{n-k}\(\tilde{g}^{-1} \circ \widetilde{S} - \frac{1}{2} I \)
	=
	\sum_{i=0}^{n-k} (-1)^{n-k-i} \frac{C_{n-k}^i}{2^{n-k-i}} p_i \(\tilde{g}^{-1} \circ \widetilde{S} \).
	\end{align}
	When $n=2$, we have
	\begin{align}\label{rel-shif-prin-radii-Schouton-2}
	p_1(\tilde{\lambda})= \frac{1}{2}(\tilde{\lambda}_1 + \tilde{\lambda}_2) = \frac{1}{4} \widetilde{R} - \frac{1}{2}.
	\end{align}
\end{cor}

In \cite{Han06, Via00}, Han, Viaclovsky proved the following Theorem A.
\begin{thmA}[\cite{Han06, Via00}]
	Let $(M^n, g)$ be an $n$-dimensional $(n \geq 3)$ compact Riemannian manifold, $\sigma_k (g^{-1} \circ S_g)$ be the $k$-th $(1 \leq k \leq n)$ elementary symmetric function of the eigenvalues of the Schouten tensor with respect to metric $g$. Let $\overline{V}$ be a conformal Killing vector field on $(M^n, g)$. When $k \geq 3$, also assume that $(M^n,g)$ is locally conformally flat. Then
	\begin{equation*}
	\int_M \metric{\overline{V}}{\nabla \sigma_k (g^{-1} \circ S_g)} d \vol_g =0.
	\end{equation*}
\end{thmA}
\begin{rem}
	When $n \geq 3$ and $k=1$, Theorem A degenerates to the classical Kazdan-Warner identity, i.e.
	\begin{equation*}
	\int_M \metric{\overline{V}}{ \nabla R_g} d \vol_g =0.
	\end{equation*}
\end{rem}

Recently, Li, Lu, and Lu \cite{LLL21} considered the $\sigma_2$-Nirenberg problem and obtained the following Kazdan-Warner type identity on $\mathbb{S}^2$. Here, we state their result by using the notations in this paper.
\begin{thmB}[\cite{LLL21}]
	Let $\overline{V}$ be a conformal Killing vector field on $(\mathbb{S}^2, g_0)$, and let $\tilde{g} = \frac{1}{\g^2} g_0$ be a conformal metric to $g_0$ on $\mathbb{S}^2$, where $\g(z)$ is a smooth, positive function on $\mathbb{S}^2$. Then
	\begin{equation*}
	\int_{\mathbb{S}^2} \g^{-2}\overline{V} \(\sigma_2 \(\g A[\g] + \frac{1}{2}I \) \) d\sigma=0.
	\end{equation*} 
\end{thmB}
By using \eqref{shifted curvature-support function}, we remark that Theorem B is equivalent to
\begin{align}\label{LLL-new}
\int_{\mathbb{S}^2} \g^{-2} \overline{V} \(\sigma_2 \(\tilde{\lambda}+\frac{1}{2} \) \) d\sigma=0.
\end{align}

Now we can give the proof of Theorem \ref{thm-necess-cond-p=-n}.
\begin{proof}[Proof of Theorem \ref{thm-necess-cond-p=-n}]
	Assume that $\tilde{g} = \frac{1}{\g^2(z)} g_0$, where $g_0$ is the canonical metric on the unit sphere $\mathbb{S}^n$.
	Then we have $d \vol_{\tilde{g}} = \g^{-n} d\sigma$, where $d\sigma$ is the volume element of $\mathbb{S}^n$.
	Taking $p =-n$ in \eqref{eq-p-CM problem-with-lambda}, we have
	\begin{equation}\label{eq-CM-p=-n}
	f(z) = p_{n-k}(\tilde{\lambda}).
	\end{equation}
	The proof of \eqref{eq-necess-cond-p=-n-conformal Killing} is divided into three cases.
		
	\textbf{Case 1.}  $n \geq 3$ and $0 \leq k\leq n-1$.
	Using \eqref{rel-shif-prin-radii-Schouton-1} in \eqref{eq-CM-p=-n}, we have by Theorem A
	\begin{align*}
	\int_{\mathbb{S}^n} \g^{-n} \overline{V}(f(z)) d\sigma
	=&\int_{\mathbb{S}^n} \g^{-n} \overline{V}(  p_{n-k}(\tilde{\lambda}) ) d\sigma\\
	=&\sum_{i=0}^{n-k} (-1)^{n-k-i} \frac{C_{n-k}^i}{2^{n-k-i}} \int_{\mathbb{S}^n} \g^{-n} \overline{V} \(p_i (\tilde{g}^{-1} \circ\widetilde{S})\) d\sigma =0
	\end{align*}
	for all $0 \leq k \leq n$, where $\overline{V}$ is any conformal Killing vector field on $\mathbb{S}^n$.
	
	\textbf{Case 2.} $n=2$ and $k=1$.
	Using \eqref{rel-shif-prin-radii-Schouton-2} in \eqref{eq-CM-p=-n}, we have by the Kazdan-Warner identity
	\begin{equation*}
	\int_{\mathbb{S}^2} \g^{-2} \overline{V}(f(z)) d\sigma
	= \int_{\mathbb{S}^2} \g^{-2} \overline{V}\( p_1 (\tilde{\lambda}) \) d\sigma
	= \frac{1}{4} \int_{\mathbb{S}^2} \g^{-2} \overline{V}(\widetilde{R}) d\sigma=0.
	\end{equation*}
	
	\textbf{Case 3.} $n=2$ and $k=0$.
	By using \eqref{eq-CM-p=-n}, \eqref{LLL-new} and the Kazdan-Warner identity, we have
	\begin{align*}
	\int_{\mathbb{S}^2} \g^{-2}\overline{V}(f(z)) d\sigma
	=&\int_{\mathbb{S}^2} \g^{-2}\overline{V}\(  p_2 (\tilde{\lambda} ) \)d\sigma\\
	=& \int_{\mathbb{S}^2}\g^{-2}\overline{V}\( p_2 \(\tilde{\lambda} + \frac{1}{2} \) -p_1(\tilde{\lambda}) -\frac{1}{4}\) d\sigma\\
	=&\int_{\mathbb{S}^2} \g^{-2} \(\overline{V}\( p_2 \( \tilde{\lambda} +\frac{1}{2} \)\)-\frac{1}{4} \overline{V}( \widetilde{R})  \)d\sigma=0.
	\end{align*}
	Thus we obtain \eqref{eq-necess-cond-p=-n-conformal Killing}.
	
	In particular, we can choose the conformal Killing vector field $\overline{V}$ in \eqref{eq-necess-cond-p=-n-conformal Killing} as $D x_i$, $ 0 \leq i \leq n$, where $x_0, x_1,\ldots, x_n$ are the coordinate functions of $\mathbb{S}^n \subset \mathbb{R}^{n+1}$. Therefore we obtain \eqref{eq-necess-cond-p=-n-coordinate}. This completes the proof of Theorem \ref{thm-necess-cond-p=-n}.
\end{proof}

\subsection{Direct proof of \eqref{eq-necess-cond-p=-n-coordinate}}\label{subsec-direct pf of necc cond-p=-n}$ \ $

In Theorem \ref{thm-necess-cond-p=-n}, we gave Kazdan-Warner type obstructions to existence of solutions to equation \eqref{eq-p=-n-CM problem}. The proof in Subsection \ref{Subsec-p=-n-necess-cond} is intrinsic, and it is based on several theorems in \cite{Han06, LLL21, Via00}. In this subsection, we will give a geometric proof of the necessary condition \eqref{eq-necess-cond-p=-n-coordinate}. The following Lemma \ref{lem-KZ-non-shifted} is crucial.

\begin{lem}\label{lem-KZ-non-shifted}
	Let $n \geq 1$ and $M=\partial \Omega$ be a smooth closed hypersurface in $\mathbb{H}^{n+1} \subset \mathbb{R}^{n+1,1}$. Then
	\begin{align}\label{KZ-non-shifted}
	\int_M (\cosh r \nu- \tilde{u}X) p_k(\kappa) d\mu =0, \quad \forall \ 0 \leq k \leq n.
	\end{align}
\end{lem}

\begin{proof}
	Let $\vec{a}$ be any fixed constant vector in $\mathbb{R}^{n+1,1}$. We divide the proof of Lemma \ref{lem-KZ-non-shifted} into two cases.
	
	\textbf{Case 1.} $n \geq 1$ and $1 \leq k \leq n$. By using \eqref{Xij}, we have
	\begin{align}\label{X,a-ij}
	\nabla_j \nabla_i \metric{X}{\vec{a}}
	=
	\nabla_j \metric{\partial_i X}{\vec{a}}
	=
	\metric{-{h}_{ij} \nu + g_{ij} X}{\vec{a}}
	=
	-{h}_{ij} \metric{\nu}{\vec{a}}+ g_{ij} \metric{X}{\vec{a}}.
	\end{align}
	By \eqref{X,a-ij} and  \eqref{pm-ij h-ij}, we have
	\begin{align}
	\int_{M} \cosh r \dot{p}_k^{ij}(\kappa) \nabla_j \nabla_i \metric{X}{\vec{a}} d\mu
	=&
	\int_{M}  \cosh r \dot{p}_k^{ij}(\kappa) \( -{h}_{ij} \metric{\nu}{\vec{a}}+ g_{ij} \metric{X}{\vec{a}} \) d\mu  \nonumber\\
	=&\int_{M} \(-k\cosh r \metric{\nu}{\vec{a}} p_k(\kappa)+  \cosh r \metric{X}{\vec{a}} \dot{p}_k^{ij}(\kappa) g_{ij}\)d\mu.\label{KZ-identity-1 leq k-1}
	\end{align}
	Since $\dot{p}_k^{ij}$ is divergence-free, by using integration by parts, \eqref{di cosh r} and \eqref{conf-vf}, we also have
	\begin{align}
	\int_{M} \cosh r \dot{p}_k^{ij}(\kappa) \nabla_j \nabla_i \metric{X}{\vec{a}} d\mu
	=&\int_{M} \metric{X}{\vec{a}} \dot{p}_k^{ij}(\kappa) \nabla_j \nabla_i \cosh r d\mu  \nonumber\\
	=&\int_{M} \metric{X}{\vec{a}} \dot{p}_k^{ij}(\kappa) \nabla_j \metric{V}{\partial_i X} d\mu  \nonumber\\
	=&  \int_{M}   \metric{X}{\vec{a}}\dot{p}_k^{ij}(\kappa) \(-h_{ij} \tilde{u} + \cosh r g_{ij}  \) d\mu \nonumber\\
	=&	\int_{M} \(-k\tilde{u} \metric{X}{\vec{a}} p_k(\kappa)+  \cosh r \metric{X}{\vec{a}} \dot{p}_k^{ij}(\kappa) g_{ij}\)d\mu.\label{KZ-identity-1 leq k-2}
	\end{align}
	Comparing \eqref{KZ-identity-1 leq k-1} with \eqref{KZ-identity-1 leq k-2}, we have
	\begin{align}\label{KZ-identity-k geq 1}
	\int_M \metric{\cosh r \nu -\tilde{u}X}{\vec{a}}p_k(\kappa) d\mu =0,
	\end{align}
	for all $1 \leq k \leq n$ and $\vec{a} \in \mathbb{R}^{n+1,1}$.
	
	\textbf{Case 2.} $n \geq 1$ and $k=0$. By using \eqref{conf-vf} and \eqref{div-conf-vf}, it is easy to get
	\begin{equation}\label{three-div-formula}
	\begin{aligned}
	\divv_{\mathbb{H}^{n+1}} ( \cosh r \vec{a}) =& \metric{V}{\vec{a}}, \quad
	\divv_{\mathbb{H}^{n+1}}\( \cosh r \metric{X}{\vec{a}} X \) = n \cosh r \metric{X}{\vec{a}},\\
	\divv_{\mathbb{H}^{n+1}} \( \metric{X}{\vec{a}} V\) =& n \cosh r \metric{X}{\vec{a}} + \metric{V}{\vec{a}}.
	\end{aligned}
	\end{equation}
	Since $X$ is the normal vector field of the inclusion $\mathbb{H}^{n+1} \subset \mathbb{R}^{n+1,1}$, we know that the projection of $\vec{a}$ on the tangent bundle $T\mathbb{H}^{n+1}$ is $\vec{a} + \metric{X}{\vec{a}} X$, which implies that $\cosh r \vec{a} +\cosh r \metric{X}{\vec{a}}X$ is a tangential vector field on $\mathbb{H}^{n+1}$.
	Hence, by using $\tilde{u} : = \metric{V}{\nu}$, $\metric{X}{\nu} = 0$, \eqref{three-div-formula} and the divergence theorem, we have
	\begin{align}
	\int_{M} \metric{\cosh r \nu - \tilde{u} X}{\vec{a}} d\mu
	=&
	\int_{M} \metric{\cosh r \vec{a} - \metric{X}{\vec{a}} V}{\nu} d\mu  \nonumber\\
	=&
	\int_{M} \metric{\cosh r \vec{a} + \cosh r \metric{X}{\vec{a}} X - \metric{X}{\vec{a}} V}{\nu} d\mu  \nonumber\\
	=& \int_{\Omega} \divv_{\mathbb{H}^{n+1}} ( \cosh r \vec{a}) + \divv_{\mathbb{H}^{n+1}}\( \cosh r \metric{X}{\vec{a}} X \) - \divv_{\mathbb{H}^{n+1}} \( \metric{X}{\vec{a}} V\) dv \nonumber \\
	=&
	\int_{\Omega}  \( \metric{V}{\vec{a}} + n\cosh r \metric{X}{\vec{a}}- \( n \cosh r \metric{X}{\vec{a}} + \metric{V}{\vec{a}} \)\) dv =0.\label{KZ-identity-k=0}
	\end{align}
	Since the vector $\vec{a}$ is arbitrary, Lemma \ref{lem-KZ-non-shifted} follows from \eqref{KZ-identity-k geq 1} and \eqref{KZ-identity-k=0}.
\end{proof}

\begin{cor}
	Let $M$ be a smooth uniformly h-convex hypersurface in $\mathbb{H}^{n+1}$ $(n \geq 1)$ and $u(z)$ be the horospherical support function of $M$. We denote $\g(z) = e^{u(z)}$. Then 
	\begin{align}\label{KZ-shifted-version}
	\int_M (\cosh r \nu- \tilde{u} X)p_k(\tilde{\kappa}) d\mu = 0, \quad \forall \ 0 \leq k \leq n,
	\end{align}
	and
	\begin{align}\label{KZ-identity-sphere version}
	\int_{\mathbb{S}^n} \( \frac{D \g}{\g}(z) +z \) \g^{-k}(z)p_{n-k}(A[\g(z)])) d\sigma =0, \quad \forall \ 0 \leq k \leq n.
	\end{align}
\end{cor}

\begin{proof}
	By choosing suitable linear combinations, it is easy to see that \eqref{KZ-shifted-version} follows from \eqref{KZ-non-shifted}. Formula \eqref{coshr nu-uX} states 
	\begin{align}\label{cosh r nu- u X}
	\cosh r \nu- \tilde{u} X = \(- \frac{D\g}{\g}-z,0 \).
	\end{align}
	Formulas \eqref{rel-area element} and \eqref{shifted curvature-support function} imply
	\begin{align}\label{pk-dmu-d-sigma}
	p_k(\tilde{\kappa}) d\mu
	=\frac{p_{n-k}}{p_n} (\tilde{\lambda}) \g^{-n} p_n(\tilde{\lambda}) d\sigma
	=\g^{-n}p_{n-k} (\tilde{\lambda}) d\sigma
	= \g^{-k} p_{n-k}(A[\g]) d\sigma.
	\end{align}
	Then  we get \eqref{KZ-identity-sphere version} by substituting \eqref{cosh r nu- u X} and \eqref{pk-dmu-d-sigma} into \eqref{KZ-shifted-version}.
\end{proof}

Now we can give a direct proof of the obstruction \eqref{eq-necess-cond-p=-n-coordinate}.
\begin{proof}[Direct proof of \eqref{eq-necess-cond-p=-n-coordinate}]
	Let $\vec{a}$ be any constant vector in $\mathbb{R}^{n+1}$. We denote its projection on the tangent bundle of $\mathbb{S}^n$ as $\vec{a}^T = \vec{a} - \metric{z}{\vec{a}} z$.  Recall the equation \eqref{eq-p=-n-CM problem},
	\begin{align*}
	\g^{n-k}(z)p_{n-k}(A[\g(z)]) = f(z).
	\end{align*}
	By integration by parts, we have 
	\begin{align}
	\int_{\mathbb{S}^n} \g^{-n} \metric{Df}{\vec{a}} d\sigma
	=& \int_{\mathbb{S}^n} \g^{-n} \metric{D \(\g^{n-k}p_{n-k}(A[\g])\)}{\vec{a}}d\sigma  \nonumber\\
	=&\int_{\mathbb{S}^n} \g^{-n}\metric{D \(\g^{n-k}\)}{\vec{a}}  p_{n-k} (A[\g]) d\sigma
	+\int_{\mathbb{S}^n} \g^{-k}\metric{D p_{n-k}(A[\g])}{\vec{a}} d\sigma   \nonumber\\
	=&\int_{\mathbb{S}^n} \( (n-k)\g^{-k-1} \metric{D\g}{\vec{a}} -  \divv_{\mathbb{S}^n} (\g^{-k} \vec{a}^T)\) p_{n-k}(A[\g])d\sigma.  \label{new pf-p=-n-eq-1}
	\end{align}
	A direct calculation yields 
	\begin{align}
	\divv_{\mathbb{S}^n} (\g^{-k} \vec{a}^T)
	=&  \metric{D\(\g^{-k}\)}{\vec{a}} + \g^{-k}\divv_{\mathbb{S}^n} (\vec{a} - \metric{z}{\vec{a}} z) \nonumber\\
	=& -k\g^{-k-1} \metric{D\g}{\vec{a}} - n\g^{-k}\metric{z}{\vec{a}}.
	\label{div-phi^-k-a}
	\end{align} 
	Inserting \eqref{div-phi^-k-a} into the right-hand side of \eqref{new pf-p=-n-eq-1} and using \eqref{KZ-identity-sphere version}, we have
	\begin{align*}
	\int_{\mathbb{S}^n} \g^{-n} \metric{Df}{\vec{a}} d\sigma
	=n \int_{\mathbb{S}^n} \metric{\frac{D\g}{\g}+z}{\vec{a}}\g^{-k}p_{n-k}(A[\g]) d\sigma=0.
	\end{align*}
	 Since $\vec{a}\in \mathbb{R}^{n+1}$ is arbitrary, we complete the proof of \eqref{eq-necess-cond-p=-n-coordinate}.  
\end{proof}

\begin{rem}
	From the above proof, we know that \eqref{eq-necess-cond-p=-n-coordinate} holds for all $n \geq 1$ in Theorem \ref{thm-necess-cond-p=-n}.
\end{rem}

\section{Horospherical $p$-Brunn-Minkowski inequality and Horospherical $p$-Minkowski inequalities}\label{sec: BM inequalities}

In this section, we propose a conjecture of horospherical $p$-Brunn-Minkowski inequalities in Conjecture \ref{conj-BM} and two conjectures of horospherical $p$-Minkowski inequalities in Conjecture \ref{conj-Min ineq} and Conjecture \ref{conj-Min ineq-weak}. We unify some well-known geometric inequalities from a horospherically convex geometric point of view. Furthermore, we prove some new inequalities which are special cases of these conjectures.

\subsection{An introduction of horospherical $p$-Brunn-Minkowski inequality and horospherical $p$-Minkowski inequalities} \label{subsec-10.1}$ \ $

Let $\omega_n$ denote the volume of the unit n-sphere $\mathbb{S}^n$ and $B(r)$ denote a geodesic ball of radius $r>0$  centered at the origin in $\mathbb{H}^{n+1}$.
At first, we calculate the $k$-th modified quermassintegrals of $B(r)$, i.e. $\widetilde{W}_k(B(r) )$. Since the shifted principal curvatures $\lbrace \tilde{\kappa}_1, \ldots, \tilde{\kappa}_n \rbrace$ on $\partial B(r)$ are all equal to $\frac{e^{-r}}{\sinh r }$ and the area of $\partial B(r)$ is $\omega_n \sinh^{n} r$. By taking $\mathscr{F} \equiv 1$ in Lemma \ref{lem-variation of modified quermass}, we have
\begin{align*}
\frac{d}{dr} \widetilde{W}_k(B(r)) = \omega_n \(\frac{e^{-r}}{\sinh r} \)^{k} \sinh^n r= \omega_n \sinh^{n-k} r e^{-k r}.
\end{align*}
Hence we can define a monotonically increasing function $I_k : [0, +\infty) \to \mathbb{R}$ by
\begin{align}\label{I_k(r)}
I_k(r) : = \widetilde{W}_k\( B(r)\) = \omega_n \int_0^r \sinh^{n-k} t e^{-kt} dt.
\end{align}
Particularly, we have 
\begin{align}\label{I_n(r)}
I_n(r) = \frac{\omega_n}{n}(1- e^{-nr}).
\end{align}
\begin{defn}\label{def-mod k-mean radius}
	The modified $k$-mean radius of a domain $\Omega \subset \mathbb{H}^{n+1}$ is defined by
	\begin{equation}\label{mod k-mean radius}
	r^k_\Omega := I_k^{-1} (\widetilde{W}_k(\Omega)).
	\end{equation}
\end{defn}

Based on the observation of hyperbolic $p$-sum of geodesic balls in Theorem \ref{thm-sum of balls 0.5--2}, we can prove the following geometric inequality. 

\begin{thm}\label{thm-horo-BM balls}
	Let $n \geq 1$ and $0 \leq k \leq n$ be integers, and let $\frac{1}{2} \leq p \leq 2$ be a real number. Assume that $a$ and $b$ are positive numbers such that $a+b \geq 1$.
	Let $K= B(X, r_1)$ and $L= B(Y, r_2)$ be two geodesic balls in $\mathbb{H}^{n+1}$, and let $\Omega = a \cdot K +_p b \cdot L$.  Then
	\begin{equation}\label{horo-BM-ineq-balls}
	\exp \left( p I_k^{-1} ( \widetilde{W}_k (\Omega) ) \right)
	\geq a \exp \left( p I_k^{-1} ( \widetilde{W}_k (K) ) \right)
	+b \exp \left( p I_k^{-1} ( \widetilde{W}_k (L) ) \right), \quad k=0,1, \ldots, n.
	\end{equation}
	Equality holds if and only if $X=Y$, i.e. $K$ and $L$ are centered at the same point.
\end{thm}

\begin{proof}
	By the definition of the increasing function $I_k$ in \eqref{I_k(r)}, we have 
	\begin{equation}\label{RHS of BM balls}
	\exp \left( p I_k^{-1} ( \widetilde{W}_k (K) ) \right) = e^{p r_1}, \quad
	\exp \left( p I_k^{-1} ( \widetilde{W}_k (L) ) \right) = e^{p r_2}.
	\end{equation}
	Theorem \ref{thm-sum of balls 0.5--2} asserts that we can find a geodesic ball $\widetilde{B}$ of radius $\frac{1}{p} \log \(ae^{pr_1}  + b e^{pr_2} \)$ lying inside $\Omega$, and $\widetilde{B} \subset \neq \Omega$ unless $X=Y$. Combining this with Proposition \ref{prop-domain monotone modified quermass}, we have
	\begin{equation}\label{LHS of BM balls}
	\exp \left( p I_k^{-1} ( \widetilde{W}_k (\Omega) ) \right) \geq 
		\exp \left( p I_k^{-1} ( \widetilde{W}_k (\widetilde{B}) ) \right)
	=ae^{pr_1}  + b e^{pr_2},
	\end{equation}
	with equality if and only if $X=Y$. Then Theorem \ref{thm-horo-BM balls} follows by substituting \eqref{RHS of BM balls} and \eqref{LHS of BM balls} into \eqref{horo-BM-ineq-balls}.
\end{proof}

Question arises whether inequality \eqref{horo-BM-ineq-balls} still holds by replacing $K$ and $L$ by general h-convex domains.
 For this purpose, we propose the following  Conjecture \ref{conj-BM}, which is named as the horospherical $p$-Brunn-Minkowski inequalities for modified quermassintegrals. Recall the definition of hyperbolic dilates in Definition \ref{def-hyperbolic dilates}.
\begin{con}[Horospherical $p$-Brunn-Minkowski inequalities]\label{conj-BM}
	Let $n \geq 1$ and $0 \leq k \leq n$ be integers, and let $\frac{1}{2} \leq p \leq 2$ be a real number. Assume that $a \geq 0$ and $b \geq 0$ are real numbers such that $a+b \geq 1$.
	Let $K$ and $L$ be two smooth uniformly h-convex bounded domains in $\mathbb{H}^{n+1}$, and let $\Omega = a \cdot K +_p b \cdot L$.  Then 
	\begin{equation}\label{horo-BM-ineq}
	\exp \left( p I_k^{-1} ( \widetilde{W}_k (\Omega) ) \right)
	\geq a \exp \left( p I_k^{-1} ( \widetilde{W}_k (K) ) \right)
	+b \exp \left( p I_k^{-1} ( \widetilde{W}_k (L) ) \right), \quad k=0,1, \ldots, n.
	\end{equation}
	\begin{enumerate}
		\item When $0 \leq k \leq n-1$,  equality holds in \eqref{horo-BM-ineq}  if and only if one of the following conditions holds,
		\begin{enumerate}
			\item $a=1$ and $b=0$,
			\item $a=0$ and $b=1$,
			\item $K = L$ and $a+b=1$,
			\item $K$ and $L$ are geodesic balls centered at the same point.
		\end{enumerate}
		\item When $k=n$, equality holds in \eqref{horo-BM-ineq} if and only if one of the following conditions holds,
		\begin{enumerate}
			\item $a \geq 1$ and $b =0$,
			\item $a=0$ and $b \geq 1$,
			\item $K$ and $L$ are hyperbolic dilates.
		\end{enumerate}
	\end{enumerate}
\end{con}

When $k=n$, we will prove Conjecture \ref{conj-BM} by showing that it is equivalent to a Minkowski inequality for functions on the unit sphere $\mathbb{S}^n$.

\begin{lem}\label{lem-Wn-explicit}
	For a smooth uniformly h-convex bounded domain $\Omega$ in $\mathbb{H}^{n+1}$ $(n \geq 1)$, we have
	\begin{align}
	\int_{\partial \Omega} p_n (\tilde{\kappa}) d\mu =& \omega_n - n \widetilde{W}_n(\Omega),\label{Wn-explicit-kappa}\\
	\widetilde{W}_n (\Omega) =& \frac{1}{n} \int_{\mathbb{S}^n} (1 - \g^{-n}_{\Omega}) d\sigma.\label{Wn-explicit-phi}
	\end{align}
	Consequently, it holds that
	\begin{align}\label{phi^-n-Wn}
	\int_{\mathbb{S}^n} \g_{\Omega}^{-n} d\sigma = \omega_n e^{-n r^n_\Omega},
	\end{align} 
	where $r^n_\Omega := I_n^{-1} \( \widetilde{W}_n (\Omega) \)$.
\end{lem}
\begin{proof}
	Along a general flow $\partial_t X = \mathscr{F} \nu$, by Lemma \ref{lem-variation of modified quermass} and \eqref{scalr-flow DT's trick}, we have
	\begin{equation*}
	\frac{d}{dt} \widetilde{W}_n(\Omega_t) =  \int_{\Omega_t} p_n (\tilde{\kappa}) \mathscr{F} d \mu_t
	,\quad \frac{\partial}{\partial t} \g (z,t) = \g(z,t) \mathscr{F}.
	\end{equation*}
	Using  \eqref{rel-area element} and \eqref{shifted curvature-support function}, we have
	\begin{align}\label{dt Wn-dt phi^-n}
	\frac{d}{dt} \widetilde{W}_n(\Omega_t) 
	= \int_{\mathbb{S}^n}  p_n (\tilde{\kappa}) \mathscr{F}\g^{-n} p_n (\tilde{\lambda}) d\sigma
	=  \int_{\mathbb{S}^n} \g^{-n} \mathscr{F} d\sigma
	=\frac{d}{dt} \int_{\mathbb{S}^n} \(-\frac{1}{n} \g^{-n}(z,t) \) d\sigma. 
	\end{align}
	Let $N = (0,1)$ be the origin of $\mathbb{H}^{n+1} \subset \mathbb{R}^{n+1,1}$. Then $\widetilde{W}_n(N) =0$, and $\g_N(z) =1$ for all $z \in \mathbb{S}^n$. By Proposition \ref{prop-p-sum}, we can define $\Omega_t$ by $\g_{\Omega_t}(z) = (1-t) + t \g_{\Omega}(z)$ such that $\Omega_0 = N$, $\Omega_1 = \Omega$, and $\partial \Omega_t$ is smooth and uniformly h-convex for $t \in (0,1]$.
	Integrating the both sides of \eqref{dt Wn-dt phi^-n} from $0$ to $1$ then yields 
	\begin{align*}
	\widetilde{W}_n (\Omega) = \frac{1}{n} \int_{\mathbb{S}^n} (1 - \g^{-n}_{\Omega}) d\sigma
	=\frac{1}{n} \omega_n - \frac{1}{n} 	\int_{\partial \Omega} p_n (\tilde{\kappa}) d\mu
	\end{align*}
	for any smooth uniformly h-convex bounded domain $\Omega$. Thus we obtain \eqref{Wn-explicit-kappa} and \eqref{Wn-explicit-phi}. 
	By  \eqref{mod k-mean radius} and \eqref{I_n(r)}, we have
	\begin{align}\label{In-r}
	\widetilde{W}_n(\Omega) = I_n(r^n_{\Omega})= \frac{\omega_n}{n} (1-e^{-nr^n_\Omega}).
	\end{align}
	Comparing \eqref{Wn-explicit-phi} with \eqref{In-r}, we obtain the desired formula \eqref{phi^-n-Wn}.
	We complete the proof of Lemma \ref{lem-Wn-explicit}.
\end{proof}

\begin{lem}\label{lem-Minineq-func}
	Let $n \geq 1$ be an integer and $q \in (-1,0) \cup (0,+\infty)$ be a real number. Assume that $f$ and $g$ are two positive continuous functions on $\mathbb{S}^n$, then  
	\begin{equation}\label{Minineq-func}
	\left( \int_{\mathbb{S}^n} (f+g)^{-q}  d\sigma \right)^{-\frac{1}{q}}
	\geq 
	\(\int_{\mathbb{S}^n} f^{-q} d\sigma \)^{-\frac{1}{q}}
	+\(\int_{\mathbb{S}^n} g^{-q} d\sigma \)^{-\frac{1}{q}}.
	\end{equation}
	Equality holds if and only if $f^{-1} g$ is constant.
\end{lem}

\begin{proof}
	\textbf{Case 1.} $q \in (-1,0)$.
	 By the H\"older inequality, we have
	\begin{align*}
	\int_{\mathbb{S}^n} f^{-q} d\sigma
	\leq&
	\( \int_{\mathbb{S}^n} f(f+g)^{-q-1} d \sigma   \)^{-q}
	\( \int_{\mathbb{S}^n} (f+g)^{-q}  d\sigma   \)^{q+1},
	\end{align*}
	with equality if and only if $f (f+g)^{-1}$ is constant on $\mathbb{S}^n$.  
	Then,
	\begin{align}
	\int_{\mathbb{S}^n} (f+g)^{-q} d\sigma
	=& \int_{\mathbb{S}^n} (f+g) (f+g)^{-q-1} d\sigma \nonumber\\
	=& \int_{\mathbb{S}^n} f(f+g)^{-q-1} d\sigma
	+\int_{\mathbb{S}^n} g(f+g)^{-q-1} d\sigma \nonumber\\
	\geq& \left(\int_{\mathbb{S}^n} (f+g)^{-q} d\sigma \right)^{\frac{q+1}{q}} \left( \(\int_{\mathbb{S}^n} f^{-q} d\sigma \)^{-\frac{1}{q}}
	+\(\int_{\mathbb{S}^n} g^{-q} d\sigma \)^{-\frac{1}{q}}
	\right),\label{func Minkowski ineq -1,0}
	\end{align}
	with equality if and only if both $f (f+g)^{-1}$ and $g (f+g)^{-1}$ are constants, equivalently,  $f^{-1} g$ is constant. Thus \eqref{Minineq-func} is proved for $q \in (-1,0)$.
	
	\textbf{Case 2.} $q \in (0,+\infty)$.
	By the H\"older inequality, we have
	\begin{equation*}
	\int_{\mathbb{S}^n} (f+g)^{-q}  d\sigma
	\leq
	\left(\int_{\mathbb{S}^n} f(f+g)^{-q-1} d \sigma \right)^{\frac{q}{q+1}}
	\(\int_{\mathbb{S}^n} f^{-q} d\sigma \)^{\frac{1}{q+1}},  
	\end{equation*}
	with equality if and only if $f (f+g)^{-1}$ is constant on $\mathbb{S}^n$.
	Hence inequality \eqref{func Minkowski ineq -1,0} is also valid for positive $q$. In this case, equality holds in \eqref{func Minkowski ineq -1,0} if and only if $f^{-1} g$ is constant on $\mathbb{S}^n$.
	This completes the proof of Lemma \ref{lem-Minineq-func}.
\end{proof}

Now we are able to prove the following Theorem \ref{thm-BM-k=n}, which is a special case of Conjecture \ref{conj-BM}.

\begin{thm}\label{thm-BM-k=n}
	 Conjecture \ref{conj-BM} holds in the case $n \geq 1$ and $k=n$.
\end{thm}
\begin{proof}
	\textbf{Case 1.} $a>0$, $b>0$ and $a+b \geq 1$. By the assumption in Conjecture \ref{conj-BM} and Definition \ref{def-p sum}, we have $\g_\Omega^p(z)= a \g_K^p(z) + b \g_L^p(z)$.
	Taking $q= \frac{n}{p} >0$, $f(z) = a \g_K^p(z)$ and $g(z) = b \g_L^p(z)$ in \eqref{Minineq-func}, we have
	\begin{align}\label{BM-equiv-form}
	\(\int_{\mathbb{S}^n} \g^{-n}_\Omega d \sigma \)^{-\frac{p}{n}}
	\geq 
	a \(\int_{\mathbb{S}^n} \g^{-n}_K d \sigma \)^{-\frac{p}{n}}
	+
	b\(\int_{\mathbb{S}^n} \g^{-n}_L d \sigma \)^{-\frac{p}{n}}.
	\end{align}
	 Equality holds in \eqref{BM-equiv-form} if and only if $\g_K^{-1} \g_L$ is constant on $\mathbb{S}^n$, which means $K$ and $L$ are hyperbolic dilates, see Definition \ref{def-hyperbolic dilates}.
    Using \eqref{phi^-n-Wn}, we have
	\begin{align}\label{W_n-phi}
	\exp \left( p I_n^{-1} ( \widetilde{W}_n (\Omega) ) \right)
	= e^{p r^n_\Omega}
	= \omega_n^{\frac{p}{n}} \(\int_{\mathbb{S}^n} \g^{-n}_\Omega  d\sigma \)^{-\frac{p}{n}}.
	\end{align}
	Substituting \eqref{W_n-phi} into \eqref{BM-equiv-form}, we obtain
	\begin{align}\label{BM-k=n}
	\exp \left( p I_n^{-1} ( \widetilde{W}_n (\Omega) ) \right)
	\geq
	a\exp \left( p I_n^{-1} ( \widetilde{W}_n (K) ) \right)
	+b\exp \left( p I_n^{-1} ( \widetilde{W}_n (L) ) \right),
	\end{align}
	with equality if and only if $K$ and $L$ are hyperbolic dilates.
	
	\textbf{Case 2.} Either $a \geq 1$ and $b=0$, or $a =0$ and $b \geq 1$. It is easy to see that equality holds in \eqref{BM-equiv-form}. Hence, in this case, equality holds in \eqref{BM-k=n} for all smooth uniformly h-convex bounded domains $K$ and $L$.
	
	We complete the proof of Theorem \ref{thm-BM-k=n}.
\end{proof}

\begin{rem}
	If we omit the geometric meaning of $\g_\Omega$, then Lemma \ref{lem-Minineq-func} implies that inequality \eqref{BM-equiv-form} holds for $p \in [ -n,0) \cup (0,+\infty)$.
\end{rem}

Next, we will show that Conjecture \ref{conj-BM} holds in the case $K=L$. The following Theorem C was proved by Andrews et al. \cite[Corollary 1.9]{ACW18}, see also an alternative proof in \cite[Remark 1.7]{HLW20}.
\begin{thmC}[\cite{ACW18}]
	Let $M = \partial \Omega$ be a smooth, closed and uniformly h-convex hypersurface in $\mathbb{H}^{n+1}$ $(n \geq 1)$. Then for any $0 \leq l <k \leq n$, there holds
	\begin{equation*}
	\widetilde{W}_k (\Omega) \geq I_k \circ I_l^{-1} ( \widetilde{W}_l (\Omega)),
	\end{equation*}
	with equality holding if and only if $\Omega$ is a geodesic ball.
\end{thmC}
From Theorem C, we can obtain the following Alexandrov-Fenchel type inequalities.
\begin{cor}\label{cor-AF-modif-quermass}
	Let $\Omega$ be a smooth uniformly h-convex bounded domain in $\mathbb{H}^{n+1}$ $(n \geq 1)$. Then 
	\begin{equation}\label{AF-modif-quermass}
	\int_{\partial \Omega} p_k (\tilde{\kappa}) d \mu \geq \omega_n \sinh^{n-k} r^k_\Omega e^{-k r^k_\Omega}, \quad k=0, 1, \ldots, n-1.
	\end{equation}
	Equality holds if and only if $\Omega$ is a geodesic ball.
\end{cor}
\begin{proof}
	By Theorem C, we have
	\begin{equation*}
	\widetilde{W}_{k+1} (\Omega) \geq I_{k+1} \circ I_k^{-1} (\widetilde{W}_k (\Omega)) = I_{k+1} (r^k_\Omega).
	\end{equation*}
	Hence, by using \eqref{I_k(r)}, $e^{-t} = \cosh t-\sinh t$ and integration by parts, we have
	\begin{align}
	\widetilde{W}_{k+1} (\Omega)
	\geq&\omega_n \int_0^{r^k_\Omega} \sinh^{n-k-1} t e^{-(k+1) t} dt \nonumber\\
	=& \omega_n \int_0^{r^k_\Omega} \sinh^{n-k-1} t \cosh t e^{-k t} dt
	-\omega_n \int_0^{r^k_\Omega} \sinh^{n-k} t e^{-k t} dt  \nonumber\\
	=& \left.\frac{\omega_n }{n-k} \sinh^{n-k} t e^{-kt}  \right|_{t=0}^{t= r^k_\Omega}
	+\frac{k\omega_n}{n-k} \int_{0}^{r^k_\Omega} \sinh^{n-k} t e^{-kt} dt
	-\widetilde{W}_k(\Omega) \nonumber\\
	=& \frac{\omega_n}{n-k}  \sinh^{n-k} r^k_\Omega e^{-k r^k_\Omega}- \frac{n-2k}{n-k} \widetilde{W}_k (\Omega).\label{AF-inq-shifted-quermass}
	\end{align}
	Equality holds if and only if $\Omega$ is a geodesic ball.
	Then inequality \eqref{AF-modif-quermass} follows directly from \eqref{AF-inq-shifted-quermass} and \eqref{induction def of modi-quer-intg}. We complete the proof of Corollary \ref{cor-AF-modif-quermass}.
\end{proof}

\begin{prop}\label{prop-BM-K=L}
	Conjecture \ref{conj-BM} holds in the case $K=L$. More generally, for any $n \geq 1$, $0 \leq k \leq n-1$, $p > 0$ and $t \geq 1$, we have
	\begin{equation}\label{BM-K=L-original}
	\exp \( p I_k^{-1} ( \widetilde{W}_k (t \cdot \Omega) ) \)
	\geq t \exp \(  p I_k^{-1} ( \widetilde{W}_k(\Omega) ) \)
	\end{equation}
	holds for any smooth uniformly h-convex bounded domain $\Omega$. If $t>1$, then equality holds in \eqref{BM-K=L-original} if and only if $\Omega$ is a geodesic ball.
\end{prop}
\begin{proof}
	Let $\Omega_t = t \cdot \Omega $, $r^k_t := I_k^{-1} ( \widetilde{W}_k(\Omega_t) )$ and $\chi(t)  = e^{p r^k_t}$.
	Then the desired inequality \eqref{BM-K=L-original} can be written shortly as 
	\begin{equation}\label{BM-K=L}
	\chi(t) \geq t \chi(1), \quad  \forall \ t \geq 1. 
	\end{equation}
	 Definition \ref{def-p dilation} gives
	 \begin{equation*}
	 u_{\Omega_t}(z) := u(z,t) = u(z,1) + \frac{1}{p} \log t.
	 \end{equation*}
     Hence,  at $t=1$ we have
     \begin{equation*}
     \frac{\partial}{\partial t} u(z,t) =\frac{1}{p}.
     \end{equation*}
	By Lemma \ref{lem-variation of modified quermass}, \eqref{scalr-flow DT's trick} and \eqref{I_k(r)}, at $t=1$ we have
	\begin{align*}
	\frac{d}{dt} \widetilde{W}_k(t \cdot \Omega) = \frac{1}{p} \int_{\partial \Omega} p_k (\tilde{\kappa}) d \mu,
	\quad
	I_k'(r^k_\Omega) = \omega_n \sinh^{n-k} r^k_\Omega e^{-k r^k_\Omega}.
	\end{align*}	
	Hence, by using \eqref{AF-modif-quermass}, we have
	\begin{align}
	\chi'(1) =& p e^{p r^k_\Omega} \frac{d}{dt} r^k_t
	=  p e^{p r^k_\Omega}  \( I'_k(r^k_\Omega )\)^{-1} \frac{d}{dt} \widetilde{W}_k (t \cdot \Omega) \nonumber\\
	=& \(\omega_n \sinh^{n-k} r^k_\Omega e^{-k r^k_\Omega}\)^{-1} e^{pr^k_\Omega} \int_{\partial \Omega} p_k(\tilde{\kappa}) d\mu \nonumber\\
	\geq& e^{p r^k_\Omega} = \chi(1),\label{chi'(1) geq chi(1)}
	\end{align}
	with equality if and only if $\Omega$ is a geodesic ball.
	It is easy to see that \eqref{chi'(1) geq chi(1)} is necessary for \eqref{BM-K=L}; while, we claim that it is also sufficient.	
	
	Instead of $\Omega$, if we set $\overline{\Omega} = {t_0} \cdot \Omega$ be the initial domain for $t_0 \geq 1$, and $\zeta(s) = \chi(t_0 s)$, then inequality \eqref{chi'(1) geq chi(1)} implies $\zeta'(1) \geq \zeta(1)$. Hence 
	\begin{equation*}
	\chi'(t_0) t_0 = \zeta'(1) \geq \zeta(1) = \chi(t_0).
	\end{equation*}
	 Letting $t_0$ runs from $1$ to infinity, we have $\frac{d}{dt} \log \(\frac{\chi(t)}{t} \)\geq 0$ for all $t \geq 1$. Then we get \eqref{BM-K=L}.
	 By the equality case of \eqref{chi'(1) geq chi(1)}, if $t>1$, then equality holds in \eqref{BM-K=L} if and only if $\Omega$ is a geodesic ball. We complete the proof of Proposition \ref{prop-BM-K=L}. 
\end{proof}

By using Theorem \ref{thm-BM-k=n} and Proposition \ref{prop-BM-K=L}, we can simplify Conjecture \ref{conj-BM}.
\begin{prop}\label{prop-reduce conj BM - a+b=1 case}
		Assume that Conjecture \ref{conj-BM} holds in the case that $a \geq 0$ and $b \geq 0$ with $a+b=1$. Then  Conjecture \ref{conj-BM} holds in the case that $a \geq 0$ and $b \geq 0$ with $a+b > 1$.
\end{prop}

\begin{proof}
	 We have proved in Theorem \ref{thm-BM-k=n} and in Proposition \ref{prop-BM-K=L} that Conjecture \ref{conj-BM} holds when either $n \geq 1$ and $k=n$,  or $ab =0$. 
	Then we can suppose that $n \geq 1$, $0 \leq k \leq n-1$, $a>0$, $b> 0$ and $a+b>1$ without loss of generality.
	For a smooth h-convex bounded domain $\Omega$, we set
	\begin{equation*}
	\chi(\Omega) = \exp \(p I_k^{-1} (\widetilde{W}_k(\Omega)) \),
	\end{equation*}
	where $0 \leq k \leq n-1$ and $\frac{1}{2} \leq p \leq 2$.
	By \eqref{horo-BM-ineq} and the assumption in Proposition \ref{prop-reduce conj BM - a+b=1 case}, we have that inequality
	\begin{equation}\label{BM-a+b=1}
	\chi \((1-t)\cdot K +_p t \cdot L \) \geq (1-t) \chi(K) +t \chi(L) 
	\end{equation}
	holds for any smooth uniformly h-convex bounded domains $K$ and $L$, and $t \in [0,1]$. If $t \in (0,1)$,  then equality holds in \eqref{BM-a+b=1} if and only if either $K=L$, or $K$ and $L$ are geodesic balls centered at the same point.
	 We set
	 \begin{equation*}
	 \overline{\Omega} = \frac{a}{a+b} \cdot K +_p \frac{b}{a+b} \cdot L,
	 \end{equation*}
	 which is well-defined by Theorem \ref{thm-def p sum-well defined}. 
	Using Proposition \ref{prop-BM-K=L} and taking $t = \frac{b}{a+b} \in (0,1)$ in \eqref{BM-a+b=1}, we have
	\begin{equation}\label{chi-aK-bL}
	\chi \(a\cdot K +b \cdot L\)= \chi \((a+b) \cdot \overline{\Omega}\) \geq  (a+b) \chi (\overline{\Omega})
	\geq  a \chi(K) +b \chi (L), 
	\end{equation}
	thus we obtain \eqref{horo-BM-ineq}. Furthermore, if $a+b>1$, then Proposition \ref{prop-BM-K=L} asserts that equality holds only if $\overline{\Omega}$ is a geodesic ball.
	
	 Then equality holds in \eqref{chi-aK-bL} 
	 if and only if  equality holds in \eqref{BM-a+b=1} and $\overline{\Omega}$ is a geodesic ball. Consequently,  equality holds in \eqref{chi-aK-bL} if and only if $K$ and $L$ are geodesic balls centered at the same point.
	We complete the proof of Proposition \ref{prop-reduce conj BM - a+b=1 case}.
\end{proof}

In the classical convex geometry in Euclidean space, we know that the Minkowski inequality can be obtained by calculating the variation of the Brunn-Minkowski inequality. Furthermore, these two inequalities can deduce each other. Similarly, we propose the following conjectures of Minkowski type inequalities in hyperbolic space. Since the statement does not rely on the definition of hyperbolic $p$-sum in Definition \ref{def-p sum}, we can relax the restriction on $p$.
Since Conjecture \ref{conj-BM} has been solved in the case $k=n$ in Theorem \ref{thm-BM-k=n}, we mainly focus on the case $0 \leq k \leq n-1$.

\begin{con}[Horospherical $p$-Minkowski inequalities of type \uppercase\expandafter{\romannumeral1}]\label{conj-Min ineq}
	Let $n \geq 1$ and $0 \leq k \leq n$ be integers, and let $p$ be a real number.
	Suppose that $K$ and $L$ are smooth uniformly h-convex bounded domains in $\mathbb{H}^{n+1}$. If $p>0$, then	
	\begin{align}\label{formula-conj-Min ineq}
	\int_{\mathbb{S}^n} \g_L^p \g_K^{-p-n} p_{n-k}(\tilde{\lambda}_K) d\sigma
	- \int_{\mathbb{S}^n} \g_K^{-n}p_{n-k} (\tilde{\lambda}_K)d\sigma
	\geq
	\omega_n \sinh^{n-k}r^k_K e^{-kr^k_K} \left( e^{p(r^k_L - r^k_K)}-1 \right),
	\end{align} 
	where $r^k_\Omega$ is defined by $r^k_\Omega = I_k^{-1} (\widetilde{W}_k (\Omega))$ for a h-convex bounded domain $\Omega \subset \mathbb{H}^{n+1}$.
	\begin{enumerate}
		\item When $0 \leq k \leq n-1$,  equality holds in \eqref{formula-conj-Min ineq}  if and only if one of the following conditions holds,
			\begin{enumerate}
				\item $K =L$,
				\item $K$ and $L$ are geodesic balls centered at the same point.
			\end{enumerate}
		\item When $k=n$, equality holds in \eqref{formula-conj-Min ineq}  if and only if $K = c\cdot L$ or $L = c\cdot K$ for some constant $c \geq 1$.
	\end{enumerate}
	If $-n \leq p <0$, then
	\begin{align}\label{formula-conj-Min ineq-2}
	\int_{\mathbb{S}^n} \g_L^p \g_K^{-p-n} p_{n-k}(\tilde{\lambda}_K) d\sigma
	- \int_{\mathbb{S}^n} \g_K^{-n}p_{n-k} (\tilde{\lambda}_K)d\sigma
	\leq
	\omega_n \sinh^{n-k}r^k_K e^{-kr^k_K} \left( e^{p(r^k_L - r^k_K)}-1 \right).
	\end{align} 
	\begin{enumerate}
		\item When $0 \leq k \leq n-1$ and $-n<p<0$,  equality holds in \eqref{formula-conj-Min ineq-2} if and only if one of the following conditions holds,
		\begin{enumerate}
			\item $K =L$,
			\item $K$ and $L$ are geodesic balls centered at the same point.
		\end{enumerate}
		\item When  $0 \leq k \leq n-1$ and $p=-n$, equality holds in \eqref{formula-conj-Min ineq-2} if and only if one of the following conditions holds,
				\begin{enumerate}
					\item $K =L$,
					\item $K$ and $L$ are geodesic balls.
				\end{enumerate}
		\item When $k=n$ and $-n<p <0$,  equality holds in \eqref{formula-conj-Min ineq-2} if and only if $K$ and $L$ are hyperbolic dilates.
		\item When $k=n$ and $p=-n$, equality always holds in  \eqref{formula-conj-Min ineq-2}.
	\end{enumerate}
\end{con}
By using \eqref{rel-area element}, \eqref{shifted curvature-support function} and \eqref{1/phi, coshr-u},  we note that \eqref{formula-conj-Min ineq} is equivalent to
\begin{align*}
\int_{\partial K} \(\g_L^p (z) (\cosh r -\tilde{u})^{p} -1\)  p_k (\tilde{\kappa}) d\mu
\geq \omega_n \sinh^{n-k}r^k_K e^{-kr^k_K} \left( e^{p(r^k_L - r^k_K)}-1 \right),
\end{align*}
and \eqref{formula-conj-Min ineq-2} is equivalent to
\begin{align*}
\int_{\partial K} \(\g_L^p (z) (\cosh r -\tilde{u})^{p} -1\)  p_k (\tilde{\kappa}) d\mu
\leq \omega_n \sinh^{n-k}r^k_K e^{-kr^k_K} \left( e^{p(r^k_L - r^k_K)}-1 \right).
\end{align*}

\begin{rem}[Relationship between Conjecture \ref{conj-Min ineq} and the flow \eqref{flow-HCMF}]
	We can assume $p>0$ at first.
	If we fix the smooth uniformly h-convex bounded domain $K$ in Conjecture \ref{conj-Min ineq} and set $f(z) =\g_K^{-p-n}(z) p_{n-k}(\tilde{\lambda}_K(G_K^{-1}(z)  ))$, then we have the left-hand side of \eqref{formula-conj-Min ineq} is non-increasing and the right-hand side of \eqref{formula-conj-Min ineq} is invariant along the flow \eqref{flow-HCMF} by \eqref{mono quan-ineq} and \eqref{mono quan-eq}, respectively. For the case $-n \leq p \leq 0$, the argument is roughly the same. Hence, flow \eqref{flow-HCMF} is a possible tool to solve Conjecture \ref{conj-Min ineq}.  
\end{rem}

In the case $\frac{1}{2}\leq p \leq 2$, we will show  Conjecture \ref{conj-Min ineq} can be deduced from Conjecture \ref{conj-BM} in the following way. 
\begin{prop}\label{prop-rel-BM-Min ineq-strong}
	Let $n$, $k$, $p$, $K$ and $L$ satisfy the assumptions in Conjecture \ref{conj-BM}.
	Assume that $\Omega_t:= (1-t) \cdot K +_p t \cdot L$. Then Conjecture \ref{conj-Min ineq} is equivalent to
	\begin{equation*}
	\left. \frac{d}{dt} \right|_{t=0} \exp(p I_k^{-1} (\widetilde{W}_k(\Omega_t))) \geq \exp (p I_k^{-1} (\widetilde{W}_k(L))) - \exp(p I_k^{-1} (\widetilde{W}_k(K))).
	\end{equation*}
	Thus Conjecture \ref{conj-Min ineq} can be obtained from Conjecture \ref{conj-BM} in this case.
\end{prop}

\begin{proof}
	Denote by $u(z,t)$ the horospherical support function of $\Omega_t$, and $\g(z,t) = e^{u(z,t)}$. Then we have from the definition of $\Omega_t$ that
	\begin{equation*}
	\g^p(z,t) = (1-t) \g^p_K(z) + t \g^p_{L}(z).
	\end{equation*}
	Hence 
	\begin{equation*}
	\left. \frac{\partial}{\partial t} \right|_{t=0} u(z,t) =\frac{1}{p}\frac{\g_L^p(z) - \g_K^p(z)}{\g_K^p (z)}.   
	\end{equation*}
	Applying a similar calculation as in \eqref{chi'(1) geq chi(1)}, we have
	\begin{align*}
	&\left.\frac{d}{dt} \right|_{t=0} \exp\(p I_k^{-1} (\widetilde{W}_k(\Omega_t))\)\\
	=&
	p e^{p r^k_K} \( \omega_n \sinh^{n-k} r^k_K e^{-kr^k_K} \)^{-1} \int_{\partial K} \(  \frac{1}{p}\frac{\g_L^p(z) - \g_K^p(z)}{\g_K^p (z)}\)p_k(\tilde{\kappa}) d\mu\\
	=&\( \omega_n \sinh^{n-k} r^k_K e^{-(k+p)r^k_K} \)^{-1}\( \int_{\partial K} \g_L^p \g_K^{-p} p_k (\tilde{\kappa}) d\mu- \int_{\partial K} p_k(\tilde{\kappa}) d\mu\)\\
	=&\( \omega_n \sinh^{n-k} r^k_K e^{-(k+p)r^k_K} \)^{-1}\( \int_{\mathbb{S}^n} \g_L^p \g_K^{-(n+p)} p_{n-k} (\tilde{\lambda}_K) d\sigma- \int_{\mathbb{S}^n} \g_K^{-n} p_{n-k}(\tilde{\lambda}_K) d\sigma\),
	\end{align*}
	where we used \eqref{rel-area element} and \eqref{shifted curvature-support function} in the last equality. Hence, by use of \eqref{mod k-mean radius}, we can rewrite Conjecture \ref{conj-Min ineq} as
	\begin{equation}\label{BM-a=1-t-b=t-deriv at 0}
	\left. \frac{d}{dt} \right|_{t=0} \exp(p I_k^{-1} (\widetilde{W}_k(\Omega_t))) \geq 
	e^{pr^k_L}- e^{pr^k_K}=\exp (p I_k^{-1} (\widetilde{W}_k(L))) - \exp(p I_k^{-1} (\widetilde{W}_k(K))).
	\end{equation}
	
	 On the other hand, taking $a=1-t$ and $b=t$ in Conjecture \ref{conj-BM} yields
	\begin{equation}\label{BM-a=1-t-b=t}
	\exp(p I_k^{-1} (\widetilde{W}_k(\Omega_t))) \geq 
	(1-t)\exp(p I_k^{-1} (\widetilde{W}_k(K)))+
	t\exp(p I_k^{-1} (\widetilde{W}_k(L))). 
	\end{equation}
	Hence \eqref{BM-a=1-t-b=t-deriv at 0} can be obtained by taking the derivative of \eqref{BM-a=1-t-b=t} with respect to $t$ at $t=0$.
\end{proof}

However, if we set $a=1$, $b=t \geq 0$ in Conjecture \ref{conj-BM} and take the derivative of \eqref{horo-BM-ineq} with respect to $t$ at $t=0$, then we can get the following conjecture by a similar method as in the proof of Proposition \ref{prop-rel-BM-Min ineq-strong}. Here we omit calculations in details.

\begin{con}[Horospherical $p$-Minkowski inequalities of type \uppercase\expandafter{\romannumeral2}]\label{conj-Min ineq-weak}
	Let $n \geq 1$ and $0 \leq k \leq n$ be integers, and let $p$ be a positive number. 
	Suppose that $K$ and $L$ are two smooth uniformly h-convex bounded domains in $\mathbb{H}^{n+1}$. Then 	
	\begin{align}\label{formula-conj-Min ineq-weak}
	\int_{\mathbb{S}^n} \g_L^p \g_K^{-p-n} p_{n-k}(\tilde{\lambda}_K) d\sigma
	\geq
	\omega_n \sinh^{n-k}r^k_K e^{-(k+p)r^k_K} e^{pr^k_L},
	\end{align} 
	where $r^k_\Omega$ is defined by $r^k_\Omega = I_k^{-1} (\widetilde{W}_k (\Omega))$ for a h-convex bounded domain $\Omega \subset \mathbb{H}^{n+1}$. 
	\begin{enumerate}
		\item When $0 \leq k \leq n-1$, equality holds in \eqref{formula-conj-Min ineq-weak} if and only if $K$ and $L$ are geodesic balls centered at the same point.
		\item When $k=n$, equality holds in \eqref{formula-conj-Min ineq-weak} if and only if $K = c \cdot L$ or $L =c \cdot K$ for some constant $c \geq 1$.
	\end{enumerate}
\end{con}
 By using \eqref{rel-area element}, \eqref{shifted curvature-support function} and \eqref{1/phi, coshr-u}, we note that \eqref{formula-conj-Min ineq-weak} is equivalent to
\begin{equation*}
\int_{\partial K} \g_L^p (z) (\cosh r -\tilde{u})^{p}   p_k (\tilde{\kappa}) d\mu
\geq \omega_n \sinh^{n-k}r^k_K e^{-(k+p)r^k_K}  e^{pr^k_L}.
\end{equation*}
\begin{rem}\label{rem-rel-strong Min-weak Min}
	It is easy to see that \eqref{formula-conj-Min ineq-weak} follows from \eqref{formula-conj-Min ineq} and \eqref{AF-modif-quermass}. Hence Conjecture \ref{conj-Min ineq-weak} is weaker than Conjecture \ref{conj-Min ineq}.
	
	When $L = B(r_L)$, by \eqref{mod k-mean radius}, we have $\g_L^p(z) = e^{p r_L} = e^{p r^k_L}$ for all $z \in \mathbb{S}^n$. Hence, if  $L$ is a geodesic ball centered at the origin, then Conjecture \ref{conj-Min ineq-weak} is equivalent to
	 \begin{align*}
	 \int_{\partial K}  (\cosh r -\tilde{u})^{p}   p_k (\tilde{\kappa}) d\mu
	 \geq \omega_n \sinh^{n-k}r^k_K e^{-(k+p)r^k_K} , 
	 \end{align*}
	 where $n \geq 1$, $0 \leq k \leq n$, $p>0$ and $K$ is a smooth uniformly h-convex bounded domain in $\mathbb{H}^{n+1}$.
\end{rem}

\subsection{Horospherical $p$-Minkowski inequalities}\label{subsec-10.2}$ \ $

In this subsection, we discuss Conjecture \ref{conj-Min ineq} and Conjecture \ref{conj-Min ineq-weak}. For fixed $n$, $k$ and $p$, we first show that these conjectures only depend on the relative position of smooth uniformly h-convex bounded domains $K$ and $L$.
\begin{prop}\label{prop-Conj9.3-Conj9.4-rel position}
	Let $n \geq 1$ and $0 \leq k \leq n$ be integers, and let $p$ be a real number.
	Assume that $K$ and $L$ are smooth uniformly h-convex bounded domains in $\mathbb{H}^{n+1}$.
	Then the integral
	\begin{align*}
	\int_{\mathbb{S}^n} \g_L^p(z) \g_K^{-p-n}(z) p_{n-k}(\tilde{\lambda}_K) d\sigma
	\end{align*}
only depends on the relative position of $K$ and $L$. Namely, we have for any isometry $f$ of $\mathbb{H}^{n+1}$ that
\begin{align*}
\int_{\mathbb{S}^n} \g_{f(L)}^p(z) \g_{f(K)}^{-p-n}(z) p_{n-k}(\tilde{\lambda}_{f(K)}) d\sigma
= \int_{\mathbb{S}^n} \g_L^p(z) \g_K^{-p-n}(z) p_{n-k}(\tilde{\lambda}_K) d\sigma.
\end{align*}
\end{prop}

\begin{proof}
	In this proof, we let $f$ be an isometry of $\mathbb{H}^{n+1}$. It is well-known that $f \in O^+(n+1,1)$, i.e. $f$ is a linear isomorphism of $\mathbb{R}^{n+1,1}$ that preserves the Lorentzian metric and $\{ Y = (y,y_{n+1}) \in \mathbb{R}^{n+1,1}:  y_{n+1}>0 \}$, see e.g. \cite[p. 47]{MB16}.
	For any fixed  $X \in \partial K$, we let $z=G_K(X)$, where $G_K$ is the horospherical Gauss map of $K$. Then we have that the supporting horosphere of $K$ at $X$ is given by 
	\begin{align*}
	H_{z}(u_K(z)) = \{ Y=(y,y_{n+1}) \in \mathbb{H}^{n+1}:  \metric{-Y}{(z,1)}= \g_K(z) \}.
	\end{align*} 
	Since $f$ preserves the Lorentzian metric, condition $\metric{-Y}{(z,1)}= \g_K(z)$ is equivalent to $\metric{-f(Y)}{f\((z,1)\)} = \g_K(z)$. Besides, it is easy to see that $f\( (z,1) \)$ lies in the future light cone of $(0,0) \in \mathbb{R}^{n+1,1}$.
	Then we have
	\begin{align}\label{trans-horo-sphere}
	f\( H_{z}(u_K(z)) \)= \{ Y \in \mathbb{H}^{n+1}:  \metric{-Y}{(\tilde{z},1)}= \chi(z)\g_K(z) \},
	\end{align} 
	where $\chi(z)$ and $\tilde{z}$ were uniquely determined by $\chi(z) f\((z,1)\)= \(\tilde{z}, 1\)$.
	Let $\widetilde{K} = f(K)$ and $\widetilde{X} = f(X)$. We know that $f\( H_{z}(u_K(z)) \)$ is the supporting horosphere of $\widetilde{K}$ at $\widetilde{X}$ and $G_{\widetilde{K}} ( \widetilde{X} ) = \tilde{z}$. By using \eqref{trans-horo-sphere}, we obtain
	$\g_{\widetilde{K}} (\tilde{z} ) = \chi(z)\g_K (z)$.
	Since $\chi(z)$ only depends on the choice of $f$, we have $\g_{\widetilde{L}} (\tilde{z}) = \chi (z) \g_L (z)$ by the same argument. That gives
	\begin{align}\label{phiL-phiK-inv}
  \frac{\g_{\widetilde{L}} (\tilde{z})}{\g_{\widetilde{K}} (\tilde{z} )} =\frac{\g_L (z) }{\g_K (z)}.
  	\end{align}
  	Since $f$ is an isometry of $\mathbb{H}^{n+1}$, we have 
  	\begin{align}\label{curv-mea-inv}
  	p_{k}(\tilde{\kappa}_K ) d\mu_{\partial K} = f^* \(	p_{k}(\tilde{\kappa}_{\widetilde{K}} ) d\mu_{\partial \widetilde{K}} \).
  	\end{align}
  	Combining \eqref{phiL-phiK-inv} with \eqref{curv-mea-inv} yields
	\begin{align}\label{eq-conj 9.3-conj 9.4-rel posi}
	\int_{\partial K} \g_L^p(z)\g_K^{-p}(z)  p_{k}(\tilde{\kappa} ) d\mu_{\partial K}
=\int_{\partial \widetilde{K}} \g_{\widetilde{L}}^p(\tilde{z})\g_{ \widetilde{K} }^{-p}(\tilde{z}) 	p_{k}(\tilde{\kappa}) d\mu_{\partial \widetilde{K}}.
	\end{align}
	Then Proposition \ref{prop-Conj9.3-Conj9.4-rel position} follows by using \eqref{rel-area element} and \eqref{shifted curvature-support function} in \eqref{eq-conj 9.3-conj 9.4-rel posi}.
\end{proof}

\begin{cor}\label{cor-Conj9.3-Conj9.4-rel position}
	Assume that $n$, $k$, $p$ satisfy the assumptions in Conjecture \ref{conj-Min ineq} (or Conjecture \ref{conj-Min ineq-weak}). Then Conjecture \ref{conj-Min ineq} (or Conjecture \ref{conj-Min ineq-weak}) only depends on the relative position of the smooth uniformly h-convex bounded domains $K$ and $L$. Equivalently, the following statements are equivalent,
	\begin{enumerate}
		\item Conjecture \ref{conj-Min ineq} (or Conjecture \ref{conj-Min ineq-weak}) holds for $K$ and $L$,
		\item Conjecture \ref{conj-Min ineq} (or Conjecture \ref{conj-Min ineq-weak}) holds for $f(K)$ and $f(L)$,
	\end{enumerate}
where $f$ is any isometry of $\mathbb{H}^{n+1}$.
\end{cor}

\begin{proof}
	Let $f \in O^+(n+1,1)$ be any isometry of $\mathbb{H}^{n+1}$.
		By using \eqref{rel-area element} and \eqref{shifted curvature-support function}, we have 
		\begin{align*}
		 \int_{\mathbb{S}^n} \g_K^{-n}p_{n-k} (\tilde{\lambda}_K)d\sigma
		 = \int_{\partial K} p_k (\tilde{\kappa}) d\mu_K,
		\end{align*}
		which is preserved along $f$.
		Besides, we know that the modified quermassintegrals of h-convex bounded domains  are preserved along $f$, i.e. $\widetilde{W}_k (K) = \widetilde{W}_k (f(K))$ and $\widetilde{W}_k (L) = \widetilde{W}_k (f(L))$ for all $k=0,1,\ldots,n$. Then Corollary \ref{cor-Conj9.3-Conj9.4-rel position} follows directly from Proposition \ref{prop-Conj9.3-Conj9.4-rel position}.
\end{proof}

\subsubsection{Results about Conjecture \ref{conj-Min ineq}}
We will consider Conjecture \ref{conj-Min ineq} in several cases at first. When $K = L$,  Conjecture \ref{conj-Min ineq} is trivial. 
\begin{enumerate}
	\item When $n \geq 1$ and $k=n$, we will show that Conjecture \ref{conj-Min ineq} is equivalent to the H\"older inequalities in Theorem \ref{thm-k=n-strong Min}. 
	\item When $K$ is a geodesic ball centered at the origin and $L$ is a general smooth uniformly h-convex bounded domain, Conjecture \ref{conj-Min ineq} holds, which will be proved in Theorem \ref{thm-strong Min ineq-K ball}.
	\item  When $L$ is a geodesic ball and $K$ is general, we prove Conjecture \ref{conj-Min ineq} in Theorem \ref{thm-p=1-strong Min ineq, L ball} with assumption $p=1$. 
\end{enumerate}

\begin{thm}\label{thm-k=n-strong Min}
	Conjecture \ref{conj-Min ineq} holds in the case $n \geq 1$ and $k=n$.
\end{thm}

\begin{proof}
Let $u_K(z)$ and $u_L(z)$ be the horospherical support functions of $K$ and $L$ respectively, and let $\g_K(z) =e^{u_K(z)}$ and $\g_L(z) = e^{u_L(z)}$. 

\textbf{Case 1.} $p>0$. By the H\"older inequality, we have
\begin{align*}
\int_{\mathbb{S}^n} \g_K^{-n} d\sigma \leq \(\int_{\mathbb{S}^n} \g_L^{-n} d\sigma \)^{\frac{p}{n+p}} \(\int_{\mathbb{S}^n} \g_L^p \g_K^{-n-p} d\sigma \)^{\frac{n}{n+p}},
\end{align*}
with equality if and only if $\g_L^{-1}(z) \g_K(z)$ is constant on $\mathbb{S}^n$.
We rewrite the above inequality as
\begin{align}\label{k=n stron-Min phi p>0}
\int_{\mathbb{S}^n} \g_L^p \g_K^{-n-p} d\sigma 
- \int_{\mathbb{S}^n} \g_K^{-n} d\sigma
\geq \(\int_{\mathbb{S}^n} \g_K^{-n} d\sigma\)\( \(\int_{\mathbb{S}^n} \g_L^{-n} d\sigma \)^{-\frac{p}{n}}
\(\int_{\mathbb{S}^n} \g_K^{-n} d\sigma\)^{\frac{p}{n}}-1 \).
\end{align}
By \eqref{phi^-n-Wn}, we have
\begin{align}\label{k=n quer K,L}
\int_{\mathbb{S}^n} \g_K^{-n} d\sigma  =\omega_n e^{-n r^n_K}, \quad
\int_{\mathbb{S}^n} \g_L^{-n} d\sigma  =\omega_n e^{-n r^n_L}.
\end{align}
Substituting \eqref{k=n quer K,L} into \eqref{k=n stron-Min phi p>0}, we obtain
\begin{align*}
\int_{\mathbb{S}^n} \g_L^p \g_K^{-n-p} d\sigma 
- \int_{\mathbb{S}^n} \g_K^{-n} d\sigma
\geq
\omega_n e^{-n r^n_K} \( e^{p(r^n_L -r^n_K)}-1\),
\end{align*}
thus we obtain \eqref{formula-conj-Min ineq} in the case $n \geq 1$, $k=0$ and $p>0$.
Equality holds if and only if $K$ and $L$ are hyperbolic dilates.

\textbf{Case 2.} $-n<p<0$. By the H\"older inequality, we have
\begin{align*}
\int_{\mathbb{S}^n} \g_L^p \g_K^{-n-p} d\sigma
\leq  \(\int_{\mathbb{S}^n} \g_L^{-n} d\sigma \)^{-\frac{p}{n}}  \( \int_{\mathbb{S}^n} \g_K^{-n} d\sigma \)^{\frac{p}{n} +1},
\end{align*}
with equality if and only if $\g_L^{-1}(z) \g_K(z)$ is constant on $\mathbb{S}^n$.
We rewrite the above inequality as
\begin{align}\label{k=n stron-Min phi p<0}
\int_{\mathbb{S}^n} \g_L^p \g_K^{-n-p} d\sigma 
- \int_{\mathbb{S}^n} \g_K^{-n} d\sigma
\leq \(\int_{\mathbb{S}^n} \g_K^{-n} d\sigma\)\( \(\int_{\mathbb{S}^n} \g_L^{-n} d\sigma \)^{-\frac{p}{n}}
\(\int_{\mathbb{S}^n} \g_K^{-n} d\sigma\)^{\frac{p}{n}}-1 \).
\end{align}
Substituting \eqref{k=n quer K,L} into \eqref{k=n stron-Min phi p<0}, we obtain
\begin{align*}
\int_{\mathbb{S}^n} \g_L^p \g_K^{-n-p} d\sigma 
- \int_{\mathbb{S}^n} \g_K^{-n} d\sigma
\leq
\omega_n e^{-n r^n_K} \( e^{p(r^n_L -r^n_K)}-1\),
\end{align*} 
thus we obtain \eqref{formula-conj-Min ineq-2} in the case $n \geq 1$, $k=n$ and $-n<p<0$.
Equality holds if and only if $K$ and $L$ are hyperbolic dilates.

\textbf{Case 3.} $p=-n$. Using \eqref{k=n quer K,L}, we have
\begin{align*}
\int_{\mathbb{S}^n} \g_L^{-n} d\sigma - \int_{\mathbb{S}^n} \g_K^{-n} d\sigma
= \omega_n e^{-n r^n_K}\( e^{-n \(r^n_L -r^n_K \)} -1\),
\end{align*}
thus we obtain that equality always holds in \eqref{formula-conj-Min ineq-2} when $n \geq 1$, $k=n$ and $p=-n$. 
We complete the proof of Theorem \ref{thm-k=n-strong Min}.
\end{proof}
If $n \geq 1$, $k=n$ and $p>0$, then by use of \eqref{k=n quer K,L},
 it is easy to see that Conjecture \ref{conj-Min ineq-weak} is equivalent to Conjecture \ref{conj-Min ineq} in this case. Then we get the following Corollary \ref{cor-k=n-conj Min ineq-weak}.
 \begin{cor}\label{cor-k=n-conj Min ineq-weak}
 		Conjecture \ref{conj-Min ineq-weak} holds in the case $n \geq 1$ and $k=n$.
 \end{cor}

If we let $K =B(r)$ be a geodesic ball of radius $r>0$ centered at the origin, then it is straightforward to see that
\begin{align*}
 \g_K^{-n}p_{n-k} (\tilde{\lambda}_K)=&
 \sinh^{n-k}r^k_K e^{-kr^k_K},
 \\
  \g_K^{-p-n} p_{n-k}(\tilde{\lambda}_K)
  =&\sinh^{n-k}r^k_K e^{-(k+p)r^{k}_K}.
\end{align*}
Therefore inequalities \eqref{formula-conj-Min ineq} and \eqref{formula-conj-Min ineq-2} in Conjecture \ref{conj-Min ineq}  are equivalent to 
\begin{align*}
\int_{\mathbb{S}^n} \g_L^p d\sigma \geq& \omega_n e^{pr^k_L}, \quad p>0,\\
\int_{\mathbb{S}^n} \g_L^p d\sigma \leq& \omega_n e^{pr^k_L}, \quad  -n \leq p <0
\end{align*} 
respectively, where $n \geq 1$, $ 0 \leq k \leq n$, and $L$ is a smooth uniformly h-convex bounded domain in $\mathbb{H}^{n+1}$. In the next theorem, we will prove Conjecture \ref{conj-Min ineq} in the case that $K =B(r)$.

\begin{thm}\label{thm-strong Min ineq-K ball}
	Conjecture \ref{conj-Min ineq} holds when  $K$ is a geodesic ball centered at the origin. Precisely, if we assume $n \geq 1$, $0 \leq k \leq n$, $p \geq -n$ and $L$ is a smooth uniformly h-convex bounded domain in $\mathbb{H}^{n+1}$, then 
	\begin{equation}\label{Min ineq-K ball}
	\( \frac{\int_{\mathbb{S}^n} \g_L^p d\sigma}{\omega_n} \)^{\frac{1}{p}} \geq e^{r^k_L}.
	\end{equation}
	\begin{enumerate}
		\item When $0 \leq k \leq n-1$ and $p>-n$, equality holds if and only if $L$ is a geodesic ball centered at the origin.
		\item  When $0 \leq k \leq n-1$ and  $p=-n$, equality holds if and only if $L$ is a geodesic ball.
		\item When $k=n$ and $p>-n$, equality holds if and only if $L$ is a geodesic ball centered at the origin.
		\item When $k=n$ and $p=-n$, equality always holds.
	\end{enumerate}
\end{thm}

\begin{proof}
    Using the H\"older inequality (see e.g., \cite[p. 146]{GT01}), \eqref{phi^-n-Wn} and Theorem C, we have
	\begin{align}\label{pf-Min-K ball}
	\( \frac{\int_{\mathbb{S}^n} \g_L^p d\sigma}{\omega_n} \)^{\frac{1}{p}}
	\geq  \( \frac{\int_{\mathbb{S}^n} \g_L^{-n} d\sigma}{\omega_n} \)^{-\frac{1}{n}}= e^{ r^n_L} \geq e^{r^k_L}, 
	\end{align} 
	where $p \geq -n$. Thus we obtain \eqref{Min ineq-K ball}. Now we characterize the equality cases of \eqref{Min ineq-K ball}.
	
	If $0 \leq k \leq n$ and $p >-n$, then equality holds in the first inequality of \eqref{pf-Min-K ball}  if and only if  $\g_L(z)$ is constant on $\mathbb{S}^n$, which means that $L$ is a geodesic ball centered at the origin. By Theorem C, equality holds in the second inequality of \eqref{pf-Min-K ball}  when $L$ is a geodesic ball. Hence the Case (1) and the Case (3) in Theorem \ref{thm-strong Min ineq-K ball} are proved.
	
	If $0 \leq k \leq n-1$ and $p=-n$, then Theorem C implies that equality holds in the second inequality of \eqref{pf-Min-K ball}  if and only if $L$ is a geodesic ball. Thus Case (2) is proved. 
	Case (4) follows directly form \eqref{phi^-n-Wn}.
    We complete the proof of Theorem \ref{thm-strong Min ineq-K ball}.
\end{proof}

\begin{rem}
	Taking $p \to 0^+$ in \eqref{Min ineq-K ball}, the corresponding inequality is
	\begin{equation*}
	\int_{\mathbb{S}^n} u_L(z) d\sigma \geq \omega_n r^k_L,
	\end{equation*}
	equality holds if and only if $L$ is a geodesic ball centered at the origin,
	where $u_L(z)$ is the horospherical support function of a smooth uniformly h-convex bounded domain $L$.
\end{rem}

Since equality holds in \eqref{AF-modif-quermass}  when $\Omega$ is a geodesic ball, the following Corollary \ref{cor-weak Min-K ball} is a direct consequence of Theorem \ref{thm-strong Min ineq-K ball}, see Remark \ref{rem-rel-strong Min-weak Min}.
\begin{cor}\label{cor-weak Min-K ball}
		Conjecture \ref{conj-Min ineq-weak} holds when  $K$ is a geodesic ball centered at the origin.
\end{cor}

In \cite{HLW20}, Hu et al. introduced the following locally constrained curvature flow,
\begin{equation}\label{flow-HLW20}
\left\{\begin{aligned}
\frac{\partial}{\partial t}X(x,t)=&\left( \(\cosh r-\tilde{u}\) \frac{p_{k-1} (\tilde{\kappa})}{p_k (\tilde{\kappa})} - \tilde{u}\right) \nu , \quad k=1, \ldots,n,\\
X(\cdot,0)=& M_0,
\end{aligned}\right.
\end{equation}
 They proved the long time existence and convergence for flow \eqref{flow-HLW20} in the following Theorem D. Here we only focus on the case that $M_0$ is uniformly h-convex.

\begin{thmD}[\cite{HLW20}]
	Let $n \geq 2$ and $1 \leq k \leq n$, and let $M_0 =\partial \Omega_0$ be a smooth uniformly h-convex hypersurface in $\mathbb{H}^{n+1}$ such that $\Omega_0$ contains the origin in its interior. Then the flow \eqref{flow-HLW20} has a smooth solution for all time $t \in[0, +\infty)$, and $M_t$ is uniformly h-convex for each $t>0$ and converges smoothly and exponentially to a geodesic sphere of radius $r_\infty$ such that $\widetilde{W}_k (B_{r_\infty}) = \widetilde{W}_k(\Omega_0)$.
\end{thmD}

Using the flow \eqref{flow-HLW20}, Hu et al. \cite{HLW20} gave an alternative proof of Theorem C. Furthermore, they proved the following weighted Alexandrov-Fenchel inequalities. We state their result by use of the modified $k$-mean radius defined in \eqref{mod k-mean radius}.
\begin{thmE}[\cite{HLW20}]
	Let $n \geq 2$ and $1 \leq k \leq n$ be integers. Assume that $\Omega$ is a smooth uniformly h-convex bounded domain in $\mathbb{H}^{n+1}$ which contains the origin in its interior.  Then there holds
	\begin{equation}\label{ineq-HLW20}
	\int_{\partial \Omega} (\cosh r - \tilde{u}) p_k(\tilde{\kappa}) d\mu \geq \omega_n \sinh^{n-k} r^k_\Omega e^{-(k+1) r^k_\Omega}.
	\end{equation}
	Equality holds if and only if $\Omega$ is a geodesic ball centered at the origin.
\end{thmE}

Hu et al. \cite{HLW20} found that the left-hand side of \eqref{ineq-HLW20} is monotone non-increasing along the flow \eqref{flow-HLW20}, while the right-hand side of \eqref{ineq-HLW20} is invariant along the flow. Hence they proved Theorem E by applying the convergence result in Theorem D. In the following Lemma \ref{lem-mono-quan-2-HLW20}, we find a new geometric quantity which is also monotone non-increasing along the flow \eqref{flow-HLW20}.
\begin{lem}\label{lem-mono-quan-2-HLW20}
	Let $n \geq 2$ and $1 \leq k \leq n-1$ be integers. Along the flow \eqref{flow-HLW20}, we have
	\begin{equation}\label{mono-quan-2-HLW-1}
	\frac{d}{dt} \int_{\mathbb{S}^n} \(\frac{c_0}{\g(z,t)}-1 \) \g^{-n}(z,t) p_{n-k} (\tilde{\lambda}(G_t^{-1}(z))) d\sigma \leq 0,
	\end{equation}
	for any fixed number $c_0 \geq 1$, which is equivalent to
	\begin{equation}\label{mono-quan-2-HLW-2}
	\frac{d}{dt} \int_{M_t} \({c_0} ( \cosh r -\tilde{u}) -1 \) p_k (\tilde{\kappa}) d\mu_t \leq 0.
	\end{equation}
	Equality holds if and only if $M_t$ is a geodesic sphere centered at the origin.
\end{lem}

\begin{proof}
	The equivalence of \eqref{mono-quan-2-HLW-1} and \eqref{mono-quan-2-HLW-2} follows from  \eqref{rel-area element},  \eqref{shifted curvature-support function} and \eqref{1/phi, coshr-u}.
	The following evolution equation was calculated in \cite[Section 9.1]{HLW20},
	\begin{equation*}
	\frac{d}{dt} \int_{M_t} c_0 (\cosh r - \tilde{u}) p_k (\tilde{\kappa}) d \mu_t
	= c_0 \int_{M_t} \(  (n-k) \cosh r p_{k+1}(\tilde{\kappa}) - (k+1) (\cosh r-\tilde{u}) p_k (\tilde{\kappa}) \) \mathscr{F} d\mu_t,
	\end{equation*}
	where $\mathscr{F}$ represents the speed function of the flow \eqref{flow-HLW20}. Taking the time derivative of \eqref{induction def of modi-quer-intg} and using Lemma \ref{lem-variation of modified quermass}, we have
	\begin{equation*}
	\frac{d}{dt}\int_{{M_t}} p_k (\tilde{\kappa}) d\mu_t
	= (n-2k) \int_{M_t} p_k(\tilde{\kappa}) \mathscr{F} d\mu_t
	+ (n-k) \int_{M_t} p_{k+1} (\tilde{\kappa}) \mathscr{F} d\mu_t.
	\end{equation*}
	Hence, 
	\begin{align}
	&\frac{d}{dt} \int_{M_t} \({c_0} ( \cosh r -\tilde{u}) -1 \) p_k (\tilde{\kappa}) d\mu_t  \nonumber\\
	=& c_0(n-k) \int_{M_t} \cosh r p_{k+1} (\tilde{\kappa}) \mathscr{F} d\mu_t
	- (n-k) \int_{M_t} p_{k+1} (\tilde{\kappa}) \mathscr{F} d\mu_t  \nonumber\\
	& -c_0(k+1) \int_{M_t}  (\cosh r -\tilde{u})p_k (\tilde{\kappa})\mathscr{F} d\mu_t
	-(n -2k) \int_{{M_t}} p_k(\tilde{\kappa}) \mathscr{F} d\mu_t.\label{new mono quan-HLW-1}
	\end{align}
	Since in the flow \eqref{flow-HLW20}, 
	\begin{equation}\label{speed func-HLW20}
	\mathscr{F} = (\cosh r-\tilde{u}) \frac{p_{k-1}  (\tilde{\kappa})}{p_k (\tilde{\kappa})}  - \tilde{u},
	\end{equation}
	we have
	\begin{align}
	&c_0(n-k) \int_{M_t} \cosh r p_{k+1} (\tilde{\kappa}) \mathscr{F} d\mu_t
	- (n-k) \int_{M_t} p_{k+1} (\tilde{\kappa}) \mathscr{F} d\mu_t  \nonumber\\
	=&c_0(n-k) \int_{M_t} \cosh r \( (\cosh r-\tilde{u}) \frac{p_{k+1} (\tilde{\kappa}) p_{k-1} (\tilde{\kappa})}{p_k (\tilde{\kappa})}  -p_{k+1} (\tilde{\kappa}) \tilde{u}\) d\mu_t  \nonumber\\
	&-(n-k) \int_{M_t} \( (\cosh r-\tilde{u})\frac{p_{k+1} (\tilde{\kappa}) p_{k-1} (\tilde{\kappa})}{p_k (\tilde{\kappa})} -  p_{k+1} (\tilde{\kappa}) \tilde{u}\) d\mu_t. \label{before N-M-mono-quan}
	\end{align}
	As $c_0 \geq 1$, we have $c_0(n-k) \cosh r- (n-k) \geq (c_0 -1) (n-k) \geq 0$. Using \eqref{Newton ineq} in the right-hand side of \eqref{before N-M-mono-quan}, we get
	\begin{align}
	&c_0(n-k) \int_{M_t} \cosh r p_{k+1} (\tilde{\kappa}) \mathscr{F} d\mu_t
	- (n-k) \int_{M_t} p_{k+1} (\tilde{\kappa}) \mathscr{F} d\mu_t   \nonumber\\
	\leq& c_0(n-k) \int_{M_t} \cosh r \( (\cosh r -\tilde{u})p_k(\tilde{\kappa}) - p_{k+1} (\tilde{\kappa}) \tilde{u} \) d\mu_t   \nonumber\\
	&-
	(n-k)\int_{M_t} \( (\cosh r -\tilde{u})p_k(\tilde{\kappa}) - p_{k+1} (\tilde{\kappa}) \tilde{u} \) d\mu_t  \nonumber\\
	=&c_0(n-k)\int_{M_t} \cosh r \( (\cosh r -\tilde{u})p_k(\tilde{\kappa}) - p_{k+1} (\tilde{\kappa}) \tilde{u} \) d\mu_t,\label{new mono quan-HLW-1 line}
	\end{align}
	where we used \eqref{eq-shifted Minkowski formula} in the equality.
	Using \eqref{speed func-HLW20} in the last line of \eqref{new mono quan-HLW-1}, we have
	\begin{align}
	&-c_0(k+1) \int_{M_t}  (\cosh r -\tilde{u})p_k (\tilde{\kappa})\mathscr{F} d\mu_t
	-(n -2k) \int_{{M_t}} p_k(\tilde{\kappa}) \mathscr{F} d\mu_t   \nonumber\\
	=& -c_0(k+1) \int_{M_t} (\cosh r-\tilde{u}) \( (\cosh r-\tilde{u})p_{k-1}(\tilde{\kappa} ) - p_k (\tilde{\kappa}) \tilde{u} \) d\mu_t    \nonumber\\
	&-(n-2k) \int_{M_t} \( (\cosh r-\tilde{u})p_{k-1}(\tilde{\kappa})- p_k (\tilde{\kappa})\tilde{u} \)d\mu_t  \nonumber\\
	=&-c_0(k+1) \int_{M_t} (\cosh r-\tilde{u})\( (\cosh r-\tilde{u})p_{k-1}(\tilde{\kappa} ) - p_k (\tilde{\kappa}) \tilde{u} \) d\mu_t,\label{new mono quan-HLW-2 line}
	\end{align}
	where we used \eqref{eq-shifted Minkowski formula} in the last equality.
	Substituting \eqref{new mono quan-HLW-1 line} and \eqref{new mono quan-HLW-2 line} into \eqref{new mono quan-HLW-1} yields
	\begin{align*}
	&\frac{d}{dt} \int_{M_t} \({c_0} ( \cosh r -\tilde{u}) -1 \) p_k (\tilde{\kappa}) d\mu_t\\
	\leq& 
	c_0(n-k)\int_{M_t} \cosh r \( (\cosh r -\tilde{u})p_k(\tilde{\kappa}) - p_{k+1} (\tilde{\kappa}) \tilde{u} \) d\mu_t\\
	&-c_0(k+1) \int_{M_t} (\cosh r-\tilde{u})\( (\cosh r-\tilde{u})p_{k-1}(\tilde{\kappa} ) - p_k (\tilde{\kappa}) \tilde{u} \) d\mu_t\\
	=& c_0 \frac{n-k}{k+1} \int_{M_t} \cosh r \dot{p}_{k+1}^{ij} (\tilde{\kappa}) \nabla_j \nabla_i \cosh r d\mu_t
	-c_0 \frac{k+1}{k} \int_{M_t} (\cosh r-\tilde{u}) \dot{p}_k^{ij}(\tilde{\kappa}) \nabla_j \nabla_i \cosh r d\mu_t,
	\end{align*}
	where we used \eqref{shifted-Min-fml-diverg-formula} in the equality. Since $\dot{p}_{k+1}^{ij} (\tilde{\kappa})$ and $\dot{p}_k^{ij} (\tilde{\kappa})$ are positive definite and divergence-free, integration by parts gives
	\begin{align}
	&\frac{d}{dt} \int_{M_t} \({c_0} ( \cosh r -\tilde{u}) -1 \) p_k (\tilde{\kappa}) d\mu_t \nonumber\\
	\leq&
	-c_0 \frac{n-k}{k+1} \int_{M_t}\dot{p}_{k+1}^{ij}(\tilde{\kappa}) \nabla_i \cosh r \nabla_j \cosh r d\mu_t
	+c_0 \frac{k+1}{k} \int_{M_t} \dot{p}_k^{ij}(\tilde{\kappa}) \nabla_i \cosh r \nabla_j (\cosh r-\tilde{u}) d\mu_t  \nonumber\\
	=&-c_0 \frac{n-k}{k+1} \int_{M_t}  \dot{p}_{k+1}^{ij}(\tilde{\kappa}) \metric{V}{\partial_i X} \metric{V}{\partial_j X} d\mu_t
	-c_0 \frac{k+1}{k} \int_{M_t} \dot{p}_k^{ij}(\tilde{\kappa}) \tilde{h}_{i}{}^l \metric{V}{\partial_l X} \metric{V}{\partial_j X} d\mu_t  \nonumber\\
	\leq& 0, \label{mono-last-2-HLW}
	\end{align}
	where we used \eqref{1/phi i new} in the equality. Then obtain \eqref{mono-quan-2-HLW-2}.
	
    By the uniform h-convexity of $M_t$ and \eqref{di cosh r}, equality holds in the last inequality of \eqref{mono-last-2-HLW} if and only if $\nabla \cosh r =0$, which means that $M_t$ is a geodesic sphere centered at the origin. This completes the proof of Lemma \ref{lem-mono-quan-2-HLW20}.
\end{proof}

Using the new monotone quantity of the flow \eqref{flow-HLW20}  proved in Lemma \ref{lem-mono-quan-2-HLW20}, we can prove that Conjecture \ref{conj-Min ineq} holds in a special case in the following Theorem \ref{thm-p=1-strong Min ineq, L ball}.

\begin{thm}\label{thm-p=1-strong Min ineq, L ball}
	Conjecture \ref{conj-Min ineq} holds when $n \geq 2$, $1 \leq k \leq n-1$, $p=1$, $K$ is a smooth uniformly h-convex bounded domain which contains the origin in its interior, and $L$ is a geodesic ball centered at the origin. Precisely, if $L = B(r_L)$, then
	\begin{align}\label{p=1-strong Min ineq, L ball}
	e^{ r_L}\int_{\mathbb{S}^n}  \g_K^{-1-n} p_{n-k}(\tilde{\lambda}_K) d\sigma
	- \int_{\mathbb{S}^n} \g_K^{-n}p_{n-k} (\tilde{\lambda}_K)d\sigma
	\geq
	\omega_n \sinh^{n-k}r^k_K e^{-kr^k_K} \left( e^{(r_L - r^k_K)}-1 \right)
	\end{align}  
	for any smooth uniformly h-convex bounded domain $K$ which contains the origin in its interior. Particularly, when $L$ degenerates to the origin, we have
	\begin{align}\label{p=1-strong Min ineq, L point}
	\int_{\partial K} (\cosh r-\tilde{u})p_k (\tilde{\kappa}) d\mu
	-\omega_n \sinh^{n-k} r^k_K e^{-(k+1) r^k_K}
	\geq
	\int_{\partial K} p_k(\tilde{\kappa}) d\mu - \omega_n \sinh^{n-k} r^k_K e^{-k r^k_K}
	\geq 0.
	\end{align}
	In the both inequalities, equality holds if and only if $K$ is a geodesic ball centered at the origin.
\end{thm}

\begin{proof}
	Taking $\Omega_0 = K$ in Theorem D, as the flow \eqref{flow-HLW20} preserves $\widetilde{W}_k (\Omega_t)$, we have $r_\infty = r^k_K$. Taking $c_0 = e^{r_L} \geq 1$ in Lemma \ref{lem-mono-quan-2-HLW20} and using  Theorem D, we obtain
	\begin{align*}
	\int_{\mathbb{S}^n} (e^{r_L} \g_K^{-1}-1) \g_K^{-n} p_{n-k} (\tilde{\lambda}_K) d\sigma
	\geq& \omega_n (e^{r_L} e^{-r_\infty} -1) e^{-n r_\infty} \sinh^{n-k} r_\infty e^{(n-k) r_\infty} \nonumber\\
	=& \omega_n (e^{r_L} e^{-r^k_K} -1) e^{-n r^k_K} \sinh^{n-k} r^k_K e^{(n-k) r^k_K},
	\end{align*}
	and hence obtain \eqref{p=1-strong Min ineq, L ball}. Equality holds in \eqref{p=1-strong Min ineq, L ball} if and only if $K$ is a geodesic ball centered at the origin.
	
	By using  \eqref{rel-area element}, \eqref{shifted curvature-support function} and \eqref{1/phi, coshr-u}, we have that inequality \eqref{p=1-strong Min ineq, L ball} is equivalent to
	\begin{equation*}
	e^{r_L} \int_{\partial K} (\cosh r-\tilde{u})p_k(\tilde{\kappa}) d\mu -
	\int_{\partial K} p_k(\tilde{\kappa}) d\mu
	\geq \omega_n  \sinh^{n-k} r^k_K ( e^{r_L-r^k_K} -1) e^{-k r^k_K}.
	\end{equation*}
	Taking $r_L =0$ in the above inequality and using \eqref{AF-modif-quermass}, we obtain \eqref{p=1-strong Min ineq, L point}. We complete the proof of Theorem \ref{thm-p=1-strong Min ineq, L ball}.
\end{proof}

\begin{rem}
	We notice that Theorem  \ref{thm-p=1-strong Min ineq, L ball} implies Theorem E directly.
\end{rem}

\subsubsection{Results about Conjecture \ref{conj-Min ineq-weak}}
In the last part of this subsection, we will study Conjecture \ref{conj-Min ineq-weak}. Although Conjecture \ref{conj-Min ineq-weak} can be obtained form Conjecture \ref{conj-Min ineq}  by using  \eqref{AF-modif-quermass}, Conjecture \ref{conj-Min ineq} has not been proved in general case. It makes sense for us to discuss Conjecture \ref{conj-Min ineq-weak} here.

Let $K$ contain the origin and $L$ be a geodesic ball centered at the origin in Conjecture \ref{conj-Min ineq-weak}. By Theorem \ref{thm-p=1-strong Min ineq, L ball} or Theorem E, we have Conjecture \ref{conj-Min ineq-weak} holds for $n \geq 2$, $1 \leq k \leq n-1$, $p=1$ in this case. 
Now we prove Conjecture \ref{conj-Min ineq-weak} in the case $n\geq 2$, $1 \leq k \leq n-1$ and $p \geq 1$ by using Theorem \ref{thm-p=1-strong Min ineq, L ball}. 

\begin{prop}\label{prop-cor of strong Min ineq}
	Let $n \geq 2$, $1 \leq k \leq n-1$ be integers, $p \geq 1 $ be a positive number and  $\Omega \subset \mathbb{H}^{n+1}$ be a smooth uniformly h-convex bounded domain which contains the origin in its interior, then
	\begin{equation*}
	\int_{\partial \Omega} (\cosh r-\tilde{u})^p p_k( \tilde{\kappa}) d\mu
	\geq \(\int_{\partial \Omega} p_k (\tilde{\kappa}) d\mu \) e^{- pr^k_\Omega}.
	\end{equation*}
	Equality holds if and only if $\Omega$ is a geodesic ball centered at the origin.
\end{prop}

\begin{proof}
	By Theorem  \ref{thm-p=1-strong Min ineq, L ball}, we have
	\begin{align}
	\int_{\partial \Omega} (\cosh r-\tilde{u})p_k (\tilde{\kappa})
	\geq&  \int_{\partial \Omega} p_k(\tilde{\kappa}) d\mu - \omega_n \sinh^{n-k} r^k_\Omega e^{-k r^k_\Omega} (1- e^{-r^k_K}) \nonumber\\
	\geq& \(\int_{\partial \Omega} p_k (\tilde{\kappa}) d\mu \) e^{- r^k_\Omega}, \label{mid step 1}
	\end{align}
	where the second inequality follows from \eqref{AF-modif-quermass}. Equality holds if and only if $ \Omega$ is a geodesic ball centered at the origin.
	Hence the $p=1$ case of Proposition \ref{prop-cor of strong Min ineq} is proved. From the H\"older inequality and \eqref{mid step 1}, we can easily get 
	\begin{equation*}
	\(  \frac{\int_{\partial \Omega} (\cosh r -\tilde{u})^p p_k(\tilde{\kappa}) d\mu}{ \int_{\partial \Omega} p_k(\tilde{\kappa})  d\mu} \)^{\frac{1}{p}} \geq \frac{\int_{\partial \Omega} (\cosh r -\tilde{u}) p_k(\tilde{\kappa}) d\mu}{ \int_{\partial \Omega} p_k(\tilde{\kappa})  d\mu} \geq e^{-r^k_\Omega}
	\end{equation*}
	for $p \geq 1$. Equality holds if and only if $\Omega$ is a geodesic ball centered at the origin.
	 This completes the proof of Proposition \ref{prop-cor of strong Min ineq}. 
\end{proof}
Using \eqref{AF-modif-quermass} in Proposition \ref{prop-cor of strong Min ineq}, we can get the following Theorem \ref{thm-HLW- p geq 1}.
\begin{thm}\label{thm-HLW- p geq 1}
	Conjecture \ref{conj-Min ineq-weak} holds in the case $n \geq 2$, $1 \leq k \leq n-1$, $p \geq 1$, $K$ is a smooth uniformly h-convex bounded domain which contains the origin in its interior and $L$ is a geodesic ball centered at the origin. Precisely, we have
	\begin{equation}\label{HLW- p geq 1}
	\int_{\partial \Omega} (\cosh r-\tilde{u})^p p_k(\tilde{\kappa}) d\mu
	\geq
	\omega_n \sinh^{n-k} r^k_\Omega e^{-(k+p) r^k_\Omega},
	\end{equation}
	where $\Omega$ is a smooth uniformly h-convex bounded domain which contains the origin in its interior.
	Equality holds in \eqref{HLW- p geq 1} if and only if $\Omega$ is a geodesic ball centered at the origin.
\end{thm}

Now we study the case $n \geq 1$, $k=0$ and $p \geq 1$ in Conjecture \ref{conj-Min ineq-weak}. The following Proposition \ref{prop-ineq-weighted vol-cano col} will give a comparison between weighted volume and the canonical volume of bounded domains in $\mathbb{H}^{n+1}$.
\begin{prop}\label{prop-ineq-weighted vol-cano col}
	Let $\Omega$ be a bounded domain in $\mathbb{H}^{n+1}$ $(n \geq 1)$, then we have
	\begin{align}\label{ineq-weighted vol-cano col}
	\Vol_w(\Omega) =\int_{\Omega} \cosh r dv \geq \frac{\omega_n}{n+1} \sinh^{n+1} r^0_\Omega,
	\end{align}
	with equality if and only if $\Omega$ is a geodesic ball centered at the origin.
\end{prop}
\begin{proof}
	It is easy to see that $\Vol \(\Omega \backslash B (r^0_\Omega) \) = \Vol \(B(r^0_\Omega) \backslash \Omega \)$ by the definition of $r^0_\Omega$ in \eqref{mod k-mean radius}. 
	Noting that $\cosh r \geq \cosh r^0_\Omega$ on  $\Omega \backslash B (r^0_\Omega)$ and $\cosh r \leq  \cosh r^0_\Omega$ in $B(r^0_\Omega)$, we have
	\begin{align*}
	\Vol_w(\Omega) =& \Vol_w \(\Omega \cap  B (r^0_\Omega) \) + \Vol_w \(\Omega \backslash B(r^0_\Omega) \)\\
	\geq& \Vol_w \(\Omega \cap  B (r^0_\Omega) \)+ \cosh r^0_\Omega \Vol\( \Omega \backslash B(r^0_\Omega) \)\\
	=&\Vol_w \( B (r^0_\Omega) \cap \Omega \) + \cosh r^0_\Omega \Vol\( B(r^0_\Omega) \backslash \Omega \)\\
	\geq& \Vol_w \( B (r^0_\Omega) \cap \Omega \) + \Vol_w \( B(r^0_\Omega\backslash \Omega) \)
	= \Vol_w \( B(r^0_\Omega)\).
 	\end{align*}	
	Equality holds if and only if $\Omega$ is a geodesic ball centered at the origin.
\end{proof}
Here we state a weighted isoperimetric inequality proved by Scheuer and Xia \cite{SX19}.
\begin{thmF}[\cite{SX19}]
	Let $\Omega$ be a smooth uniformly h-convex bounded domain in $\mathbb{H}^{n+1}$ $(n \geq 1)$ which contains the origin in its interior. Then
	\begin{equation*}
	\int_{\partial \Omega} \cosh r d\mu
	\geq
	\( \( (n+1)\Vol_w(\Omega) \)^2+
	\omega_n^{ \frac{2}{n+1}} \( (n+1)\Vol_w(\Omega) \)^{\frac{2n}{n+1}}
	\)^{\frac{1}{2}}.
	\end{equation*}
	Equality holds if and only if $\Omega$ is a geodesic ball centered at the origin.
\end{thmF}
\begin{rem}
	We remark that Theorem F has been recently improved by the authors \cite{LX} that the above weighted isoperimetric inequality actually holds for any domain $\Omega$ with $C^1$ boundary. Besides, Theorem F was recently generalized to a class of weighted Alexandrov-Fenchel inequalities by Hu and the first author \cite{HL21}.
\end{rem}
In the following Theorem \ref{thm-HLW-k=0-p=1}, by using Theorem F and Proposition \ref{prop-ineq-weighted vol-cano col}, we will prove that \eqref{ineq-HLW20} holds in the case $k=0$.
\begin{thm}\label{thm-HLW-k=0-p=1}
	 Conjecture \ref{conj-Min ineq-weak} holds in the case $n \geq 1$, $k=0$, $p=1$, $K$ is a smooth h-convex bounded domain in $\mathbb{H}^{n+1}$ which contains the origin in its interior and $L$ is a geodesic ball centered at the origin.
	 Precisely, we have
	\begin{equation}\label{HLW-k=0-p=1}
	 \int_{\partial \Omega} (\cosh r- \tilde{u}) d\mu \geq \omega_n \sinh^n r^0_\Omega e^{-r^0_\Omega},
	 \end{equation}
	where $\Omega \subset \mathbb{H}^{n+1}$ is a smooth uniformly h-convex bounded domain which contains the origin in its interior. 
	Equality holds in \eqref{HLW-k=0-p=1} if and only if $\Omega$ is a geodesic ball centered at the origin.
\end{thm}
\begin{proof}
	By \eqref{weighted volume} and Theorem F, we have
	\begin{align}
	&\int_{\partial \Omega} (\cosh r- \tilde{u}) d\mu
	= \int_{\partial \Omega} \cosh r d\mu - (n+1) \Vol_w(\Omega)  \nonumber\\
	\geq& 
	\( \( (n+1)\Vol_w(\Omega) \)^2+
	\omega_n^{ \frac{2}{n+1}} \( (n+1)\Vol_w(\Omega) \)^{\frac{2n}{n+1}}
	\)^{\frac{1}{2}} - (n+1) \Vol_w(\Omega).\label{poly-Vw}
	\end{align}
	It is easy to see that the right-hand side of \eqref{poly-Vw} is monotone increasing as a single-variable function of  $\Vol_w(\Omega)$. Then by using \eqref{ineq-weighted vol-cano col}, we have
	\begin{align*}
	\int_{\partial \Omega} (\cosh r- \tilde{u}) d\mu
	\geq& \( \(\omega_n \sinh^{n+1} r^0_\Omega \)^2 + \omega_n^{\frac{2}{n+1} } \(\omega_n \sinh^{n+1} r^0_\Omega \)^{\frac{2n}{n+1}}    \)^{\frac{1}{2}} -\omega_n \sinh^{n+1} r^0_\Omega\\
	=& \omega_n \sinh^n r^0_\Omega \( \(\sinh^2 r^0_\Omega+1\)^{\frac{1}{2}} - \sinh r^0_\Omega  \)\\
	=& \omega_n \sinh^n r^0_\Omega e^{-r^0_\Omega}.
	\end{align*}
	Then we obtain \eqref{HLW-k=0-p=1} and the equality holds if and only if $\Omega$ is a geodesic ball centered at the origin. Hence, Theorem \ref{thm-HLW-k=0-p=1} follows from  \eqref{HLW-k=0-p=1} and Remark \ref{rem-rel-strong Min-weak Min}.
\end{proof}

\subsection{Horospherical $p$-Brunn-Minkowski inequalities} $ \ $

In this subsection, we study the horospherical $p$-Brunn-Minkowski inequalities which we proposed in Conjecture \ref{conj-BM}.

	We have shown in Proposition \ref{prop-reduce conj BM - a+b=1 case} that Conjecture \ref{conj-BM} can be obtained from the special case where $a \geq 0$, $b \geq 0$ and $a+b=1$. We proved that Conjecture \ref{conj-BM} holds for $k=n$ in Theorem \ref{thm-BM-k=n}. When $K =L$, we proved Conjecture \ref{conj-BM} in Proposition \ref{prop-BM-K=L}. In the following propositions, we will prove Conjecture \ref{conj-BM} holds for 
	\begin{enumerate}
		\item $n \geq 2$, $1 \leq  k \leq n-1$, $1 \leq p \leq 2$, $a=1$, $b \geq 0$, $K$ is a smooth uniformly h-convex bounded domain in $\mathbb{H}^{n+1}$ which contains the origin in its interior and $L$ is a geodesic ball centered at the origin,
		\item $n \geq 1$, $k =0$, $p=1$, $a=1$, $b \geq 0$, $K$ is a smooth uniformly h-convex bounded domain in $\mathbb{H}^{n+1}$ which contains the origin in its interior and $L$ is a geodesic ball centered at the origin.
	\end{enumerate}

\begin{prop}\label{prop-BM-L-ball-p geq 1-k geq 1}
	Conjecture \ref{conj-BM} holds in the case $n \geq 2$, $1 \leq k \leq n-1$, $1 \leq p \leq 2$, $a=1$, $b =t \geq 0$, $K$ is a smooth uniformly h-convex bounded domain in $\mathbb{H}^{n+1}$ which contains the origin  in its interior and $L$ is a geodesic ball of radius $r_L \geq 0$ centered at the origin. Precisely, we have
	\begin{align}\label{BM-L-ball-p geq 1-k geq 1}
	\exp \( p I_k^{-1} (\widetilde{W}_k(\Omega_t)) \) \geq \exp \( p I_k^{-1} (\widetilde{W}_k(K))\)
	+t \exp \( p I_k^{-1} (\widetilde{W}_k(B(r_L))) \),
	\end{align}
	where $t \geq  0$ and $\Omega_t = K +_p t\cdot B(r_L)$. If $t>0$, then equality holds in \eqref{BM-L-ball-p geq 1-k geq 1} if and only if $K$ is a geodesic ball centered at the origin.
\end{prop}

\begin{proof}
	From Definition \ref{def-p sum}, we have $ \g_{\Omega_t}(z) :=\g(z,t)= \(\g_K^p(z) + t \g_L^p(z) \)^{\frac{1}{p}}$. Hence, at $t=0$ we have 
	\begin{equation*}
	\frac{\partial}{\partial t} u (z,t)= \g_K^{-1} \partial_t \g(z,t) = \frac{1}{p} \(\frac{\g_L(z)}{\g_K(z)} \)^p.  
	\end{equation*}
	Therefore, by using Lemma \ref{lem-variation of modified quermass}, \eqref{1/phi, coshr-u} and $\g_L(z) \equiv e^{r_L}$, we have
	\begin{equation}\label{dt Wk-p geq 1- k geq 1-a=1}
	\left. \frac{d}{dt} \right|_{t=0} \widetilde{W}_k(\Omega_t) = \frac{1}{p}\int_{\partial K} \(\frac{\g_L}{\g_K} \)^p p_k (\tilde{\kappa}) d\mu
	=\frac{e^{p r_L}}{p} \int_{\partial K} (\cosh r-\tilde{u})^p p_k (\tilde{\kappa}) d\mu
	\end{equation}
	For abbreviation, let $r^k_t = I_k^{-1} ((\widetilde{W}_k (\Omega_t))$ and $\chi (t) = e^{p r^k_t}$.
	A similar calculation as in the proof of Proposition \ref{prop-BM-K=L} gives
	\begin{align}
	\chi'(0) =& p \( \omega_n \sinh^{n-k} r^k_K e^{-k r^k_K} \)^{-1} e^{p r^k_K} \left. \frac{d}{dt} \right|_{t=0} \widetilde{W}_k(\Omega_t)   \nonumber\\
	=&e^{p r_L} \( \omega_n \sinh^{n-k} r^k_K e^{-k r^k_K} \)^{-1} e^{p r^k_K} \int_{\partial K} (\cosh r-\tilde{u})^p p_k (\tilde{\kappa}) d\mu
	\geq e^{p r_L},\label{eq-use general HLW}
	\end{align}
	where we used Theorem \ref{thm-HLW- p geq 1} in the last inequality. Equality holds in \eqref{eq-use general HLW} if and only if $K = \Omega_0$ is a geodesic ball centered at the origin.
	Instead of $K$, if we set $\Omega_{t_0}$ as the initial uniformly h-convex bounded domain, then we can get $\chi'(t_0) \geq e^{p r_L}$ by the same calculation as in \eqref{eq-use general HLW}.
	Hence 
	\begin{align}\label{chi(t) geq chi(0)}
	\chi(t) \geq \chi(0) +t e^{p r_L} = \chi(0) + t \exp \(p I_k^{-1}(\widetilde{W}_k (B(r_L) )) \).
	\end{align}
	Thus \eqref{BM-L-ball-p geq 1-k geq 1} follows by the above inequality and the definition of $\chi(t)$ in this proof.
	
	By the equality case of \eqref{eq-use general HLW}, if $t>0$ and equality holds in \eqref{chi(t) geq chi(0)}, then $K$ is a geodesic ball centered at the origin. Conversely, if $K$ is a geodesic ball centered at the origin, then for all $t \geq 0$, it is easy to see that equality holds in \eqref{chi(t) geq chi(0)}. 
	This completes the proof of Proposition \ref{prop-BM-L-ball-p geq 1-k geq 1}.
\end{proof}

Instead of using Theorem \ref{thm-HLW- p geq 1}, if we let $p=1$ and use Theorem \ref{thm-HLW-k=0-p=1} in \eqref{eq-use general HLW}, then we can prove the following Proposition \ref{prop-BM-k=0-p=1-L ball} by a similar argument as in Proposition \ref{prop-BM-L-ball-p geq 1-k geq 1}. Here we omit the proof.
\begin{prop}\label{prop-BM-k=0-p=1-L ball}
	Conjecture \ref{conj-BM} holds in the case $n\geq 1$, $k = 0$, $p=1$, $a=1$, $b =t \geq 0$,  $K$ is a smooth uniformly h-convex bounded domain in $\mathbb{H}^{n+1}$ which contains the origin in its interior and $L$ is a geodesic ball of radius $r_L \geq 0$ centered at the origin, i.e.
	\begin{equation*}
	\exp \(I_0^{-1} (\widetilde{W}_0(\Omega_t)) \) \geq \exp \(  I_0^{-1} (\widetilde{W}_0(K))\)
	+t \exp \(I_0^{-1} (\widetilde{W}_k(B(r_L))) \),
	\end{equation*}
	where $\Omega_t = K +_1 t\cdot B(r_L)$. If $t>0$, then equality holds if and only if $K$ is a geodesic ball centered at the origin.
\end{prop}

\subsection{A summary of the results about conjectures}$ \ $

In this subsection, we summarize the results of Conjecture \ref{conj-BM}, Conjecture \ref{conj-Min ineq} and Conjecture \ref{conj-Min ineq-weak},  which have been proved in the previous subsections. According to Proposition \ref{prop-sums not rely on origin} and  Corollary \ref{cor-Conj9.3-Conj9.4-rel position}, the assumptions in these results can be weakened.
\begin{thm}\label{thm-sum-BM}
	Conjecture \ref{conj-BM} holds in one of the following cases,
	\begin{enumerate}
		\item $n \geq 1$ and $k =n$,
		\item $K=L$,
		\item $n \geq 2$, $1 \leq k \leq n-1$, $1 \leq p \leq 2$, $a=1$, $b \geq 0$, $K$ is a smooth uniformly h-convex bounded domain in $\mathbb{H}^{n+1}$ and $L$ is a geodesic ball centered at an interior point of $K$,
		\item $n \geq 1$, $k=0$, $p=1$, $a =1$, $b \geq 0$, $K$ is a smooth uniformly h-convex bounded domain in $\mathbb{H}^{n+1}$  and  $L$ is a geodesic ball centered at an interior point of $K$.
	\end{enumerate}
\end{thm}
\begin{proof}
	Case $(1)$ was proved in Theorem \ref{thm-BM-k=n}. Case $(2)$ was proved in Proposition \ref{prop-BM-K=L}.
	For fixed $n$, $k$, $p$, $a$ and $b$,
	Proposition \ref{prop-sums not rely on origin} asserts that Conjecture \ref{conj-BM} only depends on the relative position of $K$ and $L$. Hence we can assume that $K$ contains the origin and $L$ is a geodesic ball centered at the origin in Case $(3)$ and Case $(4)$.
	Then Case $(3)$ and Case $(4)$ follow from  Proposition \ref{prop-BM-L-ball-p geq 1-k geq 1} and Proposition \ref{prop-BM-k=0-p=1-L ball} respectively.
\end{proof}
\begin{thm}\label{thm-sum-Min ineq-strong}
	Conjecture \ref{conj-Min ineq} holds in one of the following cases,
	\begin{enumerate} 
		\item $n\geq 1$ and $k=n$,
		\item $K =L$,
			\item $n\geq 2$, $1 \leq k \leq n-1$, $p=1$, $K$ is a smooth uniformly h-convex bounded domain in $\mathbb{H}^{n+1}$  and $L$ is a geodesic ball centered at an interior point of $K$,
		\item $K$ is a geodesic ball and $L$ is a smooth uniformly h-convex bounded domain in $\mathbb{H}^{n+1}$. 
	\end{enumerate}
\end{thm}
\begin{proof}
	Case $(1)$ was proved in Theorem \ref{thm-k=n-strong Min}.
	Case $(2)$ is trivial.
	By Corollary \ref{cor-Conj9.3-Conj9.4-rel position},
	 we can assume that $K$ contains the origin and $L$ is a geodesic ball centered at the origin in Case $(3)$ without loss of generality. 
	Then Case $(3)$ follows from Theorem \ref{thm-p=1-strong Min ineq, L ball}.
	Case $(4)$ was proved in  Theorem \ref{thm-strong Min ineq-K ball}.
\end{proof}

\begin{thm}\label{thm-sum-Min ineq-weak}
	Conjecture \ref{conj-Min ineq-weak} holds in one of the following cases,
	\begin{enumerate}
		\item $n \geq 1$ and $k=n$,
		\item $K =L$, 
		\item $n \geq 1$, $k=0$, $p=1$, $K$ is a smooth uniformly h-convex bounded domain in $\mathbb{H}^{n+1}$ and $L$ is a geodesic ball centered at an interior point of $K$. 
		\item $n \geq 2$, $1 \leq k \leq n-1$, $p \geq 1$, $K$ is a smooth uniformly h-convex bounded domain in $\mathbb{H}^{n+1}$ and $L$ is a geodesic ball centered at an interior point of $K$,
		\item $K$ is a geodesic ball  and $L$ is a smooth uniformly h-convex bounded domain in $\mathbb{H}^{n+1}$.
	\end{enumerate}
\end{thm}

\begin{proof}
	Case $(1)$ was proved in Corollary \ref{cor-k=n-conj Min ineq-weak}.
	Case $(2)$ is trivial.
	By Corollary \ref{cor-Conj9.3-Conj9.4-rel position},  we can assume that $K$ contains the origin and $L$ is centered at the origin in Case $(3)$ and Case $(4)$.
	Then Case $(3)$ and Case $(4)$ follow from Theorem \ref{thm-HLW-k=0-p=1} and Theorem \ref{thm-HLW- p geq 1}  respectively.
	 Case $(5)$ was proved in Corollary \ref{cor-weak Min-K ball}.
\end{proof}

\section{Weighted horospherical $p$-Brunn-Minkowski theory}\label{sec: weighted h-Brunn-Minkowski}
Different from Theorem C, there are some geometric inequalities for hypersurfaces in $\mathbb{H}^{n+1}$ which contain the quantity $\cosh r$, see \cite{HL21, HLW20, SX19, Xia16}. Here $r$ denotes the geodesic distance from points to a fixed origin in $\mathbb{H}^{n+1}$. 
Various assumptions of convexity of hypersurfaces were considered in these papers. Nevertheless, h-convex hypersurfaces satisfy most of them. Then we can ask: \emph{Whether there exists a horospherical Brunn-Minkowski theory with weight $\cosh r$?} 

In  Section \ref{sec: weighted h-Brunn-Minkowski}, we approach this problem with similar works as in the previous sections. We recall that $\Vol_w(\Omega): = \int_{\Omega} \cosh r dv$, where $\Omega$ is a bounded domain in $\mathbb{H}^{n+1}$.
\subsection{Weighted horospherical $p$-Minkowski problem}
$ \ $
The following definitions are similar to Definition \ref{def-p mixed k-th modified qmintegral}-- \ref{def-(p,k)surface area measure}.
\begin{defn}
	Let $n \geq 1$ be an integer and $\frac{1}{2} \leq p \leq 2$ be a real number.
	Assume that $K$ and $L$ are two smooth uniformly h-convex bounded domains in $\mathbb{H}^{n+1}$. 
	The $p$-mixed weighted volume of  $K$, $L$ is defined by
	\begin{align*}
	V_{\omega, p}(K,L) = \lim_{t \to 0^+} \frac{\Vol_w(K +_p t \cdot L) - \Vol_w(K)}{t}.
	\end{align*}
\end{defn}
\begin{defn}
	Let $n \geq 1$ be an integer and $p$ be a real number.
	Assume that $K$ is a smooth uniformly h-convex bounded domain in $\mathbb{H}^{n+1}$.
	The weighted horospherical $p$-surface area measure of  $K$ is defined by
	\begin{align*}
	dS_{\omega,p}(K,z) = \cosh r_K \g_K^{-p} \det (A [\g_K]) d\sigma,
	\end{align*} 
	where
	\begin{align*}
	\cosh r_K = \frac{1}{2} \frac{|D \g_K |^2}{\g_K}(z)+ \frac{1}{2} \(\g_K(z) +\frac{1}{\g_K(z)} \).
	\end{align*} 
\end{defn}
Then, by a weighted version of Lemma \ref{lem-formula-Wpk-K-L}, we see that
\begin{align*}
	V_{\omega, p}(K,L) = \frac{1}{p}\int_{\mathbb{S}^n} \g_L^p  dS_{\omega,p}(K,z).
\end{align*}

Similar to Problem \ref{prob-Horospherical p-Minkowski problem}, we can propose the following prescribed measure problem for h-convex bounded domains.
\begin{prob}[Weighted horospherical $p$-Minkowski problem]\label{prob-weighted-horospherical p-Minkowski problem}
	Let $n \geq 1$ be an integer and $p$ be a real number.
	Given a positive function $f(z)$ defined on $\mathbb{S}^n$, what are  necessary and sufficient conditions for $f$, such that there exists a smooth uniformly h-convex bounded domain $K \subset \mathbb{H}^{n+1}$  satisfying 
	\begin{align*}
	d S_{\omega, p}(K,z) = f(z)d\sigma. 
	\end{align*} 
	That is, finding a smooth positive solution to
	\begin{align*}
	\frac{1}{2}\g^{-p}(z)\(\frac{|D \g|^2}{\g}(z) + \g(z)+\frac{1}{\g(z)} \) p_n (A[\g(z)]) =f(z),
	\end{align*}
	such that $A[\g(z)]>0$ for all $z \in \mathbb{S}^n$.
\end{prob}

\subsection{Weighted horospherical $p$-Brunn-Minkowski inequality and Weighted horospherical $p$-Minkowski inequalities}$ \ $

Let $\Omega$ be a bounded domain in $\mathbb{H}^{n+1}$. Define
\begin{align}\label{S-Omega}
\mathcal{S}(\Omega) = (n+1)^{\frac{1}{n+1}} \omega_n^{  -\frac{1}{n+1}} \Vol_\omega^{\frac{1}{n+1}} (\Omega).
\end{align}
Geometrically, we have $\mathcal{S}(B(r)) = \sinh r$. Similar to Conjecture \ref{conj-BM}, we propose the following Brunn-Minkowski type inequality.

\begin{con}[Weighted horospherical $p$-Brunn-Minkowski inequality] \label{conj-weighted-h-BM}
	Let $n \geq 1$ be an integer and $\frac{1}{2} \leq p \leq 2$ be a real number.
	Assume that $K$ and $L$ are two smooth origin symmetric, uniformly h-convex bounded domains in $\mathbb{H}^{n+1}$. Let $\Omega = a \cdot K +_p b \cdot L$, where  $a \geq 0$, $b \geq 0$ and $a+b \geq 1$. Then 
	\begin{align*}
	\( \mathcal{S}(\Omega) + \sqrt{  \mathcal{S}^2(\Omega)+1 }  \)^p \geq
	a\( \mathcal{S}(K) + \sqrt{  \mathcal{S}^2(K)+1 }  \)^p
	+
	b\( \mathcal{S}(L) + \sqrt{  \mathcal{S}^2(L)+1 }  \)^p.
	\end{align*}
	Equality holds if and only if one of the following conditions holds,
	\begin{enumerate}
		\item $a=1$ and $b=0$,
		\item $a =0$ and $b=1$,
		\item $K =L$ and $a+b =1$, 
		\item $K$ and $L$ are geodesic balls centered at the origin.
	\end{enumerate}
\end{con}

\begin{rem}\label{rem-origin sym can not remove}
	By calculating the weighted volume of geodesic balls not centered at the origin, we find that
	the origin-symmetric assumption in Conjecture \ref{conj-weighted-h-BM} can not be removed. Furthermore, 
	it can't be replaced by the assumption that $K$ and $L$ both contain the origin in their interiors.
	We will discuss it in Subsection \ref{subsec-origin-sym cant remove}.
\end{rem}

Here we recall Theorem F.
\begin{thmF}[\cite{SX19}]
	Let $\Omega$ be a smooth uniformly h-convex bounded domain in $\mathbb{H}^{n+1}$ $( n\geq 1)$ which contains the origin in its interior. Then
	\begin{align}\label{SX19-weighted-iso}
	\int_{\partial \Omega} \cosh r d\mu
	\geq
	\( \( (n+1)\Vol_w(\Omega) \)^2+
	\omega_n^{ \frac{2}{n+1}} \( (n+1)\Vol_w(\Omega) \)^{\frac{2n}{n+1}}
	\)^{\frac{1}{2}}.
	\end{align}
	Equality holds if and only if $\Omega$ is a geodesic ball centered at the origin.
\end{thmF}
By a similar proof as in Proposition \ref{prop-BM-K=L}, one can prove the following special case of Conjecture \ref{conj-weighted-h-BM}.
\begin{prop}\label{prop-weighted-scaling}
	Conjecture \ref{conj-weighted-h-BM} holds in the case $K =L$ and $K$ is a smooth uniformly h-convex bounded domain which contains the origin in its interior. Precisely, if we assume that $p>0$, $t \geq 1$, and $\Omega$ is a smooth uniformly h-convex bounded domain that contains the origin in its interior, then
	\begin{align}\label{weighted-Minkowski-scaling}
	\( \mathcal{S}(t \cdot \Omega) + \sqrt{  \mathcal{S}^2(t \cdot \Omega)+1 }  \)^p \geq
	t\( \mathcal{S}(\Omega) + \sqrt{  \mathcal{S}^2(\Omega)+1 }  \)^p.
	\end{align}
	If $t>1$, 
     then equality holds in \eqref{weighted-Minkowski-scaling} if and only if  $\Omega$ is a geodesic ball centered at the origin.
\end{prop}

\begin{proof}
Let $t \geq 1$, $\Omega_t : = t \cdot \Omega$, and $u(z,t)$ be the horospherical support function of $\Omega_t$.
Denote $u(z):=u(z,1)$. From Definition \ref{def-p dilation}, we have $\g^p (z, t) =t \g^p(z) $ and $u(z,t) = u(z,1) + \frac{1}{p}\log t$, which implies  $\partial_t u(z,t) = \frac{1}{pt}$. 
By using \eqref{S-Omega}, \eqref{scalr-flow DT's trick} and the co-area formula, we have
\begin{align}
\frac{d}{dt}  \mathcal{S}(\Omega_t) 
=& \frac{d}{dt} \( (n+1)^{\frac{1}{n+1}} \omega_n^{  -\frac{1}{n+1}} \Vol_\omega^{\frac{1}{n+1}} (\Omega_t) \) \nonumber\\
=& (n+1)^{-\frac{n}{n+1}}\omega_n^{  -\frac{1}{n+1}}  \Vol_\omega^{-\frac{n}{n+1}} (\Omega_t)
\int_{\partial \Omega_t} \cosh r \frac{1}{pt} d\mu_t  \nonumber\\
=&\frac{1}{pt}\omega_n^{-1} \mathcal{S}^{-n}(\Omega_t) \int_{\partial \Omega_t} \cosh r  d\mu_t.\label{dt-S-Omega t}
\end{align}	
By \eqref{S-Omega}, Theorem F asserts that
\begin{align}
 \int_{\partial \Omega_t} \cosh r  d\mu_t
 \geq& \( \(\omega_n \mathcal{S}^{n+1} (\Omega_t)  \)^2 +\omega_n^{\frac{2}{n+1}} \( \omega_n \mathcal{S}^{n+1}  (\Omega_t) \)^{\frac{2n}{n+1}} \)^{\frac{1}{2}}  \nonumber\\
=& \(\omega_n^2 \mathcal{S}^{2n+2}  (\Omega_t)+ \omega_n^2 \mathcal{S}^{2n}  (\Omega_t)\)^{\frac{1}{2}} \nonumber \\
=& \omega_n \mathcal{S}^n (\Omega_t) \sqrt{\mathcal{S}^2 (\Omega_t) +1}. \label{weighted Mink-theorem F}
\end{align}
Substituting \eqref{weighted Mink-theorem F} into \eqref{dt-S-Omega t}, we have
\begin{align*}
\frac{d}{dt}  \mathcal{S}(\Omega_t)  \geq \frac{1}{pt}\sqrt{\mathcal{S}^2 (\Omega_t) +1}.
\end{align*}
Hence,
\begin{align}
&\frac{d}{dt}\log \( \mathcal{S}(\Omega_t) + \sqrt{\mathcal{S}^2 (\Omega_t) +1} \) \nonumber\\
=& \( \mathcal{S}(\Omega_t) + \sqrt{\mathcal{S}^2 (\Omega_t) +1} \)^{-1}\(\frac{d}{dt}\mathcal{S}(\Omega_t)
+
\frac{1}{2} \frac{2\mathcal{S}(\Omega_t)}{  \sqrt{\mathcal{S}^2 (\Omega_t) +1} } \frac{d}{dt}\mathcal{S}(\Omega_t)\) \nonumber\\
=&\( \mathcal{S}(\Omega_t) + \sqrt{\mathcal{S}^2 (\Omega_t) +1} \)^{-1} \frac{ \mathcal{S}(\Omega_t) + \sqrt{\mathcal{S}^2 (\Omega_t) +1}}{ \sqrt{\mathcal{S}^2 (\Omega_t) +1} } \frac{d}{dt}\mathcal{S}(\Omega_t)
\geq \frac{1}{pt}.\label{varia-weighted-Mink-Scaling}
\end{align}
Integrating the both sides of \eqref{varia-weighted-Mink-Scaling}, we obtain
\begin{align*}
\log \( \mathcal{S}(\Omega_t) + \sqrt{\mathcal{S}^2 (\Omega_t) +1} \)
\geq \frac{1}{p}\log t+ \log \( \mathcal{S}(\Omega) + \sqrt{\mathcal{S}^2 (\Omega) +1} \),
\end{align*}
thus we obtain \eqref{weighted-Minkowski-scaling}.

 By Theorem F, equality holds in \eqref{weighted Mink-theorem F} if and only if $\Omega_t$ is a geodesic ball centered at the origin, which means that $\g(z,t) = t^{\frac{1}{p}} \g(z)$ is constant on $\mathbb{S}^n$. Hence, equality holds in \eqref{weighted Mink-theorem F}  for some $t \geq 1$ if and only if $\Omega$ is a geodesic ball centered at the origin. Then we have that equality holds in \eqref{weighted-Minkowski-scaling} if and only if either $t=1$, or $\Omega$ is a geodesic ball centered at the origin. This completes the proof of Proposition \ref{prop-weighted-scaling}.
\end{proof}

Letting $a =1-t$ and $b=t$ in Conjecture \ref{conj-weighted-h-BM} and using a similar calculation as in Proposition \ref{prop-rel-BM-Min ineq-strong}, we can propose the following Conjecture \ref{conj-weighted-horo-Mink-ineq}, which is the weighted version of Conjecture \ref{conj-Min ineq}.

\begin{con}[Weighted horospherical $p$-Minkowski inequality of type \uppercase\expandafter{\romannumeral1}] \label{conj-weighted-horo-Mink-ineq}
	Let $K$ and $L$ be two smooth origin-symmetric, uniformly h-convex bounded domains in $\mathbb{H}^{n+1}$ $(n \geq 1)$. If $p>0$, then 
	\begin{equation*}
		\begin{aligned}
		&\int_{\partial K} \g_L^p \cosh r (\cosh r- \tilde{u})^p d\mu
		- \int_{\partial K} \cosh r d\mu\\
		\geq&
		\omega_n \mathcal{S}^n (K) \sqrt{\mathcal{S}^2(K) +1}
		\(  \(  \frac{\mathcal{S}(L) + \sqrt{\mathcal{S}^2(L) +1} }{\mathcal{S}(K) + \sqrt{\mathcal{S}^2(K) +1}}  \)^p-1 \). 
		\end{aligned}
	\end{equation*}
If $p<0$, then
	\begin{equation*}
	\begin{aligned}
	&\int_{\partial K} \g_L^p \cosh r (\cosh r- \tilde{u})^p d\mu
	- \int_{\partial K} \cosh r d\mu\\
	\leq&
	\omega_n \mathcal{S}^n (K) \sqrt{\mathcal{S}^2(K) +1}
	\(  \(  \frac{\mathcal{S}(L) + \sqrt{\mathcal{S}^2(L) +1} }{\mathcal{S}(K) + \sqrt{\mathcal{S}^2(K) +1}}  \)^p-1 \). 
	\end{aligned}
	\end{equation*}
 In the both inequalities, equality holds if and only if either
 \begin{enumerate}
 	\item $K=L$, or
 	\item $K$ and $L$ are geodesic balls centered at the origin.
 \end{enumerate}
\end{con}

\begin{prop}\label{prop-weig Min from weig BM}
	Assume that $\frac{1}{2} \leq p \leq 2$ and $\Omega_t := (1-t) \cdot K +_p t \cdot L$, then Conjecture \ref{conj-weighted-horo-Mink-ineq} is equivalent to
	\begin{equation}\label{equi-form-weight-Mink}
	\left.	\frac{d}{dt}\right|_{t=0} \(\mathcal{S}(\Omega_t) + \sqrt{\mathcal{S}^2(\Omega_t) +1}   \)^p 
	\geq  \(\mathcal{S}(L) + \sqrt{\mathcal{S}^2(L) +1}  \)^p - \( \mathcal{S}(K) + \sqrt{\mathcal{S}^2(K) +1} \)^p.
	\end{equation}
	Equality holds if and only if either $K=L$, or $K$ and $L$ are geodesic balls centered at the origin.
	Thus Conjecture \ref{conj-weighted-horo-Mink-ineq} can be obtained from Conjecture \ref{conj-weighted-h-BM} in this case.
\end{prop}

\begin{proof}
	At first, we show the equivalence between \eqref{equi-form-weight-Mink} and Conjecture \ref{conj-weighted-horo-Mink-ineq}.
	Let $u(z,t)$ denote the horospherical support function of $\Omega_t$. Then we have $\g^p(z,t) = (1-t)\g_K^p (z) + t \g_L^p(z)$ by the definition of $\Omega_t$. Thus
	\begin{align*}
	\partial_t u(z,t) = \frac{\partial_t \g^p(z,t)}{p \g^p(z,t)}
	= \frac{1}{p} \frac{\g_L^p(z)  -\g_K^p(z)  }{(1-t)\g_K^p (z) + t \g_L^p(z) }.
	\end{align*}
	Using \eqref{scalr-flow DT's trick} and a similar calculation as in \eqref{dt-S-Omega t},  we have
	\begin{align}\label{dt-0-S Omega-t}
	\left. \frac{d}{dt}\right|_{t=0}  \mathcal{S}(\Omega_t)
	= \omega_n^{-1} \mathcal{S}^{-n}(K) \int_{\partial K} \cosh r \(\frac{1}{p} \frac{\g_L^p(z) - \g_K^p(z)}{\g_K^p(z)}  \) d\mu. 
	\end{align}
	A direct calculation gives
	\begin{align}\label{dt-0-log S Omega-t}
		\left. \frac{d}{dt}\right|_{t=0} \log  \(  \mathcal{S}(\Omega_t) + \sqrt{\mathcal{S}^2 (\Omega_t) +1}  \)
		=& \frac{1}{    \sqrt{\mathcal{S}^2(K) +1}} 	\left. \frac{d}{dt}\right|_{t=0} \mathcal{S}(\Omega_t).
	\end{align}
	Inserting \eqref{dt-0-S Omega-t} into \eqref{dt-0-log S Omega-t}, we have
	\begin{align}
		&\left.	\frac{d}{dt}\right|_{t=0} \(\mathcal{S}(\Omega_t) + \sqrt{\mathcal{S}^2(\Omega_t) +1}   \)^p \nonumber\\
	=&p \(\mathcal{S}(K) + \sqrt{\mathcal{S}^2(K) +1}   \)^p	\left.	\frac{d}{dt}\right|_{t=0} \log  \(  \mathcal{S}(\Omega_t) + \sqrt{\mathcal{S}^2 (\Omega_t) +1}  \) \nonumber\\
	=&\(\mathcal{S}(K) + \sqrt{\mathcal{S}^2(K) +1}   \)^p  \frac{  \int_{\partial K} \g_L^p(z) \cosh r \( \cosh r- \tilde{u} \)^p  d\mu -\int_{\partial K} \cosh r d\mu    }{  \omega_n \mathcal{S}^n (K) \sqrt{\mathcal{S}^2(K) +1} },\label{dt-0-S^p Omega t}
	\end{align}
	where we used \eqref{1/phi, coshr-u} in the second equality.
	Substituting \eqref{dt-0-S^p Omega t} into the left-hand side of \eqref{equi-form-weight-Mink}, we have
	\begin{align*}
	&\int_{\partial K} \g_L^p(z) \cosh r \( \cosh r- \tilde{u} \)^p  d\mu -\int_{\partial K} \cosh r d\mu  \\
	\geq& 
	 \omega_n \mathcal{S}^n (K) \sqrt{\mathcal{S}^2(K) +1} \( \(\frac{\mathcal{S}(L) + \sqrt{\mathcal{S}^2(L) +1}}{\mathcal{S}(K) + \sqrt{\mathcal{S}^2(K) +1}}\)^p -1 \),
	\end{align*}
	thus we have that \eqref{equi-form-weight-Mink} is equivalent to Conjecture \ref{conj-weighted-horo-Mink-ineq} in the case $\frac{1}{2} \leq p \leq 2$.
	
	If we admit Conjecture \ref{conj-weighted-h-BM},  then we have 
	\begin{align}\label{weigh conj-a=1-t,b=t}
	\(\mathcal{S}(\Omega_t) + \sqrt{\mathcal{S}^2(\Omega_t) +1}   \)^p
	\geq (1-t)  \( \mathcal{S}(K) + \sqrt{\mathcal{S}^2(K) +1} \)^p+
	t  \(\mathcal{S}(L) + \sqrt{\mathcal{S}^2(L) +1}  \)^p
	\end{align}
	for all $t \in [0,1]$.
	At this time, \eqref{equi-form-weight-Mink} can be achieved by differentiating the both sides of \eqref{weigh conj-a=1-t,b=t} at $t=0$. 
	Thus we have that Conjecture \ref{conj-weighted-horo-Mink-ineq} can be obtained from Conjecture \ref{conj-weighted-h-BM}.
	We complete the proof of Proposition \ref{prop-weig Min from weig BM}. 
\end{proof}

When $p>0$, the following Conjecture \ref{conj-weigh-Min-ineq} can be considered as either a weighted version of Conjecture \ref{conj-Min ineq-weak}, or a corollary of Conjecture \ref{conj-weighted-horo-Mink-ineq} by Theorem F. At the same time, we can also get it by letting $a=1$ and $b=t$ in Conjecture \ref{conj-weighted-h-BM} and doing a similar calculation as in Proposition \ref{prop-rel-BM-Min ineq-strong}.
\begin{con}[Weighted horospherical $p$-Minkowski inequality of type \uppercase\expandafter{\romannumeral2}]\label{conj-weigh-Min-ineq}
	Let $n \geq 1$ be an integer and $p>0$ be a real number.
	Assume that $K$ and $L$ are smooth uniformly h-convex bounded domains in $\mathbb{H}^{n+1}$. Then
	\begin{align*}
	&\int_{\partial K} \g_L^p(z) \cosh r (\cosh r- \tilde{u})^p d\mu\\
	\geq&
		\omega_n \mathcal{S}^n (K)\sqrt{\mathcal{S}^2(K) +1}
	 \(  \frac{\mathcal{S}(L) + \sqrt{\mathcal{S}^2(L) +1} }{\mathcal{S}(K) + \sqrt{\mathcal{S}^2(K) +1}}  \)^p. 
	\end{align*}
	Equality holds if and only if $K$ and $L$ are geodesic balls centered at the origin.
\end{con}

In the study of Conjecture \ref{conj-Min ineq-weak}, we proved a special case in Theorem \ref{thm-HLW- p geq 1}. Hence we can propose the following Conjecture \ref{conj-last of weighted isoperi type}, which is the corresponding weighted version of Theorem \ref{thm-HLW- p geq 1}.
\begin{con}\label{conj-last of weighted isoperi type}
	Let $\Omega$ be a smooth uniformly h-convex bounded domain in $\mathbb{H}^{n+1}$. If $0 \leq p<1$, then
	\begin{align}\label{Weighted horospherical Minkowski inequality-L point-p>0}
	\int_{\partial \Omega} \cosh r(\cosh r- \tilde{u})^p d\mu
	\geq \omega_n \mathcal{S}^n (\Omega)
	 \sqrt{\mathcal{S}^2(\Omega) +1}\( \mathcal{S}(\Omega) + \sqrt{\mathcal{S}^2(\Omega) +1}\)^{-p}.
	\end{align}
	Equality holds if and only if $\Omega$ is a geodesic ball centered at the origin.
\end{con}

\begin{rem}
	It is natural to propose Conjecture \ref{conj-last of weighted isoperi type} for $p \geq 0$.
	However, if $p \geq 1$, then we will prove \eqref{Weighted horospherical Minkowski inequality-L point-p>0} in the follow-up Subsection \ref{subsec-Proof of conj weighted iso special case}.
\end{rem}

\subsection{Proof of \eqref{Weighted horospherical Minkowski inequality-L point-p>0} in the case $p \geq 1$} \label{subsec-Proof of conj weighted iso special case}$ \ $

The main result in Subsection \ref{subsec-Proof of conj weighted iso special case} is the following Theorem \ref{thm-p=1 case in conj last of weighted iso type}.
\begin{thm}\label{thm-p=1 case in conj last of weighted iso type}
	Let $\Omega$ be a smooth bounded domain in $\mathbb{H}^{n+1}$. If $p \geq 1$, then
	\begin{equation}\label{ineq-p=1 case in conj last of weighted iso type}
	\int_{\partial \Omega} \cosh r(\cosh r- \tilde{u})^p d\mu
	\geq \omega_n \mathcal{S}^n (\Omega)
	\sqrt{\mathcal{S}^2(\Omega) +1}\( \mathcal{S}(\Omega) + \sqrt{\mathcal{S}^2(\Omega) +1}\)^{-p}.
	\end{equation}
	Equality holds if and only if $\Omega$ is a geodesic ball centered at the origin.
\end{thm}
The authors \cite{LX} have gave a new proof of Theorem F and have shown that $\Omega$ need not contain the origin, and the h-convex assumption of $\Omega$ can be removed. That means that the $p=0$ case of Conjecture \ref{conj-last of weighted isoperi type} has been completely solved in \cite{LX}.
 Fortunately, the method used in \cite{LX} can also be used to prove Theorem \ref{thm-p=1 case in conj last of weighted iso type}. 

In the first step, we define the vertical projection map $\Pi$ from $\mathbb{H}^{n+1} \subset \mathbb{R}^{n+1,1}$ to $\mathbb{R}^{n+1} := \mathbb{R}^{n+1} \times \{0 \} \subset \mathbb{R}^{n+1,1}$ by $\Pi(x,x_{n+1}) =x$. Denote by $\rho := |x|$ the radial function in $\mathbb{R}^{n+1}$. Here we remark that Brendle and Wang \cite{BW14} used the map $\Pi$ to prove the Penrose inequality for surface lying in a convex static timelike hypersurface in $\mathbb{R}^{n+1,1}$, see also \cite{MS14}. For a bounded domain $\Omega$ in $\mathbb{H}^{n+1}$, let $\overline{\Omega}$ denote $\Pi(\Omega)$ in this subsection. The support function and the outward unit normal of $\overline{\Omega}$ are denoted by $\overline{u}$ and $\overline{\nu}$, respectively. 

If we write $dv:= \sinh^n r dr d\sigma$ and $d \overline{v}:= \rho^n d\rho d\sigma$ for the volume elements of $\mathbb{H}^{n+1}$ and $\mathbb{R}^{n+1}$ respectively, then 
\begin{equation}\label{vertical proj-volume element}
\pi^* (d \overline{v}) = \pi^* \(\rho^n d\rho d\sigma\) = \sinh^n r \cosh r dr d\sigma = \cosh r dv.
\end{equation}
Therefore, $\Vol_w \(\Omega\)$ is exactly the classical volume of $\overline{\Omega}$ in $\mathbb{R}^{n+1}$. Consequently, the right-hand side of \eqref{ineq-p=1 case in conj last of weighted iso type} is a function of $\Vol \(\overline{\Omega}\)$. In the second step, we will rewrite the left-hand side of \eqref{ineq-p=1 case in conj last of weighted iso type} by using the geometric quantities of $\overline{\Omega}$. The outward unit normal at $X$ of $\partial \Omega$ is given by (see \cite{LX})
\begin{equation}\label{vertical proj-nu}
\nu \(X\) = \frac{\(\overline{\nu} + \overline{u} x, x_{n+1} \overline{u}\)}{ \( \overline{u}^2+1 \)^{\frac{1}{2}}},
\end{equation}
where $\overline{\nu}$ and $\overline{u}$ are evaluated at $x \in \partial \overline{\Omega}$.
For the conformal Killing vector field $V = \sinh r \partial_r$, we know
\begin{equation*}
V =\sinh r \partial_r \( \sinh r \theta, \cosh r \)= \( \cosh r x, \sinh^2 r \)= \(x_{n+1} x, |x|^2\).
\end{equation*}
This together with \eqref{vertical proj-nu} gives
\begin{equation*}
\widetilde{u}\(X\) : = \metric{V}{\nu} = \frac{x_{n+1} \overline{u}}{\( \overline{u}^2+1 \)^{\frac{1}{2}}}.
\end{equation*}
Therefore
\begin{equation}\label{vertical proj cosh r-u}
\(\cosh r-\widetilde{u} \) \(X\)
= x_{n+1} \( 1- \frac{\overline{u}}{\( \overline{u}^2+1 \)^{\frac{1}{2}}}  \).
\end{equation}
Letting $d \overline{\mu}$ denote the area element of $\partial \overline{\Omega}$, we have by \eqref{vertical proj-nu}
\begin{equation}\label{vertical proj area element}
\Pi_* \( \cosh r d\mu \) =\Pi_* \nu \lrcorner d \overline{v} = \metric{\Pi_* \nu}{\overline{\nu}} d\overline{\mu}= \(\overline{u}^2+1 \)^{\frac{1}{2}} d \overline{\mu}.
\end{equation}
Using \eqref{vertical proj cosh r-u} and \eqref{vertical proj area element}, we can write the left-hand side of \eqref{ineq-p=1 case in conj last of weighted iso type} as
\begin{equation*}
\int_{\partial \Omega} \cosh r\(\cosh r - \widetilde{ u}\)^p d\mu
= \int_{\partial \overline{\Omega}} x_{n+1}^p \(1- \frac{\overline{u}}{\( \overline{u}^2+1 \)^{\frac{1}{2}} }\)^p \( \overline{u}^2+1 \)^{\frac{1}{2}} d\overline{\mu}.
\end{equation*}
Since $x_{n+1} = \(|x|^2+1\)^{\frac{1}{2}} \geq \( \overline{u}^2+1 \)^{\frac{1}{2}}$, we have
\begin{equation}\label{LHS of weig iso geq integ of u}
\int_{\partial \Omega} \cosh r\(\cosh r - \widetilde{ u}\)^p d\mu
\geq \int_{\partial \overline{\Omega}} \( \overline{u}^2+1 \)^{\frac{1}{2}} \(\( \overline{u}^2+1 \)^{\frac{1}{2}} - \overline{u}  \)^p d \overline{\mu}, \quad \forall p \geq 0.
\end{equation}
Finally, we are ready to prove Theorem \ref{thm-p=1 case in conj last of weighted iso type}.
\begin{proof} [Proof of Theorem \ref{thm-p=1 case in conj last of weighted iso type}]
	At first, we define a positive function $\chi (t)$ on $\mathbb{R}$ by
	\begin{equation*}
	\chi(t) := \( t^2+1 \)^{\frac{1}{2}} \( \(t^2 +1\)^{\frac{1}{2}} -t \)^p.
	\end{equation*}
	Then \eqref{LHS of weig iso geq integ of u} can be written shortly as
	\begin{equation}\label{short-LHS of weig iso geq integ of u}
	\int_{\partial \Omega} \cosh r\(\cosh r - \widetilde{ u}\)^p d\mu
	\geq  \int_{\partial \overline{\Omega}} \chi (\overline{u}) d\overline{\mu}.
	\end{equation}
	Now we study the properties of $\chi (t)$.
	A direct calculation gives
	\begin{equation}\label{log chi '}
	\( \log \chi (t)\)' 
	= \frac{1}{2} \(\log\(t^2+1\)\)' +p \(\log \( \(t^2+1\)^{\frac{1}{2}}-t \)\)'
	= \frac{t - p \(t^2+1\)^{\frac{1}{2}}}{t^2+1}. 
	\end{equation}
	As we assumed $p \geq 1$, $\chi(t)$ is strictly decreasing on $\mathbb{R}$. For the second derivative of $\chi(t)$, we have by use of \eqref{log chi '}
	\begin{equation*}
	\begin{aligned}
	\frac{\chi''(t)}{\chi(t)} =& \(\log \chi(t)\)''+ \(\log \chi(t)\)'^2\\
	=&\( \frac{t - p \(t^2+1\)^{\frac{1}{2}}}{t^2+1} \)'+ \( \frac{t - p \(t^2+1\)^{\frac{1}{2}}}{t^2+1} \)^2\\
	=&\(t^2+1\)^{-2} \( \(t^2+1 \)\(t- p \(t^2+1\)^{\frac{1}{2}}\)' - 2t \(t- p \(t^2+1\)^{\frac{1}{2}} \)+ \(t- p \(t^2+1\)^{\frac{1}{2}}\)^2  \)\\
	=& \(t^2+1\)^{-2} \( \(t^2+1 \)\(1-pt \(t^2+1\)^{-\frac{1}{2}}\)- 2t^2+t^2+p^2 \(t^2+1\)  \)\\
	=& \(t^2+1\)^{-2} \(1-pt\(t^2+1\)^{\frac{1}{2}}+p^2 \(t^2+1\)  \)\\
	=&  \(t^2+1\)^{-2}\(1+ p\(t^2+1\)^{\frac{1}{2}} \(p \(t^2+1\)^{\frac{1}{2}}-t\)\)>0,
	\end{aligned}
	\end{equation*}
	where we used $p \geq 1$ in the last inequality. Since $\chi''(t)>0$ on $\mathbb{R}$,  Jensen's inequality yields that
	\begin{equation}\label{Jess-ineq}
	\frac{\int_{\partial \overline{\Omega}} \chi \(\overline{u}\)d \overline{\mu}   }{\int_{\partial \overline{\Omega}}d \overline{\mu}}
	\geq 
	\chi \( \frac{\int_{\partial \overline{\Omega}} \overline{u} d\overline{\mu} }{\int_{\partial \overline{\Omega}}d \overline{\mu}  }  \),
	\end{equation}
	with equality if and only if $\overline{u}$ is constant on $\partial \overline{\Omega}$.
	By the divergence theorem and the classical isoperimetric inequality, we have
	\begin{align}
	\int_{\partial \overline{\Omega}} \overline{u} d\overline{\mu} =& (n+1) \Vol \(\overline{\Omega} \)>0, \nonumber\\
		\frac{\int_{\partial \overline{\Omega}} \overline{u} d\overline{\mu} }{\int_{\partial \overline{\Omega}}d \overline{\mu}  }
	=& \frac{(n+1) \Vol \(\overline{\Omega} \)}{ {\rm Area}\( \partial \overline{\Omega} \)  }
	\leq  \(n+1\)^{\frac{1}{n+1}} \omega_n^{-\frac{1}{n+1}}  \Vol\(\overline{\Omega}\)^{\frac{1}{n+1}}. \label{iso-ineq}
	\end{align}
	Equality holds in \eqref{iso-ineq} if and only if $\overline{\Omega}$ is a geodesic ball in $\mathbb{R}^{n+1}$.
	Now we let $\overline{\chi} (t) : = t^{-1} \chi (t)$, which is strictly decreasing on $\mathbb{R}^+$. 
	Using \eqref{Jess-ineq} and \eqref{iso-ineq}, we arrive at
	\begin{equation}\label{vertical proj-Eucl ineq}
	\int_{\partial \overline{\Omega}} \chi \(\overline{u}\)d \overline{\mu}
	\geq
	\int_{\partial \overline{\Omega}} \overline{u} d\overline{\mu} \cdot 
	\overline{\chi} \(  \frac{\int_{\partial \overline{\Omega}} \overline{u} d\overline{\mu} }{\int_{\partial \overline{\Omega}}d \overline{\mu}  } \)
	\geq
	\(n+1\) \Vol\(\overline{\Omega}\) \overline{\chi} \(\(n+1\)^{\frac{1}{n+1}} \Vol\(\overline{\Omega}\)^{\frac{1}{n+1}}\).	
	\end{equation}
	Recall that \eqref{vertical proj-volume element} implies $\Vol \(\overline{\Omega} \) = \Vol_w \(\Omega\)$. Combining \eqref{vertical proj-Eucl ineq} with \eqref{short-LHS of weig iso geq integ of u} and \eqref{S-Omega}, we have
	\begin{align}
	\int_{\partial \Omega} \cosh r\(\cosh r - \widetilde{ u}\)^p d\mu
	\geq& 
	\int_{\partial \overline{\Omega}} \chi (\overline{u}) d\overline{\mu} \nonumber\\
	\geq&
	\(n+1\) \Vol\(\overline{\Omega}\) \overline{\chi} \(\(n+1\)^{\frac{1}{n+1}}\omega_n^{-\frac{1}{n+1}} \Vol\(\overline{\Omega}\)^{\frac{1}{n+1}}\) \nonumber\\
	=&	\(n+1\) \Vol_w\(\Omega\) \overline{\chi} \(\(n+1\)^{\frac{1}{n+1}}\omega_n^{-\frac{1}{n+1}} \Vol_w\(\Omega\)^{\frac{1}{n+1}}\) \nonumber\\
	=& \omega_n \mathcal{S}^{n+1} \(\Omega\) \overline{\chi} \( \mathcal{S} \(\Omega\)\) \nonumber\\
	=& \omega_n \mathcal{S}^n (\Omega)
	\sqrt{\mathcal{S}^2(\Omega) +1}\( \mathcal{S}(\Omega) + \sqrt{\mathcal{S}^2(\Omega) +1}\)^{-p}. \label{ineq-p=1 case-last step-in conj last of weighted iso type}
	\end{align}
	Hence we obtain \eqref{ineq-p=1 case in conj last of weighted iso type}.
	
	By \eqref{Jess-ineq} and \eqref{iso-ineq}, if equality holds in \eqref{ineq-p=1 case-last step-in conj last of weighted iso type}, then $\overline{\Omega}$ is a geodesic ball, and $\overline{u}$ is constant on $\overline{\Omega}$. Consequently, equality holds in \eqref{ineq-p=1 case in conj last of weighted iso type} if and only if $\Omega$ is a geodesic ball centered at the origin. This completes the proof of Theorem \ref{thm-p=1 case in conj last of weighted iso type}.
\end{proof}

\begin{rem}
	Letting $\chi (t): = \(t^2+1\)^{\frac{1}{2}}$, we can check directly that $t^{-1}\chi(t)$ is strictly deceasing on $\mathbb{R}^+$, and $\chi(t)$ is strictly convex on $\mathbb{R}$.
	Therefore the above proof can be used to prove Theorem F without major changes, just as we have done in \cite{LX}.
\end{rem}

\section{Relationship between the hyperbolic $p$-sum in hyperbolic space and the Minkowski-Firey $L_p$-addition in Euclidean space} \label{sec-hyperbolic sum and Firey's sum}
The arguments in Section \ref{sec-hyperbolic sum and Firey's sum}  are motivated by formula \eqref{phi(coshr nu-uX)}. For a convex body $\widehat{K}$ with smooth support function $\g(z)$ in Euclidean space $\mathbb{R}^{n+1}$, the position vectors of $\partial K$ are exactly given by $D\g(z)+ \g(z)z$, $z \in \mathbb{S}^n$, see e.g. \cite[p. 47]{Sch14}.  Here, we identify the $(n+1)$-dimensional Euclidean space $\mathbb{R}^{n+1}$ with the hyperplane $\mathbb{R}^{n+1} \times \{0\}$ in $\mathbb{R}^{n+1,1}$. We will show that formula \eqref{phi(coshr nu-uX)} actually gives a map $\pi$ from h-convex bounded domains in $\mathbb{H}^{n+1}$ to convex bounded domains in $\mathbb{R}^{n+1}$. 
Subsections \ref{subsec-A light cone projection}--\ref{subsec-proof of prop lc-map-supp} are devoted to figure out the detailed properties of the map $\pi$.
Among those properties, Proposition \ref{prop-lc proj map-supp func} plays a key role.
Here we remark that map $\pi$ closely relates to the Penrose inequality in general relativity, see \cite{MS14}.
 In the next step, we study the counterpart of $p$-Brunn-Minkowski inequality, $p$-Minkowski inequality and $p$-isoperimetric inequality in hyperbolic space. These works will be done in Subsections \ref{subsec-weighted p-BM type ineq}--\ref{subse-pf of weighted p-BM type ineq}.

In order to compare the geometry of domains in $\mathbb{H}^{n+1}$ and $\mathbb{R}^{n+1}$, we make some conventions on the notations. Different from the horospherical Gauss map $G$, we use $\widehat{G}_{\widehat{\Omega}}$ to denote the Euclidean Gauss map of a convex body $\widehat{\Omega}$ in $\mathbb{R}^{n+1}$.  We denote the support function of $\widehat{\Omega}$ as $\widehat{u}_{ \widehat{\Omega} } (z)$, i.e.
\begin{equation*}
\widehat{u}_{\widehat{\Omega}} (z) = \max_{x \in \widehat{\Omega}} \metric{x}{z}, \quad z \in \mathbb{S}^n.
\end{equation*}
Now we recall the definition of Minkowski-Firey $L_p$-addition \cite{F62} (Firey's $p$-sum for short)  for convex bodies in $\mathbb{R}^{n+1}$.
 Through Section \ref{sec-hyperbolic sum and Firey's sum}, we  use $+_p^{\mathcal{R}}$ to denote the Firey's $p$-sums for convex bodies in $\mathbb{R}^{n+1}$ and use $+_p^{\mathcal{H}}$ to denote the hyperbolic $p$-sums for h-convex bounded domains in $\mathbb{H}^{n+1}$.
 
 Let $p >0$ be a positive number. Assume that $\widehat{K}$ and $\widehat{L}$ are convex bodies in $\mathbb{R}^{n+1}$ that contain the origin in their interiors. Let $a \geq 0$ and $b \geq 0$ be two real numbers.
Let $\widehat{\Omega} := a \cdot \widehat{K} +^{\mathcal{R}}_p b \cdot \widehat{L}$ denote the $p$-linear combination of $\widehat{K}$ and $\widehat{L}$.
 If $p \geq 1$, then $\widehat{\Omega}$ is defined by
\begin{equation*}
	\widehat{u}_{\widehat{\Omega}}^p (z) = a \widehat{u}_{\widehat{K}}^p (z) + b \widehat{u}_{\widehat{L}}^p(z).
\end{equation*}
If $p \in (0,1)$, then the definition of $\widehat{\Omega}$ was proposed by B\"{o}r\"{o}czky, Lutwak, Yang, and Zhang \cite{BLYZ12}, i.e.
\begin{equation*}
\widehat{\Omega}:=
\bigcap_{z \in \mathbb{S}^n} \left\{ x \in \mathbb{R}^{n+1} \left| \metric{x}{z} \leq  a \widehat{u}_{\widehat{K}}^p (z) + b \widehat{u}_{\widehat{L}}^p (z) \right. \right\}.
\end{equation*}
The above definitions can also be found in \cite[Chapter 9]{Sch14}.
In the case $p \in (0,1)$, we should remark the following fact. If $\widehat{u}(z):= \(a \widehat{u}_{\widehat{K}}^p (z) + b \widehat{u}_{\widehat{L}}^p (z)\)^{\frac{1}{p}}$ satisfies $D^2 \widehat{u}(z) +\widehat{u} (z) I >0$ on $\mathbb{S}^{n}$, then $\widehat{u}_{\widehat{\Omega}}(z) = \widehat{u}(z)$.

For a positive number $p \geq 1$, the \emph{$p$-mixed volume} $\Vol_p \(\widehat{K}, \widehat{L} \)$ of convex bodies $\widehat{K}$ and $\widehat{L}$ is defined by
\begin{equation}  \label{p-mixed volume}
\Vol_p \( \widehat{K}, \widehat{L} \) := \frac{1}{n+1} \int_{\partial \widehat{K}} \widehat{u}_{\widehat{L}}^p (z)\widehat{u}^{1-p}_{\widehat{K}} (z) d \widehat{\mu}, 
\end{equation}
where $z \in \mathbb{S}^n$ denoted the outward unit normal of $\partial \widehat{K} \subset \mathbb{R}^{n+1}$.
To avoid confusion, we may write $d\widehat{\mu}_K$ for $d \widehat{\mu}$ in the following text.

The following $L_p$ Brunn-Minkowski inequality was proved by Firey \cite{F62}, see also an independent proof by Lutwak \cite{Lut93}.
\begin{thmG}[$L_p$ Brunn-Minkowski inequality]
	Let $p>1$ be a real number. Assume that $\widehat{K}$ and $\widehat{L}$ are convex bodies in $\mathbb{R}^{n+1}$ that contain the origin in their interiors.  Then
	\begin{equation}\label{classical p-BM ineq}
	\Vol \( \widehat{K} +_p^{\mathcal{R}} \widehat{L} \)^{\frac{p}{n+1}} \geq \Vol\(\widehat{K} \)^{\frac{p}{n+1}} + \Vol\(\widehat{L} \)^{\frac{p}{n+1}},
	\end{equation}
	with equality if and only if $\widehat{K}$ and $\widehat{L}$ are dilates.
\end{thmG}
The following $p$-Minkowski inequality was proved by Lutwak \cite{Lut93}.
\begin{thmH}[$p$-Minkowski inequality]
		Let $p>1$ be a real number. Assume that $\widehat{K}$ and $\widehat{L}$ are convex bodies in $\mathbb{R}^{n+1}$ that contain the origin in their interiors.  Then
		\begin{equation}\label{classcial p-Min ineq}
		\Vol_p \( \widehat{K}, \widehat{L} \)^{n+1}\geq \Vol\( \widehat{K}\)^{n+1-p}\Vol\(\widehat{L}\)^{p},
		\end{equation}
		with equality if and only if $\widehat{K}$ and $\widehat{L}$ are dilates.
\end{thmH}
In Theorem H, by taking $\widehat{L}$ as the  geodesic balls centered at the origin in $\mathbb{R}^{n+1}$, one will obtain an isoperimetric type inequality for each $p>1$, which is known as the $p$-isoperimetric inequality. 
\begin{thmI}[$p$-isoperimetric inequality]
	Let $p>1$ be a real number. Assume that $\widehat{\Omega}$ is a convex body in $\mathbb{R}^{n+1}$ that contains the origin in its interior. Then
	\begin{equation}\label{p-isoperimetric inequality}
	\int_{\partial \widehat{\Omega}} \widehat{u}_{\widehat{\Omega}}^{1-p}(z) d \widehat{\mu}
	\geq (n+1)^{\frac{n+1-p}{n+1}} \omega_n^{\frac{p}{n+1}} \Vol \( \widehat{\Omega}\)^{\frac{n+1-p}{n+1}},
	\end{equation}
	with equality if and only if $\widehat{\Omega}$ is a geodesic ball centered at the origin.
\end{thmI}

\begin{rem} $ \ $
	\begin{enumerate}
		\item If $p=1$, then inequalities \eqref{classical p-BM ineq} and \eqref{classcial p-Min ineq} also hold, which are the classical Brunn-Minkowski inequality and the Minkowski inequality. Each of the inequalities holds equality if and only if $\widehat{K}$ and $\widehat{L}$ are homothetic.
		\item In the case $0 \leq p <1$, the studies of $L_p$ Brunn-Minkowski inequality and $p$-Minkowski inequality become more difficult. B\"{o}r\"{o}czky et al. \cite{BLYZ12} considered the $p=0$ case for origin symmetric convex bodies in $\mathbb{R}^2$. For origin symmetric convex bodies in higher dimensional Euclidean space $\mathbb{R}^{n+1}$ ($n\geq 2$), the inequalities were proved for $p$ close to $1^-$, see \cite{CHLL20, KM20}.
	\end{enumerate}
\end{rem}

\subsection{A projection map from $\mathbb{H}^{n+1}$ to $\mathbb{R}^{n+1}$}\label{subsec-A light cone projection}$ \ $

Denote by $\widehat{B}(x, r)$ the geodesic ball of radius $r$ centered at $x$ in $\mathbb{R}^{n+1}$.
\begin{defn}\label{def-light conr proj-pi}
 Define a map $\pi$ by
	\begin{equation*}
	\pi \(\Omega\) :  = \bigcup_{X = (x, x_{n+1}) \in \Omega} \widehat{B}\(-x, x_{n+1} \),
	\end{equation*}
	where $\Omega$ is any set in $\mathbb{H}^{n+1} \subset \mathbb{R}^{n+1,1}$.
\end{defn}
Geometrically, $-\pi (X) =  \widehat{B}\(x, x_{n+1} \)$ is the domain enclosed by the intersection of the past light cone of $X= (x, x_{n+1})$ with $\mathbb{R}^{n+1} \times \{0\}$ in $\mathbb{R}^{n+1,1}$. Here, we provide a concrete example of Definition \ref{def-light conr proj-pi} in the following Figure \ref{Fig-light cone projection}.
\begin{figure}[htbp]
	\centering
	\begin{tikzpicture}[scale=1.5]
	\begin{axis}
	[
	axis lines= middle,
	axis on top,
	hide axis,
	axis equal image
	]
	\addplot3[surf, opacity=0.3, shader=interp,  
	trig format plots = rad,
	domain = 0:2*pi, domain y=0:4]
	({2+y*cos(x)}, {-2+y*sin(x)}, {0});
	\addplot3[surf, opacity=0.6, shader=interp, 
	trig format plots = rad, colormap/viridis,
	domain = 0:2*pi, domain y=0:4]
	({2+y*cos(x)}, {-2+y*sin(x)}, {4-y});
	\addplot3[domain = -2:6, ->]({x},{0},{0});
	\addplot3[domain = 2:-6, ->]({0},{x},{0});
	\addplot3[domain = -1:6, ->]({0},{0},{x});
	\addplot3[surf, domain =-3:4, domain y = -3:4, opacity=1,	colormap/viridis, shader=interp]({x}, {y}, {sqrt(x*x+y*y+1)});
	\addplot3[surf, opacity=1, shader=interp,  colormap/cool,
	trig format plots = rad,
	domain = 0:2*pi, domain y=0:1]
	({1+sqrt(2)*y*cos(x)}, {-1+y*sin(x)}, {sqrt(3+y*y+y*y*cos(x)*cos(x) +2*sqrt(2)*y*cos(x)-2*y*sin(x) )});
	\addplot3[domain = 1.55:6, ->]({0},{0},{x});
	\end{axis}
	\end{tikzpicture}
	\caption{In this figure, the green surface denotes the hyperboloid model of $\mathbb{H}^2$ in $\mathbb{R}^{2,1}$. Let $X= (\frac{\sqrt{2}}{2}, \frac{\sqrt{2}}{2}, \sqrt{2})$ be a point on $\mathbb{H}^2$. The blue domain $\Omega$ denotes the geodesic ball of radius $\frac{3}{2} \log 2$ centered at $X$ on $\mathbb{H}^2$, i.e. $\Omega = B(X, \frac{3}{2} \log r)$.  Then the purple domain lying on $\{ x_2=0 \}$ denotes $-\pi(\Omega)$. Later, Proposition \ref{prop-lc proj maps ball to ball} will imply that $-\pi(\Omega) =\widehat{B}( (2,2), 4) \subset \mathbb{R}^2$.}
	\label{Fig-light cone projection}
\end{figure}
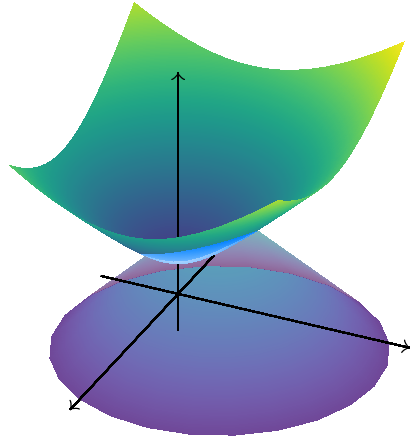

We will show that $\pi$ maps h-convex domains in $\mathbb{H}^{n+1}$ to convex domains in $\mathbb{R}^{n+1}$ in the following proposition. 
\begin{prop}\label{prop-lc proj map-supp func}
	Assume that $\Omega$ is a smooth uniformly h-convex bounded domain in $\mathbb{H}^{n+1}$ with horospherical support function $u_{\Omega}(z)$. Then $\widehat{\Omega} :=\pi \(\Omega\)$ is a smooth uniformly convex bounded domain in $\mathbb{R}^{n+1}$ with support function
	\begin{equation}\label{lc proj map-supp func}
	\widehat{u}_{\widehat{\Omega}} \(z\) = \g_{\Omega}(z) = e^{u_{\Omega}(z)}>0.
	\end{equation}
	Consequently, $\widehat{\Omega}$ contains the origin in its interior.
\end{prop}

Before proving Proposition \ref{prop-lc proj map-supp func}, we state its applications in the following Theorem \ref{thm-rel p-sums in Eucl and Hyperbolic}, Corollary \ref{cor-lc proj map-bdy map} and Proposition \ref{prop-lc proj maps ball to ball}.
\begin{thm}\label{thm-rel p-sums in Eucl and Hyperbolic}
	Let $p$, $a$, $b$, $K$ and $L$ satisfy the same assumptions as in Definition \ref{def-p sum}. Let $\Omega = a \cdot K +_p^{\mathcal{H}} b \cdot  L$. Then
	\begin{equation}\label{formula-rel Eucl sum and Hyper sum} 
	\pi (\Omega) = a \cdot \pi(K) +_p^{\mathcal{R}} b\cdot \pi(L).
	\end{equation} 
\end{thm}

\begin{proof}
	Let $\widehat{K} = \pi (K)$, $\widehat{L} = \pi (L)$, $\widehat{\Omega} = \pi(\Omega)$, and $\g_{\Omega}^p(z) = \alpha \g_K^p (z) +\beta \g_L^p (z)$. By \eqref{def A-phi} and Proposition \ref{prop-Aij >=0}, we have
	\begin{equation*}
	 D_j D_i \g_{\Omega}+ \g_{\Omega} \delta_{ij} = A_{ij}[\g_{\Omega} ] + \frac{1}{2} \frac{|D\g_{\Omega}|^2}{\g_{\Omega}}\delta_{ij} + \frac{1}{2} \( \g_{\Omega} + \g_{\Omega}^{-1}\) \delta_{ij}\geq 1>0.
	\end{equation*}
	This together with \eqref{lc proj map-supp func} gives
	\begin{align*}
	&D^2 \widehat{u}_{\widehat{\Omega}} + \widehat{u}_{\widehat{\Omega}} I =D^2 \g_{\Omega} + \g_{\Omega} I  >0,\\
	&\widehat{u}^p_{\widehat{\Omega}} (z) = \g^p_{\Omega}(z)
	= a\g^p_{K}(z) +b \g^p_{L}(z)
	= a \widehat{u}_{\widehat{K}}^p(z) +b \widehat{u}_{\widehat{L}}^p (z).
	\end{align*}
	 Therefore $\widehat{\Omega} = a \cdot \widehat{K} +_p^{\mathcal{R}} b \cdot \widehat{L}$. This completes the proof of Theorem \ref{thm-rel p-sums in Eucl and Hyperbolic}.
\end{proof}

\begin{cor}\label{cor-lc proj map-bdy map}
	Let $\Omega$ be a smooth uniformly h-convex bounded domain in $\mathbb{H}^{n+1}$ and $\widehat{\Omega} = \pi (\Omega)$. Define $\xi : \partial \Omega \to \mathbb{R}^{n+1}$ by 
	\begin{equation*}
	\xi (X) = \frac{\widetilde{u}X - \cosh r \nu}{\cosh r -\widetilde{u}}= -X + x_{n+1} (z,1),
	\end{equation*}
	where $X= (x, x_{n+1}) \in \partial \Omega$ and $z = G_{\Omega} (X)$.
	Then $\xi$ gives a diffeomorphism from $\partial \Omega$ to $\partial \widehat{\Omega}$. Moreover, $G_{\Omega} (X) = \widehat{G}_{\widehat{\Omega}} \( \xi(X) \)$ for all $X \in \partial \Omega$.  
\end{cor}

\begin{proof}
	For any $z \in \mathbb{S}^n$, it is well-known that $\widehat{G}_{\widehat{\Omega}}^{-1} (z) = \(D \widehat{u}_{\widehat{\Omega}} (z)+ \widehat{u}_{\widehat{\Omega}}(z) z, 0 \)$.
	Using \eqref{lc proj map-supp func} and \eqref{phi(coshr nu-uX)}, we obtain
	\begin{equation*}
	\widehat{G}_{\widehat{\Omega}}^{-1} (z) = \(D \g_{\Omega}(z) + \g_{\Omega}(z) z,0  \)= \xi (X(z)).
	\end{equation*}
	We complete the proof of Corollary \ref{cor-lc proj map-bdy map}.
\end{proof}

\begin{prop}\label{prop-lc proj maps ball to ball}
	Let $X = (x, x_{n+1}) \in \mathbb{H}^{n+1}$ and $\Omega:=B(X, r)$ be a ball of radius $r$ centered at $X$ in  $\mathbb{H}^{n+1}$. 
	Then $\widehat{\Omega} =\pi \( \Omega\)$ is given by $\widehat{B} (-e^rx,  e^r x_{n+1} )$.
\end{prop}

\begin{proof}
	It is well-known that the classical support function of the geodesic ball $\widehat{B} (y, \rho)$ is given by
	\begin{equation*}
	\widehat{u}_{\widehat{B} (y, \rho)  }(z) = \rho+ \metric{y}{z}, \quad z \in \mathbb{S}^n.
	\end{equation*}
	On the other hand, formulas \eqref{lc proj map-supp func} and \eqref{horo supp of geodesic ball} yield that
	\begin{equation*}
	\widehat{u}_{\widehat{\Omega}} (z) =\g_{\Omega}(z) = e^r\(x_{n+1} - \metric{x}{z}\).
	\end{equation*}
	Then Proposition \ref{prop-lc proj maps ball to ball} follows from the above formulas.
\end{proof}
\subsection{Proof of Proposition \ref{prop-lc proj map-supp func}} \label{subsec-proof of prop lc-map-supp}$ \ $

\begin{lem}\label{lem-lc proj-holds for a point}
	Proposition \ref{prop-lc proj map-supp func} holds when $\Omega$ degenerates to a single point $X= \(x, x_{n+1}\)$ in $\mathbb{H}^{n+1}$.
\end{lem}

\begin{proof}
	From Definition \ref{def-light conr proj-pi}, we have $\widehat{B}\(-x, x_{n+1}\) = \pi (X)$. Then we have
	\begin{equation*}
	\widehat{u}_{\widehat{B}\(-x, x_{n+1}\)} \(z\) = \sup_{y \in \widehat{B}\(-x, x_{n+1}\)} \metric{y}{z} = \metric{-x}{z} + x_{n+1}.
	\end{equation*}
	At the same time, we have
	\begin{equation*}
	e^{u_X(z)} = \g_X \(z\) = - \metric{X}{ \(z,1\)} = - \metric{x}{z} +x_{n+1}. 
	\end{equation*}
	The above two formulas imply $\widehat{u}_{\widehat{B}\(-x, x_{n+1}\)} (z) = e^{u_X \(z\)}$. We complete the proof of Lemma \ref{lem-lc proj-holds for a point}.
\end{proof}

\begin{proof}[Proof of Proposition \ref{prop-lc proj map-supp func}]
	In this proof, we will use Lemma \ref{lem-lc proj-holds for a point} frequently without mention.
	By the strict horo-convexity of $\Omega$, we have $A[\g_{\Omega}] >0$. Using \eqref{def A-phi} and \eqref{coshr}, we know
	\begin{equation} \label{D^2 phi+phi I}
	D^2 \g_{\Omega}(z) + \g_{\Omega}(z) I = A [\g_{\Omega}] + \cosh r\(X(z) \) >I, \quad \forall z \in \mathbb{S}^n.
	\end{equation}
	Then we can assume that $\widetilde{\Omega}$ is the convex body in $\mathbb{R}^{n+1}$ with support function $\widehat{u}_{\widetilde{\Omega}}(z) = \g_{\Omega}(z)$.  For any  $y \in \widehat{\Omega}$, there exists $X \in \Omega$ such that $y \in \pi \(X\)$. Consequently, $\metric{y}{z} \leq \widehat{u}_{\pi(X)}(z) = -\metric{x}{z}+x_{n+1}$ for all $z \in \mathbb{S}^n$.
	 Since $\g_X \(z\) \leq \g_\Omega(z)$ for all $z \in \mathbb{S}^n$,
	we have
	\begin{equation*}
	\metric{y}{z} \leq -\metric{x}{z} +x_{n+1} = \g_X \(z\) \leq \g_\Omega\(z\) =e^{u(z)}, \quad \forall z\in \mathbb{S}^n.
	\end{equation*}
	Hence $y \in \widetilde{\Omega}$ and then $\widehat{\Omega} \subset \widetilde{\Omega}$. 
	
	For any $y_0 \in \widetilde{\Omega}$, we claim that there exists $X_0 = \(x^0, x_{n+1}^0\) \in \Omega$ satisfying
	$y_0 \in \pi (X_0) =\widehat{B}\(-x^0, x^0_{n+1}\) \subset \widehat{\Omega}$. This will imply $\widetilde{\Omega} \subset \widehat{\Omega}$ and then $\widetilde{\Omega} = \widehat{\Omega}$. 
	
	 Since $y_0 \in \widetilde{\Omega}$, we have
	\begin{equation}\label{t=0, aux leq 0}
	\metric{y_0}{z} - \g_{\Omega}\(z\)  \leq 0, \quad \forall z \in \mathbb{S}^n. 
	\end{equation}
	Let $X_1 = \(x^1, x^1_{n+1}\)$ be an interior point of $\Omega$. Obviously, $\g_{X_1}(z) < \g_\Omega(z)$. If $y_0 \in \widehat{B}\(-x^1, x^1_{n+1}\) = \pi \(X_1\)$, then we can choose $X_0 := X_1$. Hence we can assume that $y_0$ lies outside $\pi \(X_1\)$ without loss of generality, which implies
	\begin{equation} \label{t=1, aux > 0}
	\max_{z \in \mathbb{S}^n} \( \metric{y_0}{z} - \g_{X_1} \(z\) \) >0.
	\end{equation}
	We define an auxiliary function 
	\begin{equation}\label{phi_t aux-func}
	\g_t(z) :=  t \g_{\Omega}(z) + \(1-t\) \g_{X_1}(z), \quad t \in (0,1].
	\end{equation}
	From \eqref{t=0, aux leq 0} and \eqref{t=1, aux > 0}, there exists $t_0 \in (0,1]$ such that
	\begin{align*}
	\max_{z \in \mathbb{S}^n} \(\metric{y_0}{z} - \g_{t_0} \(z\) \)= 0,
	\end{align*}
	and we assume that the maximum is attained at $z_0 \in \mathbb{S}^n$, which implies $\metric{y_0}{z_0} = \g_{t_0} \(z_0\)$. At $z_0$, we have
	\begin{equation*}
	\metric{y_0}{e} - \metric{D\g_{t_0} \(z_0\)}{e} =0, \quad \forall e \in T_{z_0} \mathbb{S}^n.
	\end{equation*}
	Then we can assume $y_0 - D \g_{t_0} \( z_0\) = \lambda z_0$ for some $\lambda \in \mathbb{R}$. Taking the inner product with $z_0$ on both sides of this formula, we get $\lambda = \metric{y_0}{z_0} = \g_{t_0} \(z_0\)$. Hence
	\begin{equation*}
	y_0 = D \g_{t_0} \(z_0\) + \g_{t_0}(z_0) z_0.
	\end{equation*}
	By Definition \ref{def-p sum} and \eqref{phi_t aux-func}, we know that $\log \g_{t_0}(z)$ is the horospherical support function of a uniformly h-convex bounded domain $\Omega_{t_0}$, see also \cite[Proposition 5.1]{ACW18}. Since $\g_{t_0}(z) \leq  \g_{\Omega}(z)$, we know $\Omega_{t_0} \subset \Omega$.
	By using \eqref{X(z)},  we can find a point $X_0 := X(z_0) \in \partial \Omega_{t_0} \subset \Omega$. By \eqref{X-coshr (z,1)}, we obtain
	\begin{equation*}
	X_0 - x_{n+1}^0 (z_0,1)
	= -\( D \g_{t_0}(z_0) + \g_{t_0}(z_0) z_0,0 \)=-y_0.
	\end{equation*}
	It is easy to see that $\(X_0 - x_{n+1}^0 \(z_0,1\)\) \in \widehat{B}(x_0, x_{n+1}^0)$, which implies 
	\begin{equation*}
	y_0 \in \widehat{B} \(-x_0, x^0_{n+1}\)= \pi (X_0) \subset \widehat{\Omega}.
	\end{equation*}
	Then $\widetilde{\Omega} \subset \widehat{\Omega}$ and hence $\widetilde{\Omega} = \widehat{\Omega}$.
	This shows that $\widehat{\Omega} =\pi (\Omega)$ is a convex body with support function $\g_{\Omega}(z)$ in $\mathbb{R}^{n+1}$. As $\g_{\Omega}(z)= e^{u_{\Omega} (z)}>0 $, we know that $\pi (\Omega)$ contains the origin in its interior.
	We complete the proof of Proposition \ref{prop-lc proj map-supp func}.
\end{proof}

\subsection{Weighted $p$-Brunn-Minkowski type inequality  and weighted $p$-Minkowski type inequality for h-convex domains}
\label{subsec-weighted p-BM type ineq}$ \ $

The $p$-Brunn-Minkowski inequality (Theorem G) and $p$-Minkowski inequality (Theorem H) play important roles in the convex geometric analysis.
Proposition \ref{prop-lc proj map-supp func} and Theorem \ref{thm-rel p-sums in Eucl and Hyperbolic} suggest us to interpret them  to  geometric inequalities for h-convex domains in $\mathbb{H}^{n+1}$ by use of map $\pi$. The results in this subsection will be proved in Subsection \ref{subse-pf of weighted p-BM type ineq}.
 
To state our results, we should define two geometric functionals.
\begin{defn}
	Let $p$ be a real number.
	Let $K$ and $L$ be two smooth uniformly h-convex bounded domain in $\mathbb{H}^{n+1}$. We define
	\begin{align}
	\mathcal{V} (K) :=& \frac{1}{n+1} \int_{\partial K} \frac{p_n \( \cosh r \kappa - \widetilde{u} \)}{\(\cosh r - \widetilde{u}\)^{n+1}} d\mu, \label{mathcal V-K}\\
	\mathcal{V}_p \( K, L \):=& \frac{1}{n+1}\int_{\partial K} \g_L^p (z) \frac{p_n \( \cosh r \kappa - \widetilde{u} \)}{\(\cosh r - \widetilde{u}\)^{n+1-p}} d\mu. \label{mathcal V-K,L}
	\end{align}
\end{defn}

\begin{prop}\label{prop-mathcal V monotone}
	The geometric functional $\mathcal{V}$ is monotone increasing with respect to inclusions for h-convex domains.
\end{prop}
Theorem G can be used to get a weighted horospherical $p$-Brunn-Minkowski type inequality for h-convex domains in $\mathbb{H}^{n+1}$. 
\begin{thm}\label{thm-weighted p-BM ineq-lc proj}
	Let $1 \leq p \leq 2$, $a \geq 0$, $b \geq 0$  and $a+b \geq 1$. Assume that $K$ and $L$ are two smooth uniformly h-convex bounded domains in $\mathbb{H}^{n+1}$.
	Let $\Omega = a \cdot K +_p^{\mathcal{H}} b \cdot L$. Then
	\begin{equation}\label{a cls of weighted BM ineq}
	\mathcal{V}^{\frac{p}{n+1}} \( \Omega\) 
	\geq a^\frac{1}{p} \mathcal{V}^{\frac{p}{n+1}} \( K\) + b^\frac{1}{p} \mathcal{V}^{\frac{p}{n+1}} \( L\).
	\end{equation}
	When $p=1$, equality holds in \eqref{a cls of weighted BM ineq} if $K$ and $L$ are hyperbolic dilates.
	When $1<p \leq 2$ and $ab \neq 0$, equality holds in \eqref{a cls of weighted BM ineq} if and only if $K$ and $L$ are hyperbolic dilates. 
\end{thm}

Theorem H can be used to get a weighted horospherical $p$-Minkowski type inequality for h-convex domains in $\mathbb{H}^{n+1}$.
\begin{thm}\label{thm-weighted p-Min ineq-lc proj}
	Let $p \geq 1$. Assume that $K$ and $L$ are two smooth uniformly h-convex bounded domains in $\mathbb{H}^{n+1}$.
	Then
	\begin{equation} \label{a cls of weighted Min ineq}
	\mathcal{V}_p \( K, L\) \geq \mathcal{V}^\frac{n+1-p}{n+1} \( K \)\mathcal{V}^\frac{p}{n+1} \( L \).
	\end{equation}
	When $p=1$, equality holds in \eqref{a cls of weighted Min ineq}  if $K$ and $L$ are hyperbolic dilates.
	When $p>1$, equality holds in \eqref{a cls of weighted Min ineq} if and only if $K$ and $L$ are hyperbolic dilates.
\end{thm}
The following Theorem \ref{thm-weighted p-iso ineq-lc proj} is the counterpart of Theorem I in $\mathbb{H}^{n+1}$.
\begin{thm}\label{thm-weighted p-iso ineq-lc proj}
		Let $p \geq 1$. Assume that $\Omega$ is a smooth uniformly h-convex bounded domain in $\mathbb{H}^{n+1}$. Then
		\begin{equation}\label{weighted p-iso type ineq from lc proj}
		\int_{\partial \Omega} \frac{p_n \( \cosh r \kappa- \widetilde{u} \)}{ \( \cosh r -\widetilde{u}\)^{n+1-p} } d\mu
		\geq \omega_n^{\frac{p}{n+1}} \( \int_{\partial \Omega} \frac{p_n \( \cosh r \kappa- \widetilde{u} \)}{ \( \cosh r -\widetilde{u}\)^{n+1} }   d\mu   \)^{\frac{n+1-p}{n+1}}.
 		\end{equation}
		When $p>1$, equality holds in \eqref{weighted p-iso type ineq from lc proj} if and only if $\Omega$ is a geodesic ball centered at the origin. When $p=1$, equality holds in \eqref{weighted p-iso type ineq from lc proj} if and only if $\Omega$ is a geodesic ball.
\end{thm}

\subsection{Proof of results in Subsection \ref{subsec-weighted p-BM type ineq}} \label{subse-pf of weighted p-BM type ineq}$ \ $

\begin{lem}\label{lem-hyperbolic version of p-mixed volume}
	Let $p$ be a real number.
	Assume that $K$ and $L$ are smooth uniformly h-convex bounded domains in $\mathbb{H}^{n+1}$. Let $\widehat{K} = \pi (K)$ and $\widehat{L} = \pi (L)$. Then
	\begin{align}
	\mathcal{V} (K) =& \Vol \( \widehat{K} \), \label{lc-hyper version of volume} \\
	\quad \mathcal{V}_p (K,L) =& \Vol_p \( \widehat{K}, \widehat{L} \). \label{lc-hyper version of p-mixed volume}
	\end{align}
\end{lem}

\begin{proof}
	The following formula is well-known in convex geometry (see e.g., \cite[p. 121]{Sch14}),
	\begin{equation}\label{area ele Eucl Gauss map}
	d\widehat{\mu}_{\widehat{K}}\( \widehat{G}^{-1}_{\widehat{K}} (z)\) = p_n \(D^2 \widehat{u}_{\widehat{K}}(z) + \widehat{u}_{\widehat{K}}(z)\) d\sigma(z).
	\end{equation}
	Recall the definition of $p$-mixed volume in \eqref{p-mixed volume}.
	Formula \eqref{area ele Eucl Gauss map} together with \eqref{lc proj map-supp func} gives
	\begin{align}
	\Vol_p \( \widehat{K}, \widehat{L} \)
	:=& \frac{1}{n+1} \int_{\partial \widehat{K}} \widehat{u}_{\widehat{L}}^p(z)  \widehat{u}_{\widehat{K}}^{1-p}(z) d\widehat{\mu}_K \nonumber\\
	=&\frac{1}{n+1} \int_{\mathbb{S}^n} \widehat{u}^p_{\widehat{L}} \widehat{u}_{\widehat{K}}^{1-p} p_n \( D^2 \widehat{u}_{\widehat{K}} + \widehat{u}_{\widehat{K}} I\) d\sigma \nonumber\\
	=& \frac{1}{n+1} \int_{\mathbb{S}^n} \g_L^p \g_K^{1-p} p_n\(D^2 \g_K+\g_K I\) d\sigma  \label{p-mixed volume on sphere}
	\end{align}
	Using \eqref{D^2 phi+phi I}, \eqref{shifted curvature-support function} and \eqref{1/phi, coshr-u}, we have
	\begin{align}
	p_n \(D^2\g_K(z) +\g_K(z) I\) =& p_n \(\cosh r I+ A[\g_K]  \) \nonumber\\
	=& \prod_{i=1}^n \( \cosh r+ \( \cosh r- \widetilde{u} \) \( \tilde{\lambda}_K \)_i \) \nonumber\\
	=& p_n \( \tilde{\lambda}_K \) \prod_{i=1}^n \( \cosh r \(\tilde{\kappa}_K\)_i + \cosh r- \widetilde{u }\) \nonumber\\
	=& p_n \(\tilde{\lambda}_K \) p_n \( \cosh r \kappa_K - \widetilde{u} \), \label{Det of D^2 phi+ phi I}
	\end{align}
	where the right-hand side of \eqref{Det of D^2 phi+ phi I} was evaluated at $G^{-1}_K \(z\)$.
	From \eqref{rel-area element} and \eqref{shifted curvature-support function}, we obtain
	\begin{equation}\label{area ele Hyper Gauss map}
	\(G^{-1}_K\)^* \(d\sigma(z) \) = \det \(A[\g_K]\)^{-1} d\mu_K = \g_K^n p_n \( \tilde{\kappa}_K \) d\mu_K\( G^{-1}_K (z) \).
	\end{equation}
	Inserting \eqref{Det of D^2 phi+ phi I} and \eqref{area ele Hyper Gauss map} into \eqref{p-mixed volume on sphere}, and using \eqref{mathcal V-K,L}, we get
	\begin{align*}
	\Vol_p \( \widehat{K}, \widehat{L} \)
	=& \frac{1}{n+1}\int_{\partial K} \g_L^p (z) \g_K^{1-p}(z)  p_n \( \tilde{\lambda}_K\) p_n \( \cosh r\kappa_K -\widetilde{u} \) \g_K^n p_n \( \tilde{\kappa}_K\) d\mu_K\\
	=&\frac{1}{n+1}\int_{\partial K} \g_L^p (z) \frac{ p_n \( \lambda'\kappa_K -\widetilde{u} \)}{\( \cosh r-\widetilde{u}\)^{n+1-p}} d\mu_K\\
	=& \mathcal{V}_p (K,L). 
	\end{align*}
	Hence we obtain \eqref{lc-hyper version of p-mixed volume}.
	Formula \eqref{lc-hyper version of volume} follows by taking $L=K$ in \eqref{lc-hyper version of p-mixed volume}.
	We complete the proof of Lemma \ref{lem-hyperbolic version of p-mixed volume}.
\end{proof}

\begin{cor}
	Let $p$ be a real number. Assume that $K$ is a smooth uniformly h-convex bounded domain in $\mathbb{H}^{n+1}$ and $\widehat{K} = \pi (K)$. Then
	\begin{equation}\label{hyper version of p-surface area}
	\int_{\partial \widehat{K}} \widehat{u}_{\widehat{K}}^{1-p}(z) d\widehat{\mu}
	= \int_{\partial K} \frac{p_n \( \cosh r \kappa -\widetilde{u} \)}{ \(\cosh r- \widetilde{u}\)^{n+1-p}  } d\mu.
	\end{equation}
\end{cor}

\begin{proof}
	Suppose that $L \subset \mathbb{H}^{n+1}$ is a geodesic ball centered at the origin, which means  $\g_L (z) \equiv C_L$ for some constant $C_L \geq 1$. Proposition \ref{prop-lc proj map-supp func} gives $\widehat{u}_{\widehat{L}}(z) \equiv C_L$.
	Recall the definitions of $\Vol_p(\widehat{K}, \widehat{L})$ and $\mathcal{V}(K, L)$ in \eqref{p-mixed volume} and \eqref{mathcal V-K,L}, respectively. 
	Then \eqref{hyper version of p-surface area} follows directly from \eqref{lc-hyper version of p-mixed volume}.
\end{proof}

\begin{lem}\label{lem-Eucl dilates equiv Hyper parallel}
	Assume that $K$ and $L$ are smooth uniformly h-convex bounded domains in $\mathbb{H}^{n+1}$. Let $\widehat{K} = \pi (K)$ and $\widehat{L} = \pi (L)$. The following statements are equivalent,
	\begin{enumerate}
		\item $K$ and $L$ are hyperbolic dilates,
		\item $\widehat{K}$ and $\widehat{L}$ are dilates.
	\end{enumerate}
\end{lem}
\begin{proof}
	Lemma \ref{lem-Eucl dilates equiv Hyper parallel} follows directly from Proposition \ref{prop-lc proj map-supp func} and Definition \ref{def-hyperbolic dilates}.
\end{proof}

\begin{proof}[Proof of Proposition \ref{prop-mathcal V monotone}]
	Let $\Omega_1$ and $\Omega_2$ be two h-convex domains in $\mathbb{H}^{n+1}$ with $\Omega_1 \subset \Omega_2$. Since $\g_{\Omega_1} \leq  \g_{\Omega_2}$, \eqref{lc proj map-supp func} implies that $\widehat{\Omega}_1 \subset \widehat{\Omega}_2$. Therefore, Proposition \ref{prop-mathcal V monotone} follows from \eqref{lc-hyper version of volume}. 
\end{proof}

\begin{proof}[Proof of Theorems \ref{thm-weighted p-BM ineq-lc proj} -- \ref{thm-weighted p-iso ineq-lc proj} ]
	Combining Theorem G with \eqref{formula-rel Eucl sum and Hyper sum} and \eqref{lc-hyper version of volume}, we obtain \eqref{a cls of weighted BM ineq}. By Lemma \ref{lem-Eucl dilates equiv Hyper parallel}, equality holds in \eqref{a cls of weighted BM ineq} if $K$ and $L$ are hyperbolic dilates. Besides, if $1<p \leq 2$, then equality holds in \eqref{a cls of weighted BM ineq} if and only if $K$ and $L$ are hyperbolic dilates. Then we complete the proof of Theorem \ref{thm-weighted p-BM ineq-lc proj}.
	
	Theorem \ref{thm-weighted p-Min ineq-lc proj} follows from Theorem H, \eqref{lc-hyper version of volume}, \eqref{lc-hyper version of p-mixed volume} and a similar argument as above.

	Inequality \eqref{weighted p-iso type ineq from lc proj} follows from \eqref{p-isoperimetric inequality}, \eqref{hyper version of p-surface area} and \eqref{lc-hyper version of volume}. When $p>1$, by Theorem G and Proposition \ref{prop-lc proj maps ball to ball}, we know that  the equality holds in \eqref{weighted p-iso type ineq from lc proj} if and only if $\Omega$ is a geodesic ball centered at the origin. When $p=1$, by the classical isoperimetric inequality and Proposition \ref{prop-lc proj maps ball to ball},  we have that the equality holds in \eqref{weighted p-iso type ineq from lc proj} if and only if $\Omega$ is a geodesic ball. We complete the proof of Theorem \ref{thm-weighted p-iso ineq-lc proj}.
\end{proof}

In the last part of Section \ref{sec-hyperbolic sum and Firey's sum}, we would like to make the following comments on map $\pi$. 

\begin{enumerate}
	\item The Firey's $p$-sum ($p>1$) for convex bodies in Euclidean space depends on the choice of the origin of $\mathbb{R}^{n+1}$.
	Now we choose a point to be the origin of $\mathbb{H}^{n+1}$, i.e. $N=(0,1)\in \mathbb{R}^{n+1,1}$ in the hyperboloid model. Then map $\pi$ gives a perfect relationship between the hyperbolic $p$-sum and the Firey's $p$-sum, which was proved in Theorem \ref{thm-rel p-sums in Eucl and Hyperbolic}.  
	 We have shown in Proposition \ref{prop-geo-hyper-p-dilation} and Proposition \ref{prop-sums not rely on origin} that the both hyperbolic $p$-sum and hyperbolic $p$-dilation can be constructed intrinsically in hyperbolic space $\mathbb{H}^{n+1}$, which is a different phenomenon comparing with the Firey's $p$-sum.
	 
	 \item 	The approach in Section \ref{sec-hyperbolic sum and Firey's sum} relies on the h-convexity of domains in $\mathbb{H}^{n+1}$. Specifically, we used the horospherical support function to investigate the properties of the map $\pi$. Recall that we defined hyperbolic $p$-sums for sets in Definition \ref{def-p sum-p>0-new}, and the definition of $\pi$ does not require the convexity of domains.
	 For nonconvex compact sets in Euclidean space, the $p$-Brunn-Minkowski inequality $(p > 1)$ was proved in \cite{LYZ12}.
	 Thus, it is natural to study whether the results in Section \ref{sec-hyperbolic sum and Firey's sum} still hold in a more general setting.
	 For example, we can ask: \emph{Can one remove the h-convex assumption of $K$ and $L$ in Theorem \ref{thm-weighted p-BM ineq-lc proj}}?
	
	\item According to Gardner \cite{Gar02}, if inequalities are silver currency in mathematics, those that come along with precise equality conditions are gold. 
	For many important geometric inequalities of multiple domains in $\mathbb{R}^{n+1}$, the equality holds if and only if the domains are dilates. Similarly, the inequalities of a single domain in $\mathbb{R}^{n+1}$ holds equality always require that the domain is a  geodesic ball.
	So, another strong motivations that suggest us to move geometric inequalities for domains in $\mathbb{R}^{n+1}$ to those in $\mathbb{H}^{n+1}$ are Proposition \ref{prop-lc proj maps ball to ball} and Lemma \ref{lem-Eucl dilates equiv Hyper parallel}.  
	In conclusion, our approach in Section \ref{sec-hyperbolic sum and Firey's sum} can be used to find many geometric inequalities for h-convex domains in $\mathbb{H}^{n+1}$ by using inequalities for domains in Euclidean space. Furthermore, the approach can give new characterizations of geodesic balls in $\mathbb{H}^{n+1}$. Finally, we will give a concrete example in Theorem \ref{thm-hyoer version of BCD17}.
	
	\item For inequality \eqref{weighted p-iso type ineq from lc proj}, the h-convex assumption of $\Omega$ in Theorem \ref{thm-weighted p-iso ineq-lc proj} is not natural. 
	A domain $\Omega$ in $\mathbb{H}^{n+1}$ is called static-convex if the principal curvatures are greater than or equal to $\frac{\widetilde{u}}{\cosh r}$ on $\partial \Omega$. The static-convex assumption of domains appeared in many important geometric inequalities, see e.g. \cites{BW14, HL21, Xia16}.
	Therefore, we conjecture that Theorem \ref{weighted p-iso type ineq from lc proj} holds for all static-convex domains in $\mathbb{H}^{n+1}$.

	\item In this paper, we focus on geometric inequalities of  domains in $\mathbb{H}^{n+1}$. Let $\Omega$ be a domain in $\mathbb{H}^{n+1}$ and $M=\partial \Omega$ be its boundary.  It is easy to see that $\nu +X$ and $\nu-X$ are two light-like, outward pointing normal of $M \subset \mathbb{R}^{n+1,1}$, where $\nu+X$ is future directed and $\nu-X$ is past directed. Besides, $\nu-X = -\g^{-1}(z) \(z,1\)$. Then we know that $-\xi\(M\)$ is the intersection of the null hypersurface generated by $M$ and hyperplane $\mathbb{R}^{n+1} \times \{0\}$, see Corollary  \ref{cor-lc proj map-bdy map}. If we no longer require that the spacelike submanifold $M^n$ lies on $\mathbb{H}^{n+1}$, then the above structure was well studied in general relativity, see e.g.  \cites{MS12, MS14}. Hence, it is natural to ask: \emph{How to propose the results in Section \ref{sec-hyperbolic sum and Firey's sum} for general spacelike submanifold $M^n$ in $\mathbb{R}^{n+1,1}$}?
\end{enumerate}

Now we give another direct application of the map $\pi$. In \cite{BCD17}, Brendle, Choi, and Daskalopoulos classified the self-similar solutions of the $\alpha$-Gauss curvature flow for all $\alpha \geq \frac{1}{n+2}$, $n \geq 2$. 
For convenience, we do not consider the affine invariant case $\alpha= \frac{1}{n+2}$ in the following statement of their result.
\begin{thmJ}[\cite{BCD17}]
	Let $\alpha \in (\frac{1}{n+2}, +\infty)$ and $n \geq 2$. Assume that $\widehat{\Omega} \subset \mathbb{R}^{n+1}$ is a uniformly convex bounded domain containing the origin in its interior, and satisfies
	\begin{equation*}
	p_n(D^2 \widehat{u}_{\widehat{\Omega}}(z) + \widehat{u}_{\widehat{\Omega}}(z) I)^{-\alpha} = \widehat{u}_{\widehat{\Omega}}(z), \quad \forall z \in \mathbb{S}^n.
	\end{equation*}
	Then $\widehat{\Omega}$ is a geodesic ball centered at the origin.
\end{thmJ}
\begin{thm}\label{thm-hyoer version of BCD17}
	Let $\alpha \in (\frac{1}{n+2}, +\infty)$ and $n \geq 2$.
	Let $\Omega$ be a smooth uniformly h-convex bounded domain in $\mathbb{H}^{n+1}$. Assume that $\partial \Omega$ satisfies
	\begin{equation}\label{hyper version of BCD17}
	p_n \(\cosh r + \frac{\cosh r- \widetilde{u}}{\tilde{\kappa}} \)^\alpha = \gamma\(\cosh r- \widetilde{u}\)
	\end{equation}
	for some positive constant $\gamma$.
	Then $\Omega$ is a geodesic ball centered at the origin.
\end{thm}

\begin{proof}
	Using \eqref{Det of D^2 phi+ phi I}, \eqref{1/phi, coshr-u} and Proposition \ref{prop-lc proj map-supp func}, we have that equation \eqref{hyper version of BCD17} is equivalent to
	\begin{equation*}
	p_n \( D^2 \widehat{u}_{\pi \(\Omega\)}(z) + \widehat{u}_{\pi (\Omega)}(z) I\)^{\alpha} = \gamma \widehat{u}^{-1}_{\pi (\Omega)}(z), \quad 
	\forall z \in \mathbb{S}^n.
	\end{equation*}
	Letting $\widehat{\Omega} = \gamma^{-\frac{1}{1+ \alpha n}} \pi (\Omega)$ in Theorem J, we have that $\widehat{\Omega}$ is a geodesic ball centered at the origin. 
	By Proposition \ref{prop-lc proj maps ball to ball}, we know that $\Omega$ is a geodesic ball centered at the origin. Therefore, we complete the proof of Theorem \ref{thm-hyoer version of BCD17}.
\end{proof}
The study of curvature equations in space forms is a longstanding problem in geometric analysis,  see e.g. \cites{GLM18, GLW22, GRW15}. Combining results in that topic with the map $\pi$ in this section can give other characterizations of geodesic balls in $\mathbb{H}^{n+1}$.

\section{Appendix} \label{sec-appendix}
\subsection{Intrinsic proof of Lemma \ref{lem-conf-Schouten}}$ \ $

The following Lemma \ref{lem-Ric-R-conf-trans} is well known in conformal geometry (see e.g., \cite[Equation (2.1)]{GL11}).
\begin{lem}\label{lem-Ric-R-conf-trans}
	For any two conformal metrics $g$ and $\tilde{g}$, which satisfy $\tilde{g} = e^{2\lambda} g$ for some smooth function $\lambda$, we have
		\begin{align}
		\widetilde{\Ric}_{ij} =&\Ric_{ij}-(n-2)\(\nabla_j \nabla_i \lambda -\nabla_i \lambda \nabla_j \lambda\) - \Delta \lambda g_{ij}
		-(n-2) |\nabla \lambda|^2 g_{ij},\label{tilde Ric}\\
		e^{2\lambda} \widetilde{R}
		=&R-2(n-1)\Delta \lambda -(n-1)(n-2)|\nabla \lambda|^2, \nonumber
	\end{align}
	where $\widetilde{\Ric}_{ij}$ and $\widetilde{R}$ are the Ricci curvature tensor and scalar curvature with respect to metric $\tilde{g}$. 
\end{lem}

Now we give an intrinsic proof of Lemma \ref{lem-conf-Schouten}.
\begin{proof}[Proof of Lemma \ref{lem-conf-Schouten}]
	In this proof, we use normal coordinates around $z_0 \in (\mathbb{S}^n, g_0)$  such that the Ricci tensor is diagonal at $z_0$.
	Taking $\lambda=- \log \g$ in \eqref{tilde Ric}, at $z_0$ we have $\tilde{g}^{ij} = \g^2 \delta_{ij}$ and 
	\begin{align}
	\widetilde{\Ric}_{ij}
	=& (n-1) \delta_{ij} 
	-(n-2) \left( - (\log \g)_{ij} - \frac{\g_i \g_j}{\g^2}  \right)
	+ \Delta \log \g \delta_{ij} - (n-2) \frac{|D \g|^2}{\g^2}\delta_{ij}  \nonumber\\
	=&(n-1) \delta_{ij} +(n-2) \frac{\g_{ij}}{\g}
	+ \(\frac{\Delta \g}{\g}- \frac{|D \g|^2}{\g^2} \) \delta_{ij}
	-(n-2) \frac{|D \g|^2}{\g^2} \delta_{ij}.\label{phi-2-Rij}
	\end{align}
	Hence
	\begin{align}\label{phi-2-R}
	\widetilde{R} := 
	\tilde{g}^{ij} 	\widetilde{\Ric}_{ij}
	=
	n(n-1)\g^2 +2(n-1) \g \Delta \g -n(n-1) |D \g|^2.
	\end{align}
	Using \eqref{phi-2-Rij}, \eqref{phi-2-R} and \eqref{shifted curvature-support function}, at $z_0$ we have
	\begin{align}
	\tilde{g}^{ik}\widetilde{S}_{kj} =& \frac{1}{n-2}\(\g^2 \widetilde{\Ric}_{ij} - \frac{\widetilde{R}}{2(n-1)} \delta_{ij}\)
	=\g \(\g_{ij} - \frac{1}{2}\frac{|D \g|^2}{\g} \delta_{ij} + \frac{1}{2} \g \delta_{ij} \)    \nonumber\\
	=& \g A_{ij}[\g] +\frac{1}{2}\delta_{ij} = \(\tilde{\lambda}_i +\frac{1}{2}\) \delta_{ij},\label{S_ij-lambda-1/2}
	\end{align}
	where $\tilde{\lambda}_1, \ldots, \tilde{\lambda}_n$ are shifted principal  radii of curvature of $M$. This completes the proof of Lemma \ref{lem-conf-Schouten}.
\end{proof}
\subsection{Theorem \ref{thm-weighted Steiner formula} without uniformly h-convex condition}$ \ $

The original proof of Theorem \ref{thm-weighted Steiner formula} used horospherical support function, which depends on the h-convexity of $\Omega$.  In the following Proposition \ref{prop-new-pf-weighted Steiner formula}, we only assume that $\Omega \subset \mathbb{H}^{n+1}$ is a smooth convex bounded domain as we mentioned in Remark \ref{rem-WSF-convex}.

\begin{prop}\label{prop-new-pf-weighted Steiner formula}
	Let $n \geq 1$ be an integer and $\Omega$ be a smooth convex bounded domain in $\mathbb{H}^{n+1}$, then
	\begin{align*}
	\Vol_w(\Omega_\rho) - \Vol_w(\Omega)
	=& \sum_{k=0}^{n} \left( \int_{\partial \Omega} \cosh r \sigma_k (\tilde{\kappa}) d \mu \right) \int_0^\rho  e^{(n-k+1)t} \sinh^{k}t dt\\
	&-\sum_{k=0}^{n} \left( \int_{\partial \Omega} (\cosh r - \tilde{u}) \sigma_k (\tilde{\kappa}) d \mu \right) \int_0^\rho  e^{(n-k)t} \sinh^{k+1}t dt.
	\end{align*}
\end{prop}
 
\begin{proof}
	For any position vector $X \in \partial \Omega$, we know $\cosh r = - \metric{X}{ (0,1)} = x_{n+1}$ by \eqref{polar coord X}. Recall that $\overline{\nabla}$ and $\widetilde{\nabla}$ denote the connections on $\mathbb{H}^{n+1}$ and $\mathbb{R}^{n+1,1}$ respectively.
	As $V := \sinh r \partial_r = \overline{\nabla} \cosh r$, we have that $V$ is the projection of $(0,-1) =\widetilde{\nabla} x_{n+1} $ on $T_X \mathbb{H}^{n+1}$. Hence, $\tilde{u}: = \metric{V}{\nu} = -\metric{(0,1)}{\nu}$ at $X \in \partial \Omega$.
	Denote by $X_t$ the position vector on $\partial \Omega_t$. Then (see e.g., \cite[Remark 14]{EGM09})
	\begin{align*}
	X_t = \cosh t X + \sinh t \nu.
	\end{align*}	
	Let $r_t$ be the geodesic distance form $X_t$ to $N =(0,1) \in \mathbb{R}^{n+1,1}$, i.e. $r_t = d_{\mathbb{H}^{n+1}} (X_t, N)$. Then \eqref{geo-dis on hyperbolic space} implies
	\begin{align}
	\cosh r_t =& -\metric{X_t}{(0,1)}
	= -\cosh t \metric{X}{(0,1)} - \sinh t \metric{\nu}{(0,1)} \nonumber\\
	=&\cosh t \cosh r + \sinh t \tilde{u}
	=\cosh r e^t - \(\cosh r- \tilde{u}\)\sinh t.  \label{coshr-t}
	\end{align}
	Now we use normal coordinates for $\partial \Omega$ around some point $p \in \partial \Omega$ such that $g_{ij} =\delta_{ij}$ and $\(h_i{}^j\) = {\rm diag}\(\kappa_1, \ldots, \kappa_n\)$ at $p$. Then
	\begin{align*}
	\partial_i X_t = \cosh t \partial_i X + \sinh t \partial_i \nu
	=(\cosh t + \sinh t \kappa_i)\partial_i X
	=\(e^t + \sinh t \tilde{\kappa}_i \) \partial_i X.
	\end{align*}
	It follows that 
	\begin{align}\label{dmu-t}
	d \mu_t = \sqrt{\det \( \metric{\partial_i X_t}{ \partial_j X_t} \) } d \mu 
	= \prod \limits_{i=1}^n (e^t + \sinh t \tilde{\kappa}_i) d\mu
	= \sum_{k=0}^n e^{(n-k)t} \sinh^k t \s_k(\tilde{\kappa}) d\mu.
	\end{align}
	Using \eqref{coshr-t} and \eqref{dmu-t}, we have
	\begin{align}
	A_w(\Omega_t)
	:=& \int_{\partial \Omega_t} \cosh r_t d\mu_t \nonumber\\ 
	=&\int_{\partial \Omega}
	 \(\cosh r e^t - \(\cosh r- \tilde{u}\) \sinh t \)  
	\(\sum_{k=0}^n e^{(n-k)t} \sinh^k t \s_k(\tilde{\kappa}) \)d\mu  \nonumber\\
	=& 
	\sum_{k=0}^n \left( \int_{\partial \Omega} \cosh r \s_k(\tilde{\kappa}) d\mu  \right)
	e^{(n-k+1)t} \sinh^k t  \nonumber\\
	&-
	\sum_{k=0}^n \left( \int_{\partial \Omega} \( \cosh r -\tilde{u}\) \s_k(\tilde{\kappa}) d\mu  \right)
	e^{(n-k) t} \sinh^{k+1} t. \label{A_w(Omega t) - kappa}
	\end{align}
	 Proposition \ref{prop-new-pf-weighted Steiner formula} follows by integrating the both sides of \eqref{A_w(Omega t) - kappa} from $0$ to $\rho$.
\end{proof}

\subsection{Origin symmetric assumption in Conjecture \ref{conj-weighted-h-BM}} \label{subsec-origin-sym cant remove}$ \ $

\begin{lem}\label{lem-weighted volume of sym}
	Let $Y \in \mathbb{H}^{n+1}$ and $K$ be a symmetric domain with respect to $X$ in $\mathbb{H}^{n+1}$. 
	Then the weighted volume of $K$ with respect to $Y= (y, y_{n+1})$ is given by
	\begin{equation}\label{weighted volume of sym}
	\int_{K} \cosh d_{\mathbb{H}^{n+1}}(P, Y) dv = \cosh \( d_{\mathbb{H}^{n+1}}(X,Y) \) \int_K \cosh  d_{\mathbb{H}^{n+1}}(P, X) dv,
	\end{equation}
	where $P \in K$.
\end{lem}

\begin{proof}
	Without loss of generality, we assume that  $X= (0,1)$ is the north pole of the hyperboloid model. Then
	\begin{align}\label{weig vol ball-1}
	\int_K \cosh d_{\mathbb{H}^{n+1}}(P, Y) dv
	=& -\int_K \metric{P}{Y}dv   \nonumber\\
	=&\int_K p_{n+1}y_{n+1}dv - \int_K \metric{p}{y}dv \nonumber\\
	=& y_{n+1} \int_K p_{n+1} dv,
	\end{align}
	where $P =(p, p_{n+1}) \in K$, and the last equality followed from the symmetry of $K$.
	 Then \eqref{weighted volume of sym} follows by substituting $y_{n+1} = \cosh d_{\mathbb{H}^{n+1}} (X,Y)$ and $p_{n+1} = \cosh d_{\mathbb{H}^{n+1}} (X,P)$ into \eqref{weig vol ball-1}. This completes the proof of Lemma \ref{lem-weighted volume of sym}.
\end{proof}

Taking $K=B(X,r)$ and $Y =(0,1)$ in \eqref{weighted volume of sym}, we get the weighted volume of geodesic balls not centered at the origin.
\begin{cor} \label{cor-weighted volume-ball}
	Let  $B(X, r)$ be a geodesic ball of radius $r$ centered at $X=(x,x_{n+1})$ in $\mathbb{H}^{n+1}$. Then
	\begin{equation*}
	\mathcal{S}(B(X,r))= x_{n+1}^{\frac{1}{n+1}} \sinh r.
	\end{equation*}
\end{cor}

Assume that $X = (x,x_{n+1}) \in \mathbb{H}^{n+1}$. Let $r_1= \log \frac{11}{10}$, $r_{2} = \log  \frac{4}{3}$ and $r_{3} = \log \frac{73}{60}$. Let $K = B(X, r_1)$, $L = B(X, r_2)$ and $\Omega= B(X, r_3)$ be geodesic balls centered at $X$. Then $\Omega = \frac{1}{2} \cdot K +_1 \frac{1}{2} \cdot L$. In addition, if $x_{n+1} = 100^{n+1}$, then Corollary \ref{cor-weighted volume-ball} implies 
\begin{align*}
	&\( \mathcal{S}(\Omega) + \sqrt{  \mathcal{S}^2(\Omega)+1 }  \) -
	\frac{1}{2}\( \mathcal{S}(K) + \sqrt{  \mathcal{S}^2(K)+1 }  \)
	-
	\frac{1}{2}\( \mathcal{S}(L) + \sqrt{  \mathcal{S}^2(L)+1 }  \)\\
	=&0.7533\ldots >0.
\end{align*}
If $x_{n+1} = 2^{n+1}$, then Corollary \ref{cor-weighted volume-ball} implies
\begin{align*}
	&\( \mathcal{S}(\Omega) + \sqrt{  \mathcal{S}^2(\Omega)+1 }  \) -
	\frac{1}{2}\( \mathcal{S}(K) + \sqrt{  \mathcal{S}^2(K)+1 }  \)
	-
	\frac{1}{2}\( \mathcal{S}(L) + \sqrt{  \mathcal{S}^2(L)+1 }  \)\\
	=&-0.0516\ldots<0.
\end{align*}
Consequently, the origin symmetric assumption in Conjecture \ref{conj-weighted-h-BM} can not be removed.

\subsection{Another proof of Theorem \ref{thm-sum of balls 0.5--2} for $1 \leq p \leq 2$} $ \ $

In the following Lemma \ref{lem-Concavity of N}, we prove that the Minkowski norm $N$ is concave with respect to future time-like vectors.
\begin{lem}\label{lem-Concavity of N}
	Let $X$ and $Y$ be future time-like vectors in $\mathbb{R}^{n+1,1}$ and $t \in(0,1)$ be a real number. Then
	\begin{equation}\label{concavity of N}
	N\( \(1-t\) X+t Y \) \geq  (1-t) N(X) + t N(Y),
	\end{equation}
	with equality if and only if $X = kY$ for some constant $k>0$. Thus, $N(X)$ is a concave function defined on the future time-like vectors.
\end{lem}
\begin{proof}
	The proof is straightforward. Assume that $X= (x, x_{n+1})$ and $Y=(y,y_{n+1})$ with $x_{n+1} >|x| \geq 0$ and $y_{n+1}>|y| \geq 0$. A direct calculation shows
	\begin{align}
	N( (1-t)X+ tY) =& \(  \( (1-t)x_{n+1}+ t y_{n+1} \)^2-  | (1-t)x+ty|^2    \)^{\frac{1}{2}} \nonumber\\
	\geq& \(  \( (1-t)x_{n+1}+ t y_{n+1} \)^2-  \( (1-t)|x|+t|y|\)^2    \)^{\frac{1}{2}}\nonumber\\
	=& \((1-t)^2\(x_{n+1}^2-|x|^2\)+t^2\(y_{n+1}^2-|y|^2\)+2t\(1-t\)\(x_{n+1}y_{n+1}-|x||y|\)  \)^{\frac{1}{2}} \nonumber\\
	=&\( (1-t)^2N(X)^2+t^2N(Y)^2 +2t (1-t ) \(x_{n+1}y_{n+1}-|x||y|\)\)^{\frac{1}{2}}, \label{N-1}
	\end{align}
	with equality if and only if either $x=ky$ or $y=kx$ for some constant $k\geq0$.
	On the other hand, the AM-GM inequality yields 
	\begin{align}
	x_{n+1}y_{n+1}-|x||y|
	=& \( x_{n+1}^2y_{n+1}^2+|x|^2|y|^2-2 x_{n+1}y_{n+1}|x||y| \)^{\frac{1}{2}} \nonumber\\
	\geq& \(x_{n+1}^2 y_{n+1}^2+|x|^2|y|^2-x_{n+1}^2 |y|^2-y_{n+1}^2|x|^2   \)^{\frac{1}{2}} \nonumber\\
	=&\(x_{n+1}^2-|x|^2\)^{\frac{1}{2}}\(y_{n+1}^2-|y|^2\)^{\frac{1}{2}} \nonumber\\
	=&N(X)N(Y),\label{N-2}
	\end{align}
	with equality if and only if $x_{n+1}|y| = y_{n+1}|x|$. Then \eqref{concavity of N} follows by substituting \eqref{N-2} into \eqref{N-1}. For future time-like vectors $X$ and $Y$, if equality hold in the both \eqref{N-1} and \eqref{N-2}, then $X=kY$ for some constant $k>0$.
	This completes the proof of Lemma \ref{lem-Concavity of N}.  
\end{proof}

By using the above Lemma \ref{lem-Concavity of N},  we can give another proof of Theorem \ref{thm-sum of balls 0.5--2} for $1 \leq p \leq 2$.
\begin{proof}[Another proof of Theorem \ref{thm-sum of balls 0.5--2} for $1 \leq p \leq 2$]
	We first consider the case $p=1$.
	The concavity of the Minkowski norm $N$ in \eqref{concavity of N} implies that
	\begin{equation*}
	N( a e^{r_1} X + b e^{r_2} Y) \geq a N(e^{r_1} X) + b N( e^{r_2} Y)= ae^{r_1} +b e^{r_2}, 
	\end{equation*} 
	with equality if and only if $X=Y$. This together with Corollary \ref{cor-1 sum of balls} gives that $\Omega$ is a geodesic ball of radius greater than $\frac{1}{p}\log \(a e^{p r_1+ b e^{p r_2}}\)$ whenever $X \neq Y$. If $X=Y$, then the radius of $\Omega$ is $\frac{1}{p}\log \(a e^{p r_1+ b e^{p r_2}}\)$. Thus we have that Theorem \ref{thm-sum of balls 0.5--2} holds for $p=1$.
	
	Now we consider the case $1< p \leq 2$.
	Let $q$ satisfy $\frac{1}{p} + \frac{1}{q} =1$.
	Taking $A = a^{\frac{1}{p}} \g_K(z)$, $B = b^{\frac{1}{p}} \g_L(z)$ and 
	\begin{equation*}
	t = \frac{be^{pr_2}}{ae^{pr_1 } +b e^{pr_2}}
	\end{equation*}
	in inequality \eqref{jesen-ineq-p>1}, we have
	\begin{align}
	\g_{\Omega}(z)=& \(a \g_K(z)^p +b \g_L(z)^p\)^{\frac{1}{p}} \nonumber\\
	\geq& \(\frac{a e^{pr_1}}{a e^{pr_1}+b e^{pr_2}} \)^{\frac{1}{q}} a^{\frac{1}{p}} \g_K(z)+
	\(\frac{b e^{pr_2}}{a e^{pr_1+b e^{pr_2}}}  \)^{\frac{1}{q}}b^{\frac{1}{p}} \g_L(z) \nonumber\\
	=&- \metric{  \(\frac{a e^{pr_1}}{a e^{pr_1}+b e^{pr_2}} \)^{\frac{1}{q}}a^{\frac{1}{p}} e^{r_1} X+ 
		\(\frac{b e^{pr_2}}{a e^{pr_1+b e^{pr_2}}}  \)^{\frac{1}{q}} b^{\frac{1}{p}}e^{r_2} Y    }{(z,1)  } \nonumber\\
	=&- \metric{T}{(z,1)}, \label{big ball in sum}
	\end{align}
	where the first equality used the assumption $1<p \leq 2$ and Theorem \ref{thm-def p sum-well defined}, and 
	\begin{equation*}
	T:=  \(\frac{a e^{pr_1}}{a e^{pr_1}+b e^{pr_2}} \)^{\frac{1}{q}} a^{\frac{1}{p}}e^{r_1} X+ 
	\(\frac{b e^{pr_2}}{a e^{pr_1+b e^{pr_2}}}  \)^{\frac{1}{q}} b^{\frac{1}{p}}e^{r_2} Y.
	\end{equation*}
	By the convexity of the domain enclosed by the future light cone of $(0,0)$, it is obvious that $T$ is a future time-like vector.
	By the concavity of the Minkowski norm $N$, we have
	\begin{align*}
	N(T) \geq& \(\frac{a e^{pr_1}}{a e^{pr_1}+b e^{pr_2}} \)^{\frac{1}{q}}a^{\frac{1}{p}} e^{r_1}+\(\frac{b e^{pr_2}}{a e^{pr_1+b e^{pr_2}}}  \)^{\frac{1}{q}}b^{\frac{1}{p}} e^{r_2}\\
	=&    \(a e^{pr_1}+b e^{pr_2}\)^{\frac{1}{p}}\geq  (a+b)^{\frac{1}{p}} \geq 1,
	\end{align*}
	and equality holds in the first inequality if and only if $X=Y$. Let $\widetilde{B} = B(\frac{T}{N(T)}, \log N(T))$.
	Then the above inequality implies that the radius of $\widetilde{B}$ is greater than $\frac{1}{p}\(a e^{pr_1}+b e^{pr_2}\)$ whenever $X \neq Y$. On the other hand, Lemma \ref{lem-horo supp of geodesic ball} implies $\g_{\widetilde{B}} = - \metric{T}{(z,1)}$, and hence \eqref{big ball in sum} deduces $\widetilde{B} \subset \Omega$.
	Consequently, if $p>1$ and $X \neq Y$, then there is a geodesic ball $\widetilde{B}$ of radius greater than $\frac{1}{p}\(a e^{pr_1}+b e^{pr_2}\)$ lying inside $\Omega$. If $1<p \leq 2$ and $X=Y$, then the identity $\widetilde{B} = \Omega$ follows by the same argument as in the previous proof of Theorem \ref{thm-sum of balls 0.5--2}.
	
	Then we complete the proof of Theorem \ref{thm-sum of balls 0.5--2} for $1 \leq  p \leq 2$.
\end{proof}


\begin{bibdiv}
\begin{biblist}
\bibliographystyle{amsplain}


\bib{Alek37}{article}{
	author={Aleksandrov, A. D.},
	title={On the theory of mixed volumes. I. Extension of certain concepts in the theory of convex bodies},
	language={Russian; German summary},
	journal={Mat. Sbornik N.S.},
	volume={2},
	date={1937},
	pages={947--972},
}



\bib{Alek38}{article}{
	author={Aleksandrov, A. D.},
	title={ On the theory of mixed volumes of convex bodies. III. Extension of two theorems of Minkowski on convex polyhedra to arbitrary convex bodies},
	language={Russian},
	note={ English translation in Aleksandrov, A. D., Selected Works, Part 1, Chapter 4, pp. 61–97, Gordon and Breach, Amsterdam, 1996.}
	journal={ Mat. Sb.},
	volume={3},
	date={1938},
	pages={27--46},
}




\bib{And94}{article}{
	author={Andrews, Ben},
	title={Contraction of convex hypersurfaces in Riemannian spaces},
	journal={J. Differential Geom.},
	volume={39},
	date={1994},
	number={2},
	pages={407--431},
}



\bib{And04}{article}{
	title={Fully nonlinear parabolic equations in two space variables},
	author={Andrews, Ben},
	year={2004},
	eprint={arXiv: math/0402235},
	archivePrefix={arXiv},
}


\bib{And07}{article}{
	author={Andrews, Ben},
	title={Pinching estimates and motion of hypersurfaces by curvature
		functions},
	journal={J. Reine Angew. Math.},
	volume={608},
	date={2007},
	pages={17--33},
}


\bib{ACW18}{article}{
	author={Andrews, Ben},
	author={Chen, Xuzhong},
	author={Wei, Yong},
	title={Volume preserving flow and Alexandrov-Fenchel type inequalities in
		hyperbolic space},
	journal={J. Eur. Math. Soc. (JEMS)},
	volume={23},
	date={2021},
	number={7},
	pages={2467--2509},
}

\bib{AHL20}{article}{
	author={Andrews, Ben},
	author={Hu, Yingxiang},
	author={Li, Haizhong},
	title={Harmonic mean curvature flow and geometric inequalities},
	journal={Adv. Math.},
	volume={375},
	date={2020},
	pages={107393, 28},
}


\bib{AMZ13}{article}{
	author={Andrews, Ben},
	author={McCoy, James},
	author={Zheng, Yu},
	title={Contracting convex hypersurfaces by curvature},
	journal={Calc. Var. Partial Differential Equations},
	volume={47},
	date={2013},
	number={3-4},
	pages={611--665},
}


\bib{AK22}{article}{
	title={Horocyclic Brunn-Minkowski inequality},
	author={Assouline, Rotem},
	author={Klartag, Bo'az},
	year={2022},
	eprint={arXiv: 2208.09826},
	archivePrefix={arXiv},
	primaryClass={math.MG}
}

\bib{BM99}{article}{
	author={Borisenko, Alexandr A.},
	author={Miquel, Vicente},
	title={Total curvatures of convex hypersurfaces in hyperbolic space},
	journal={Illinois J. Math.},
	volume={43},
	date={1999},
	number={1},
	pages={61--78},
}

\bib{Bor22}{article}{
	title={The Logarithmic Minkowski conjecture and the $L_p$-Minkowski Problem},
	author={B\"{o}r\"{o}czky, K\'{a}roly J.},
year={2022},
eprint={ arXiv:2210.00194v2},
archivePrefix={arXiv},
primaryClass={math.AP}
}

\bib{BLYZ12}{article}{
	author={B\"{o}r\"{o}czky, K\'{a}roly J.},
	author={Lutwak, Erwin},
	author={Yang, Deane},
	author={Zhang, Gaoyong},
	title={The log-Brunn-Minkowski inequality},
	journal={Adv. Math.},
	volume={231},
	date={2012},
	number={3-4},
	pages={1974--1997},
}


\bib{BHZ16}{article}{
	author={B\"{o}r\"{o}czky, K\'{a}roly J.},
	author={Heged\H{u}s, P\'{a}l},
	author={Zhu, Guangxian},
	title={On the discrete logarithmic Minkowski problem},
	journal={Int. Math. Res. Not. IMRN},
	date={2016},
	number={6},
	pages={1807--1838},
}


\bib{BLYZ13}{article}{
	author={B\"{o}r\"{o}czky, K\'{a}roly J.},
	author={Lutwak, Erwin},
	author={Yang, Deane},
	author={Zhang, Gaoyong},
	title={The logarithmic Minkowski problem},
	journal={J. Amer. Math. Soc.},
	volume={26},
	date={2013},
	number={3},
	pages={831--852},
}

\bib{BCD17}{article}{
	author={Brendle, Simon},
	author={Choi, Kyeongsu},
	author={Daskalopoulos, Panagiota},
	title={Asymptotic behavior of flows by powers of the Gaussian curvature},
	journal={Acta Math.},
	volume={219},
	date={2017},
	number={1},
	pages={1--16},
}


\bib{BHW16}{article}{
	author={Brendle, Simon},
	author={Hung, Pei-Ken},
	author={Wang, Mu-Tao},
	title={A Minkowski inequality for hypersurfaces in the anti--de
		Sitter--Schwarzschild manifold},
	journal={Comm. Pure Appl. Math.},
	volume={69},
	date={2016},
	number={1},
	pages={124--144},
}




\bib{BW14}{article}{
	author={Brendle, Simon},
	author={Wang, Mu-Tao},
	title={A Gibbons-Penrose inequality for surfaces in Schwarzschild
		spacetime},
	journal={Comm. Math. Phys.},
	volume={330},
	date={2014},
	number={1},
	pages={33--43},
	issn={0010-3616},
}


\bib{Bru1887}{article}{
	author={Brunn, Hermann},
	title={ \"{U}ber Ovale und Eifl\"{a}chen},
	journal={Dissertation, M\"{u}nchen},
	date={1887},
}




\bib{BIS20}{article}{
	title={Christoffel-Minkowski flows},
	author={Bryan, Paul},
	author={Ivaki, Mohammad N.},
	author={Scheuer, Julian},
	journal={to appear in Trans. Amer. Math. Soc.},
	eprint={arXiv: 2005.14680},
	archivePrefix={arXiv},
	primaryClass={math.DG}
}


\bib{Bry87}{article}{
	author={Bryant, Robert L.},
	title={Surfaces of mean curvature one in hyperbolic space},
	language={English, with French summary},
	note={Th\'{e}orie des vari\'{e}t\'{e}s minimales et applications (Palaiseau,
		1983--1984)},
	journal={Ast\'{e}risque},
	number={154-155},
	date={1987},
	pages={12, 321--347, 353 (1988)},
}


\bib{CHLL20}{article}{
	author={Chen, Shibing},
	author={Huang, Yong},
	author={Li, Qi-Rui},
	author={Liu, Jiakun},
	title={The $L_ p$-Brunn-Minkowski inequality for $p<1$},
	journal={Adv. Math.},
	volume={368},
	date={2020},
	pages={107166, 21},
}


\bib{CY76}{article}{
	author={Cheng, Shiu Yuen},
	author={Yau, Shing Tung},
	title={On the regularity of the solution of the $n$-dimensional Minkowski
		problem},
	journal={Comm. Pure Appl. Math.},
	volume={29},
	date={1976},
	number={5},
	pages={495--516},
}



\bib{CW06}{article}{
	author={Chou, Kai-Seng},
	author={Wang, Xu-Jia},
	title={The $L_p$-Minkowski problem and the Minkowski problem in
		centroaffine geometry},
	journal={Adv. Math.},
	volume={205},
	date={2006},
	number={1},
	pages={33--83},
}

\bib{dW70}{article}{
	author={do Carmo, M. P.},
	author={Warner, F. W.},
	title={Rigidity and convexity of hypersurfaces in spheres},
	journal={J. Differential Geometry},
	volume={4},
	date={1970},
	pages={133--144},
}

\bib{Eps86}{article}{
	author={Epstein, Charles L.},
	title={The hyperbolic Gauss map and quasiconformal reflections},
	journal={J. Reine Angew. Math.},
	volume={372},
	date={1986},
	pages={96--135},
	issn={0075-4102},
}



\bib{EGM09}{article}{
	author={Espinar, Jos\'{e} M.},
	author={G\'{a}lvez, Jos\'{e} A.},
	author={Mira, Pablo},
	title={Hypersurfaces in $\Bbb H^{n+1}$ and conformally invariant
		equations: the generalized Christoffel and Nirenberg problems},
	journal={J. Eur. Math. Soc. (JEMS)},
	volume={11},
	date={2009},
	number={4},
	pages={903--939},
}


\bib{Fen29}{article}{
	author={Fenchel, Werner},
	title={\"{U}ber Kr\"{u}mmung und Windung geschlossener Raumkurven},
	language={German},
	journal={Math. Ann.},
	volume={101},
	date={1929},
	number={1},
	pages={238--252},
}


\bib{FJ38}{article}{
	author={Fenchel, W.},
	author={Jessen, B.},
	title={Mengenfunktionen und konvexe K\"{o}rper},
	journal={Danske Vid. Selskab. Mat.-fys. Medd.},
	volume={16},
	date={1938},
	pages={1--31},
}

\bib{Fill13}{article}{
	author={Fillastre, Fran\c{c}ois},
	title={Fuchsian convex bodies: basics of Brunn-Minkowski theory},
	journal={Geom. Funct. Anal.},
	volume={23},
	date={2013},
	number={1},
	pages={295--333},
}

\bib{F62}{article}{
	author={Firey, Wm. J.},
	title={$p$-means of convex bodies},
	journal={Math. Scand.},
	volume={10},
	date={1962},
	pages={17--24},
}

\bib{GST13}{article}{
	author={Gallego, E.},
	author={Solanes, G.},
	author={Teufel, E.},
	title={Linear combinations of hypersurfaces in hyperbolic space},
	journal={Monatsh. Math.},
	volume={169},
	date={2013},
	number={3-4},
	pages={329--354},
}


\bib{GLM18}{article}{
	author={Gao, Shanze},
	author={Li, Haizhong},
	author={Ma, Hui},
	title={Uniqueness of closed self-similar solutions to
		$\sigma_k^{\alpha}$-curvature flow},
	journal={NoDEA Nonlinear Differential Equations Appl.},
	volume={25},
	date={2018},
	number={5},
	pages={Paper No. 45, 26},
}


\bib{GLW22}{article}{
	author={Gao, Shanze},
	author={Li, Haizhong},
	author={Wang, Xianfeng},
	title={Self-similar solutions to fully nonlinear curvature flows by high
		powers of curvature},
	journal={J. Reine Angew. Math.},
	volume={783},
	date={2022},
	pages={135--157},
	issn={0075-4102},
}



\bib{Gar02}{article}{
	author={Gardner, R. J.},
	title={The Brunn-Minkowski inequality},
	journal={Bull. Amer. Math. Soc. (N.S.)},
	volume={39},
	date={2002},
	number={3},
	pages={355--405},
}




\bib{Gar06}{book}{
	author={Gardner, R. J.},
	title={Geometric tomography},
	series={Encyclopedia of Mathematics and its Applications},
	volume={58},
	edition={2},
	publisher={Cambridge University Press, New York},
	date={2006},
	pages={xxii+492},
}





\bib{GWW14}{article}{
	author={Ge, Yuxin},
	author={Wang, Guofang},
	author={Wu, Jie},
	title={Hyperbolic Alexandrov-Fenchel quermassintegral inequalities II},
	journal={J. Differential Geom.},
	volume={98},
	date={2014},
	number={2},
	pages={237--260},
}


\bib{Ger06}{book}{
	author={Gerhardt, Claus},
	title={Curvature problems},
	series={Series in Geometry and Topology},
	volume={39},
	publisher={International Press, Somerville, MA},
	date={2006},
	pages={x+323},
}


\bib{GT01}{book}{
	author={Gilbarg, David},
	author={Trudinger, Neil S.},
	title={Elliptic partial differential equations of second order},
	series={Classics in Mathematics},
	note={Reprint of the 1998 edition},
	publisher={Springer-Verlag, Berlin},
	date={2001},
	pages={xiv+517},
}

\bib{GL15}{article}{
	author={Guan, Pengfei},
	author={Li, Junfang},
	title={A mean curvature type flow in space forms},
	journal={Int. Math. Res. Not. IMRN},
	date={2015},
	number={13},
	pages={4716--4740},
}

\bib{GM03}{article}{
	author={Guan, Pengfei},
	author={Ma, Xi-Nan},
	title={The Christoffel-Minkowski problem. I. Convexity of solutions of a
		Hessian equation},
	journal={Invent. Math.},
	volume={151},
	date={2003},
	number={3},
	pages={553--577},
}



\bib{GRW15}{article}{
	author={Guan, Pengfei},
	author={Ren, Changyu},
	author={Wang, Zhizhang},
	title={Global $C^2$-estimates for convex solutions of curvature
		equations},
	journal={Comm. Pure Appl. Math.},
	volume={68},
	date={2015},
	number={8},
	pages={1287--1325},
	issn={0010-3640},
}

\bib{GX18}{article}{
	author={Guan, Pengfei},
	author={Xia, Chao},
	title={$L^p$ Christoffel-Minkowski problem: the case $1<p<k+1$},
	journal={Calc. Var. Partial Differential Equations},
	volume={57},
	date={2018},
	number={2},
	pages={Paper No. 69, 23},
}

\bib{GL11}{article}{
	author={Guo, Bin},
	author={Li, Haizhong},
	title={The second variational formula for the functional $\int
		v^{(6)}(g)dV_g$},
	journal={Proc. Amer. Math. Soc.},
	volume={139},
	date={2011},
	number={8},
	pages={2911--2925},
}



\bib{Han06}{article}{
	author={Han, Zheng-Chao},
	title={A Kazdan-Warner type identity for the $\sigma_k$ curvature},
	journal={C. R. Math. Acad. Sci. Paris},
	volume={342},
	date={2006},
	number={7},
	pages={475--478},
}

\bib{HL19}{article}{
	author={Hu, Yingxiang},
	author={Li, Haizhong},
	title={Geometric inequalities for hypersurfaces with nonnegative
		sectional curvature in hyperbolic space},
	journal={Calc. Var. Partial Differential Equations},
	volume={58},
	date={2019},
	number={2},
	pages={Paper No. 55, 20},
}


\bib{HL21}{article}{
	author={Hu, Yingxiang},
	author={Li, Haizhong},
	title={Geometric inequalities for static convex domains in hyperbolic space},
	journal={Trans. Amer. Math. Soc.},
	volume={375},
	date={2022},
	number={8},
	pages={5587--5615},
}


\bib{HLW20}{article}{
	author={Hu, Yingxiang},
	author={Li, Haizhong},
	author={Wei, Yong},
	title={Locally constrained curvature flows and geometric inequalities in
		hyperbolic space},
	journal={Math. Ann.},
	volume={382},
	date={2022},
	number={3-4},
	pages={1425--1474},
}



\bib{HLYZ16}{article}{
	author={Huang, Yong},
	author={Lutwak, Erwin},
	author={Yang, Deane},
	author={Zhang, Gaoyong},
	title={Geometric measures in the dual Brunn-Minkowski theory and their
		associated Minkowski problems},
	journal={Acta Math.},
	volume={216},
	date={2016},
	number={2},
	pages={325--388},
}

\bib{HLYZ05}{article}{
	author={Hug, Daniel},
	author={Lutwak, Erwin},
	author={Yang, Deane},
	author={Zhang, Gaoyong},
	title={On the $L_p$ Minkowski problem for polytopes},
	journal={Discrete Comput. Geom.},
	volume={33},
	date={2005},
	number={4},
	pages={699--715},
}




\bib{KM20}{article}{
	author={Kolesnikov, Alexander},
	author={Milman, Emanuel},
	title={Local $L^p$-Brunn-Minkowski inequalities for $p<1$},
	journal={Mem. Amer. Math. Soc.},
	volume={277},
	date={2022},
	number={1360},
	pages={v+78},
}



\bib{Lew38}{article}{
	author={Lewy, Hans},
	title={On differential geometry in the large. I. Minkowski's problem},
	journal={Trans. Amer. Math. Soc.},
	volume={43},
	date={1938},
	number={2},
	pages={258--270},
}


\bib{LWX}{article}{
	title={The discrete horospherical $p$-Minkowski problem in hyperbolic space, preprint}, 
	author={Li, Haizhong},
	author={Wan, Yao},
	author={Xu, Botong},
	year={2022},
}



\bib{LWX14}{article}{
	author={Li, Haizhong},
	author={Wei, Yong},
	author={Xiong, Changwei},
	title={A geometric inequality on hypersurface in hyperbolic space},
	journal={Adv. Math.},
	volume={253},
	date={2014},
	pages={152--162},
}

\bib{LX}{article}{
	author={Li, Haizhong},
	author={Xu, Botong},
	title={A class of weighted isoperimetric inequalities in hyperbolic space},
	journal={to appear in Proc. Amer. Math. Soc.}
	doi={10.1090/proc/16219}
	year={2022},
	eprint={ arXiv:2210.12632},
}




\bib{LSW20}{article}{
	author={Li, Qi-Rui},
	author={Sheng, Weimin},
	author={Wang, Xu-Jia},
	title={Flow by Gauss curvature to the Aleksandrov and dual Minkowski
		problems},
	journal={J. Eur. Math. Soc. (JEMS)},
	volume={22},
	date={2020},
	number={3},
	pages={893--923},
}




\bib{LLL21}{article}{
	author={Li, YanYan},
	author={Lu, Han},
	author={Lu, Siyuan},
	title={On the $\sigma_2$-Nirenberg problem on $\Bbb S^2$},
	journal={J. Funct. Anal.},
	volume={283},
	date={2022},
	number={10},
	pages={Paper No. 109606},
}




\bib{Lut93}{article}{
	author={Lutwak, Erwin},
	title={The Brunn-Minkowski-Firey theory. I. Mixed volumes and the
		Minkowski problem},
	journal={J. Differential Geom.},
	volume={38},
	date={1993},
	number={1},
	pages={131--150},
}

\bib{Lut96}{article}{
	author={Lutwak, Erwin},
	title={The Brunn-Minkowski-Firey theory. II. Affine and geominimal
		surface areas},
	journal={Adv. Math.},
	volume={118},
	date={1996},
	number={2},
	pages={244--294},
}

\bib{LYZ12}{article}{
	author={Lutwak, Erwin},
	author={Yang, Deane},
	author={Zhang, Gaoyong},
	title={The Brunn-Minkowski-Firey inequality for nonconvex sets},
	journal={Adv. in Appl. Math.},
	volume={48},
	date={2012},
	number={2},
	pages={407--413},
}


\bib{LYZ18}{article}{
	author={Lutwak, Erwin},
	author={Yang, Deane},
	author={Zhang, Gaoyong},
	title={$L_p$ dual curvature measures},
	journal={Adv. Math.},
	volume={329},
	date={2018},
	pages={85--132},
}


\bib{MS12}{article}{
	author={Mars, Marc},
	author={Soria, Alberto},
	title={On the Penrose inequality for dust null shells in the Minkowski
		spacetime of arbitrary dimension},
	journal={Classical Quantum Gravity},
	volume={29},
	date={2012},
	number={13},
	pages={135005, 25},
	issn={0264-9381},
}

\bib{MS14}{article}{
	author={Mars, Marc},
	author={Soria, Alberto},
	title={Geometry of normal graphs in Euclidean space and applications to
		the Penrose inequality in Minkowski},
	journal={Ann. Henri Poincar\'{e}},
	volume={15},
	date={2014},
	number={10},
	pages={1903--1918},
}

\bib{MB16}{article}{
	title={An Introduction to Geometric Topology}, 
	author={Martelli, Bruno},
	year={2016},
	eprint={arXiv: 1610.02592},
	archivePrefix={arXiv},
	primaryClass={math.GT}
}


\bib{Min1897}{article}{
	title={Allgemeine Lehrs{\"a}tze {\"u}ber die konvexen Polyeder}, 
	author={Minkowski, Hermann},
	journal={Nachr. Ges. Wiss. G{\"o}ttingen},
	date={1897},
	pages={198--219},
}

\bib{Min1903}{article}{
	author={Minkowski, Hermann},
	title={Volumen und Oberfl\"{a}che},
	journal={Math. Ann.},
	volume={57},
	date={1903},
	number={4},
	pages={447--495},
}


\bib{Nat15}{article}{
	author={Nat\'{a}rio, Jos\'{e}},
	title={A Minkowski-type inequality for convex surfaces in the hyperbolic
		3-space},
	journal={Differential Geom. Appl.},
	volume={41},
	date={2015},
	pages={102--109},
}

\bib{Nir53}{article}{
	author={Nirenberg, Louis},
	title={The Weyl and Minkowski problems in differential geometry in the
		large},
	journal={Comm. Pure Appl. Math.},
	volume={6},
	date={1953},
	pages={337--394},
}



\bib{Oss78}{article}{
	author={Osserman, Robert},
	title={The isoperimetric inequality},
	journal={Bull. Amer. Math. Soc.},
	volume={84},
	date={1978},
	number={6},
	pages={1182--1238},
}



\bib{Pog78}{book}{
	author={Pogorelov, A.V.},
	title={The Minkowski multidimensional problem},
	series={Scripta Series in Mathematics},
	note={Translated from the Russian by Vladimir Oliker;
		Introduction by Louis Nirenberg},
	publisher={V. H. Winston \& Sons, Washington, D.C.; Halsted Press [John
		Wiley \& Sons], New York-Toronto-London},
	date={1978},
	pages={106},
}


\bib{Rei73}{article}{
	author={Reilly, Robert C.},
	title={On the Hessian of a function and the curvatures of its graph},
	journal={Michigan Math. J.},
	volume={20},
	date={1973},
	pages={373--383},
}

\bib{San04}{book}{
	author={Santal\'{o}, Luis A.},
	title={Integral geometry and geometric probability},
	series={Cambridge Mathematical Library},
	edition={2},
	note={With a foreword by Mark Kac},
	publisher={Cambridge University Press, Cambridge},
	date={2004},
	pages={xx+404},
}

\bib{SX19}{article}{
	author={Scheuer, Julian},
	author={Xia, Chao},
	title={Locally constrained inverse curvature flows},
	journal={Trans. Amer. Math. Soc.},
	volume={372},
	date={2019},
	number={10},
	pages={6771--6803},
}

\bib{Sch14}{book}{
	author={Schneider, Rolf},
	title={Convex bodies: the Brunn-Minkowski theory},
	series={Encyclopedia of Mathematics and its Applications},
	volume={151},
	edition={Second expanded edition},
	publisher={Cambridge University Press, Cambridge},
	date={2014},
	pages={xxii+736},
}


\bib{TW13}{article}{
	author={Tian, Guji},
	author={Wang, Xu-Jia},
	title={A priori estimates for fully nonlinear parabolic equations},
	journal={Int. Math. Res. Not. IMRN},
	date={2013},
	number={17},
	pages={3857--3877},
}

\bib{Via00}{article}{
	author={Viaclovsky, Jeff A.},
	title={Conformal geometry, contact geometry, and the calculus of
		variations},
	journal={Duke Math. J.},
	volume={101},
	date={2000},
	number={2},
	pages={283--316},
}

\bib{WX14}{article}{
	author={Wang, Guofang},
	author={Xia, Chao},
	title={Isoperimetric type problems and Alexandrov-Fenchel type
		inequalities in the hyperbolic space},
	journal={Adv. Math.},
	volume={259},
	date={2014},
	pages={532--556},
}

\bib{WX15}{article}{
	author={Wei, Yong},
	author={Xiong, Changwei},
	title={Inequalities of Alexandrov-Fenchel type for convex hypersurfaces
		in hyperbolic space and in the sphere},
	journal={Pacific J. Math.},
	volume={277},
	date={2015},
	number={1},
	pages={219--239},
}


\bib{Xia16}{article}{
	author={Xia, Chao},
	title={A Minkowski type inequality in space forms},
	journal={Calc. Var. Partial Differential Equations},
	volume={55},
	date={2016},
	number={4},
	pages={Art. 96, 8},
}




\bib{Zhu15}{article}{
	author={Zhu, Guangxian},
	title={The $L_p$ Minkowski problem for polytopes for $0<p<1$},
	journal={J. Funct. Anal.},
	volume={269},
	date={2015},
	number={4},
	pages={1070--1094},
}


\bib{Zhu17}{article}{
	author={Zhu, Guangxian},
	title={The $L_p$ Minkowski problem for polytopes for $p<0$},
	journal={Indiana Univ. Math. J.},
	volume={66},
	date={2017},
	number={4},
	pages={1333--1350},
}



\end{biblist}
\end{bibdiv}
\end{document}